\documentclass[11pt,reqno]{amsart}
\usepackage[T1]{fontenc}
\usepackage{amsmath,amssymb,amsthm,mathtools,mathrsfs}
\usepackage[margin=1.1in]{geometry}
\usepackage{booktabs,array,longtable,enumitem,microtype,needspace}
\usepackage[all]{xy}
\usepackage{symmetric_branching}
\usepackage[hidelinks]{hyperref}
\hypersetup{colorlinks,	linkcolor={red!50!black},	citecolor={blue!50!black},	urlcolor={blue!80!black}}
\allowdisplaybreaks

\numberwithin{equation}{section}
\newtheorem{theorem}{Theorem}[section]
\newtheorem{proposition}[theorem]{Proposition}
\newtheorem{lemma}[theorem]{Lemma}
\newtheorem{corollary}[theorem]{Corollary}

\theoremstyle{definition}
\newtheorem{definition}[theorem]{Definition}
\newtheorem{example}[theorem]{Example}
\theoremstyle{remark}
\newtheorem{remark}[theorem]{Remark}

\DeclareMathOperator{\SL}{SL}
\DeclareMathOperator{\GL}{GL}
\DeclareMathOperator{\Sp}{Sp}
\DeclareMathOperator{\Spin}{Spin}
\DeclareMathOperator{\Spec}{Spec}
\DeclareMathOperator{\Frac}{Frac}
\DeclareMathOperator{\Hom}{Hom}
\DeclareMathOperator{\ord}{ord}

\DeclareMathOperator{\wt}{wt}
\DeclareMathOperator{\Ad}{Ad}
\DeclareMathOperator{\Sk}{Sk}
\DeclareMathOperator{\proj}{proj}
\DeclareMathOperator{\ind}{ind}
\DeclareMathOperator{\Ch}{Ch}
\DeclareMathOperator{\Pf}{Pf}
\DeclareMathOperator{\bwt}{bwt}
\DeclareMathOperator{\add}{add}
\DeclareMathOperator{\rad}{rad}
\DeclareMathOperator{\Irr}{Irr}
\DeclareMathOperator{\coker}{coker}
\DeclareMathOperator{\diag}{diag}

\newcommand{\op}[1]{\operatorname{#1}}
\newcommand{\C}{\mathbb C}
\newcommand{\Q}{\mathbb Q}
\newcommand{\Z}{\mathbb Z}
\newcommand{\N}{\mathbb N}
\newcommand{\R}{\mathbb R}
\newcommand{\B}{\mathcal B}
\newcommand{\UB}{\operatorname{UB}}
\newcommand{\ULP}{\mathcal U_{\mathrm{LP}}}
\newcommand{\cF}{\mathcal F}

\newcommand{\Id}{\mathrm{Id}}

\newcommand{\OO}{\mathcal O}
\newcommand{\rank}{\operatorname{rank}}
\newcommand{\supp}{\operatorname{supp}}
\newcommand{\lead}{\operatorname{lead}}
\newcommand{\CT}{\operatorname{CT}}
\newcommand{\Tr}{\operatorname{Tr}}
\newcommand{\trop}{\operatorname{trop}}

\newcommand{\pair}[2]{\langle #1,#2\rangle}
\newcommand{\Newt}{\operatorname{Newt}}
\newcommand{\Ageom}{\mathcal A_{\mathrm{geom}}}
\newcommand{\T}{\mathsf t}

\title[Symmetric branching via LP algebras]{Symmetric Branching via Laurent Phenomenon Algebras}
\author{Jiarui Fei}
\address{School of Mathematical Sciences, Shanghai Jiao Tong University, Shanghai, China}
\email{jiarui@sjtu.edu.cn}
\date{}
\thanks{The author was supported in part by the National Natural Science Foundation of China (Nos.~12131015 and 12571038).}
\hypersetup{pdftitle={Symmetric Branching via Laurent Phenomenon Algebras},pdfsubject={Laurent phenomenon algebras, mirror symmetry, and branching bases}}

\begin{document}
\begin{abstract}
We construct Laurent phenomenon structures on reductive branching algebras.
For $\SL_{2n}\downarrow\Sp_{2n}$, $n\ge2$, and the exceptional inclusions
$\Spin_8\downarrow G_2$, $E_6\downarrow F_4$, and
$F_4\downarrow\Spin_9$, we identify the branching algebras with upper
Laurent phenomenon algebras, keeping the frozen coefficients polynomial.
Degree-fibred categories of projective presentations give a common construction.
It also gives cluster seeds of type $A_1$ for
$\Spin_7\downarrow G_2$ and $G_2\downarrow A_2$.
We give uniform sufficient conditions identifying a
specialized Keel--Yu mirror algebra with an upper Laurent phenomenon
algebra admitting a suitable binomial seed. The resulting theta
basis is simultaneously adapted to the frozen boundary valuations.
For $\SL_{2n}\downarrow\Sp_{2n}$ with $2\le n\le5$ and for the other
exceptional inclusions, we obtain homogeneous theta bases parametrized
by lattice points of explicit rational polyhedral cones, whose weight
fibres compute every branching multiplicity. We also prove mutation
invariance of finite upper LP bounds over UFD.
\end{abstract}
\maketitle
\begingroup
\linespread{0.95}\selectfont
\tableofcontents
\endgroup

\clearpage
\section*{Introduction}
\label{sec:introduction}
For connected complex reductive groups $K\subset G$, the branching
algebra
\[ \B(G,K)=\C[G]^{U_G^-\times U_K}\]
has homogeneous components
\[ \B(G,K)_{\lambda,\mu} \simeq\Hom_K\bigl(V_K(\mu),V_G(\lambda)\bigr).\]
We write $G\downarrow K$ for restriction of representations from $G$
to $K$. Here the maximal unipotent subgroups are compatible, and the torus
action is fixed by \eqref{eq:torus-convention}. In the cases specified
below, we construct homogeneous theta bases parametrized by lattice points
in rational polyhedral cones. The weight fibres of each cone give
formulas for all restriction multiplicities.

We use Laurent phenomenon algebras (LPAs), introduced by Lam and
Pylyavskyy \cite{LamPylyavskyy}. Like cluster algebras, they have the
Laurent phenomenon, but their exchange data need not be governed by a
skew-symmetrizable matrix. A square matrix $C$ is skew-symmetrizable
if $DC$ is skew-symmetric for some positive diagonal matrix $D$.
The natural seeds for $A_{2n-1}\downarrow C_n$, $n\ge4$, and for
$D_4\downarrow G_2$, $E_6\downarrow F_4$, and $F_4\downarrow B_4$
have binomial exchanges, yet their principal matrices cannot be made
skew-symmetrizable. These seeds therefore lie outside the cluster
algebra setting. We identify their upper LP algebras with branching
algebras. For suitable binomial seeds, a comparison with Keel--Yu's
mirror algebra gives theta bases compatible with the frozen boundary
divisors. Regularity along these divisors gives the inequalities of
the polyhedral models.

Our examples cover the non-diagonal branching pairs listed in
\cite[Theorem~5]{PasquierRessayre}:
\[
 A_{2n-1}\downarrow C_n,\quad D_n\downarrow B_{n-1},\quad
 B_3\downarrow G_2,\quad G_2\downarrow A_2,\quad
 E_6\downarrow F_4,\quad F_4\downarrow B_4.
\]
We also treat $D_4\downarrow G_2$, arising from triality
\cite{Triality}. The $D_n\downarrow B_{n-1}$ branching algebras
are polynomial, while the two cases involving $G_2$ in the displayed
list are hypersurfaces with cluster seeds of type $A_1$.
They illustrate the same categorical rule in its simplest nontrivial
form; the triality seed exhibits non-skew-symmetrizability.
We use simply connected groups and work with the integral grading
lattices. The $\SL_{2n}\downarrow\Sp_{2n}$ upper-algebra realization
holds for every $n\ge2$, and its theta-basis construction is
established here for $2\le n\le5$. The geometric comparison also
applies to $D_4\downarrow G_2$, $E_6\downarrow F_4$, and
$F_4\downarrow B_4$; the two hypersurface algebra identifications
are proved in Section~\ref{subsec:hypersurface-cases}, and their cluster
theta bases and closed branching formulas are given in
Section~\ref{subsec:hypersurface-models}.

\subsection*{Categorical seeds beyond skew-symmetrizability}
The initial LPA seeds admit a common construction from categories of
projective presentations.  For $A_{2n-1}\downarrow C_n$ with $n\ge4$
we use the presentation category of type $A_{2n-1}$. The other
constructions use types $D_4,E_6$ and hereditary species of valued
types $F_4,B_3,G_2$. In each case
an additive branching degree partitions the indecomposable
presentations into full degree fibres, and the branching subcategory
is the full additive subcategory generated by the relevant fibres.
Admissibility is determined intrinsically from the complete solid and
translation neighbourhood of a presentation and the degrees of the
two resulting products.  The mutable fibres are precisely those
containing admissible presentations, and all admissible presentations
in the same fibre give the same unordered exchange polynomial.
Thus the exchange polynomial is attached to the fibre itself; a
choice of mark serves only to order its two monomials, and hence to
fix the sign of the corresponding row of the exchange matrix.

Theorems~\ref{prop:categorical-B} and~\ref{prop:exceptional-fibre-B}
identify these products with the branching exchanges in the
non-skew-symmetrizable cases, including the frozen factors; the two
hypersurface constructions are given at the end of Section~\ref{sec:exceptional-categories}.
In matrix form the construction is
\[ B=D_\chi R E P.\]
Here $E$ is the signed categorical incidence matrix, with species
valuations in types $F_4,B_3,G_2$; $R$ selects the marked presentations; $P$
collects whole degree fibres; and $D_\chi$ records the orders of the
two exchange terms. Thus the source and target operations are different.
Objects of the same degree are not identified, and replacing marked
sources by full fibres would change the exchange monomials. This is
not ordinary orbit folding.

This asymmetry explains why the natural seeds need not be cluster
seeds, even though their initial exchanges are binomial. We prove that
no choice of exchange-term orders makes the principal matrix
skew-symmetrizable for the $A_{2n-1}\downarrow C_n$ seeds with $n\ge4$ or for any of
the three exceptional seeds
(Corollaries~\ref{cor:symplectic-noncluster} and~\ref{cor:noncluster-seeds}).
The obstruction concerns these seeds, not every possible cluster
structure on the underlying algebra. Indeed, Yan's thesis
\cite{YanBranching} records a different cluster structure on
$\B(\Spin_8,G_2)$. 
The interpretation concerns the initial seeds, not every LP mutation.

\subsection*{Upper LP realizations and finite upper bounds}
The categorical exchange data are realized by normalized highest-weight
functions. For $\SL_{2n}\downarrow\Sp_{2n}$, the folded-minor seeds $\Sigma_n$
satisfy
\[ \B(\SL_{2n},\Sp_{2n})
   =\UB_{\C[\mathbf f]}(\Sigma_n)
   =\ULP(\Sigma_n),\qquad n\ge2,
\]
with the separate seed for $n=3$ specified in
Theorem~\ref{thm:a2c-algebra-equality}.
Corollary~\ref{thm:d4-upper-equality} and
Theorems~\ref{thm:e6-upper-equality} and~\ref{thm:f4-explicit-seed}
give the corresponding equalities for
$D_4\downarrow G_2$, $E_6\downarrow F_4$, and $F_4\downarrow B_4$.
Propositions~\ref{prop:b3-g2-hypersurface} and~\ref{prop:g2-a2-hypersurface}
give the two hypersurface cases. In every case the frozen variables
are polynomial coefficients. 

To prove these equalities, we show that every element of the upper LP
algebra is regular at the generic point of every prime divisor of
$\Spec\B(G,K)$. For
$\SL_{2n}\downarrow\Sp_{2n}$, a global Pfaffian frame gives the
normalized exchange identities, and Pfaffian--Lagrangian coordinates
give birational charts. Primality of the seed functions, the adjacent
charts, and a weighted Jacobian identity deduced from the frozen-weight
sum prove regularity in codimension one. Normality then gives the
algebra equality. The proofs for the exceptional inclusions use the
same codimension-one comparison after constructing and normalizing the
highest-weight functions. None of these algebraic realizations depends
on the mirror comparison.

The passage from finitely many charts to the entire LP pattern follows
from a result valid independently of branching. For a seed $t$, write
$L(t)$ for its Laurent ring, $\UB(t)$ for the intersection of $L(t)$
and its adjacent Laurent rings, and $\ULP(t)$ for the intersection over
all seeds. Theorem~\ref{thm:upper-invariance} proves that, for every LP
pattern over a unique factorization domain (UFD),
\[  \UB(t)=\UB(\mu_k t)=\ULP(t). \]
Thus the initial chart and its adjacent charts determine the upper LP
algebra. Du--Li prove upper-bound invariance under an additional
exchange-normalization hypothesis \cite[Theorem~4.9 and Corollary~4.14]{DuLi}; our theorem
requires no such hypothesis. In particular, it does not require the
ordinary exchange polynomials to agree with their associated exchange
Laurent polynomials throughout the mutation pattern, and the
coefficient ring need not be localized.

These results extend the use of Laurent phenomenon structures in
branching beyond the cluster constructions of Fei \cite{FeiTensor}
and Francone \cite{FranconeBranching}. Francone treats Levi restriction
and diagonal embeddings by minimal monomial lifting, also with
noninvertible frozen variables. In none of our inclusions is the smaller
group a Levi subgroup or the derived subgroup of one
\cite[Proposition~21.90]{MilneAlgebraicGroups}.
LPAs already occur in representation-theoretic geometry: Daisey--Ducat \cite{DaiseyDucat}
construct one on the homogeneous coordinate ring of the Cayley plane.
To our knowledge, the present work is the first systematic application
of LPAs to branching algebras and their restriction multiplicities.

\subsection*{Mirror algebras and boundary-compatible bases}
Keel--Yu's Frobenius structure theorem \cite{KeelYu} constructs a
mirror algebra with a distinguished theta basis for an affine log
Calabi--Yau variety containing a torus. We give sufficient conditions
for its specialization to be an upper LP algebra and determine which
theta functions extend across the frozen boundary. The resulting basis
is simultaneously adapted to all frozen boundary valuations.

Let $R_0$ be an
upper LP algebra with polynomial frozen variables $f_1,\ldots,f_s$,
and put $R=R_0[(f_1\cdots f_s)^{-1}]$. Starting from a binomial seed, we form a toric blowup model from the
transposed seed data. Suppose it embeds in a smooth affine log
Calabi--Yau variety $V$ as specified in Section~\ref{sec:binomial}.
Theorem~\ref{thm:comparison},
Proposition~\ref{prop:binomial}, and Corollary~\ref{cor:application}
give, under their geometric, finiteness, and boundary hypotheses,
\[ \Ageom(V)_1\xrightarrow{\ \sim\ }R. \]
The subscript $1$ denotes specialization of every curve-class
coefficient to $1$. The hypotheses include a saturated exchange
lattice and copositivity of the principal matrix,
$a^{\mathsf t}B_0a\ge0$ for $a\in\R_{\ge0}^r$, rather than
skew-symmetrizability. They are conditions on a chosen binomial seed
and its geometric model; later LP exchanges need not remain binomial.

The toric-boundary construction follows Gross--Hacking--Keel
\cite[Construction~3.4]{GHKBirational}. Regular functions on $V$ bound
the theta expansions, while Laurent charts of the original LP pattern
give the reverse inequalities for boundary orders. One finiteness
condition makes each theta expansion a Laurent polynomial; another
bounds the possible leading exponents at fixed degree and fixed
boundary orders. Finite leading-term subtraction then proves the
isomorphism and preserves every boundary condition. An alternative
comparison uses Newton polytopes (Proposition~\ref{thm:newton}). 

More precisely, let $\nu_f$ be the orders along the frozen divisors,
and let $\phi_f$ be the tropicalizations of the corresponding regular
boundary functions on $V$. For arbitrary integers $d_f$, the comparison
gives
\[ \{g\in R:\nu_f(g)\ge d_f\text{ for all }f\}
   =\bigoplus_{\substack{m\in L\\\phi_f(m)\ge d_f\ \forall f}}
        \C\Theta_m,
\]
where $L$ is the lattice of theta indices in the chosen torus
coordinates. In particular,
\[ \nu_f\!\left(\sum_m c_m\Theta_m\right)
       =\min_{c_m\ne0}\phi_f(m)
\]
for every finite nonzero sum. Setting all $d_f=0$ selects a basis of
$R_0$; positive bounds impose vanishing orders along the boundary.
Thus every simultaneous boundary-vanishing condition selects a subset
of the same basis.

Mirror constructions for LPAs already appear in Ducat's work on the
three-dimensional Lyness map \cite{DucatLyness}. Our criterion treats a class of seeds, identifies their upper LP
algebras with specialized mirror algebras, and determines the bases
selected by the boundary divisors. To our knowledge, this is the
first uniform construction of theta bases of a class of upper LP
algebras from Keel--Yu's theory without requiring skew-symmetrizable
seed matrices. For $\SL_6\downarrow\Sp_6$, the seed is of cluster type
$A_3$, so the cluster theta construction \cite{GHKK} gives the boundary
basis directly (Corollary~\ref{thm:a5c3-canonical-model}). Related logarithmic,
non-archimedean, and scattering constructions are developed in
\cite{ArguzGross,Johnston,KeelWhiteYu,CMMM}; our comparison uses the
specified variety, coefficient specialization, and torus coordinates
of Section~\ref{sec:marking}.

\begin{table}[!t]
\centering\small
\renewcommand{\arraystretch}{1.16}
\begin{tabular}{@{}p{.29\linewidth}p{.67\linewidth}@{}}
\toprule
Notation & Meaning\\
\midrule
$\B(G,K),\ \B_n$ & Branching algebra for $K\subset G$; $\B_n=\B(\SL_{2n},\Sp_{2n})$\\
$\UB(\Sigma),\ \ULP(\Sigma)$ & Finite upper bound and upper Laurent phenomenon algebra\\
$(\lambda;\mu),\ \wt$ & Source and target highest weights; the branching bidegree\\
$\omega_i,\ \varpi_j$ & Source and target fundamental weights\\
$B=(B_0\mid B_f),\ B^\T$ & Oriented seed matrix and its transpose\\
$B_p$ & Chosen unimodular completion of the rows of $B$\\
$W,\ H$ & Branching-weight matrix and matrix of boundary inequalities\\
$\Ageom(V)_\xi$ & Mirror algebra of $V$ after coefficient specialization $\xi$\\
$\lead g$ & Least exponent in the order fixed in Section~\ref{sec:finite-control}\\
$\vartheta_m,\ \Theta_m$ & Theta function in the mirror algebra and its Laurent expansion\\
$\nu_f,\ \Psi_f,\ \phi_f$ & Boundary valuation, regular boundary function, and its tropicalization\\
$\psi_\ell,\ J_{\rm bnd},\ \zeta$ & Regular functions, their tropical gradient matrix (Section~\ref{sec:binomial}), and the positive vector in \eqref{eq:balance}\\
$\mathcal S_n,\ \deg_\circ$ & Section of $U_G\to U_G/U_K$ and positive integer grading\\
$C_Q,\ \Phi_Q$ & Projective Cartan matrix and Coxeter transformation\\
$\mathscr C_Q$ & Category $\Ch_2(\proj\C Q)$ of projective presentations\\
$\mathscr P_{Q_F}$ & Category $\Ch_2(\proj\mathbb S_{Q_F})$ for the valued-$F_4$ species\\
$X^{\mathrm l}_{a,\ell},\ X^{\mathrm r}_{a,\ell}$ & Presentations in the degree fibre of $F_{a,\ell}$ for $A_{2n-1}\downarrow C_n$\\
\bottomrule
\end{tabular}
\caption{Notation used throughout the paper.}
\label{tab:notation}
\end{table}

\subsection*{Polyhedral branching rules}
We apply the geometric criterion to the five seeds for
\[ A_7\downarrow C_4,\qquad A_9\downarrow C_5,\qquad D_4\downarrow G_2,
 \qquad E_6\downarrow F_4,\qquad F_4\downarrow B_4.
\]
Theorem~\ref{thm:geometric-models} verifies the required affine models,
regular functions, and boundary inequalities. Together with the
polynomial case $n=2$ and the finite-type cluster case $n=3$
(Corollary~\ref{thm:a5c3-canonical-model}), this gives
branching theta bases for $A_{2n-1}\downarrow C_n$ with $2\le n\le5$ and
for all three exceptional inclusions.

For each of these applications there is a fixed rational polyhedral
cone
\[ \mathcal C_H=\{m\in\R^{r+s}:Hm\ge0\}\]
and a homogeneous basis
\[ \B(G,K)=\bigoplus_{m\in\mathcal C_H\cap\Z^{r+s}}\C\Theta_m,
 \qquad \wt(\Theta_m)=W^{\mathsf t}m.
\]
Here $r$ and $s$ are the numbers of mutable and frozen variables,
$W$ records their branching weights, and $H$ is constructed from the
supports of the regular boundary functions. The inequalities express
regularity along the frozen divisors. The homogeneous basis therefore gives
\[ \dim\Hom_K\bigl(V_K(\mu),V_G(\lambda)\bigr)
   =\#\bigl\{m\in\Z^{r+s}:Hm\ge0,\;
                   W^{\mathsf t}m=(\lambda;\mu)\bigr\}.
\]
The real weight fibres are bounded, and the grading lattice is
$\op{im}W^{\mathsf t}$. The same cone parametrizes a basis
of the entire branching algebra and computes every restriction
multiplicity. Its inequalities come from the LP geometry; the
branching weights enter through $W$.

There are earlier polyhedral descriptions of non-diagonal symmetric
branching. Schumann--Torres \cite[Sections~6--7]{SchumannTorres} give
polytopes for symplectic branching in the stable range, and
Kumar--Torres \cite[Remark~4.8]{KumarTorres} give a model using disjoint
unions of flagged hive polytopes. For $F_4\downarrow B_4$, Gornitskii
\cite{GornitskiiBranching} constructs an explicit branching semigroup.
Our models are obtained instead from the boundary valuations of a
theta basis. Their common feature is a single cone of theta
indices for each branching algebra, with all multiplicity formulas
arising from its weight fibres. In particular, the polyhedral parametrization describes both the
homogeneous theta basis and its simultaneous compatibility with all frozen
boundary valuations.

For $A_{2n-1}\downarrow C_n$, the upper-LP realization holds for every
$n\ge2$, while the geometric comparison and theta-basis construction
are proved here for $2\le n\le5$; see Remark~\ref{rem:higher-rank-scope}.
The geometric seeds used here are obtained from the initial categorical
seeds by the LP mutations and exchange-term orders specified in
Section~\ref{sec:applications}.

\subsection*{Organization}
Part~I constructs the $A_{2n-1}\downarrow C_n$ seeds, proves their upper-algebra
realization, and establishes the geometric comparison criteria.
Section~\ref{sec:upper-bounds} contains the general upper-bound
invariance theorem. Part~II gives the categorical seed constructions in
Sections~\ref{sec:categorical}--\ref{sec:exceptional-categories},
including the two hypersurface cases at the end of Section~\ref{sec:exceptional-categories},
then the remaining algebraic realizations in Section~\ref{sec:verification},
and the theta bases and branching formulas in Section~\ref{sec:applications}.
The exceptional inequality matrices are included with these applications.
Appendix~\ref{app:part-I} contains the auxiliary affine-model lemmas
and the boundary-support construction. Appendix~\ref{app:part-II}
records the categorical degree matrices, exchange data, affine covers,
regular functions, mutation sequences, and electronic supplement.
The principal notation is collected in Table~\ref{tab:notation}.

\part{The family \texorpdfstring{$A_{2n-1}\downarrow C_n$}{A2n-1 to Cn} and theta bases}
\section{Branching algebras and Laurent phenomenon seeds}
\label{sec:setup}

\subsection{The grading and the coefficient ring}
Fix opposite maximal unipotent subgroups $U_G^-,U_G$, a maximal torus
$T_G$, and compatible $U_K,T_K$. For the simply connected groups used
below, $P(G)=X^*(T_G)$ and $P(K)=X^*(T_K)$ are their weight lattices,
and $\rho_G,\rho_K$ are the half-sums of their positive roots. We write
$\omega_i,\varpi_j$ for the respective fundamental weights and
$\pi:P(G)\to P(K)$ for restriction. The same notation for weight
lattices is used with Dynkin types in place of groups.
We use the grading characterized by
\begin{equation} f(tgs)=\lambda(t)\mu(s)f(g) \quad(t\in T_G,\ s\in T_K). \label{eq:torus-convention} \end{equation}
Here left $U_G^-$-invariance means $f(ug)=f(g)$, and right
$U_K$-invariance means $f(gv)=f(g)$. With this convention, a matrix
coefficient $\langle\xi,gv\rangle$, where $\xi$ is the $U_G^-$-fixed
covector in $V_G(\lambda)^*$ and $v$ is $U_K$-highest of weight $\mu$,
has degree $(\lambda;\mu)$. Thus Peter--Weyl gives
\[ \B(G,K)_{\lambda,\mu} \simeq \Hom_K(V_K(\mu),V_G(\lambda)). \]
See \cite{GoodmanWallach} for the matrix-coefficient decomposition.
In each $V_G(\omega_i)$, choose a highest vector $v_i^+$
and the $U_G^-$-fixed covector $\xi_i$ normalized by $\xi_i(v_i^+)=1$.
The \emph{principal function} is $p_i(g)=\langle\xi_i,gv_i^+\rangle$,
of degree $(\omega_i;\pi\omega_i)$. For
$\lambda=\sum_i\lambda_i\omega_i$, set $p^\lambda=\prod_i p_i^{\lambda_i}$.

For $\SL_{2n}\downarrow\Sp_{2n}$, we put
\[ \B_n:=\B(\SL_{2n},\Sp_{2n}),\qquad
   \omega_0=\omega_{2n}=0,\quad\varpi_0=0. \]

\begin{proposition}\label{prop:branching-ufd}
The algebra $\B_n$ is a finitely generated normal UFD, and
$\B_n^\times=\C^\times$.
\end{proposition}
\begin{proof}
Finite generation is the Grosshans finite-generation theorem applied to
the two maximal-unipotent actions \cite{Grosshans}. The ring
$\C[\SL_{2n}]$ is factorial and has only constant units
\cite{PopovVinberg}. A connected unipotent group acting on a factorial
ring fixes every prime factor of an invariant: it preserves each of the
finitely many prime ideals in the factorization, and its scalar action
on a generator is a character, hence trivial. Factorization therefore
descends to the invariant ring. This proves factoriality and hence
normality; the assertion about units follows from the ambient ring.
\end{proof}

We will also use the following consequence of the same factorization
argument. If a torus acts on a factorial domain with only constant
units, every prime factor of a semi-invariant is a semi-invariant. The
torus cannot nontrivially permute the finite set of prime factors,
because it is connected. Consequently factorization in $\B_n$ respects
the dominant branching grading.

For an extended seed, let $I_{\rm mut}$ and $I_{\rm fr}$ be its mutable
and frozen index sets. Put
\[ S=\C[x_f:f\in I_{\rm fr}],\qquad L(t)=S[x_i^{\pm1}:i\in I_{\rm mut}]. \]
Frozen variables are polynomial coefficients, not units. We explicitly
write $H^{-1}$ whenever some of them are inverted.

\subsection{Exchange polynomials and upper bounds}
We use LP mutation over a UFD as in \cite{LamPylyavskyy}. An ordinary
exchange polynomial $E_i\in S[x_j:j\in I_{\rm mut}]$ is irreducible,
independent of $x_i$, and not divisible by any mutable variable. Its
associated exchange Laurent polynomial has the form
\begin{equation} \widehat E_i=E_i\prod_{j\ne i}x_j^{-a_{ij}},\qquad a_{ij}\ge0, \label{eq:lp-normalization} \end{equation}
where the exponents are determined by the LP substitution and
nondivisibility rule. Mutation replaces $x_i$ by
$x_i'=\widehat E_i/x_i$ and changes the other exchange polynomials by
the same LP rule. In particular, equality $\widehat E_i=E_i$ in one
seed does not justify imposing that equality at all mutated seeds.

Define
\begin{equation} \UB(t)=L(t)\cap\bigcap_{i\in I_{\rm mut}}L(\mu_i t), \qquad \ULP(t)=\bigcap_{s\sim t}L(s). \label{eq:upper-definitions} \end{equation}
All intersections take place in the common fraction field. We do not
assume that every mutation of a binomial initial seed remains binomial.

For $a\in\R$, write $[a]_+=\max(a,0)$, and apply this notation
coordinatewise to vectors. For an integer matrix
$C\in\Z^{r\times(r+s)}$ with $C_{ii}=0$, the notation $\Sigma(C)$
means the seed on ordered variables $z_1,\ldots,z_{r+s}$, the first
$r$ mutable and the last $s$ frozen, with ordinary exchanges $\prod_jz_j^{[C_{ij}]_+}+\prod_jz_j^{[-C_{ij}]_+}$,
whenever these satisfy the LP seed axioms over the specified frozen
coefficient ring $S$. We abbreviate $\UB_S(\Sigma(C))$ to $\UB_S(C)$.
The associated exchanges are always determined by the LP rule.

\subsection{Mutation invariance of the finite upper bound}
\label{sec:upper-bounds}
The caterpillar intersection below proves mutation invariance of the
finite upper bound.

\begin{lemma}[Caterpillar]\label{lem:caterpillar}
Consider an alternating path $t_0\xrightarrow{i}t_1\xrightarrow{j}
 t_2\xrightarrow{i}t_3$, with $i\ne j$. Then
\[ L(t_1)\cap L(t_3)\subseteq L(t_0),\qquad L(t_0)\cap L(t_2)\subseteq L(t_3). \]
\end{lemma}
\begin{proof}
Write the changing coordinates as $(x,y),(z,y),(z,u),(v,u)$ and the
unchanged ones as $\mathbf w$. The caterpillar lemmas
\cite[Lemmas~5.2 and~5.4]{LamPylyavskyy} give
\[ u,v\in L_0:=S[\mathbf w^{\pm1},x^{\pm1},y^{\pm1}],\qquad \gcd_{L_0}(z,u)=\gcd_{L_0}(z,v)=1. \]
Moreover $z$ is prime in $L_0$: it is associated to the irreducible
exchange polynomial for $x$, and none of the inverted variables
divides that polynomial. For $h\in L(t_1)\cap L(t_3)$ write
$h=f/z^a=g/(u^bv^c)$ with $f,g\in L_0$ and $a,b,c\ge0$.
The identity $fu^bv^c=gz^a$ implies $z^a\mid f$, hence $h\in L_0$.
Reversing the path proves the second inclusion.
\end{proof}

\begin{theorem}[LP starfish]\label{thm:upper-invariance}
For an LP pattern over a UFD, the upper bound is mutation invariant.
Consequently $\UB(t)=\ULP(t)$.
\end{theorem}
\begin{proof}
Let $t'=\mu_k t$. For $j\ne k$, apply
Lemma~\ref{lem:caterpillar} to
\[ \mu_jt\xrightarrow{j}t\xrightarrow{k}t' \xrightarrow{j}\mu_jt'. \]
It gives $L(\mu_jt)\cap L(t')\subseteq L(\mu_jt')$ and the
reverse-path inclusion $L(t)\cap L(\mu_jt')\subseteq L(\mu_jt)$.
Intersecting with $L(t)\cap L(t')$ proves
$\UB(t)=\UB(t')$. The cases with no mutable variable or one mutable variable are immediate. The common
upper bound is contained in every Laurent ring of the pattern, and
therefore equals their intersection.
\end{proof}

The coprimality used in Lemma~\ref{lem:caterpillar} follows
from LP mutation, not an extra hypothesis. Neither pairwise coprimality
of the ordinary exchanges nor $E_i=\widehat E_i$ throughout the pattern
is required. The coefficient ring $S$ has not been localized.

\subsubsection*{Upper bounds versus adjacent generators.}
Theorem~\ref{thm:upper-invariance} does not say that the initial and
adjacent cluster variables generate the upper bound. For example, over
$S=\C[a]$ take $E_x=y+a$ and $E_y=x+a$, and put
\[ z=\frac{y+a}{x},\qquad w=\frac{x+a}{y},\qquad u=\frac{x+y+a}{xy}=\frac{1+z}{y}=\frac{1+w}{x}. \]
These three expressions show that $u\in\UB(t)$. The ring generated by
$x,y,z,w$ is
\[ A=S[x,y,z,w]/(xz-y-a,\;yw-x-a). \]
Indeed, eliminating $a$ identifies this domain with
$\C[x,y,z,w]/(y(w+1)-x(z+1))$, and localization at $xy$ gives
$S[x^{\pm1},y^{\pm1}]$. But
$A/xA\simeq\C[a,z,w]/(a(w+1))$, where $1+w\ne0$. Since
$xu=1+w$, it follows that $u\notin A$. Thus the initial and adjacent variables need not generate the upper
bound over polynomial coefficients.

\subsection{Frozen boundary and height-one comparison}
Let $D\subseteq I_{\rm fr}$, set $H=\prod_{q\in D}x_q$, and denote
by $\nu_q$ the exponent valuation of $x_q$ in the initial coordinates.
For a valuation $\nu$, write $\mathcal O_\nu=\{h:\nu(h)\ge0\}$,
including zero.

\begin{lemma}[Frozen descent]\label{lem:frozen-descent}
Suppose that $\nu_q(x_i')=0$ for every $q\in D$ and every mutable $i$.
Then
\[ \UB(t)=\UB(t)[H^{-1}]\cap\bigcap_{q\in D}\mathcal O_{\nu_q}. \]
\end{lemma}
\begin{proof}
In the initial Laurent ring the valuation is the ordinary $x_q$-adic
valuation. In an adjacent seed, reduction modulo $x_q$ gives
$\overline{x_i'}=\overline{\widehat E_i}/\overline{x_i}$ and $\overline{x_i}=\overline{\widehat E_i}/\overline{x_i'}$. The numerator is nonzero by the hypothesis and is independent of
$x_i$. Thus the residues of the adjacent extended variables other than
$x_q$ are algebraically independent and generate the same residue
field as before. The ambient valuation is also the intrinsic
$x_q$-adic valuation of this adjacent Laurent ring. Each Laurent ring
in the finite intersection \eqref{eq:upper-definitions} is therefore
recovered from its $H$-localization by these nonnegative-order
conditions. Localization commutes with a finite intersection: the
product of finitely many denominator-clearing powers clears all
rings at once. Intersecting proves the claim.
\end{proof}

\begin{proposition}\label{prop:comparison}
Let $A$ be a normal finitely generated domain with the same fraction
field as an LP seed $t$. Suppose the selected frozen functions $x_q$,
$q\in D$, are pairwise nonassociate prime elements of $A$. Assume
\[
\begin{gathered}
 \nu_q(x_i')=0\quad(q\in D,\ i\in I_{\rm mut}),\qquad
 A[H^{-1}]=\UB(t)[H^{-1}],\\
 \nu_q=\ord_{V(x_q)}\quad\hbox{on }\Frac(A)
 \quad(q\in D).
 \end{gathered}
\]
Then $A=\UB(t)=\ULP(t)$.
\end{proposition}
\begin{proof}
Normality and the height-one intersection description give
\[ A=A[H^{-1}]\cap\bigcap_{q\in D}\mathcal O_{\ord_{V(x_q)}}. \]
The prime divisors removed by $H$ are $V(x_q)$, $q\in D$.
Now use Lemma~\ref{lem:frozen-descent} and
Theorem~\ref{thm:upper-invariance}.
\end{proof}

The inclusion $\UB(t)\subseteq A$ follows from these intersection
descriptions. For upper cluster algebras obtained by minimal monomial
lifting, Francone gives related criteria in terms of equality between
frozen-coordinate valuations and geometric divisor orders
\cite[Propositions~4.1.3 and~4.1.10]{FranconeBranching}.

For later use, a separate height-one argument gives the familiar
inclusion criterion. If $A$ is normal, contains all initial and first
mutated functions, no height-one prime contains two distinct mutable
functions, and no height-one prime contains both $x_i$ and $x_i'$, then
$\UB(t)\subseteq A$. Indeed, at a height-one prime either all mutable
functions are units, or exactly one is not and its first mutation is
a unit. The initial or the corresponding adjacent Laurent ring is
then contained in the local ring. Intersect over height-one primes.

\section{The folded-minor seed}
\label{sec:folded-seed}
Throughout Sections~\ref{sec:folded-seed}--\ref{sec:algebra-equality},
we put $N=2n$, $V=\C^N$, and $\bar i=N+1-i$.
We construct regular functions in $\B_n$ and define the corresponding
LP seed in algebraically independent symbols. Proposition~\ref{thm:initial-chart}
identifies its rational function field with $\Frac(\B_n)$.

The seed functions lie in one-dimensional highest-weight spaces. We
give Pfaffian formulas for them and fix their scalars by evaluation at
integral matrices. Their primality allows us to remove the spherical
denominators in the exchange identities. The boundary formulas then
specify the polynomial exchanges.

\subsection{Bordered Pfaffians and their signs}
Let $J_{i,\bar i}=1$ and $J_{\bar i,i}=-1$ for $1\le i\le n$, all
other entries being zero. Thus
\[ \Omega=\sum_{i=1}^n e_i\wedge e_{\bar i},\qquad \Pf(J)=1. \]
Take $K=\Sp(V)=\{g:gJg^{\mathsf t}=J\}$ with the compatible upper
triangular Borel. For $g\in\SL_N$, set
\[
A(g)=gJg^{\mathsf t},\qquad
 \mathcal K(g)=\begin{pmatrix}A(g)&g\\-g^{\mathsf t}&0\end{pmatrix}.
\]
If $R$ is an increasing row list and $I$ an increasing column list,
write $I^*$ for the corresponding indices in the second block and put
\begin{equation} \tau(R;I)=(-1)^{\binom{|I|}{2}} \Pf\mathcal K(g)_{(R,I^*)}. \label{eq:tau-definition} \end{equation}
An odd-size Pfaffian is zero. We abbreviate $\tau(R;[m])$ to
$\tau(R;m)$, where $[m]=\{1,\ldots,m\}$.

\begin{lemma}\label{lem:tau-coefficient}
For $|R|=|I|+2p$ with $p\ge0$,
\begin{equation} \tau(R;I)=\left\langle e_R^*, g\left(e_I\wedge\frac{\Omega^p}{p!}\right)\right\rangle. \label{eq:tau-coefficient} \end{equation}
In particular, $\tau(R;I)=\Delta_{R,I}(g)$ if $|R|=|I|$, and
$\tau(R;0)=\Pf A(g)_{R,R}$.
\end{lemma}
\begin{proof}
In the Pfaffian expansion, each starred index must pair with an
ordinary index. Summing these matchings gives the determinant from
$g(e_I)$; the remaining pairs give $g(\Omega^p/p!)$. Moving the
starred indices to the exterior order contributes
$(-1)^{\binom{|I|}{2}}$, canceled by the prefactor in
\eqref{eq:tau-definition}. This proves the formula with its stated
sign, including both special cases.
\end{proof}

For $0\le m\le n$ and $|R|=m+2p$, the useful expansion is
\begin{equation}
\tau(R;m)=
 \sum_{\substack{I\subseteq\{m+1,\ldots,n\}\\|I|=p}}
 \Delta_{R,([m],I,\bar I)}(g),
 \qquad
 \bar I=(\bar i_p,\ldots,\bar i_1)
 \quad\hbox{if }I=(i_1<\cdots<i_p).
 \label{eq:tau-minor-sum}
\end{equation}
Indeed, expand $\Omega^p/p!$ and discard the terms meeting $[m]$.
The order of $\bar I$ in this formula is part of the normalization.
The vector $e_{[m]}\wedge\Omega^p/p!$ is $U_K$-highest of weight
$\varpi_m$. Thus $\tau([m+2p];m)\in\B_n$, of degree
$(\omega_{m+2p};\varpi_m)$. An arbitrary row list $R$ gives a
right $U_K$-invariant coefficient, but not in general a left
$U_G^-$-invariant function.

Set $\mu(d)=\min(d,N-d)$. The sum in
\eqref{eq:tau-minor-sum} has exactly one term when
$|R|=d$ and $m=\mu(d)$, so
\begin{equation} \tau(R;\mu(d))=\Delta_{R,[d]}(g). \label{eq:flag-specialization} \end{equation}
In particular, the principal and spherical functions are
\[
\begin{aligned}
p_d&=\Delta_{[d],[d]}=\tau([d];\mu(d)),&&1\le d\le 2n-1,
 \\
 s_{2r}&=\tau([2r];0),&&1\le r\le n-1.
\end{aligned}
\]
Their degrees are $(\omega_d;\varpi_{\mu(d)})$ and
$(\omega_{2r};0)$, respectively; $s_{2n}=1$.
They are nonzero and prime in $\B_n$: their left degrees are
fundamental, and a factor of left degree zero is constant.

\subsection{The index set and the weight formula}
For odd $a$ and $1\le\ell\le n$, define
\begin{equation} r=1+|a-1|,\qquad m=\min(\ell,\ell+a-1,2n+1-a-\ell). \label{eq:folded-parameters} \end{equation}
Let $\widetilde I(n)$ consist of the pairs $(a,\ell)$ with $m\ge1$.
For such a pair put
\begin{equation} \lambda_{a,\ell}=\sum_{j=0}^{m-1}\omega_{r+2j},\qquad q=\frac{\ell-m}{2},\quad R_0=\frac{r-1}{2},\quad D_0=n-R_0-m. \label{eq:folded-degree} \end{equation}
These are integers and
\begin{equation} 0\le q\le\min(R_0,D_0). \label{eq:q-bounds} \end{equation}
For example, before either boundary is crossed $m=\ell$ and $q=0$.
At a left crossing $q=R_0$, and at a right crossing $q=D_0$.
These three cases prove \eqref{eq:q-bounds} directly from
\eqref{eq:folded-parameters}.

We shall also use the following expression for the weight. Define a
$4n$-periodic sequence by
\[ \epsilon_j=\omega_j\ (1\le j<2n),\quad \epsilon_0=\epsilon_{2n}=0,\quad \epsilon_{2n+j}=-\omega_{2n-j}\ (1\le j<2n). \]
Then, for every $(a,\ell)\in\widetilde I(n)$,
\begin{equation} \lambda_{a,\ell}=\sum_{s=0}^{\ell-1}\epsilon_{a+2s}. \label{eq:antiperiodic-dictionary} \end{equation}
To check this, cancel the terms with indices paired by $j\leftrightarrow-j$
or by $j\leftrightarrow4n-j$, using periodicity in the first case.
The interval from $a$ to $a+2\ell-2$ has length less than $2n$,
so it cannot cross both $0$ and $2n$. The remaining indices are
$r,r+2,\ldots,r+2m-2$, which gives \eqref{eq:folded-degree}.

For $n\ge4$, the mutable variables are indexed by
\begin{align}
 I_{<}(n)&=\{(a,\ell):1\le\ell<n,\ a\text{ odd},
                      \ 3-\ell\le a\le2n-\ell-1\}, \notag\\
 I_{=}(n)&=\{(a,n):a\text{ odd},\ 3\le a\le n-1\},\qquad
 I_{\rm mut}=I_<\sqcup I_=.
 \label{eq:mutable-chamber}
\end{align}
The additional pair $(1,n)$ is frozen. For fixed odd $\ell<n$ there are $n-2$ mutable variables, for fixed even $\ell<n$ there are $n-1$, and at $\ell=n$ there are
$\lfloor(n-2)/2\rfloor$. Hence
\[ |I_{\rm mut}|=n(n-2),\qquad |I_{\rm fr}|=3n-1, \qquad |I_{\rm mut}|+|I_{\rm fr}|=n^2+n-1. \]

\subsection{Multiplicity one}
The partition corresponding to $\lambda_{a,\ell}$ is
\[ \Lambda=(m^r,(m-1)^2,(m-2)^2,\ldots,1^2,0^{d_0}), \qquad d_0=2n-r-2m+2=2D_0+1. \]
For $m=1$, the intermediate pairs of equal parts are absent. Let $A_k(\Lambda)$
count partitions $\eta$ with at most $2n$ rows such that
$\eta/\Lambda$ is a vertical $k$-strip and every column of $\eta$ has
even length. Then
\begin{equation} A_k(\Lambda)= \#\{(u,v)\in\Z_{\ge0}^2: u\le R_0,\ v\le D_0, \ k=m+2(u+v)\}. \label{eq:strip-count} \end{equation}
Indeed, a vertical strip increases an initial segment of each
group of rows of the same length. If $t_j$ is the number of rows of length $j$
increased, even multiplicity of each positive row length forces
$t_m=2u$, $t_{m-1}=\cdots=t_1=1$, and $t_0=2v+1$.
The allowed ranges give exactly \eqref{eq:strip-count}.
By \eqref{eq:q-bounds}, $A_\ell=q+1$ and $A_{\ell-2}=q$,
with negative subscripts interpreted as zero.

\begin{proposition}\label{thm:folded-multiplicity}
For every $(a,\ell)\in\widetilde I(n)$,
$\dim(\B_n)_{\lambda_{a,\ell},\varpi_\ell}=1$.
\end{proposition}
\begin{proof}
Exterior Pieri and the even-column criterion for symplectic invariants
in polynomial $\GL_{2n}$-modules give
\[ \dim\Hom_{\Sp_{2n}}(\wedge^k V,\mathbb S_\Lambda V) =A_k(\Lambda). \]
Here the exterior power is self-dual as an $\Sp_{2n}$-module. Both
statements are finite-dimensional with the row bound $2n$ imposed;
see \cite{FultonYoungTableaux,GoodmanWallach}. For $\ell\le n$,
symplectic contraction gives
$\wedge^\ell V\simeq V_{\Sp_{2n}}(\varpi_\ell)\oplus\wedge^{\ell-2}V$.
Subtracting the two multiplicities yields $(q+1)-q=1$.
\end{proof}

\subsection{The normalized column--Pfaffian coefficient}
Set $t=2q+1=\ell-m+1$, $h_j=r+2(j-1)$, $k_1=t$, and
$k_j=1$ for $j>1$. Let $\op{Sh}_{\mathbf k}([\ell])$
be the ordered partitions $(I_1,\ldots,I_m)$ of $[\ell]$ into
increasing lists of sizes $k_j$. Write $\varepsilon(\mathbf I)$
for the sign of their concatenation. Define
\begin{equation} F_{a,\ell}(g)= \sum_{\mathbf I\in\op{Sh}_{\mathbf k}([\ell])} \varepsilon(\mathbf I)\prod_{j=1}^m\tau([h_j];I_j). \label{eq:folded-function} \end{equation}
All Pfaffians and signs on the right have already been fixed.
Equivalently, the right vector is the exterior coproduct of
$e_{[\ell]}$, with its $j$-th component wedged with
$\Omega^{(h_j-k_j)/2}/((h_j-k_j)/2)!$; the left covector is
$e_{[h_1]}^*\otimes\cdots\otimes e_{[h_m]}^*$.
Consequently $F_{a,\ell}$ is a regular semi-invariant of the claimed
bidegree, provided it is nonzero.

The following explicit construction proves nonvanishing with a fixed
normalization.

\begin{proposition}\label{prop:unit-frame}
For every $(a,\ell)\in\widetilde I(n)$ there is an explicitly
specified $g_{a,\ell}\in\SL_{2n}(\Z)$ such that
\begin{equation} F_{a,\ell}(g_{a,\ell})=1. \label{eq:unit-evaluation} \end{equation}
\end{proposition}
\begin{proof}
Put $\alpha_j=e_{2j-1}$ and $\beta_j=e_{2j}$, and abbreviate
$R_0$ to $R$. The bounds \eqref{eq:q-bounds} give
$q\le R$ and $R+m+q\le n$.
For $1\le i\le q$, put $j_i=R+m+i$ and choose the following two
ordered pairs:
\[
\begin{aligned}
 a_{2i-1}&=\alpha_i+\alpha_{j_i},& b_{2i-1}&=\beta_i,\\
 a_{2i}&=\beta_i-\beta_{j_i},& b_{2i}&=\alpha_{j_i}.
 \end{aligned}
\]
Next take $a_{2q+k}=\alpha_{R+k}$ and
$b_{2q+k}=\beta_{R+k}$ for $1\le k\le m$, and append the unused
standard pairs $(\alpha_j,\beta_j)$ in increasing order of $j$.
This gives $n$ pairs in total. Each mixed pair of pairs satisfies
\[ a_{2i-1}\wedge b_{2i-1}+a_{2i}\wedge b_{2i} =\alpha_i\wedge\beta_i+\alpha_{j_i}\wedge\beta_{j_i}. \]
Therefore the integral matrix
$g=(a_1\mid\cdots\mid a_n\mid b_n\mid\cdots\mid b_1)$
satisfies $gJg^{\mathsf t}=J_0$, where $J_0$ is the direct sum of
$n$ copies of $\left(\begin{smallmatrix}0&1\\-1&0\end{smallmatrix}\right)$.
Since $\Pf J=\Pf J_0=1$, it follows that $\det g=1$.

For this matrix, a singleton bordered coefficient is particularly
simple:
\[ \tau([2u-1];\{i\})(g)=g_{2u-1,i}. \]
In fact, after deleting the row paired with the starred index, the
adjacent-pair matrix $J_0$ has nonzero Pfaffian only when that row is
$2u-1$, and the resulting sign is positive. For $j\ge2$, among the
first $\ell=2q+m$ columns the row $h_j=2(R+j)-1$ has a single
nonzero entry, namely $1$ in column $2q+j$. Thus the only shuffle
which can contribute to \eqref{eq:folded-function} is
\[ I_1=[2q+1],\qquad I_j=\{2q+j\}\quad(2\le j\le m). \]
Its sign and all singleton factors are $1$.

In the first $2R+1$ rows, the first $2q+1$ columns of $g$ restrict to
$e_1,\ldots,e_{2q},e_{2R+1}$. In the exterior formula for
$\tau([2R+1];[2q+1])$, the remaining rows are filled by the adjacent
pairs $q+1,\ldots,R$ in $J_0$. Moving $e_{2R+1}$ across these
$2(R-q)$ entries contributes sign $+1$. This last factor is also $1$,
proving \eqref{eq:unit-evaluation}.
\end{proof}

\begin{corollary}\label{cor:normalized-folded-functions}
The function \eqref{eq:folded-function} spans
$(\B_n)_{\lambda_{a,\ell},\varpi_\ell}$. Its integral polynomial
expression has no nontrivial common integer factor. The evaluation
in \eqref{eq:unit-evaluation} fixes this primitive integral normalization.
\end{corollary}
\begin{proof}
The exterior coproduct and multiplication by powers of $\Omega$ are
$\Sp(V)$-equivariant, and the chosen left flag covectors have total
weight $\lambda_{a,\ell}$. Apply Proposition~\ref{prop:unit-frame}
and Proposition~\ref{thm:folded-multiplicity}. Integral evaluation at a
point with value $1$ proves the assertion about common integer factors.
\end{proof}

Set
\[ \mathsf P_{u,j}=\tau([2u-1];\{j\}),\quad 1\le u,j\le n, \qquad c_n=F_{1,n}. \]
When $a=2b-1\ge1$ and $b+\ell\le n+1$, we have $m=\ell$ and
all $k_j=1$. Hence the shuffle is a determinant:
\begin{equation} F_{2b-1,\ell}= \det\mathsf P_{[b,b+\ell-1],[\ell]},\qquad c_n=\det\mathsf P. \label{eq:solid-minor-formula} \end{equation}
In particular,
\[ \wt(c_n)=(\omega_1+\omega_3+\cdots+\omega_{2n-1};\varpi_n). \]
The initial extended functions are the $F_{a,\ell}$ in
\eqref{eq:mutable-chamber}, together with $p_1,\ldots,p_{2n-1}$,
$s_2,s_4,\ldots,s_{2n-2}$ and $c_n$.

\begin{proposition}\label{prop:initial-primes}
Every initial extended function is prime in $\B_n$, and distinct
initial variables are nonassociate primes. The sum of the frozen degrees
is $(2\rho_{A_{2n-1}};2\rho_{C_n})$.
\end{proposition}
\begin{proof}
The claim for $p_d,s_{2r}$ was proved above. A factorization of
$F_{a,\ell}$ into semi-invariants has exactly one factor with nonzero
right degree, since $\varpi_\ell$ is fundamental. All other factors
are spherical. The spherical weight monoid for $\SL_{2n}/\Sp_{2n}$
is generated by the even fundamental weights
\cite{GoodmanWallach}, whereas $\lambda_{a,\ell}$ is supported on
odd fundamental weights. Such a spherical factor is therefore
constant. Proposition~\ref{prop:branching-ufd} gives primality.
The initial degrees are distinct: the right degree separates levels,
and the support $\{r,r+2,\ldots,r+2m-2\}$ of the source weight
distinguishes the initial variables within a level. Finally, the principal degrees contain each left fundamental
weight once and each right fundamental weight twice except $\varpi_n$,
which occurs once. The spherical degrees supply the remaining even
left weights, and $c_n$ supplies the odd ones and $\varpi_n$.
\end{proof}

\subsection{A complete boundary convention}
Regard the initial seed variables as independent symbols in this subsection.
If $(a,\ell)\in\widetilde I(n)$ occurs as a neighbouring index, we
identify $F_{a,\ell}$ with the unique initial seed variable of degree
$(\lambda_{a,\ell};\varpi_\ell)$. For two such indices with the
same degree, their $r,m,\ell$ agree, so \eqref{eq:folded-function}
gives the same normalized function. A principal identification follows from
\eqref{eq:flag-specialization}.

Define rational neighbor symbols by
\begin{equation}
\mathfrak f_{a,\ell}=
 \begin{cases}
 1,&\ell=0,\\
 F_{a,\ell},&m(a,\ell)\ge1,\\
 p_\ell/s_\ell,&\ell\text{ even},\ a=1-\ell,\\
 p_{2n-\ell}/s_{2n-\ell},&\ell\text{ even},\ a=2n+1-\ell.
 \end{cases}
 \label{eq:neighbor-resolver}
\end{equation}
The only neighbours with $m<1$ have $\ell=0$, or have even $\ell$
and $a=1-\ell$ or $a=2n+1-\ell$. These are exactly the cases in
\eqref{eq:neighbor-resolver}; all other neighbours are
folded-minor indices. At an even
boundary, the numerator is a principal function and the denominator
is the corresponding spherical function.

For $(a,\ell)\in I_<$, set
\begin{equation} R_{a,\ell}= \mathfrak f_{a-2,\ell+1}\mathfrak f_{a,\ell-1}\mathfrak f_{a+2,\ell} +\mathfrak f_{a-2,\ell}\mathfrak f_{a,\ell+1}\mathfrak f_{a+2,\ell-1}. \label{eq:lower-stencil} \end{equation}
For $(a,n)\in I_=$, set
\begin{equation} R_{a,n}= \mathfrak f_{a+2,n}\mathfrak f_{a,n-1}\mathfrak f_{4-a,n-1} +\mathfrak f_{a-2,n}\mathfrak f_{a+2,n-1}\mathfrak f_{2-a,n-1}. \label{eq:middle-stencil} \end{equation}
Let $E_{a,\ell}=\kappa_{a,\ell}R_{a,\ell}$, where
\begin{equation}
\kappa_{a,\ell}=
 \begin{cases}
 s_\ell,&\ell<n\text{ even},\ a=3-\ell,\\
 s_{2n-\ell},&\ell<n\text{ even},\ a=2n-\ell-1,\\
 s_n,&n\text{ even},\ \ell=n,\ a=n-1,\\
 1,&\text{otherwise}.
 \end{cases}
 \label{eq:clearing-factor}
\end{equation}
These are polynomial binomials. For the boundary cases, write $a_L=3-\ell$ and $a_R=2n-\ell-1$. For lower even $\ell$,
\begin{align}
 E_{a_L,\ell}
 &=s_\ell p_{\ell-1}p_{\ell+1}F_{5-\ell,\ell}
   +p_\ell F_{3-\ell,\ell+1}F_{5-\ell,\ell-1},
 \label{eq:left-even-exchange}\\
 E_{a_R,\ell}
 &=F_{2n-\ell-3,\ell+1}F_{2n-\ell-1,\ell-1}p_{2n-\ell}
   +s_{2n-\ell}p_{2n-\ell-1}p_{2n-\ell+1}F_{2n-\ell-3,\ell}.
 \label{eq:right-even-exchange}
\end{align}
For even $n$, at $F_{n-1,n}$,
\begin{equation} E_{n-1,n}=s_n p_{n-1}p_{n+1}F_{n-3,n} +p_nF_{n-1,n-1}F_{5-n,n-1}. \label{eq:corrected-middle-boundary} \end{equation}
Indeed, $\mathfrak f_{n+1,n}=p_n/s_n$,
$F_{n+1,n-1}=p_{n+1}$ and $F_{3-n,n-1}=p_{n-1}$.
When $n$ is odd, $F_{n,n}=p_n$.

\begin{proposition}\label{prop:formal-seed}
The polynomials $E_{a,\ell}$ define an LP seed
$\Sigma_n^{\mathrm{for}}$ in the independent extended symbols.
They are homogeneous, and their associated exchange Laurent
polynomials satisfy $\widehat E_{a,\ell}=E_{a,\ell}$ in this seed.
\end{proposition}
\begin{proof}
Substitution of \eqref{eq:neighbor-resolver} into the two exchange formulas,
with \eqref{eq:clearing-factor}, gives two squarefree monomials with
disjoint supports in each exchange polynomial. Neither monomial contains the variable being mutated. These statements also hold at the three even boundaries by
\eqref{eq:left-even-exchange}--\eqref{eq:corrected-middle-boundary}.
Away from them they follow by separating the three levels and then
the consecutive indices within each level. The same inspection shows
that no two unordered pairs of monomial supports coincide.

The exponent difference of the two monomials is a nonzero primitive
integer vector (its nonzero entries are $1$ or $-1$). After Laurent
localization the binomial is associated to $1+z$ for a primitive
lattice monomial $z$, hence irreducible. The absence of a monomial
common factor proves irreducibility in the polynomial ring.
Homogeneity follows by substituting
\eqref{eq:antiperiodic-dictionary}; the two lists of $\epsilon$-terms
cancel pairwise. At an even boundary the added spherical factor
supplies exactly its missing even fundamental weight.

For LP normalization, fix two distinct mutable indices $i,j$. If
$x_j$ occurs in $E_i$, write $E_i=A+x_jB$ with $A,B$ monomials.
Substituting $x_j=E_j/X$ and clearing $X$ gives $XA+E_jB$.
Modulo the prime $E_j$ this is the nonzero monomial $XA$, so no
factor of $E_j$ can be removed. If $x_j$ does not occur, distinctness
of the unordered supports makes $E_i,E_j$ nonassociate irreducibles,
again excluding divisibility. Thus all exponents in
\eqref{eq:lp-normalization} are zero.
\end{proof}

Write $E_x=M_{x,+}+M_{x,-}$. We order the terms as in the exchange formulas,
\emph{except} at $F_{n-1,n}$ for even $n$, where we use the
order in \eqref{eq:corrected-middle-boundary}. Define
\begin{equation} b_{xy}=\exp_y(M_{x,+})- \exp_y(M_{x,-}). \label{eq:row-sign-convention} \end{equation}
Here $\exp_y(M)$ denotes the exponent of $y$ in the monomial $M$.
Thus this row has the opposite sign to the term order in
\eqref{eq:middle-stencil}. The exchange polynomial is unchanged. This explicit choice is
needed for the categorical out-minus-in convention of
Section~\ref{sec:categorical}; a polynomial binomial alone does not
choose a row sign. If $W$ lists the extended weights by rows, then
$BW=0$.

\begin{figure}[tbp]
\centering
\resizebox{.8\textwidth}{!}{\AnineCfive}
\caption{The initial exchange matrix for $A_9\downarrow C_5$.
A rounded vertex $(a,\ell)$ denotes $F_{a,\ell}$; rectangular
vertices are frozen. The blue vertex $(-1,4)$ denotes $F_{-1,4}$,
mutated in Section~\ref{subsec:a9-geometric}. At a mutable vertex,
an incident tail contributes $+1$ and an arrowhead
$-1$ to its exchange row. Thus a headless edge has signs $(+,+)$ and
a double-headed edge signs $(-,-)$. The two red edges record these
exceptional sign pairs. }
\label{fig:a9c5-seed}
\end{figure}

\subsubsection*{The $A_9\downarrow C_5$ seed.}
In Figure~\ref{fig:a9c5-seed}, the blue vertex $(-1,4)$ denotes
$F_{-1,4}$. Its two exchange monomials, with the conventions above,
give
\begin{equation}
 E_{-1,4}=F_{1,4}s_4p_5p_3+F_{1,3}F_{3,5}p_4.
 \label{eq:a9-running-exchange}
\end{equation}
Section~\ref{subsec:a9-geometric} mutates at this variable to obtain
the geometric seed.

\subsection{The seeds for \texorpdfstring{$n=2$ and $n=3$}{n=2 and n=3}}
For $n=2$, there are no mutable variables. The five frozen functions
are $p_1,p_2,p_3,s_2,c_2$. This is the case $D_3\downarrow B_2$
under $\Spin_6\simeq\SL_4$ and $\Spin_5\simeq\Sp_4$.

For $n=3$, use the three $2\times2$ minors of the first two columns
of $\mathsf P$:
\[ x_1=\det\mathsf P_{\{2,3\},[2]},\qquad x_2=\det\mathsf P_{\{1,3\},[2]},\qquad x_3=\det\mathsf P_{\{1,2\},[2]}. \]
Their degrees are respectively
$(\omega_3+\omega_5;\varpi_2)$,
$(\omega_1+\omega_5;\varpi_2)$,
and $(\omega_1+\omega_3;\varpi_2)$.
The frozen functions are $p_1,\ldots,p_5,s_2,s_4,c_3$, and the
separate formal exchange polynomials are
\begin{equation} E_1=x_2s_4p_3+p_4c_3,\qquad E_2=x_3p_5+x_1p_1,\qquad E_3=p_2c_3+x_2s_2p_3. \label{eq:rank-three-seed} \end{equation}
The seeds for $n=2,3$ are defined separately from
\eqref{eq:mutable-chamber}. Their exchange identities and initial charts
are proved in Section~\ref{sec:algebra-equality}.

Section~\ref{sec:exchanges} proves regularity of the first mutations.
Section~\ref{sec:algebra-equality} proves algebraic independence,
fraction-field generation, and the ring equality.

\section{The Pfaffian frame and the exchange identities}
\label{sec:exchanges}
The frame identities hold on the spherical open set of the whole group,
not only on a unitriangular section. A character of the branching
grading absorbs the contragredient signs, while preserving the
normalization of Section~\ref{sec:folded-seed}.

\subsection{Construction of the Pfaffian frame}
Write $g_i$ for row $i$ of $g$, set $s_0=s_{2n}=1$, and put
\[ H_{\rm sph}=\prod_{r=1}^{n-1}s_{2r},\qquad G^\circ=\{g\in\SL_{2n}:H_{\rm sph}(g)\ne0\}. \]
For $1\le r\le n$, define the vector
\[ \rho_r=\bigl(\tau([2r-1];\{j\})\bigr)_{j=1}^{2n}. \]
Expansion along the starred index gives
\begin{equation} \rho_r=\sum_{i=1}^{2r-1}(-1)^{i+1} \Pf A_{[2r-1]\setminus\{i\}}\,g_i =s_{2r-2}g_{2r-1}+\sum_{i<2r-1}c_i g_i. \label{eq:global-pfaffian-row} \end{equation}
On $G^\circ$, define the vectors
\[
r_s=\begin{cases}
 -g_{2(1-s)}/s_{2(1-s)},&1-n\le s\le0,\\
 \rho_s,&1\le s\le n,
 \end{cases}
 \qquad X_{b,d}=\det(r_b,\ldots,r_{b+d-1})_{[d]}.
\]
We require $1-n\le b$ and $b+d-1\le n$, and set $X_{b,0}=1$.
The entries of these vectors have denominators that are products of the spherical functions.

\begin{lemma}\label{lem:global-row-dictionary}
Let $1\le d\le n$, $b+d\le n+1$, and
$b\ge1-\lfloor d/2\rfloor$. Then
\[ X_{b,d}=\mathfrak f_{2b-1,d}\quad\text{on }G^\circ. \]
In particular, $X_{1-q,2q}=p_{2q}/s_{2q}$.
\end{lemma}
\begin{proof}
For $b\ge1$, this is \eqref{eq:solid-minor-formula}.
For $b=1-q\le0$ and $d\ge2q+1$, the triangular expansion
\eqref{eq:global-pfaffian-row} gives
\begin{equation} r_{1-q}\wedge\cdots\wedge r_{d-q} =g_1\wedge\cdots\wedge g_{2q+1} \wedge\rho_{q+2}\wedge\cdots\wedge\rho_{d-q}. \label{eq:global-row-wedge} \end{equation}
The $q$ negative even rows cancel the sign of the permutation
$(2q,2q-2,\ldots,2,1,3,\ldots,2q+1)$. Their denominators cancel
$s_2\cdots s_{2q}$, the leading coefficients of
$\rho_2,\ldots,\rho_{q+1}$.
For the folded index $1-2q$, the first subset in the shuffle has size
$2q+1$ and its column height is $2q+1$. Thus Laplace expansion of
\eqref{eq:global-row-wedge} in the first $d$ columns is exactly
\eqref{eq:folded-function}, with the same shuffle sign.
For $d=2q$, the same calculation stops one odd row earlier and gives
$s_{2q}^{-1}g_1\wedge\cdots\wedge g_{2q}$, proving the boundary
formula. No triangularity assumption on $g$ was used.
\end{proof}

\begin{lemma}\label{lem:global-top-determinant}
For odd $a$ with $3\le a\le n+1$, we have
\[ X_{(3-a)/2,n+1}=F_{a,n-1}\quad\text{on }G^\circ. \]
For even $n$ and $a=n+1$, the right side is $p_{n+1}$.
\end{lemma}
\begin{proof}
Write $a=2q+1$ and $m=n+1-2q$. Formula
\eqref{eq:global-row-wedge}, now in $n+1$ columns, writes the left
side as an exterior shuffle with column heights
$2q+1,2q+3,\ldots,2q+2m-1$ and subset sizes
$2q+1,1,\ldots,1$. Since
$e_{[n+1]}=e_{[n-1]}\wedge\Omega$, 
this is a regular branching coefficient of degree
$(\lambda_{a,n-1};\varpi_{n-1})$. It is a scalar multiple of
$F_{a,n-1}$ by Proposition~\ref{thm:folded-multiplicity}.

Evaluate at the integral frame in Proposition~\ref{prop:unit-frame}
for the index $(a,n-1)$. Its parameters are $R=q$, $q_0=q-1$,
and $m=n+1-2q$. All spherical pivots equal $1$. The $m-1$ singleton
rows force columns $2q,2q+1,\ldots,n-1$. The first subset in the ordered partition therefore consists of the column indices $\{1,\ldots,2q-1,n,n+1\}$.
In the first $2q+1$ rows these columns are
$e_1,\ldots,e_{2q-2},e_{2q+1},e_{2q-1},e_{2q}$.
Their determinant is $1$. Moving $n,n+1$ past the singleton subsets
has even sign. Hence the top determinant is $1$ at the same point
where $F_{a,n-1}=1$, and the scalar is exactly $1$.

\end{proof}

\subsection{Condensation and regular quotients}
For a matrix with ordered rows $v_i$, write
\[ T(i_1,\ldots,i_d)=\det(v_{i_1},\ldots,v_{i_d})_{[d]}. \]
For consecutive vectors $v_0,\ldots,v_{d+1}$, set $\Theta_1=T(0,2)$,
and, for $d\ge2$, set
\[ \Theta_d=T(2,\ldots,d)T(0,1,\ldots,d-1,d+1) -T(2,\ldots,d-1,d+1)T(0,1,\ldots,d). \]
Empty row lists have determinant $1$.

\begin{lemma}\label{lem:condensation}
One has
\begin{align}
 T(1,\ldots,d)\Theta_d
 &=T(0,\ldots,d)T(1,\ldots,d-1)T(2,\ldots,d+1)\notag\\
 &\quad+T(0,\ldots,d-1)T(1,\ldots,d+1)T(2,\ldots,d).
 \label{eq:condensation}
\end{align}
\end{lemma}
\begin{proof}
This is the flag-minor form of the three-term Pl\"ucker identity.
For a direct verification, on the dense open where the first $d-1$
columns in rows $1,\ldots,d-1$ are invertible, use block upper-triangular column operations preserving the spans
of the first $d-1$, $d$, and $d+1$ columns to make this block the
identity and to clear the last two columns in those rows. Expanding along the resulting identity
rows reduces \eqref{eq:condensation} to the $2\times2$ determinant
identity on the three remaining rows. Both sides have the same degree
in every column, so the normalizing column scalings cancel. The identity
extends to every matrix because it is polynomial. For $d=1$ it is
$T(1)T(0,2)=T(0,1)T(2)+T(0)T(1,2)$.
\end{proof}

\begin{proposition}\label{prop:one-sided-exchanges}
For $n\ge4$, suppose either $(a,\ell)\in I_<$ with
$a=2b-1$ and $b+\ell\le n$, or $(a,\ell)\in I_=$. Then
$E_{a,\ell}/F_{a,\ell}$ is a regular element of $\B_n$.
\end{proposition}
\begin{proof}
In the lower case apply \eqref{eq:condensation} to
$r_{b-1},\ldots,r_{b+\ell}$. Lemma~\ref{lem:global-row-dictionary}
identifies its seven contiguous minors and gives
\[ F_{a,\ell}\Theta=R_{a,\ell},\qquad \frac{E_{a,\ell}}{F_{a,\ell}}=\kappa_{a,\ell}\Theta \in\C[G][H_{\rm sph}^{-1}]. \]
For $\ell=n$, use centre $c=2-b$ and the vectors $r_{c-1},\ldots,r_{c+n}$. The central minor is
$F_{2-a,n}=F_{a,n}$. The two size-$n+1$ minors are
$F_{a+2,n-1}$ and $F_{a,n-1}$ by
Lemma~\ref{lem:global-top-determinant}. The remaining four neighbors
are the ones in \eqref{eq:middle-stencil}; their two products occur
in the reverse order, which does not change their sum. The same
conclusion follows, including the factor $s_n$ in $E_{n-1,n}$ for even $n$.

The quotient is a rational invariant. Multiplying it by a sufficiently
large power of the invariant $H_{\rm sph}$ gives a regular invariant,
so it belongs to $\B_n[H_{\rm sph}^{-1}]$. The prime $F_{a,\ell}$ is
nonassociate to every spherical pivot. Clearing the denominator and
canceling these prime factors in the UFD $\B_n$ proves
$F_{a,\ell}\mid E_{a,\ell}$ in $\B_n$ itself.
\end{proof}

The possible denominators are products of spherical functions,
which are relatively prime to $F_{a,\ell}$. Their removal therefore
follows from primality in $\B_n$.

\subsection{Reflection as a character of the grading}
Let $\delta(g)=Jg^{-\mathsf t}J^{-1}$. This involution preserves
$U_G^-$ and fixes $K$ pointwise. It therefore acts on $\B_n$, sending
$(\lambda;\mu)$ to $(\lambda^\vee;\mu)$, where
$\omega_j^\vee=\omega_{N-j}$.
Define
\begin{equation} \alpha_j=\lfloor j/2\rfloor+\max(0,j-n),\qquad \beta_\ell=\binom\ell2, \qquad \varepsilon(\lambda;\mu) =(-1)^{\sum_j\alpha_j\lambda_j+\sum_\ell\beta_\ell\mu_\ell}. \label{eq:reflection-character} \end{equation}
This is a character of the full branching weight lattice.

\begin{lemma}[Bordered-Pfaffian duality]\label{lem:exact-hodge}
For $I\subseteq[n]$, $k=|I|$, and $h=k+2p$ with
$0\le p\le n-k$,
\begin{equation} \tau([h];I)(\delta g)=(-1)^{\theta_n(h,k)}\tau([N-h];I)(g), \quad \theta_n(h,k)=\binom{k+1}{2}+\binom h2+h(N-h)+\max(0,h-n). \label{eq:exact-hodge} \end{equation}
\end{lemma}
\begin{proof}
Let $\star_h(e_I)=\op{sgn}(I,I^c)e_{I^c}$ and
$D_h=\star_h\circ\wedge^h(J^{-1})$. Complementing symplectic
pairs in the exterior basis gives
\[ D_h\left(e_I\wedge\frac{\Omega^p}{p!}\right) =(-1)^{\binom{k+1}{2}} e_I\wedge\frac{\Omega^{n-k-p}}{(n-k-p)!}. \]
The divided powers make each admissible pair subset occur with
coefficient $1$. The complement sign is independent of that subset:
interchanging a selected and an unselected pair moves two entries past
two entries, while the $k$ unpaired entries give
$(-1)^{k(k+1)/2}$. For the prefix covector, direct complementation gives
\[ (D_h^{-1})^*e_{[h]}^* =(-1)^{\binom h2+h(N-h)+\max(0,h-n)}e_{[N-h]}^*. \]
Here the three contributions are reversal of the $h$ barred indices,
exchange of the two complementary lists, and the negative signs of
$J^{-1}$ on the last $\max(0,h-n)$ prefix vectors.
Finally $D_h\wedge^h\delta(g)=\wedge^{N-h}g\,D_h$.
Applying \eqref{eq:tau-coefficient} proves the formula.
\end{proof}

\begin{proposition}\label{thm:sign-free-reflection}
On a homogeneous element define
\[ \mathfrak r(f)=\varepsilon(\wt(f))\,\delta^*f. \]
Then $\mathfrak r$ is an algebra involution and
\begin{equation} \mathfrak r(p_j)=p_{N-j},\quad \mathfrak r(s_{2r})=s_{N-2r},\quad \mathfrak r(F_{a,\ell})=F_{\rho_\ell(a),\ell},\qquad \rho_\ell(a)=2n-2\ell+2-a. \label{eq:sign-free-reflection} \end{equation}
It acts on rational neighbors by the same sign-free rule. For every
mutable index with $\ell<n$,
\begin{equation} \mathfrak r(E_{a,\ell})=E_{\rho_\ell(a),\ell}. \label{eq:sign-free-exchange-reflection} \end{equation}
\end{proposition}
\begin{proof}
The one-column cases of Lemma~\ref{lem:exact-hodge}, with involutivity
for $j>n$, give $\delta^*p_j=p_{N-j}$ and
$\delta^*s_{2r}=(-1)^{\alpha_{2r}}s_{N-2r}$.
For a folded index we keep $r,m,q,R_0$ from
\eqref{eq:folded-degree}. Duality complements its column heights,
then reverses their order; multiplicity one makes the result a scalar
multiple of the reflected coefficient.

We compute the scalar by evaluating at the
unit-value frame for the reflected index. Each of the $m-1$ singleton
subsets in the dualized shuffle forces a different column, in reverse
order $\ell-1,\ldots,2q+1$. The remaining subset is
$\{1,\ldots,2q,\ell\}$, and its bordered coefficient is $1$ by the
same adjacent-pair calculation as in Proposition~\ref{prop:unit-frame}.
The shuffle has sign $(-1)^{\binom m2}$. Multiplying this by the
$\theta_n(h_j,k_j)$ signs in \eqref{eq:exact-hodge} gives
\[ \delta^*F_{a,\ell}=(-1)^{\chi_n(a,\ell)}F_{\rho_\ell(a),\ell}, \quad \chi_n(a,\ell)=q+mR_0+\sum_{j=0}^{m-1}\max(0,r+2j-n). \]
Since $\ell=m+2q$,
\[ \chi_n(a,\ell)\equiv \binom\ell2+\sum_{j=0}^{m-1}\alpha_{r+2j}\pmod2. \]
Thus this sign is exactly $\varepsilon(\lambda_{a,\ell};\varpi_\ell)$.
The principal case has $\varepsilon(\omega_j;\varpi_{\mu(j)})=1$.
This proves \eqref{eq:sign-free-reflection}.

Multiplicativity follows because $\varepsilon$ is a character.
Moreover $\alpha_{N-j}\equiv\alpha_j\pmod2$, so
$\varepsilon(\lambda^\vee;\mu)=\varepsilon(\lambda;\mu)$ and
$\mathfrak r^2=1$. Ratios $p_d/s_d$ transform without a sign.
Reflection interchanges the two products in each lower recurrence
and sends its clearing factor to the reflected clearing factor.
This proves \eqref{eq:sign-free-exchange-reflection}. In particular,
there is no separate exchange-sign compatibility assumption.
\end{proof}

\begin{theorem}\label{thm:regular-exchanges}
For $n\ge4$ and every mutable index $x=(a,\ell)$, there is a unique
regular function $F_x'\in\B_n$ with
$F_xF_x'=E_x$.
The coefficient of each monomial in $E_x$ is $1$ with the
fixed normalization of Section~\ref{sec:folded-seed}.
\end{theorem}
\begin{proof}
Proposition~\ref{prop:one-sided-exchanges} treats the one-sided region
and the level $\ell=n$. For a remaining index with $\ell<n$, write $a=2b-1$.
Then $b+\ell>n$ and the reflected index has
$b^\vee=n-\ell-b+2$, so $b^\vee+\ell=n-b+2\le n$.
Apply $\mathfrak r$ to its regular exchange. The sign-free formulas
prove the exchange exactly. Uniqueness follows in the
domain $\B_n$. Nonvanishing of the quotients will also follow from
algebraic independence in Proposition~\ref{thm:initial-chart}.
\end{proof}

\section{Localization and the branching algebra}
\label{sec:algebra-equality}
We prove the algebraic equality by combining explicit coordinates with
a weighted Jacobian identity. Throughout, $n\ge2$; we treat the seed
for $n=3$ separately where necessary.

The principal and spherical localizations serve different purposes.
Put $P=\prod_{j=1}^{N-1}p_j$, $H_{\rm Pf}=c_nH_{\rm sph}$, and
$M=\prod_{i\in I_{\rm mut}}x_i$, where $x_i$ is the function attached
to the initial mutable index $i$. Table~\ref{tab:localizations} records
the coordinate rings that we shall use.

\begin{table}[htbp]
\centering\small
\renewcommand{\arraystretch}{1.16}
\begin{tabular}{@{}p{.35\linewidth}p{.61\linewidth}@{}}
\toprule
Ring & Purpose\\
\midrule
$\B_n[P^{-1}]$ & The principal torus and the polynomial section $\mathcal S_n$; a positive grading for the Jacobian argument\\
$\B_n[H_{\rm sph}^{-1}]$ & Pfaffian elimination and the reduction to invariants on $\Sp_{2n}$\\
$\B_n[H_{\rm Pf}^{-1}]$ & The $\GL_n/U_n$ factor and symmetric coordinates\\
$\B_n[(H_{\rm Pf}M)^{-1}]$ & The initial Laurent chart and the full fraction field\\
\bottomrule
\end{tabular}
\caption{The localizations used in the algebraic comparison.}
\label{tab:localizations}
\end{table}

The principal functions remain polynomial in the last ring.
The two restrictions of a homogeneous function are
\[ \bar f=f|_{\mathcal S_n},\qquad f^\sharp(k)=f(\Pi k)\quad(k\in K). \]
Here $\mathcal S_n\subset U_G$ is the section constructed in
Proposition~\ref{prop:gauss-chart}, and $\Pi$ is the permutation matrix
introduced for skew elimination. The first restriction separates the
principal torus; the second is used for the spherical localization.

Theorem~\ref{thm:regular-exchanges} has proved regularity of the first
mutations. The coordinate calculations below establish the initial
chart and the equality of fraction fields in
Proposition~\ref{thm:initial-chart}. The weighted Jacobian identity of
Proposition~\ref{thm:short-jacobian} then gives the adjacent charts and
frozen divisor valuations. Finally, Proposition~\ref{prop:comparison}
restores the frozen divisors and proves the equality in
Theorem~\ref{thm:a2c-algebra-equality}.

\subsection{The unipotent quotient and a positive grading}
Put $P=\prod_{j=1}^{N-1}p_j$ and
\[ \mathcal I_n=\{(i,j):1\le i<j\le N,\ i+j<N+1\},\qquad \mathcal S_n=I+\bigoplus_{(i,j)\in\mathcal I_n}\C E_{ij}. \]
The allowed matrix positions are closed under multiplication. We use
$\mathcal S_n$ as a section of the quotient map $U_G\to U_G/U_K$:
every coset has a unique representative in $\mathcal S_n$, depending
regularly on the coset, as the following proposition proves.

\begin{proposition}\label{prop:gauss-chart}
Multiplication induces an isomorphism of varieties
$\mathcal S_n\times U_K\longrightarrow U_G$. Consequently
\begin{equation} (\B_n)_P\cong \C[p_1^{\pm1},\ldots,p_{N-1}^{\pm1}] \otimes\C[u_{ij}:(i,j)\in\mathcal I_n]. \label{eq:polynomial-gauss-chart} \end{equation}
There are $n(n-1)$ polynomial coordinates in the second factor.
\end{proposition}
\begin{proof}
For upper unitriangular $u$, the skew matrix $A=uJu^{\mathsf t}$
has $A_{i,\bar i}=1$ for $i\le n$ and $A_{ij}=0$ for
$i<j$, $i+j>N+1$. Conversely every skew matrix with these properties
has a unique expression $vJv^{\mathsf t}$ with $v\in\mathcal S_n$.
Here is a polynomial inverse. Write
\[
A=\begin{pmatrix}0&r&1\\-r^{\mathsf t}&A'&0\\-1&0&0\end{pmatrix},
 \qquad
 v=\begin{pmatrix}1&a&0\\0&v'&0\\0&0&1\end{pmatrix}.
\]
Inductively recover $v'$ from $A'=v'J'v'^{\mathsf t}$, and set
$a=r(v'^{\mathsf t})^{-1}(J')^{-1}$. All inverses here are
unitriangular and polynomial. Thus $v\mapsto vJv^{\mathsf t}$ is
an isomorphism onto the stated affine space. For arbitrary $u$,
$k=v^{-1}u$ preserves $J$ and belongs to $U_K$, giving the claimed
product decomposition and its polynomial inverse.
Let $G_P=\{g\in G:P(g)\ne0\}$. On this big cell, Gauss decomposition
is the unique factorization $g=u_-tu_+$, with
$u_-\in U_G^-$, $t\in T_G$, and $u_+\in U_G$
\cite[Theorem~21.84]{MilneAlgebraicGroups}. It gives
$G_P\simeq U_G^-\times T_G\times U_G$, and hence
$U_G^-\backslash G_P\simeq T_G\times U_G$. The principal minors are
independent fundamental characters of $T_G$; taking right
$U_K$-invariants proves \eqref{eq:polynomial-gauss-chart}.
Counting the allowed positions gives $n(n-1)$.
\end{proof}

We call the principal open $\Spec\B_n[P^{-1}]$, with the coordinates
in \eqref{eq:polynomial-gauss-chart}, the \emph{Gauss chart}.
Its torus factor records $t$ in $g=u_-tu_+$, while $\mathcal S_n$
records the right coset of $u_+$ modulo $U_K$. Thus the chart is
$T_G\times\mathcal S_n$. This principal localization is distinct
from the mutable localization used for the initial Laurent chart.

The restriction of a homogeneous function $f$ to
$T_G\times\mathcal S_n$ is a principal Laurent character times $\bar f$. The one-parameter subgroup
\[ t\longmapsto\diag(t^{2i-N-1})_{i=1}^N\in T_K \]
acts on the section by inverse conjugation, with
\begin{equation} \deg_\circ u_{ij}=2(j-i)>0,\qquad \deg_\circ\bar f=\sum_{j=1}^{N-1}\lambda_j j(N-j) -\sum_{\ell=1}^{n}\mu_\ell\ell(N-\ell) \quad\text{if }\wt(f)=(\lambda;\mu). \label{eq:positive-section-grading} \end{equation}
In particular $\bar p_j=1$ has degree zero. The positive grading on
$\C[\mathcal S_n]$ is the reason the Jacobian has the form required
for the comparison. Differentiating an exchange gives
\[ f_i'J_{\rm Jac}=-f_iJ_{{\rm Jac},i}'. \]
Coprimality forces each mutable prime to divide $J_{\rm Jac}$.
The frozen weight sum then gives degree zero for the quotient by their
product, and positivity makes that quotient a constant. We prove this
in Lemma~\ref{lem:weighted-jacobian} after the coordinate calculations
establish algebraic independence and the polynomial companions.

\subsection{The spherical and Lagrangian quotients}
For a skew matrix $A$ with Pfaffian $1$, write
$h_r=\Pf A_{[2r]}$, $h_0=h_n=1$, and $d_r=h_r/h_{r-1}$.
Let $J_{\rm ad}$ be the adjacent-pair matrix used in
Proposition~\ref{prop:unit-frame}. Let $\Pi$ have ordered rows
$e_1^{\mathsf t},e_{\bar1}^{\mathsf t},\ldots,
 e_n^{\mathsf t},e_{\bar n}^{\mathsf t}$. Its determinant is $1$,
and $\Pi J\Pi^{\mathsf t}=J_{\rm ad}$.

\begin{lemma}[Skew elimination]\label{lem:skew-elimination}
When $h_1\cdots h_{n-1}\ne0$, there is a unique lower unitriangular
matrix $\ell_0(A)$ with $(\ell_0)_{2r,2r-1}=0$ such that
\begin{equation}
\ell_0(A)A\ell_0(A)^{\mathsf t}
 =J(\mathbf h):=\bigoplus_{r=1}^n
       \begin{pmatrix}0&d_r\\-d_r&0\end{pmatrix}.
 \label{eq:skew-elimination}
\end{equation}
Its entries are regular after the $h_r$ are inverted. The stabilizer
of $J(\mathbf h)$ in $U_G^-$ is
$U^-_{\rm pair}=\prod_r\{I+tE_{2r,2r-1}\}$.
\end{lemma}
\begin{proof}
For $A=\left(\begin{smallmatrix}E&B\\-B^\T&C\end{smallmatrix}\right)$
with invertible $2\times2$ pivot $E$, congruence by
$\left(\begin{smallmatrix}I&0\\B^\T E^{-1}&I\end{smallmatrix}\right)$
produces $E\oplus(C+B^\T E^{-1}B)$. Leading Pfaffians of the
Schur complement are $h_{r+1}/h_1$. Induction gives
\eqref{eq:skew-elimination} with only leading Pfaffian denominators.
No within-pair shear is introduced.
For a lower unitriangular stabilizer, its off-diagonal $2\times2$
blocks vanish successively: the upper-right block of the congruence
equation is $AEC^{\mathsf t}=0$, forcing $C=0$. Each diagonal
block is the free lower shear. This proves the stabilizer description.
Two normalized eliminators differ by these shears; their prescribed
zero entries force every shear parameter to vanish.
\end{proof}

Set $D(\mathbf h)=\diag(d_1,1,\ldots,d_n,1)$ and
$\gamma(\mathbf h)=D(\mathbf h)\Pi$. Since $\prod d_r=1$,
$\gamma(\mathbf h)\in\SL_N$. For $g\in G^\circ$,
\[ k(g)=\gamma(\mathbf h)^{-1}\ell_0(gJg^{\mathsf t})g\in K. \]
Put $U^-_{\rm long}=\Pi^{-1}U^-_{\rm pair}\Pi$.

\begin{proposition}[Spherical quotient]\label{prop:spherical-quotient}
There is an isomorphism
\[ \B_n[H_{\rm sph}^{-1}] \cong\C[h_1^{\pm1},\ldots,h_{n-1}^{\pm1}] \otimes\C[K]^{U^-_{\rm long}\times U_K}. \]
A homogeneous function restricts as
$f(\gamma(\mathbf h)k)=\lambda(D(\mathbf h))f^\sharp(k)$,
where $f^\sharp(k)=f(\Pi k)$.
\end{proposition}
\begin{proof}
Left multiplication by $U_G^-$ preserves every leading Pfaffian.
The two elimination matrices for $g$ and $ug$ differ, after composing
with $u$, by the stabilizer in Lemma~\ref{lem:skew-elimination}.
Consequently $k(ug)$ differs from $k(g)$ by left
$\gamma^{-1}U^-_{\rm pair}\gamma=U^-_{\rm long}$, while
$k(gv)=k(g)v$ for $v\in U_K$. The maps
$g\mapsto(\mathbf h,k(g))$ and $(\mathbf h,k)\mapsto\gamma(\mathbf h)k$
therefore induce homomorphisms of the indicated invariant rings.
For $g=\gamma(\mathbf h)k$, the skew matrix is already $J(\mathbf h)$,
so the first composite is the identity on $(\mathbf h,k)$.
The other composite sends $g$ to $\ell_0(gJg^{\mathsf t})g$.
Since $\ell_0(gJg^{\mathsf t})\in U_G^-$, it acts trivially on the
left invariants. All the formulas are regular on the indicated opens,
so the induced homomorphisms of invariant rings are inverse.
The final assertion is the left torus convention
\eqref{eq:torus-convention}.
\end{proof}

Use the symplectic basis $(e_1,\ldots,e_n,e_{\bar1},\ldots,e_{\bar n})$
and write $k=\left(\begin{smallmatrix}X&Y\\Z&W\end{smallmatrix}\right)$.
On $D(\det X)$ there is the unique factorization
\begin{equation}
k=\begin{pmatrix}I&0\\C&I\end{pmatrix}
   \begin{pmatrix}X&0\\0&X^{-\mathsf t}\end{pmatrix}
   \begin{pmatrix}I&S\\0&I\end{pmatrix},
 \qquad C=C^{\mathsf t},\quad S=S^{\mathsf t}.
 \label{eq:lagrangian-factorization}
\end{equation}
Here $C=ZX^{-1}$, $S=X^{-1}Y$, and
$W=CXS+X^{-\mathsf t}$.

\begin{proposition}[Pfaffian--Lagrangian decomposition]
\label{thm:lagrangian-ring}
Put $H_{\rm Pf}=c_nH_{\rm sph}$. Then
\begin{equation} \B_n[H_{\rm Pf}^{-1}] \cong\C[h_1^{\pm1},\ldots,h_{n-1}^{\pm1}] \otimes\C[\GL_n]^{U_n} \otimes\C[C_{ij}:1\le i<j\le n]. \label{eq:lagrangian-ring} \end{equation}
The normalized section has $S=0$, $C_{ii}=0$, and the only remaining
action is $X\mapsto XA$, $A\in U_n$. On it, we have
\[ \mathsf P(\Pi k)=X,\qquad c_n^\sharp=\det X, \qquad c_n(\gamma(\mathbf h)k)=H_{\rm sph}\det X. \]
\end{proposition}
\begin{proof}
Since $(\Pi k)J(\Pi k)^{\mathsf t}=J_{\rm ad}$, the singleton
coefficient in row $2r-1$ is exactly row $r$ of $X$. This proves
the first two identities. Left multiplication by $D(\mathbf h)$
multiplies row $r$ of $\mathsf P$ by $h_r$. Their product is
$H_{\rm sph}$, proving the last identity.

In \eqref{eq:lagrangian-factorization}, left $U^-_{\rm long}$ adds an
arbitrary diagonal matrix to $C$. Right $U_K$ consists of the matrices
$\left(\begin{smallmatrix}A&AT\\0&A^{-\mathsf t}\end{smallmatrix}\right)$,
where $A\in U_n$, $T=T^{\mathsf t}$, and acts by
\[ (X,C,S)\longmapsto(XA,C,A^{-1}SA^{-\mathsf t}+T). \]
First take invariants under the normal additive subgroup translating
$S$, and then under the independent diagonal translations of $C$.
The remaining invariant ring is
$\C[\GL_n]^{U_n}\otimes\C[C_{ij}:i<j]$.
Proposition~\ref{prop:spherical-quotient} and the exact formula for
$c_n$ give \eqref{eq:lagrangian-ring}.
\end{proof}

\subsection{An elementary flag-minor chart}
For $1\le\ell\le n$, $1\le b\le n-\ell+1$, write
\[ D_{b,\ell}=\det X_{[b,b+\ell-1],[\ell]},\qquad \mathscr D=\{D_{b,\ell}\},\qquad \mathscr D^*=\mathscr D\setminus\{D_{1,1},D_{n,1}\}. \]
These are the \emph{flag minors with consecutive rows}: they use the first $\ell$ columns and rows $b,\ldots,b+\ell-1$.

We localize at these minors except for the two extremal entries $D_{1,1}$ and $D_{n,1}$.

\begin{proposition}[Localization at flag minors]
\label{prop:solid-chart}
The minors $D_{b,\ell}$ are algebraically independent, and
\begin{equation} \C[\GL_n]^{U_n}[(D)^{-1}:D\in\mathscr D^*] =\C[D_{1,1},D_{n,1},D^{\pm1}:D\in\mathscr D^*]. \label{eq:solid-chart} \end{equation}
\end{proposition}
\begin{proof}
For $n=2$, write $a=X_{11}$, $b=X_{21}$, and $d=\det X$.
On $D(a)$ and $D(b)$, elementary column operations give the invariant
rings $\C[a^{\pm1},b,d^{\pm1}]$ and
$\C[a,b^{\pm1},d^{\pm1}]$, respectively. These opens cover
$\GL_2$, and their intersection in the common field is
$\C[a,b,d^{\pm1}]$. This proves the assertion for $n=2$.
The affine quotient allows $a=b=0$; that point does not have a
matrix representative in $\GL_2$.

Assume $n\ge3$. Right upper-unitriangular column operations give a unique representative
$X^0$ satisfying
$X^0_{ij}=0$ for $2\le i\le j$.
To normalize column $j$, solve against the preceding $j-1$ columns
in rows $2,\ldots,j$; their determinant is $D_{2,j-1}$, an allowed
unit. The first column is unchanged.

Conversely, recover $X^0$ from the prescribed values of $D_{b,\ell}$. Its first
column is $X^0_{i1}=D_{i,1}$. For $j\ge2$,
$X^0_{1j}=(-1)^{j+1}\frac{D_{1,j}}{D_{2,j-1}}$.
For $b=2,\ldots,n-j+1$, expansion of $D_{b,j}$ in column $j$
recovers $X^0_{b+j-1,j}$ with coefficient $D_{b,j-1}$.
All earlier entries of that column have already been recovered, so
this is a triangular procedure. Every divided minor belongs to
$\mathscr D^*$: for $j=2$ the denominators are the internal entries
$D_{2,1},\ldots,D_{n-1,1}$, and for $j>2$ they have size at least two.
The determinant of the constructed matrix is the prescribed unit
$D_{1,n}$.

The construction works over the Laurent polynomial ring in independent
symbols on the right of \eqref{eq:solid-chart}, and has the prescribed values of $D_{b,\ell}$. This proves their algebraic independence. For any
right-invariant regular function on $\GL_n$, evaluation at $X^0$
belongs to that Laurent polynomial ring and agrees with the original
function on the dense column-normalization open. This proves one
inclusion; the other is immediate. 
\end{proof}

\subsection{The symmetric coordinates and reflection}
On the normalized section of Proposition~\ref{thm:lagrangian-ring}, put
$x_i=\op{row}_iX$ and $z_i=\op{row}_i(CX)$.
For $2q\le\ell\le n$ define
\begin{align}
 Y^+_{q,\ell}
 &=(-1)^q\det(z_q,\ldots,z_1,x_1,\ldots,x_{\ell-q})_{[\ell]},\notag\\
 Y^-_{q,\ell}
 &=(-1)^{q\ell}\det(z_n,\ldots,z_{n-q+1},
                         x_{n-\ell+q+1},\ldots,x_n)_{[\ell]}.
 \label{eq:exact-symmetric-minors}
\end{align}
For $q=0$ these mean $D_{1,\ell}$ and $D_{n-\ell+1,\ell}$.

\begin{proposition}\label{prop:symmetric-dictionary}
For $q\ge1$ and $2q\le\ell\le n$, we have
\begin{equation} Y^+_{q,\ell}=\mathfrak f_{1-2q,\ell}^{\sharp},\qquad Y^-_{q,\ell}=\mathfrak f_{2n-2\ell+1+2q,\ell}^{\sharp}. \label{eq:exact-symmetric-dictionary} \end{equation}
At $\ell=2q$ these are $p_{2q}^\sharp$ and $p_{N-2q}^\sharp$;
at $\ell=2q+1$ they are $p_{2q+1}^\sharp$ and
$p_{N-2q-1}^\sharp$. At $\ell=n$ the two functions coincide.
\end{proposition}
\begin{proof}
On $\Pi k$, the first $n$ entries of the extended rows are
$-z_n,\ldots,-z_1,x_1,\ldots,x_n$. Lemma~\ref{lem:global-row-dictionary} gives the formula for $Y^+$.
For completeness the reflected sign can be determined exactly.
If $gJg^{\mathsf t}=J_{\rm ad}$, then
$\delta(g)=-JJ_{\rm ad}g$. Put
$a_i=(-1)^{\alpha_{2i-1}}$.
The extended odd rows of $\delta(g)$ in the first $n$ columns are
$a_i x_{n+1-i}$, by \eqref{eq:exact-hodge}; the normalized even
rows are $-a_i^{-1}z_{n+1-i}$, as follows by taking consecutive
Pfaffian pivot ratios. Thus the reflected row determinant is the
bottom determinant in \eqref{eq:exact-symmetric-minors}, multiplied by
$(-1)^{q+\binom q2+\binom{\ell-q}{2}} \prod_{i=q+1}^{\ell-q}a_i$.
For the original left index, the character in
\eqref{eq:reflection-character} is
$(-1)^{\binom\ell2}\prod_{i=q+1}^{\ell-q}a_i$.
Their product is $(-1)^{q\ell}$, since $q^2-q$ is even. The same
calculation works at $\ell=2q$, with an empty product of $a_i$ and
the rational boundary expression. This proves the second formula, with
no unspecified scalar. The principal specializations follow from
\eqref{eq:flag-specialization}. At $\ell=n$ the two indices have the
same $r,m,\ell$ and hence the same normalized function.
\end{proof}

\begin{lemma}\label{lem:symmetric-recovery}
For $q\ge1$, set $j=\ell-q+1$, $r=n-q+1$, and $i=n-\ell+q$.
Then
\[
\begin{aligned}
Y^+_{q,\ell}&=(-1)^\ell Y^+_{q-1,\ell}C_{qj}+K^+_{q,\ell},\\
 Y^-_{q,\ell}&=(-1)^\ell Y^-_{q-1,\ell}C_{ir}+K^-_{q,\ell}.
\end{aligned}
\]
The first remainder involves rows $1,\ldots,q-1$ of $C$ and the
entries $C_{qk}$ with $k>j$. The second involves columns
$n,\ldots,r+1$ and entries $C_{kr}$ with $k<i$.
Consequently all $C_{ij}$, $i<j$, are recovered by processing, for
$q=1,\ldots,\lfloor n/2\rfloor$, first
$Y^+_{q,n},Y^+_{q,n-1},\ldots,Y^+_{q,2q}$, and then
$Y^-_{q,n-1},\ldots,Y^-_{q,2q}$.
\end{lemma}
\begin{proof}
Expand $z_q=\sum_k C_{qk}x_k$. The rows $x_1,\ldots,x_{j-1}$
already occur, so terms with $k<j$ vanish. Moving $x_j$ from the
first to the last position gives coefficient
$(-1)^\ell Y^+_{q-1,\ell}$. For the bottom determinant expand
$z_r=\sum_kC_{kr}x_k$. Rows $x_{i+1},\ldots,x_n$ already occur;
the $k=i$ term is in the order for $Y^-_{q-1,\ell}$, and the
normalizing signs have ratio $(-1)^\ell$. These are the stated
recurrences.

At stage $q$, the earlier rows and later columns have already been
recovered. The $Y^+$ recursion obtains
$C_{q,n-q+1},\ldots,C_{q,q+1}$; the $Y^-$ recursion obtains
$C_{q+1,n-q+1},\ldots,C_{n-q,n-q+1}$. Each pair occurs once,
in stage $\min(i,n-j+1)$ for $C_{ij}$.
For $q=1$ the pivots are minors $D_{b,\ell}$ of size at least two, including
$c_n^\sharp$; for $q>1$ they are normalized mutable folded minors.
In particular, this recovery never divides by a principal frozen
function $p_j$.
\end{proof}

\begin{proposition}\label{thm:initial-chart}
Let $M=\prod_{i\in I_{\rm mut}}x_i$, with $M=1$ for $n=2$.
For $n=2$ and $n\ge4$, the initial functions are algebraically
independent, generate $\Frac(\B_n)$, and satisfy
\begin{equation} \B_n[(H_{\rm Pf}M)^{-1}] =\C[p_1,\ldots,p_{N-1},s_2^{\pm1},\ldots,s_{N-2}^{\pm1}, c_n^{\pm1},x_i^{\pm1}:i\in I_{\rm mut}]. \label{eq:initial-chart} \end{equation}
The same conclusions hold for the separate seed for $n=3$.
\end{proposition}
\begin{proof}
First take $n=2$ or $n\ge4$. On the normalized Lagrangian section,
the determinantal initial functions are exactly
$D_{b,\ell}$: formula \eqref{eq:solid-minor-formula} gives them,
with $D_{1,1}=p_1^\sharp$, $D_{n,1}=p_{N-1}^\sharp$ and
$D_{1,n}=c_n^\sharp$. Their number is $n(n+1)/2$.
The remaining $\binom n2$ nonspherical functions are exactly
\[ \{Y^+_{q,\ell}:1\le q\le\lfloor n/2\rfloor,\ 2q\le\ell\le n\} \ \sqcup\ \{Y^-_{q,\ell}:1\le q\le\lfloor n/2\rfloor,\ 2q\le\ell<n\}. \]
This follows from \eqref{eq:exact-symmetric-dictionary}: in each
$Y^+$ or $Y^-$ family, the functions with $\ell=2q,2q+1$ are
principal, those with $\ell>2q+1$ are the non-determinantal folded
functions, and the common level $\ell=n$ is counted once. Lemma~\ref{lem:symmetric-recovery} and
Proposition~\ref{prop:solid-chart} therefore recover every generator
of \eqref{eq:lagrangian-ring} with precisely the inversions in
\eqref{eq:initial-chart}. Undoing the spherical Laurent characters
does not change the ring. The reverse inclusion is immediate.
This also proves field generation. There are $n^2+n-1$ initial
functions, equal to the dimension given by
\eqref{eq:lagrangian-ring}, so they are algebraically independent.
The verification for $n=3$ is given below using its separately defined
seed.
\end{proof}

\medskip\noindent\textbf{The case $n=3$.}\quad
Here write $a_i=X_{i1}$, $d_{ij}=\det X_{\{i,j\},[2]}$, and
$d=\det X$. The standard flag-minor presentation in size three is
\begin{equation} \C[\GL_3]^{U_3} =\C[a_1,a_2,a_3,d_{12},d_{13},d_{23},d^{\pm1}] /(a_1d_{23}-a_2d_{13}+a_3d_{12}). \label{eq:rank-three-flag-ring} \end{equation}
It is the flag incidence relation between a vector and a two-vector,
with the determinant as an invertible coordinate; equivalently it follows from the
flag-minor straightening presentation \cite{FultonYoungTableaux}.
On the normalized Lagrangian section the chosen seed has
\[
\begin{gathered}
(x_1^\sharp,x_2^\sharp,x_3^\sharp)=(d_{23},d_{13},d_{12}),\qquad
 (p_1^\sharp,p_5^\sharp,c_3^\sharp)=(a_1,a_3,d),\\
 p_2^\sharp=C_{12}d_{12}+C_{13}d_{13},\quad
 p_3^\sharp=-C_{13}d,\quad
 p_4^\sharp=C_{13}d_{13}+C_{23}d_{23}.
\end{gathered}
\]
The first line is the determinant definition; expanding the indicated
principal minors gives the second line. Thus, when the three $x_i$
and $d$ are inverted, these formulas recover $a_2$ and all three
$C_{ij}$. They prove the assertion for $n=3$ in
Proposition~\ref{thm:initial-chart}, including independence by the
dimension $11$.

The three normalized exchange quotients are
\[ (E_1/x_1)^\sharp=dC_{23},\qquad (E_2/x_2)^\sharp=a_2,\qquad (E_3/x_3)^\sharp=dC_{12}. \]
They follow by substitution in \eqref{eq:rank-three-seed} and the
single relation in \eqref{eq:rank-three-flag-ring}. Each belongs to
$\B_3[H_{\rm Pf}^{-1}]$ after restoring its spherical character.
The three $x_i$ are nonzero by these matrix formulas and prime by the
fundamental-right-degree argument of Proposition~\ref{prop:initial-primes};
their left weights are supported on odd fundamentals. They are
nonassociate to $s_2,s_4,c_3$. Prime cancellation therefore makes
all three quotients regular in $\B_3$. This proves the exact
exchanges for $n=3$ before the spherical parameters are set to one: each relation is homogeneous, so the same
spherical character multiplies both sides.

\begin{corollary}\label{cor:seed-and-coprimality}
For every $n\ge2$ the initial datum is a normalized LP seed in
$\Frac(\B_n)$. Every first mutation is regular. For each mutable
$x$, its first mutation $x'$ satisfies $\gcd_{\B_n}(x,x')=1$.
\end{corollary}
\begin{proof}
The formal seed axioms in Proposition~\ref{prop:formal-seed} and the
corresponding verification for $n=3$ hold in the fraction
field by Proposition~\ref{thm:initial-chart}. Independence makes every
exchange polynomial nonzero. For $n\ge4$, the right degree of $x'$
is $\varpi_{\ell-1}+\varpi_{\ell+1}$ for $\ell<n$ and $2\varpi_{n-1}$ for $\ell=n$. Subtracting the right degree
$\varpi_\ell$ of $x$ is not dominant. Hence $x$ cannot divide $x'$.
For $n=3$ the mutated right degrees are respectively
$\varpi_3,\varpi_1,\varpi_3$, and the same argument applies with
$\varpi_2$. Primality of $x$ proves the assertion.
\end{proof}

\subsection{The Jacobian from the frozen weight sum}
We first prove the following weighted Jacobian identity.

\begin{lemma}\label{lem:weighted-jacobian}
Let $R=\C[z_1,\ldots,z_d]$ have positive weights $w_j$ and let
$f_1,\ldots,f_d$ be algebraically independent homogeneous polynomials.
Select pairwise nonassociate prime members $f_i$, $i\in I$. Assume
that each has a polynomial companion $f_i'$ relatively prime to $f_i$
and satisfies $f_if_i'=E_i(f_j:j\ne i)$. If
$\sum_{i\notin I}\deg f_i=\sum_jw_j$, then
\[ \det\left(\frac{\partial f_i}{\partial z_j}\right) =c\prod_{i\in I}f_i\qquad(c\in\C^\times). \]
\end{lemma}
\begin{proof}
Let $J_{\rm Jac}$ be this Jacobian, and let $J_{{\rm Jac},i}'$ be the
Jacobian obtained by replacing $f_i$ by $f_i'$ in the list of polynomials. Differentiate the exchange and wedge with
all $df_j$, $j\ne i$. Since $E_i$ is independent of $f_i$, this gives
\[ f_i'J_{\rm Jac}=-f_iJ_{{\rm Jac},i}'. \]
Coprimality implies $f_i\mid J_{\rm Jac}$, so $\prod_{i\in I}f_i\mid J_{\rm Jac}$.
The quotient has weighted degree
$\sum_{i\notin I}\deg f_i-\sum_jw_j=0$. Positivity of the coordinate
weights makes it a constant. Independence in characteristic zero makes
$J_{\rm Jac}\ne0$, so the constant is nonzero.
\end{proof}

On $\mathcal S_n$ let $\mathbf f$ be the $d=n(n-1)$ nonprincipal
initial functions, and let $\mathbf u=(u_{ij})_{\mathcal I_n}$ in a
fixed order. Each mutable restriction is prime and remains relatively
prime to its companion: localize the corresponding UFD statements at
$P$ in \eqref{eq:polynomial-gauss-chart} and remove the principal
character units. Distinct mutable restrictions remain nonassociate.

\begin{proposition}\label{thm:short-jacobian}
For every $n\ge2$, including the separate seed for $n=3$, we have
\begin{equation} \det\frac{\partial\mathbf f}{\partial\mathbf u} =c_n^{\rm Jac}\prod_{i\in I_{\rm mut}}\bar x_i, \qquad c_n^{\rm Jac}\in\C^\times. \label{eq:short-jacobian} \end{equation}
No value of $c_n^{\rm Jac}$ is needed in the comparison theorem.
\end{proposition}
\begin{proof}
The restrictions are independent by Proposition~\ref{thm:initial-chart}
and the independent principal torus factor in
\eqref{eq:polynomial-gauss-chart}. Their companions are polynomial
on $\mathcal S_n$, and each exchange polynomial is independent of the variable being mutated. The remaining $n$ functions are
$s_2,s_4,\ldots,s_{N-2},c_n$.
Formula \eqref{eq:positive-section-grading} gives
\[ \deg_\circ\bar s_{2r}=4r(n-r),\qquad \deg_\circ\bar c_n=\frac{n(n-1)(2n-1)}3. \]
Consequently
\[ \sum_{r=1}^{n-1}\deg_\circ\bar s_{2r}+\deg_\circ\bar c_n =\frac{n(n-1)(4n+1)}3 =\sum_{(i,j)\in\mathcal I_n}2(j-i). \]
Apply Lemma~\ref{lem:weighted-jacobian}. The balance is precisely the
positive one-parameter specialization of the frozen degree sum
$(2\rho_G;2\rho_K)$ after the principal torus characters are removed.
\end{proof}

Algebraic independence, regular coprime first mutations, and the frozen
weight sum determine the Jacobian divisor.

\subsection{Adjacent charts and height-one descent}
The following local criterion uses the Jacobian to identify the
adjacent Laurent charts.

\begin{lemma}\label{lem:adjacent-chart}
Let $C$ be a Laurent polynomial ring over a polynomial ring over $\C$,
let $E\in C$ be prime, and let $R$ be a normal finitely generated
$C$-domain in $\Frac(C)(x)$. Suppose
\[ x,y\in R,\quad xy=E,\quad R[x^{-1}]=C[x^{\pm1}], \]
$x$ is prime in $R$, and $y\notin(x)$. If the birational morphism
$\Spec R\to\Spec C[y]$ is \textup{\'{e}}tale at the generic point
of $V_R(x)$, then
\[ R\subseteq C[y^{\pm1}],\qquad R[y^{-1}]=C[y^{\pm1}]. \]
\end{lemma}
\begin{proof}
The contraction of $(x)$ to $C[y]$ has height one by generic
\'{e}taleness and contains the prime $E$, so it equals $(E)$.
The two height-one local rings are discrete valuation rings (DVRs)
in the same fraction field;
since $E=xy$ with $y$ a unit, they coincide. In particular
\[ \ord_x(c)=\ord_E(c)\ (c\in C\setminus\{0\}),\qquad \C(V_R(x))=\Frac(C/(E))(\bar y), \]
where $\bar y$ is transcendental over $\Frac(C/(E))$.
Write $f\in R$ in the initial chart as $\sum_j c_jx^j$ and factor
$c_j=E^{a_j}d_j$, $E\nmid d_j$. If
$m=\min_j(a_j+j)<0$, the residue of $x^{-m}f$ at $(x)$ is
\[ \sum_{a_j+j=m}\bar d_j\bar y^{a_j}\ne0. \]
The exponents $a_j$ in this sum are distinct, so nonvanishing follows
from the transcendence of $\bar y$. This would give $\ord_x(f)=m<0$,
a contradiction. Thus $E^{-j}\mid c_j$ whenever $j<0$, and replacing
$x$ by $E/y$ puts $f$ in $C[y^{\pm1}]$. The reverse inclusion after
inverting $y$ follows from $C,y\subset R$.
\end{proof}

\begin{proposition}\label{thm:localized-upper-equality}
For every $n\ge2$, we have
\begin{equation} \B_n[H_{\rm Pf}^{-1}]=\UB(\Sigma_n)[H_{\rm Pf}^{-1}]. \label{eq:localized-upper-equality} \end{equation}
Moreover the initial and all first adjacent Laurent rings, after this
localization, are coordinate rings of their corresponding principal
opens in $\Spec\B_n[H_{\rm Pf}^{-1}]$.
\end{proposition}
\begin{proof}
The initial assertion about charts is
Proposition~\ref{thm:initial-chart}. For a mutable $x$, put
\[ C=\C[x_f:f\in I_{\rm fr}][H_{\rm Pf}^{-1}, x_v^{\pm1}:v\ne x],\qquad R=\B_n[H_{\rm Pf}^{-1},x_v^{-1}:v\ne x]. \]
The ring $C$ is a Laurent polynomial ring with the principal frozen
functions not inverted. The initial chart gives
$R[x^{-1}]=C[x^{\pm1}]$. The binomial $E_x$ remains prime in $C$,
since it is an irreducible binomial with no variable factor. The prime
$x$ and its regular companion $y=x'$ are relatively prime.

The generic point of $V(x)\subset\Spec\B_n[H_{\rm Pf}^{-1}]$
belongs to $\Spec\B_n[(PH_{\rm Pf})^{-1}]$, identified by
\eqref{eq:polynomial-gauss-chart} with the corresponding open subset
of $T_G\times\mathcal S_n$. Indeed, $x$ is nonassociate to every $p_j$. Write $J_{\rm Jac}$ for the determinant
in \eqref{eq:short-jacobian}. There the principal torus characters are
independent, and the Jacobian of the adjacent nonprincipal functions
is, by differentiating the exchange,
\[ J_{{\rm Jac},x}'=-\frac{\bar x'}{\bar x}J_{\rm Jac} =-c_n^{\rm Jac}\bar x'\prod_{v\ne x}\bar x_v. \]
It is nonzero at the generic point of $V(x)$ by coprimality. Including
the principal torus coordinates only adds an invertible diagonal
factor to this Jacobian. The adjacent coordinate map is therefore
\'{e}tale there. Lemma~\ref{lem:adjacent-chart} proves
$R[(x')^{-1}]=C[(x')^{\pm1}]$, the required adjacent chart.

The initial and adjacent opens cover every height-one point of
$\Spec\B_n[H_{\rm Pf}^{-1}]$: a height-one prime missing the initial
open contains exactly one mutable prime $x$, and its coprime companion
$x'$ is a unit there. Normality identifies the ring with the
intersection of these open coordinate rings. This proves
\eqref{eq:localized-upper-equality}. For $n=2$ there is no mutable
index, and the initial chart already gives the statement.
\end{proof}

\begin{proposition}\label{prop:actual-frozen-valuations}
For every frozen function $q$, including the principal functions,
the initial-coordinate $q$-adic valuation equals $\ord_{V(q)}$ on
$\Frac(\B_n)$.
\end{proposition}
\begin{proof}
If $q=p_j$, its generic divisor point belongs to the initial chart
\eqref{eq:initial-chart}, since $p_j$ is nonassociate to all factors
of $H_{\rm Pf}M$. That chart identifies its local ring with the
ordinary coordinate $p_j$-adic DVR, proving the assertion.

For $q\in\{s_2,s_4,\ldots,s_{N-2},c_n\}$, the generic point of
$V(q)\subset\Spec\B_n$ belongs to $\Spec\B_n[P^{-1}]$,
identified with $T_G\times\mathcal S_n$, and no mutable prime
vanishes there. Thus \eqref{eq:short-jacobian} is nonzero there.
The birational map from the polynomial Gauss quotient to its initial
coordinate affine space is \'{e}tale at this generic point.
The contraction of $(q)$ therefore has height one and contains its
initial coordinate $q$, so is precisely $(q)$. The corresponding
DVRs in the common field coincide. The valuation on the initial
coordinate ring is its ordinary coordinate exponent valuation.
Inverting the principal characters changes neither valuation, so the
same conclusion holds on the full fraction field.
\end{proof}

\begin{theorem}\label{thm:a2c-algebra-equality}
For every $n\ge2$, with the separate seed specified for $n=3$, we have the equality
\begin{equation} \B_n=\UB(\Sigma_n)=\ULP(\Sigma_n). \label{eq:a2c-algebra-equality} \end{equation}
All frozen variables remain polynomial coefficients. For $n=2$ this gives
$\B_2=\C[p_1,p_2,p_3,s_2,c_2]$.
\end{theorem}
\begin{proof}
We take $H=H_{\rm Pf}$ in Proposition~\ref{prop:comparison}.
The frozen factors are pairwise nonassociate primes, and the localized
ring equality is Proposition~\ref{thm:localized-upper-equality}.
For each chosen frozen $q$, the two distinct monomials of an initial
exchange have disjoint supports, so the exchange has coordinate
$q$-valuation zero. The mutable denominator also has valuation zero;
hence $\nu_q(x_i')=0$. Proposition~\ref{prop:actual-frozen-valuations}
identifies these coordinate valuations with the divisor orders.
The comparison criterion and Theorem~\ref{thm:upper-invariance} prove
\eqref{eq:a2c-algebra-equality}.
\end{proof}

\section{Comparison with mirror algebras}
\label{sec:mechanism}
\label{sec:mirror}
We now construct bases of the upper Laurent phenomenon algebras
considered in this paper. In the applications, $R_0$ is the branching
algebra and $R$ is its localization at the frozen functions. The
mirror algebra of \cite{KeelYu} has a distinguished theta basis.
We first give an algebraic criterion under which its Laurent
expansions define an isomorphism with $R$ and determine which theta
functions belong to $R_0$. We then construct these expansions for the
toric blowups used below. The final subsection gives an alternative
criterion using Newton polytopes.

\subsection{Mirror algebras and the comparison map}
\label{sec:marking}
We recall the parts of \cite{KeelYu} used below. Let $V$ be a smooth
connected affine log Calabi--Yau variety over $\C$, containing an open
algebraic torus $T_V$. We fix a logarithmic volume form $\omega_V$
whose restriction to $T_V$ is a nonzero constant multiple of the
invariant volume form. For the non-archimedean construction, $\C$
has the trivial absolute value.

\subsubsection*{The skeleton and the coefficient ring.}
The \emph{integer points of the essential skeleton} are
\[ \Sk(V,\Z)=\{0\}\sqcup\{a\ord_E:a\in\Z_{>0},\ \omega_V\text{ has a pole along }E\}, \]
where $E$ ranges over prime divisors on birational models of $V$
\cite[Section~1 and Lemma~2.2]{KeelYu}. Choose a smooth projective
compactification $V\subset Y$ whose boundary has simple normal
crossings (snc). Let $\op{NE}(Y)\subset N_1(Y,\Z)$ be the monoid
of effective integral curve classes modulo numerical equivalence,
and put
\[ \mathscr R_Y=\C[\op{NE}(Y)]=\bigoplus_{\gamma\in\op{NE}(Y)}\C z^\gamma,\qquad z^\gamma z^{\gamma'}=z^{\gamma+\gamma'}. \]
Here $\op{NE}(Y)$ denotes an integral monoid, not the real
cone of curves.

\subsubsection*{Multiplication, trace, and specialization.}
The \emph{mirror algebra} of \cite[Theorems~1.2 and~1.6]{KeelYu},
after extension of scalars from $\Z$ to $\C$, is
\[ \Ageom(V)=\bigoplus_{p\in\Sk(V,\Z)}\mathscr R_Y\vartheta_p. \]
The basis elements $\vartheta_p$ are the \emph{theta functions}, and
this basis is the \emph{theta basis}.
We suppress the fixed compactification $Y$ in this notation. The
multiplication is
\[ \vartheta_{p_1}\vartheta_{p_2}=\sum_{p,\gamma}\chi(p_1,p_2,p,\gamma)z^\gamma\vartheta_p. \]
The nonnegative integers $\chi(p_1,p_2,p,\gamma)$ count the analytic
disks defined in \cite[Definition~1.5]{KeelYu}; the sum is finite.
This multiplication is commutative and associative, with unit
$\vartheta_0$. The \emph{trace} is the coefficient of $\vartheta_0$:
\[ \Tr\!\left(\sum_p a_p\vartheta_p\right)=a_0\in\mathscr R_Y. \]
The traces of products are the rational-curve counts of
\cite[Definition~1.1 and Theorem~1.2]{KeelYu}.
Lemma~\ref{lem:geometric-product} compares this trace with the
constant coefficient of the corresponding Laurent expansions.

For a homomorphism $\xi:\mathscr R_Y\to\C$, write
\[ \Ageom(V)_\xi=\Ageom(V)\otimes_{\mathscr R_Y,\xi}\C. \]
The subscript $1$ means $\xi(z^\gamma)=1$ for every curve class.
This specialization is defined for the algebraic coefficient ring
above, and the resulting algebra is independent of the compactification
\cite[Remark~1.3 and Section~17]{KeelYu}. We use the same notation
$\Tr$ for the specialized trace. Thus $\Ageom(V)_1$ denotes this specialization of the mirror algebra.
The theta functions are intrinsically indexed by $\Sk(V,\Z)$,
and their multiplication is independent of the open torus
\cite[Proposition~13.4]{KeelYu}. A different torus changes the coordinates,
not this algebra with its theta basis. In comparing the transposed
constructions from two seeds, we must additionally compare the resulting
varieties and their homomorphisms to $R$.

\subsubsection*{The two lattices and monomial valuations.}
The inclusion $T_V\subset V$ identifies
$\Sk(V,\Z)$ with $X_*(T_V)$
\cite[Lemma~2.2 and Section~1.1]{KeelYu}. Let $R\subseteq\C[L]$,
where $L$ is the character lattice of the original torus, and write
$\C(L)=\Frac\C[L]$. We choose an isomorphism $X_*(T_V)\simeq L$.
Thus theta functions are indexed
by $m\in L$, whereas the characters on $T_V$ are $y^h$ with
$h\in L^*:=\Hom(L,\Z)$.

For $w\in L_\R^*$, the \emph{monomial valuation} on $\C(L)$ is
\[ v_w\!\left(\sum_m c_mx^m\right)=\min_{c_m\ne0}\langle w,m\rangle,\qquad v_w(g/h)=v_w(g)-v_w(h). \]
We put $v_w(0)=+\infty$ and give nonzero constants value zero. This
is also called a Gauss valuation. In particular, $v_{e_f^*}$ is the
$x_f$-adic valuation with all other coordinates assigned value zero.
Its equality with a divisor valuation on a branching variety is a
geometric assertion, proved in Proposition~\ref{prop:actual-frozen-valuations}.
For a nonzero regular function $\psi$ on $V$, restriction to $T_V$
is a Laurent polynomial. Its \emph{tropicalization} in these coordinates is
\[ \psi=\sum_{h\in L^*}c_hy^h,\qquad \psi^{\trop}(m)=\min_{c_h\ne0}\langle h,m\rangle\quad(m\in L_\R). \]
On the original Laurent ring, we write $\CT(g)=[x^0]g$ for the
constant coefficient. Lemma~\ref{lem:geometric-product} will compare
this functional with the trace on the mirror algebra.

Our aim is to construct an algebra homomorphism
\[ \Phi:\Ageom(V)_\xi\longrightarrow R\subseteq\C[L], \]
whose value on $\vartheta_m$ is a finite Laurent polynomial with
least term $x^m$. The order on exponents is fixed in the next
subsection. We then prove that $\Phi$ is an isomorphism and that the
divisor orders select a basis of $R_0$.

\subsection{Valuation filtrations and finite subtraction}
\label{sec:finite-control}
\subsubsection*{Leading exponents and Laurent expansions.}
We fix a \emph{translation-invariant total order} $<$ on $L$: if
$m<n$, then $m+\ell<n+\ell$ for every $\ell\in L$. For a nonzero
Laurent polynomial $g=\sum_m c_mx^m$, define
\[ \lead g=\min_{<}\{m\in L:c_m\ne0\}. \]
For a rational function we set $\lead(g/h)=\lead g-\lead h$;
multiplicativity of the least terms makes this well defined. All
leading exponents below use this order. In Section~\ref{sec:binomial}
we choose a specific lexicographic order.

For the algebraic comparison, let $\mathcal A$ be a $\C$-algebra
with basis $\{\vartheta_m:m\in\Lambda\}$, where $\Lambda\subseteq L$.
We assume that an algebra homomorphism has finite expansions
\begin{equation} \Phi:\mathcal A=\bigoplus_{m\in\Lambda}\C\vartheta_m\longrightarrow R,\qquad \Theta_m:=\Phi(\vartheta_m)=x^m+\sum_{n>m}a_{m,n}x^n. \label{eq:marked} \end{equation}
In the mirror application, $\mathcal A=\Ageom(V)_\xi$ and $\Lambda=L$.
Distinct leading exponents make $\Phi$ injective. Theorems below
prove surjectivity; Lemmas~\ref{lem:geometric-expansions}
and~\ref{lem:geometric-product} construct the homomorphism for the
binomial cases.

We use the following lemma in the proofs of Theorem~\ref{thm:comparison} and Proposition~\ref{thm:newton}.
\begin{lemma}\label{lem:subtraction}
Let $E\subseteq R$ be a vector subspace. Suppose the set of leading
exponents of its nonzero elements is finite, and that
\[ 0\ne g\in E,\quad m=\lead g \quad\Longrightarrow\quad m\in\Lambda\ \text{ and }\ \Theta_m\in E. \]
Then $E$ has basis $\{\Theta_m:\Theta_m\in E\}$.
\end{lemma}
\begin{proof}
We subtract the leading coefficient of $g$ times $\Theta_{\lead g}$.
The remainder stays in $E$, and its leading exponent strictly
increases unless it is zero. There are only finitely many possible leading exponents, so the
process terminates. The least occurring exponent proves
independence of every finite family.
\end{proof}

For the boundary version, let $q:L\to\Gamma$ be a lattice grading.
Assume $R$ is graded and each $\Theta_m$ is homogeneous of degree
$q(m)$. Let $\nu_1,\ldots,\nu_s$ be integer-valued valuations of
$\Frac R$, trivial on $\C^*$, and let $\phi_f:\Lambda\to\Z$ be
functions. For prescribed integers $\mathbf d=(d_1,\ldots,d_s)$, we put
\[ R(\mathbf d)=\{g\in R:\nu_f(g)\ge d_f\text{ for every }f\},\qquad \nu_f(0)=+\infty. \]
These spaces form the \emph{filtration by the valuations}:
\[ \mathbf d'\ge\mathbf d\ \Longrightarrow\ R(\mathbf d')\subseteq R(\mathbf d),\qquad R(\mathbf d)R(\mathbf e)\subseteq R(\mathbf d+\mathbf e). \]
For $\nu_f=\ord_{D_f}$, a nonnegative $d_f$ prescribes vanishing
of order at least $d_f$ along $D_f$; a negative $d_f$ permits a pole
of order at most $-d_f$. In particular, $d_f=0$ means regularity at
the generic point of $D_f$.

\begin{theorem}\label{thm:comparison}
Assume \eqref{eq:marked}. Suppose every
leading exponent occurring in $R$ belongs to $\Lambda$, and
\begin{equation} \nu_f(\Theta_m)\ge\phi_f(m),\qquad \nu_f(g)\le\phi_f(\lead g)\quad(0\ne g\in R). \label{eq:sandwich} \end{equation}
If, for every $\gamma\in\Gamma$ and $\mathbf d\in\Z^s$, the set
\begin{equation} \{m\in\Lambda:q(m)=\gamma,\ \phi_f(m)\ge d_f\ \text{for all }f\} \quad\text{is finite}, \label{eq:finite-labels} \end{equation}
then $\Phi$ is an isomorphism. Moreover
\begin{equation}
R(\mathbf d)=\bigoplus_{\substack{m\in\Lambda\\
                      \phi_f(m)\ge d_f\ \forall f}}\C\Theta_m,
 \qquad
 \nu_f\!\left(\sum_m c_m\Theta_m\right)
       =\min_{c_m\ne0}\phi_f(m).
 \label{eq:all-thresholds}
\end{equation}
The sum in the valuation formula is finite and nonzero.
\end{theorem}
\begin{proof}
Let $m=\lead g$, and let $c$ be the coefficient of $x^m$ in $g$.
The two inequalities in \eqref{eq:sandwich} give
\[ \nu_f(\Theta_m)\ge\phi_f(m)\ge\nu_f(g). \]
Hence $g-c\Theta_m$ satisfies every valuation bound satisfied by $g$.
Its least exponent increases unless the remainder is zero.

For fixed $\gamma,\mathbf d$, we take
$E=R_\gamma\cap R(\mathbf d)$. The upper inequality puts every leading exponent
in the finite set \eqref{eq:finite-labels}; the lower inequality puts
its $\Theta_m$ back in $E$. Lemma~\ref{lem:subtraction} applies.
Every homogeneous $g$ lies in such an $E$, taking $d_f=\nu_f(g)$.
The finite homogeneous decomposition of an arbitrary element proves
spanning, and hence surjectivity of $\Phi$.

Applying both inequalities to $\Theta_m$ gives
$\nu_f(\Theta_m)=\phi_f(m)$. We now take a finite basis sum in
$R(\mathbf d)$. Its least index $m$ is its leading exponent, so the
upper inequality gives $\phi_f(m)\ge d_f$ for all $f$. Removing that
entire basis component preserves each inequality $\nu_f\ge d_f$. Induction on the
number of components gives the first formula in
\eqref{eq:all-thresholds}. For a single valuation the same argument,
with the other bounds $d_j$ chosen below all component values, gives
the second formula.
\end{proof}

The upper inequality is required for \emph{all} nonzero $g$, not only
homogeneous ones. The proof does not assume that the spaces
$R(\mathbf d)$ are graded. Homogeneity is used only to obtain spanning
in finite sets; the last argument proves the decomposition of every
space $R(\mathbf d)$.

Distinct monic leading terms alone would not suffice. The monic family
$x^n+x^{n+1}$, $n\in\Z$, spans only
$(1+x)\C[x^{\pm1}]$. Its leading-term subtraction for $1$ runs
through infinitely many exponents. The finiteness condition
\eqref{eq:finite-labels} ensures that the subtraction process terminates.

\begin{corollary}	\label{cor:locally-finite}
	Assume all the hypotheses of Theorem~\ref{thm:comparison}, except
	possibly the finiteness condition \eqref{eq:finite-labels}.
	Suppose in addition that $R=R_0[f_1^{-1},\ldots,f_s^{-1}]$,
	where $R_0$ is a $\Gamma$-graded algebra with finite-dimensional
	homogeneous components and the $f_j$ are homogeneous. Assume that
	for every $\mathbf d\in\mathbb Z^s$,
	$R(\mathbf d)=f^{\mathbf d}R_0$ and $f^{\mathbf d}=\prod_{j=1}^s f_j^{d_j}$.
	Then the finiteness condition \eqref{eq:finite-labels} is automatic.
	Consequently, $\Phi$ is an isomorphism and the formulas \eqref{eq:all-thresholds} hold.
\end{corollary}
\begin{proof}
For fixed $\gamma,\mathbf d$, every such $\Theta_m$ lies in
$(f^{\mathbf d}R_0)_\gamma$. Distinct leading terms are linearly
independent, so this finite-dimensional space contains only finitely many
such indices. This proves \eqref{eq:finite-labels} and permits application
of the theorem.
\end{proof}

We can also prove finiteness directly from the boundary forms.
\begin{proposition}\label{prop:balance}
Suppose $\phi_f(m)=\min_{h\in S_f}\pair hm$ for finite nonempty
sets $S_f\subseteq L^*$, and let $H$ have all these forms as its rows.
After choosing coordinates $L\simeq\Z^d$, let $B\in\R^{r\times d}$
have full row rank and satisfy $\ker q_\R=\op{im}B^\T$.
For $\gamma\in\Gamma_\R$ and $\mathbf d\in\R^s$, set
\[ P_{\gamma,\mathbf d}=\{m\in L_\R:q_\R(m)=\gamma,\ \langle h,m\rangle\ge d_f\text{ for }h\in S_f\}. \]
If
\begin{equation} \rank(HB^\T)=r,\qquad (HB^\T)^\T\zeta=0, \qquad\zeta>0, \label{eq:balance} \end{equation}
then every nonempty $P_{\gamma,\mathbf d}$ is bounded. For
integral data it is a rational polytope, so
\eqref{eq:finite-labels} follows for any $\Lambda\subseteq L$.
\end{proposition}
\begin{proof}
If $HB^\T a\ge0$, multiplication by $\zeta^\T$ forces
$HB^\T a=0$, and full rank forces $a=0$. Thus the recession cone
$\{v:q(v)=0,\ Hv\ge0\}$ is zero. Hence $P_{\gamma,\mathbf d}$ is
bounded when nonempty. When the defining data are integral, its
finite linear description makes it a rational polytope.
\end{proof}

The criterion depends only on the lattice, the grading, and the boundary forms.

\subsection{Laurent expansions from toric blowups}
\label{sec:binomial}
We now construct the homomorphism \eqref{eq:marked} required by the
comparison theorem. Set $d=r+s$, $L=\Z^d$, and take a normalized
oriented binomial seed
\[ B=(B_0\mid B_f)\in\Z^{r\times d},\qquad b_i=B^\T e_i,\quad E_i=x^{[b_i]_+}+x^{[-b_i]_+},\quad x_i'=E_i/x_i. \]
The last $s$ coordinates are frozen, $B_{ii}=0$, and the ordinary and
associated exchanges agree at this initial seed. Subsequent
mutations use the associated exchanges. Fix a unimodular completion
$B_p\in\mathrm{GL}_d(\Z)$ whose first $r$ rows are $B$; this completion
is a choice of lattice coordinates, not a canonical matrix.
Thus the subgroup generated
by those rows is saturated in $\Z^d$; this is stronger than full real
row rank. Relative to this completion, put
\[
 B^\vee=-\bigl(B_pe_1,\ldots,B_pe_r\bigr)^\T.
\]
We call $\Sigma(B^\vee)$ the \emph{transposed seed} when it satisfies
the LP seed axioms. Its first $r$ variables are mutable and its last
$s$ variables are frozen; in the dual construction the frozen variables
are inverted. This is an $r\times d$ matrix, not the rectangular
transpose $B^\T$; in $B_pe_i$, $e_i$ is the $i$th basis vector of
$\Z^d$. Here $e_i$ in $B^\T e_i$ is the $i$th basis vector of
$\Z^r$, whereas the coordinates of $L=\Z^d$ are indexed by
$1,\ldots,d$. Then
\[ \Pi=B^\T\N^r,\qquad q(m)=\bigl((B_p^{-\T}m)_{r+1},\ldots,(B_p^{-\T}m)_d\bigr),\qquad \ker q=B^\T\Z^r. \]
For this construction, we use the order
\[ m<n\quad\Longleftrightarrow\quad B_p^{-\T}m<_{\mathrm{lex}}B_p^{-\T}n, \]
where coordinates are compared in the order $1,\ldots,d$. Since
$B_p^{-\T}B^\T a=(a_1,\ldots,a_r,0,\ldots,0)^\T$,
we have $m<m+B^\T a$ whenever $0\ne a\in\N^r$. Thus $x^m$ is the
least term of each Laurent expansion in Proposition~\ref{prop:binomial}.

The Laurent rings used here invert the frozen variables. In the
notation of Section~\ref{sec:setup}, they are
$\mathscr L_t=L(t)[x_f^{-1}:f\in I_{\rm fr}]$. 
Let $\mathscr L_0$ denote the initial ring and $\mathscr L_i$ the
ring obtained by mutation at $i$, all inside $\C(L)$. The assumption on the original algebra is
\begin{equation} R=\bigcap_{i=0}^r\mathscr L_i, \qquad R\subseteq\mathscr L_t \quad\text{for every additional Laurent chart used below}. \label{eq:original-input} \end{equation}
These containments make $R$ graded by $q$.

\subsubsection*{The blowup construction.}
We use \cite[Construction~3.4]{GHKBirational} with lattice pair
$(L,L^*)$, vectors $(b_i,e_i^*)$, and all constants and multiplicities
equal to $1$. Its orthogonality condition is
$\langle e_i^*,b_i\rangle=0$; no skew-symmetrizability of $B$ is required.
Let $T_V=\Spec\C[L^*]$, with characters $y_j=y^{e_j^*}$.
The vectors $b_i=B^\T e_i$ are primitive, since the rows of $B$
are part of the unimodular matrix $B_p$. In the cocharacter space
$L_\R$, take the fan consisting of the zero cone and the rays
$\rho_i=\R_{\ge0}b_i$, and let $X_\Sigma$ be its toric variety.
Write $D_i^\circ$ for the open orbit of the divisor corresponding
to $\rho_i$. Its character lattice is $b_i^\perp\cap L^*$.
Since $\langle e_i^*,b_i\rangle=B_{ii}=0$, $y_i$ is an invertible
regular function on $D_i^\circ$. The \emph{blowup centre} is
\[ Z_i=\{1+y_i=0\}\subset D_i^\circ. \]
It has codimension one in $D_i^\circ$ and codimension two in
$X_\Sigma$. Blow up the disjoint centres $Z_i$ and remove the strict
transform of the toric boundary. Denote the resulting open variety by
$V^\circ$.

Near $Z_i$, choose a toric parameter $t_i$ with
$\ord_{D_i}(t_i)=1$, put $z_i=y_i$, and complete these to coordinates
with the remaining $d-2$ coordinates invertible. The ideal of $Z_i$
is $(t_i,1+z_i)$. After deleting the strict transform of $D_i$, the
blowup is the chart
\[ t_i s_i=1+z_i,\qquad z_i\ne0, \]
times a $(d-2)$-dimensional torus, as in
\cite[Lemma~3.1]{GHKBirational}. Thus regularity across the exceptional
divisor is a divisibility condition by powers of $1+z_i$.

\subsubsection*{The affine variety and its compactification.}
Assume that $V^\circ\subset V$ is a big open immersion into a smooth
affine log Calabi--Yau variety, preserving the torus volume and the
centre equations. Choose an snc compactification $V\subset Y$
satisfying \cite[Assumption~2.4]{KeelYu}. After toroidal blowups,
its essential boundary defines a smooth toric fan containing the rays
$\rho_i$ \cite[Lemma~2.2 and Assumption~2.4]{KeelYu}.
Take the subfan consisting of the zero cone and all rays, blow up the original
centres $Z_i$, and call the resulting partial toric blowup $X_{\rm bl}$.
No centre is added on the extra rays. In particular, $X_{\rm bl}$
contains $V^\circ$.

The identity on the torus gives $\varphi:Y\dashrightarrow X_{\rm bl}$.
It is an isomorphism on $V^\circ$ and near the generic point of every
pole divisor: \cite[Lemma~2.6]{KeelYu} identifies the latter with the
corresponding toric divisor, and the centres avoid these generic points.
If $Y_0$ is the isomorphism locus and $D_{\rm pole}$ is the union of
the pole divisors, then
\[
 Z=\overline{(V\setminus V^\circ)\cup(D_{\rm pole}\setminus Y_0)}\subset Y
 \quad\text{has codimension at least two in }Y.
\]
Indeed, $V\setminus V^\circ$ has codimension at least two, and no
component of $D_{\rm pole}$ is contained in $Z$.
Skeletal curves have image in $(V\cup D_{\rm pole})^{\rm an}$ and
avoid $Z^{\rm an}$ by \cite[Lemma~8.21(1)]{KeelYu}; hence they map
to $X_{\rm bl}$ in the same torus coordinates.
An unmarked toric contact then lies over an original centre $Z_i$,
just as in the proof of \cite[Lemma~21.16]{KeelYu}.
Only codimension in $Y$ is needed, not codimension inside each pole
divisor. Section~\ref{app:affine} constructs the required affine
varieties and big open immersions.

\subsubsection*{Local theta functions.}
At a generic point $Q\in L_\R$, Keel--Yu define a \emph{local theta
function} as a formal sum of characters, with coefficients counting
curves with prescribed incoming end and endpoint $Q$
\cite[Definition~20.20]{KeelYu}. The tropical image of the subtree
joining the prescribed ends is the \emph{spine}; the other attached
subtrees are \emph{twigs}
\cite[Definitions~4.12, 4.17 and Proposition~4.24]{KeelYu}.
We use regular functions on $V$ to bound the exponents in one such
expansion. The boundary functions introduced in
Section~\ref{sec:boundary} will instead determine which of these
theta functions extend over the frozen divisors of $\Spec R_0$.

\subsubsection*{The matrix of gradients.}
Let $\psi_1,\ldots,\psi_N$ be nonzero regular functions on $V$.
Choose a common generic point $Q\in L_\R$ with $Q_i>0$ for
$1\le i\le r$ such that each $\psi_\ell^{\trop}$ has a unique
minimizing exponent $h_\ell\in L^*$. Near $Q$ it is the linear
function $m\mapsto\langle h_\ell,m\rangle$. Define the matrix
$J_{\rm bnd}$ by $(J_{\rm bnd})_{\ell j}=\langle h_\ell,e_j\rangle$.
Thus $J_{\rm bnd}$ records the gradients of these tropicalizations in
the chamber containing $Q$, and $J_{\rm bnd}B^\T$ is an $N\times r$
matrix.

\begin{proposition}[Copositive criterion]
\label{prop:binomial}
Let $R$ satisfy \eqref{eq:original-input}, and let $V$ contain the
open blowup model $V^\circ$ as a big open, with the specified volume
form. Suppose
\begin{equation} a^\T B_0a\ge0\qquad(a\in\R_{\ge0}^r). \label{eq:copositive} \end{equation}
Suppose, in addition, that the regular functions $\psi_1,\ldots,\psi_N$
and their gradient matrix $J_{\rm bnd}$ satisfy
\begin{equation} \{a\in\R_{\ge0}^r:J_{\rm bnd}B^\T a\ge0\}=\{0\}. \label{eq:finite-gradient} \end{equation}
Then the Laurent expansions of the theta functions give a homomorphism
\eqref{eq:marked}
\[ \Phi:\Ageom(V)_1\hookrightarrow R,\qquad \Theta_m=x^m+\sum_{0\ne a\in\N^r}c_{m,a}x^{m+B^\T a}, \qquad m\in L, \]
with finite sums and nonnegative integer coefficients. If
$\psi\in\OO(V)\setminus0$ has unique minimizing exponent $h_\psi$ at $Q$,
every occurring exponent $e$ satisfies
\begin{equation} \pair{h_\psi}e\ge \psi^{\trop}(m),\qquad \psi^{\trop}(m)=\min_{h\in\supp \psi}\pair hm. \label{eq:slope} \end{equation}
\end{proposition}

Lemma~\ref{lem:geometric-expansions} constructs finite Laurent
expansions in $R$, and Lemma~\ref{lem:geometric-product} proves that
they respect multiplication. We then use the divisor bounds in
Section~\ref{sec:boundary} to apply Theorem~\ref{thm:comparison}.

Condition~\eqref{eq:finite-gradient} bounds the exponents in a single
theta expansion. Condition~\eqref{eq:balance} bounds the theta indices
at fixed degree and fixed divisor orders. Finiteness of the
contributing curve classes follows separately from
\cite[Lemma~20.31]{KeelYu}, as used in the proof below.

\begin{lemma}\label{lem:geometric-expansions}
Under the hypotheses of Proposition~\ref{prop:binomial}, the
local expansion of a theta function attached to each $m\in L$ is a finite sum
\[ \Theta_m=x^m+\sum_{0\ne a\in\N^r}c_{m,a}x^{m+B^\T a}, \qquad c_{m,a}\in\Z_{\ge0}. \]
Its contributing curve-class sum is finite, it belongs to the initial
Laurent ring and the Laurent rings obtained by one mutation, and it satisfies
\eqref{eq:slope}. Moreover $\Theta_0=1$, and the assignment
$\vartheta_m\mapsto\Theta_m$ is an injective homogeneous linear map
into $R$. Multiplicativity is proved in Lemma~\ref{lem:geometric-product}.
\end{lemma}
\begin{proof}
\emph{Contacts and the positive chamber.}
Let $f$ be a skeletal curve contributing to a local count.
The preceding comparison places its image in the common open of
$Y$ and $X_{\rm bl}$. On the blowup model, an unmarked contact over a toric
divisor lies above one of the centres $Z_i$. Since $y_i=-1$ there,
the corresponding end has $i$th tropical coordinate zero and
direction a positive integral multiple of $b_i$. This is the contact
calculation in the proof of \cite[Lemma~21.16]{KeelYu}; only the
centre equation and the toric parameter are used here.

Consider a twig whose root maps to $Q$. Orient its edges away from
the root. Truncate each unbounded end at a point on its final
straight segment. An end of type $i$ still has $i$th coordinate zero,
because $\langle e_i^*,b_i\rangle=B_{ii}=0$. Let $a_i$ be the sum
of the positive integral multiplicities of the type-$i$ ends. For a
finite edge $E$, let $a_E\in\N^r$ count the end multiplicities in
the component beyond $E$, and let $\ell_E\ge0$ be its length.
Balancing gives direction $B^\T a_E$ along that edge.

Write the zero-coordinate equation at every end, multiply it by the
end's multiplicity, and sum. The contribution of $E$ is
$\ell_E a_E^\T B_0^\T a_E=\ell_E a_E^\T B_0a_E$. The result is
\[ 0=\sum_i a_iQ_i+\sum_E\ell_E\,a_E^\T B_0a_E. \]
When all $Q_i>0$, a nonzero $a$ makes the first sum positive, while
\eqref{eq:copositive} makes the second nonnegative. Thus no nontrivial
twig is rooted in this chamber. At a general point of its facet
$Q_i=0$, with the other $Q_j>0$, the same identity permits only
type-$i$ ends. More generally, the change in exponent at a bend is
the sum of the incident twig directions, hence belongs to
$\Pi=B^\T\N^r$. This applies also to the infinitesimal cylinders
of \cite[Assumption~20.27]{KeelYu}.

\emph{The slope formula and the exponent bound.}
Fix a nonzero $\psi\in\OO(V)$ and a contributing curve
$f:C\to Y^{\rm an}$. Cut the incoming and outgoing annular ends
at points of their straight segments. Let $C_0$ be the remaining
compact analytic domain, including every attached component that
contains a zero of $\psi\circ f$. We choose the cuts away from
these zeros and from the tropical hypersurface of $\psi$. The image
of $C_0$ lies in $V^{\rm an}$, so $\psi\circ f$ has no pole there.
After a finite extension of the valued field, choose a semistable
model of $C_0$ with the zeros specializing to its proper components.

Let $-r_{\rm in}$ be the outgoing slope of
$-\log|\psi\circ f|$ at the incoming boundary of $C_0$, and let
$r_{\rm out}$ be its outgoing slope at the other boundary.
The slope formula for $\psi\circ f$ gives
\[ r_{\rm out}-r_{\rm in}=\deg\op{div}_0((\psi\circ f)|_{C_0})\ge0. \]
This is \cite[Lemmas~15.3 and~15.6]{KeelYu}, applied to the principal
divisor of $\psi$; the contributions at internal vertices cancel,
and all its zeros on these components are included.
Factoring the toric order at the incoming contact shows
$r_{\rm in}\ge\psi^{\trop}(m)$: the remaining regular factor may
vanish. On the outgoing segment, the unique minimizing exponent
$h_\psi$ gives $r_{\rm out}=\langle h_\psi,e\rangle$.
This proves \eqref{eq:slope}.

For $e=m+B^\T a$, the auxiliary functions $\psi_\ell$ give
\begin{equation} a\ge0,\qquad J_{\rm bnd}B^\T a\ge(\psi_\ell^{\trop}(m))_\ell-J_{\rm bnd}m. \label{eq:bend-polytope} \end{equation}
Condition~\eqref{eq:finite-gradient} makes its recession cone zero.
Thus \eqref{eq:bend-polytope} has only finitely many integral solutions.

\emph{Finiteness of the curve classes.}
The columns of $B^\T$ are independent, so $\Pi\simeq\N^r$ is a
strictly convex integral monoid. Write $\mathfrak m_\Pi$ for the
ideal of $\C[\Pi]$ generated by its nonconstant monomials, and choose
$K$ larger than $\sum_i a_i$ for every integral solution of
\eqref{eq:bend-polytope}. We have
\[ x^{B^\T a}\in\mathfrak m_\Pi^K\quad\Longleftrightarrow\quad\sum_i a_i\ge K. \]
Every possible exponent correction therefore has nonzero image in
$\C[\Pi]/\mathfrak m_\Pi^K$. The nonnegative bend condition proved
above is precisely \cite[Assumption~20.27]{KeelYu} for this monoid.
Lemma~20.31 there states that, for fixed $m$ and $Q$, only finitely
many curve classes contribute to exponent corrections whose monomials are not in
$\mathfrak m_\Pi^K$. Hence only finitely many classes
contribute to the expansion. Lemma~20.34 gives the straight term
$x^m$ with coefficient one.

\emph{Regularity in the adjacent Laurent charts.}
At a generic point of the facet $Q_i=0$, all the twig ends have type
$i$. The count therefore reduces to the single blowup chart
$t_is_i=1+y_i$, times a torus. The two-dimensional calculation used
in \cite[Lemma~21.25]{KeelYu} gives the crossing transformation
\[ T_i(x^m)=x^m(1+x^{b_i})^{m_i}. \]
At every positive integer $K'$, consistency
\cite[Proposition~20.35]{KeelYu} identifies the transformation of the
local theta function with the local count on the other side, modulo
$\mathfrak m_\Pi^{K'}$. The curves in each such count have
nonnegative integral multiplicities. Since this holds for arbitrarily
large $K'$, every coefficient of the resulting formal expansion is
nonnegative.

The rational function $T_i(\Theta_m)$ has no possible denominator
other than a power of $1+x^{b_i}$. Separate its finitely many
numerator monomials into cosets of $L/\Z b_i$ and put $z=x^{b_i}$.
Within a coset, a monomial multiple has the form $p(z)/(1+z)^a$,
where $p(z)\in\Z[z]$. If a pole at $z=-1$ remains, polynomial
division and partial fractions show that its Taylor coefficients
are, for all sufficiently large $j$, $(-1)^j$ times a nonzero
polynomial in $j$. They cannot all be nonnegative. Therefore
$(1+z)^a$ divides $p(z)$. The compatible truncations at every $K'$
justify applying this argument to the whole Taylor series.
Thus $T_i(\Theta_m)$ is Laurent.

Let $\mu_i$ be the monomial map from the adjacent exponent lattice,
with $\mu_i(e_i')=[-b_i]_+-e_i$ and the other basis vectors fixed.
The substitution of the original rational exchange is
$T_i^{-1}\circ\mu_i$. Hence the Laurentness just proved is
regularity in the $i$th adjacent chart. Equation~\eqref{eq:original-input}
puts $\Theta_m$ in $R$.

For $m=0$, the right side of \eqref{eq:bend-polytope} is zero, so
\eqref{eq:finite-gradient} forces $a=0$ and $\Theta_0=1$.
Since $q(m+B^\T a)=q(m)$, each expansion is homogeneous.
Distinct least exponents make the linear map injective.

\end{proof}

\begin{lemma}[Trace identities and multiplication]
\label{lem:geometric-product}
For the expansions of Lemma~\ref{lem:geometric-expansions},
\eqref{eq:trace} holds for products of two and three theta functions, where the trace is
the coefficient-of-$\vartheta_0$ map on the mirror algebra
$\Ageom(V)_1$, as defined in Section~\ref{sec:marking}.
Consequently the linear map of that lemma is a unital injective
homogeneous algebra homomorphism.
\end{lemma}
\begin{proof}
We first prove
\begin{equation} \CT(\Theta_{m_1}\cdots\Theta_{m_k})=\Tr(\vartheta_{m_1}\cdots\vartheta_{m_k}),\qquad k=2,3. \label{eq:trace} \end{equation}
Fix the indices $m_j$. Lemma~\ref{lem:geometric-expansions} bounds the
classes contributing to their local functions, and
\cite[Lemma~3.8 and Theorem~1.2]{KeelYu} bounds the classes in the
rational-curve trace count. Choose an ample divisor on $Y$. The
classes of components obtained by cutting or degenerating these
curves are effective, and their degrees are bounded by the degree
of the whole curve. The auxiliary toric tails are specified by the
finitely many end exponents, so their degrees are bounded as well.
We can therefore fix a common bound on curve degrees and the finite
collection of walls needed for the local functions, the trace curves,
and the effective decompositions used in gluing
\cite[Section~7 and Theorems~11.9, 12.4]{KeelYu}.

Choose a small ball around the generic endpoint $Q$, contained in the
positive chamber and disjoint from the walls and fan boundaries for
this bound. Cut the incoming branches where they enter this ball.
For three inputs, the central four-pointed domain has an internal
edge; choose its length positive and sufficiently small that all
central vertices stay in the ball. This is possible uniformly
because the possible terminal exponents are finite.
The central curve is then toric. Its balancing condition is
$\sum_j e_j=0$, and its count is one by
\cite[Lemma~12.3]{KeelYu}. Deformation invariance and gluing
\cite[Theorems~11.9 and~12.4]{KeelYu} identify the product of the
local counts, subject to this balancing condition, with the trace
count. Evaluation multiplicities and the prescribed toric tails are
included in both counts. This gives \eqref{eq:trace} for nonzero
indices and nonzero internal direction.

Here these restrictions have the following elementary interpretation.
A nonzero index cannot have terminal exponent zero at the chosen
endpoint. A straight spine has the index as its terminal exponent;
otherwise its last bend, followed by a constant segment, would give
a nontrivial twig rooted at $Q$, which the positive-chamber calculation
excludes. Thus $\CT(\Theta_m)=0$ for $m\ne0$, while $\Theta_0=1$.
If all three indices are nonzero and $e_1+e_2+e_3=0$, then
$e_1+e_2=-e_3\ne0$. The internal edge used above therefore has
nonzero direction. If one input is $\vartheta_0=1$, the identity
reduces to the two-input identity; if two inputs are the unit, it
reduces to $\CT(\Theta_m)=\Tr(\vartheta_m)$. These cases complete
the proof of \eqref{eq:trace}.

It remains to pass from trace identities to multiplication. This is
the two- and three-point recognition principle of
\cite[Proposition~3.5]{Johnston}; the following leading-term argument
proves the version needed here, without presupposing that the
expansions span $R$. If $G\ne0$ is Laurent with least exponent $v$, the form of the leading term gives
\[ \CT(G\Theta_{-v})=[x^v]G\ne0: \]
any other contribution would require an exponent $v-B^\T a<v$ in
$G$. Apply this to the difference between the product of two local
functions and the expansion of their product in the mirror algebra.
The two- and three-point identities \eqref{eq:trace} make this
difference pair to zero with every $\Theta_u$, so it is zero.
This proves multiplicativity. Unitality, injectivity, and homogeneity
were established in Lemma~\ref{lem:geometric-expansions}.
\end{proof}

\begin{proof}[Proof of Proposition~\ref{prop:binomial}]
Lemma~\ref{lem:geometric-expansions} constructs the Laurent
expansions in the chosen coordinates and their slope bounds. Lemma~\ref{lem:geometric-product}
identifies their multiplication with that of the mirror algebra.
These are exactly the asserted realization and inequalities.
\end{proof}

If $J_{\rm bnd}m=(\psi_\ell^{\trop}(m))_\ell$, then
\eqref{eq:bend-polytope} gives $\Theta_m=x^m$.
Useful sufficient conditions for \eqref{eq:finite-gradient} are
$J_{\rm bnd}B^\T=-I_r$, a regular product whose gradient is strictly
negative on every $b_i$, or full column rank of $J_{\rm bnd}B^\T$
together with a strictly positive linear relation among its rows.

\subsubsection*{The $A_9\downarrow C_5$ example.}
The seed in Figure~\ref{fig:a9c5-seed} need not itself satisfy
\eqref{eq:copositive}. The single mutation described
in Section~\ref{subsec:a9-geometric} gives
$a^\T B_0a=2a_{14}a_{15}$; the regular functions of
Section~\ref{app:bounds} then give $J_{\rm bnd}B^\T=-I_{15}$.
These verify the two numerical conditions in the Laurent criterion.
Section~\ref{app:affine} identifies the affine model and
Lemma~\ref{lem:geometric-product} proves multiplicativity.

\subsection{Divisor inequalities and theta bases}
\label{sec:boundary}
We index the $s$ boundary valuations by the frozen positions
$r<f\le d$ and set
\[ \nu_f=v_{e_f^*},\qquad \nu_f\!\left(\sum_m c_mx^m\right)=\min_{c_m\ne0}m_f. \]
We now construct functions on $V$ corresponding to the frozen divisors
of $\Spec R_0$. Choose nonzero regular $\Psi_f\in\OO(V)$ with
unique minimizing exponent $e_f^*$ at the point $Q$ used above. Put
\begin{equation} S_f=\supp\Psi_f\subseteq L^*,\qquad \phi_f(m)=\min_{h\in S_f}\pair hm,\qquad H=(h)_{f,h\in S_f}. \label{eq:boundary-data} \end{equation}
The sets $S_f$ are arbitrary finite subsets of $L^*$; their elements
need not have nonnegative coordinates. Proposition~\ref{prop:binomial}
gives $\nu_f(\Theta_m)\ge\phi_f(m)$.

For each of the five copositive cases, we prove regularity of the
boundary functions by the following divisibility criterion. The
criterion applies to arbitrary finite Laurent supports.

\begin{lemma}[Regularity along the exceptional divisors]
\label{lem:center}
Let $\psi=y^h\sum_{k\in K}c_k y^k$, where $h\in L^*$ and
$K\subseteq\Z^r$ is finite. Fix $i$. Group the terms by their
exponents in the variables other than $y_i$:
\[ \sum_{k\in K}c_k y^k=\sum_{\widehat k}\left(\prod_{j\ne i}y_j^{k_j}\right)p_{i,\widehat k}(y_i),\qquad p_{i,\widehat k}(T)=\sum_{(\widehat k,a)\in K}c_{(\widehat k,a)}T^a. \]
Here $\widehat k=(k_j)_{j\ne i}$ and $(\widehat k,a)$ places $a$
in the $i$th position. Each $p_{i,\widehat k}$ is a one-variable
Laurent polynomial. Put
\[ \delta_{i,\widehat k}=-\langle h,b_i\rangle-\sum_{j\ne i}B_{ij}k_j. \]
If $(1+T)^{[\delta_{i,\widehat k}]_+}$ divides
$p_{i,\widehat k}(T)$ in $\C[T^{\pm1}]$ for every $i,\widehat k$,
then $\psi$ is regular on $V$. For $h=e_f^*$, the required exponent is
$\delta_{i,\widehat k}=-B_{if}-\sum_{j\ne i}B_{ij}k_j$.
\end{lemma}
\begin{proof}
Each group of terms has order $-\delta_{i,\widehat k}$ along the
toric divisor $D_i$. This order is independent of the power of $y_i$
because $B_{ii}=0$. In the blowup chart $t_i s_i=1+y_i$, a zero of
order $a$ of $p_{i,\widehat k}$ at $y_i=-1$ contributes the factor
$t_i^a s_i^a$. Divisibility by
$(1+y_i)^{[\delta_{i,\widehat k}]_+}$ therefore removes the possible
pole along $t_i=0$. This proves regularity on every blowup chart
and on the torus. Their union is $V^\circ$, and normality extends
regularity across $V\setminus V^\circ$, of codimension at least two.
\end{proof}

For example, $T^{-1}+2+T=T^{-1}(1+T)^2$ satisfies the condition
when $\delta_{i,\widehat k}=2$. Divisibility is tested in the Laurent
polynomial ring; thus the same divisibility criterion applies to
negative powers of $T$.

In the applications, we use regular functions of the form
\[ \Psi_f=y_f\sum_{k\in K_f}\prod_{i=1}^r y_i^{k_i},\qquad K_f\subseteq\{0,1\}^r. \]
Their full supports in \eqref{eq:boundary-data} are
\[ S_f=\left\{e_f^*+\sum_{i=1}^r k_i e_i^*:k\in K_f\right\}. \]
The construction in Section~\ref{app:boundary-data} gives
$0\in K_f$ and verifies the divisibility in Lemma~\ref{lem:center}.
It follows that $e_f^*$ is the unique minimizing exponent of
$\Psi_f$ whenever all mutable coordinates of $Q$ are positive.

\begin{lemma}[Valuation bounds in mutated coordinates]
\label{lem:normal}
Suppose an original LP mutation sequence leads to a chart with coordinate
leading exponents given by the rows of $\mathsf L_t\in\mathrm{GL}_d(\Z)$.
Assume no frozen variable divides any ordinary exchange used along
the sequence. Mutations use the
associated exchange Laurent polynomials. Then
\[ \nu_f(g)\le\pair{\mathsf L_t^{-1}e_f}{\lead g} \qquad(0\ne g\in\mathscr L_t). \]
If every $h\in S_f$ occurs as $\mathsf L_t^{-1}e_f$ for a Laurent chart
containing $R$, the upper inequality in \eqref{eq:sandwich} follows.
\end{lemma}
\begin{proof}
A monomial valuation is also monomial in coordinates obtained by
$x_i\mapsto F/x_i$, with $F$ independent of $x_i$: on the associated graded fraction
field this replaces one transcendental coordinate by its inverse
times the nonzero initial form $\op{in}(F)$ of $F$ for this valuation.
This remains an invertible
rational substitution even when several monomials of $F$ have the same valuation.
The hypothesis on frozen factors and the fact that associated
denominators use only mutable variables show inductively that the mutable values
stay zero and the frozen value vector stays $e_f$.

Write $g$ in the Laurent chart. Its distinct Laurent monomials have
distinct initial leading exponents because $\mathsf L_t$ is unimodular.
Let $u$ be the exponent of the monomial whose expansion has the least
exponent in the initial coordinates. Then $m=\lead g=\mathsf L_t^\T u$.
Since the valuation is monomial in the chart coordinates, its value
on $g$ is the minimum of the frozen exponents, and is at most $u_f$.
The latter equals
$\pair{\mathsf L_t^{-1}e_f}m$. Taking all the resulting normal vectors proves the claim.
\end{proof}

The mutation sequence is applied to the original rational functions.
Some sequences for the seeds with four mutable variables pass through
trinomial exchanges. We keep their associated denominators and all
terms of minimum valuation
\cite[Section~2]{LamPylyavskyy}.

\begin{corollary}\label{cor:application}
Assume the hypotheses of Proposition~\ref{prop:binomial}. Suppose the
regular boundary functions \eqref{eq:boundary-data} have the original
charts satisfying Lemma~\ref{lem:normal}, and satisfy
\eqref{eq:balance}. Then
\[
\Ageom(V)_1\xrightarrow{\ \sim\ }R,\qquad
 R(\mathbf0)=\bigoplus_{\substack{m\in\Z^d\\Hm\ge0}}\C\Theta_m.
\]
Equation~\eqref{eq:all-thresholds} gives every space $R(\mathbf d)$,
for $\mathbf d\in\Z^s$, and the minimum formula for each frozen
valuation.
For fixed intrinsic degree and fixed valuation bounds, the set of
real exponents is a bounded rational polytope.
\end{corollary}
\begin{proof}
Proposition~\ref{prop:binomial} supplies the homomorphism \eqref{eq:marked}
and the lower bounds. Lemma~\ref{lem:normal} supplies the upper bounds,
and Proposition~\ref{prop:balance} supplies the required finiteness.
Apply Theorem~\ref{thm:comparison}.
\end{proof}

\subsubsection*{Regularity along the frozen divisors.}
Let $R_0$ be normal and let $\Spec R\subseteq\Spec R_0$ be the
complement of the frozen prime divisors, whose orders are $\nu_f$.
Regularity at all height-one points gives $R_0=R(\mathbf0)$.
If the frozen units $x_f$ satisfy $\nu_f(x_g)=\delta_{fg}$, then
$R(\mathbf d)=\left(\prod_f x_f^{d_f}\right)R(\mathbf0)$.

\subsubsection*{Polynomial frozen coefficients.}\label{app:polynomial-frozen}
For the applications, this conclusion also follows directly from the
finite Laurent intersection. Let $L_i$, $0\le i\le r$, be the initial
and adjacent Laurent rings with coefficient ring
$\C[f_1,\ldots,f_s]$. Set
$U=\bigcap_{i=0}^rL_i$ and $R=U[(f_1\cdots f_s)^{-1}]$.
Assume every initial exchange is nonzero modulo each $f_j$.
Its reduction then remains a birational change of the mutable
coordinate, so the $f_j$-coordinate valuation is the same in all
initial and adjacent charts. Since the intersection is finite,
\[ U[(\textstyle\prod f_j)^{-1}]=\bigcap_{i=0}^rL_i[(\textstyle\prod f_j)^{-1}]. \]
Nonnegative values of all frozen valuations put the coefficients back
in $\C[f_1,\ldots,f_s]$ in every chart. Thus $R(\mathbf0)=U$ and
$R(\mathbf d)=f^{\mathbf d}U$. Contraction from a Laurent ring also
shows that the frozen ideals are distinct primes.

\subsubsection*{Polyhedral consequences.}
If the coordinate weight matrix $W$ has $BW=0$ and rank $s$, then
$\ker W^\T=B^\T\R^r$. Consequently
\[ \{m\in\R^d:Hm\ge0,\ W^\T m=\gamma\} \]
is bounded, and its integral points count the corresponding
homogeneous basis elements.  For the supports $K_f\subseteq\{0,1\}^r$ above, each normal is
primitive and has frozen part $e_f^*$. A conical dependence on other normals
would have to be a convex dependence within that same frozen group;
a vertex of $[0,1]^r$ cannot be a convex combination of its other
vertices. Thus these inequalities are irredundant. Strict feasibility
is given by mutable part zero and frozen part all ones.
The calculations in Section~\ref{app:boundary-data} give
$\rank H=d$, so the lineality space $\ker H$ is zero and each cone
is pointed.


\subsection{Newton polytopes and coefficient specialization}
\label{sec:polyhedral-newton}
We now give an alternative comparison which uses neither a grading
nor boundary inequalities. For a lattice polytope $P\subseteq L_\R$, put
$R_P=\{g\in R:\supp g\subseteq P\}$, including zero.

\begin{proposition}\label{thm:newton}
Assume \eqref{eq:marked}. Suppose that for every nonzero $g\in R$,
\begin{equation} m=\lead g\in\Lambda,\qquad \Newt(\Theta_m)\subseteq\Newt(g). \label{eq:newton-control} \end{equation}
Then $\Phi$ is an isomorphism and
\begin{equation} R_P=\bigoplus_{\Newt(\Theta_m)\subseteq P}\C\Theta_m. \label{eq:newton-filtration} \end{equation}
For every finite nonzero sum,
\begin{equation} \Newt\!\left(\sum_m c_m\Theta_m\right) =\op{conv}\!\bigcup_{c_m\ne0}\Newt(\Theta_m). \label{eq:newton-minimum} \end{equation}
In particular every monomial valuation $v_w$ defined in
Section~\ref{sec:marking} satisfies the minimum formula on this basis.
\end{proposition}
\begin{proof}
For $E=R_P$, the possible leading exponents lie in the finite set $P\cap L$,
and \eqref{eq:newton-control} puts their functions back in $R_P$.
Lemma~\ref{lem:subtraction} proves \eqref{eq:newton-filtration} and,
since every Laurent polynomial has a Newton polytope, proves spanning.
For a finite sum $g$, take $P=\Newt(g)$. Uniqueness of the basis
expansion and \eqref{eq:newton-filtration} put every occurring
$\Newt(\Theta_m)$ inside $P$. The opposite containment is immediate
from the support of a sum. Finally minimize any linear form on both
sides of \eqref{eq:newton-minimum}.
\end{proof}

Theorem~\ref{thm:comparison} applies finite subtraction in a fixed
degree and with fixed bounds $\nu_f\ge d_f$; Proposition~\ref{thm:newton}
applies it inside a fixed Newton polytope. The five copositive applications use the first criterion.

\subsubsection*{Coefficients and specialization.}
The subtraction proofs also work over a domain $D$: replace vector
spaces by $D$-submodules and require valuations to be trivial on
$D\setminus0$. Every function is monic, so subtraction divides by no
coefficient. Applied to a homomorphism into
$R\otimes_\C D\subseteq D[L]$, they give a direct-sum $D$-basis.
For any specified homomorphism $D\to D'$, the based isomorphism then
specializes to $R\otimes_\C D'$; the distinct monic leading terms
survive. No assertion that all Newton supports stay unchanged under
specialization is needed.

The algebraic ring $\mathscr R_Y$ of Section~\ref{sec:marking} has
the specialization $z^\gamma\mapsto1$. For a construction over a
completed coefficient ring, this must be checked separately. For
example, $x^m+(1-t)^{-1}x^n$ is finite as a Laurent polynomial in $x$,
but there is no homomorphism $\C[[t]]\to\C$ sending $t$ to $1$:
the unit $1-t$ would map to zero.
For curve lattice $\Gamma_{\rm curv}$, the local ring
$D_{\rm id}=\C[\Gamma_{\rm curv}]_{ (t^\gamma-1:\gamma\in\Gamma_{\rm curv})}$
does carry the identity specialization. Descent of the
normalized functions to this ring is a sufficient, separate
hypothesis. The five copositive applications use the algebraic specialization.
Lemma~\ref{lem:geometric-expansions} proves that the corresponding
local expansions are finite.

\part{Categorical constructions and branching algebras}
\section{Categorical construction for \texorpdfstring{$A_{2n-1}\downarrow C_n$}{A2n-1 to Cn}}
\label{sec:categorical}
\subsection{The branching presentation category}
\label{subsec:original-presentation-category}
Fix $n\ge4$. Label the vertices of $A_{2n-1}$ by
\[
 V_n=\{1_{\mathrm l},\ldots,(n-1)_{\mathrm l},n,
              (n-1)_{\mathrm r},\ldots,1_{\mathrm r}\}.
\]
The vertex $\ell_{\mathrm l}$ or $\ell_{\mathrm r}$ has parity $\ell$;
all arrows point from even vertices to adjacent odd vertices:
\[
\begin{array}{ll}
\vcenter{\hbox{$\xymatrix@C=1.2em{
1_{\mathrm l}&2_{\mathrm l}\ar[l]\ar[r]&\cdots&
(n-1)_{\mathrm l}\ar[l]\ar[r]&n&
(n-1)_{\mathrm r}\ar[l]\ar[r]&\cdots&
2_{\mathrm r}\ar[l]\ar[r]&1_{\mathrm r}
}$}}&n\text{ odd},\\[1.1em]
\vcenter{\hbox{$\xymatrix@C=1.2em{
1_{\mathrm l}&2_{\mathrm l}\ar[l]\ar[r]&\cdots\ar[r]&
(n-1)_{\mathrm l}&n\ar[l]\ar[r]&
(n-1)_{\mathrm r}&\cdots\ar[l]&
2_{\mathrm r}\ar[l]\ar[r]&1_{\mathrm r}
}$}}&n\text{ even}.
\end{array}
\]
Write $\bar v$ for the vertex obtained by interchanging $\mathrm l,\mathrm r$,
with $\bar n=n$, and put
\[
 \iota(\ell_{\mathrm l})=\ell,\qquad
 \iota(\ell_{\mathrm r})=2n-\ell,\qquad \iota(n)=n.
\]
This identifies the vertex order with the fundamental-weight order of
Part~I. We abbreviate $\omega_v=\omega_{\iota(v)}$ and write
$\ell(v)=\ell$ for $v=\ell_{\mathrm l},\ell_{\mathrm r}$, with $\ell(n)=n$.
We work with left $\C A_{2n-1}$-modules and set
\begin{equation}
 \mathscr C_n=\Ch_2(\proj\C A_{2n-1}).
 \label{eq:original-presentation-category}
\end{equation}
An object is a two-term presentation
$d_X:P_+(X)\longrightarrow P_-(X)$; the positive term is the
\emph{domain}. A morphism $a:X\to Y$ is a commutative square:
\[
\begin{gathered}
 a_+:P_+(X)\longrightarrow P_+(Y),\qquad
 a_-:P_-(X)\longrightarrow P_-(Y),\qquad d_Ya_+=a_-d_X.
\end{gathered}
\]
The exact structure consists of the sequences split exact in each of
the two projective terms. In particular, $\mathscr C_n$ is the category
of presentations itself, not its homotopy category
\cite[Section~2.2]{FeiTensor}. In the right-module notation of that
reference, our left-module convention uses the opposite quiver.

For $v\in V_n$, let $P_v$ and $J_v$ be the indecomposable projective
and injective. In the categorical sections, $\tau$ denotes
Auslander--Reiten translation on modules, with its extension to the
presentation indexing specified below. It is unrelated to the
bordered-Pfaffian coefficient $\tau(R;I)$ of \eqref{eq:tau-definition}.
Put
\[
 O_v^+=(P_v\to0),\qquad O_v^-=(0\to P_v),\qquad
 \Id_v=(P_v\xrightarrow{1}P_v).
\]
The neutral presentation $\Id_v$ is a nonzero object; it is not the
identity morphism of $O_v^+$. Write
$f(M):P_1(M)\to P_0(M)$ for the minimal projective presentation of a
module. The indecomposables of $\mathscr C_n$ are the positive and
neutral presentations and the minimal presentations of indecomposable
modules, including $O_v^-=f(P_v)$.

In the vertex order above, the Cartan matrix $C_A$ and Coxeter
transformation $\Phi_A$ satisfy
\[
 C_Ae_v=\begin{cases}
 e_v,&v\text{ odd},\\
 e_v+\sum_{w\sim v}e_w,&v\text{ even},
 \end{cases}
 \qquad \Phi_A=-C_A^{\mathsf t}C_A^{-1},
\]
where $w\sim v$ means that the vertices are adjacent. Every
indecomposable module is supported on an interval of this chain.
If $d$ is its dimension vector and $u$ its top multiplicity vector, then
\begin{equation}
 P_0(M)=P(u),\qquad P_1(M)=P(u-C_A^{-1}d),\qquad
 \underline{\dim}\tau M=\Phi_Ad
 \quad(M\text{ nonprojective}),
 \label{eq:categorical-projective-calculation}
\end{equation}
where $P(\mathbf m)=\bigoplus_{w\in V_n}P_w^{m_w}$. The top consists of the
interval vertices receiving no arrow from another vertex in the interval.
The injective $J_v$ is simple for even $v$; for odd $v$ it is supported
on $v$ and its neighbours.

We use $\tau O_v^+=f(J_v)$ and
$\tau^tO_v^+=f(\tau^{t-1}J_v)$ for $1\le t\le n$.
The Coxeter calculation lists every interval module once and gives
$\tau^nO_v^+=O_{\bar v}^-$. To index the presentations by the same
pairs $(a,\ell)$ as the functions $F_{a,\ell}$, put
\[
 t_\ell(a)=\frac{a-1}{2}+\left\lfloor\frac\ell2\right\rfloor,
 \qquad L_\ell=n-(\ell\bmod2),\qquad h=\left\lfloor\frac n2\right\rfloor.
\]
For odd $a$ and $1\le\ell\le n$, define
\begin{equation}
\begin{aligned}
 X^{\mathrm l}_{a,\ell}&=\tau^{t_\ell(a)}O_{\ell_{\mathrm l}}^+,
 &X^{\mathrm r}_{a,\ell}&=\tau^{L_\ell-t_\ell(a)}O_{\ell_{\mathrm r}}^+
 &&(\ell<n),\\
 X^{\mathrm l}_{a,n}&=\tau^{h-(a-1)/2}O_n^+,
 &X^{\mathrm r}_{a,n}&=\tau^{h+(a-1)/2}O_n^+
 &&(a\ge1).
\end{aligned}
 \label{eq:original-ar-coordinates}
\end{equation}
Each expression is used when its translation exponent lies in $[0,n]$.
At $a=1$ on the central orbit the two expressions coincide; write
$X_{1,n}=\tau^hO_n^+$. At mutable indices, the superscripts distinguish the two presentations
in one degree fibre, not two different variables.

\subsubsection*{Morphism arrows and translation arrows.}
For indecomposable $X,Y$, we put
\[
 \Irr_{\mathscr C_n}(X,Y)=
 \rad_{\mathscr C_n}(X,Y)/\rad_{\mathscr C_n}^{\,2}(X,Y).
\]
The square of the radical is computed through \emph{all} objects of
$\mathscr C_n$. A solid arrow has multiplicity
$\dim_\C\Irr_{\mathscr C_n}(X,Y)$. For adjacent vertices $o,e$, with
$o$ odd and $e$ even, \cite[Proposition~3.6]{FeiTensor} gives
\begin{align}
 \tau^tO_o^+&\longrightarrow\tau^tO_e^+ &&(0\le t\le n),\notag\\
 \tau^{t+1}O_e^+&\longrightarrow\tau^tO_o^+ &&(0\le t<n).
 \label{eq:ambient-solid-arrows}
\end{align}
The remaining solid arrows are
\begin{align}
 \tau O_e^+&\longrightarrow\Id_e\longrightarrow O_e^+
 &&(e\text{ even}),\notag\\
 O_{\bar o}^-&\longrightarrow\Id_{\bar o}\longrightarrow\tau^{n-1}O_o^+
 &&(o\text{ odd}).
 \label{eq:ambient-neutral-arrows}
\end{align}
All these multiplicities are one. A dashed arrow records translation,
not an additional irreducible morphism. In the function-compatible indices,
\begin{equation}
\begin{aligned}
 X^{\mathrm l}_{a,\ell}&\dashrightarrow X^{\mathrm l}_{a+2,\ell},&
 X^{\mathrm r}_{a,\ell}&\dashrightarrow X^{\mathrm r}_{a-2,\ell}
 &&(\ell<n),\\
 X^{\mathrm l}_{a,n}&\dashrightarrow X^{\mathrm l}_{a-2,n},&
 X^{\mathrm r}_{a,n}&\dashrightarrow X^{\mathrm r}_{a+2,n}
 &&(a\ge3\text{ on the left},\ a\ge1\text{ on the right}).
\end{aligned}
 \label{eq:ambient-translation-arrows}
\end{equation}
These formulas apply wherever both presentations are defined; at the
central meeting point use $X^{\mathrm l}_{1,n}=X^{\mathrm r}_{1,n}=X_{1,n}$.
Equivalently, translation sends $\tau^tO_v^+$ to $\tau^{t+1}O_v^+$.
Both solid and dashed arrows enter the exchange products.

For $P_+(X)=\bigoplus_{w\in V_n}P_w^{m_w^+(X)}$, define
\begin{equation}
 \begin{aligned}
 \bwt(\tau^tO_v^+)=
 \left(\sum_wm_w^+(\tau^tO_v^+)\omega_w;\varpi_{\ell(v)}\right),\qquad \bwt(\Id_v)=(\omega_v;0).
 \end{aligned}
 \label{eq:categorical-bidegree}
\end{equation}
It extends additively under direct sums. The first component records the
positive projective term; the second records the folded translation orbit.

\begin{definition}\label{def:branching-presentation-category}
The branching presentation category $\mathscr C_n^{\mathrm{br}}$ is the full additive subcategory
of $\mathscr C_n$ generated by
\begin{equation}
 \{\tau^tO_v^+:v\in V_n,\ 0\le t<n\}
 \ \cup\ \{\Id_v:v\in V_n\text{ even}\}.
 \label{eq:branching-subcategory}
\end{equation}
Equivalently, omit the negative presentations and the odd neutral
presentations from $\ind\mathscr C_n$. Its morphisms and composition
are those of the original presentation category:
\[
 \Hom_{\mathscr C_n^{\mathrm{br}}}(U,V)
       =\Hom_{\mathscr C_n}(U,V).
\]
For an initial seed variable $x$, mutable or frozen, its
\emph{full degree fibre} is
\begin{equation}
 \cF_n(x)=\{X\in\ind\mathscr C_n:\bwt(X)=\wt(x)\}.
 \label{eq:definition-full-degree-fibre}
\end{equation}
Proposition~\ref{prop:categorical-fibres} shows that these fibres
partition $\ind\mathscr C_n^{\mathrm{br}}$.
\end{definition}

Indecomposables are taken up to isomorphism. The term
\emph{degree-fibred} refers to their partition by degree; all morphisms
remain the original commutative squares. No objects are identified,
and no quotient or orbit category is taken.

\subsection{The degree fibres}
\label{subsec:branching-fibred-category}
The frozen fibres, with $j=\iota(v)$, are
\begin{align}
 \cF_n(c_n)&=\{X_{1,n}\},&
 \cF_n(s_j)&=\{\Id_v\}\quad(v\text{ even}),\notag\\
 \cF_n(p_j)&=
 \begin{cases}
 \{O_v^+\},&v\text{ even},\\
 \{O_v^+,\tau^{n-1}O_{\bar v}^+\},&v\text{ odd}.
 \end{cases}
 \label{eq:principal-fibres}
\end{align}
These fibres supply the frozen columns of the exchange matrix; no
marked representative is needed for them.

\begin{proposition}\label{prop:categorical-fibres}
For $\ell<n$ and every mutable index $(a,\ell)$,
\begin{equation}
 \cF_n(F_{a,\ell})=
 \{X^{\mathrm l}_{a,\ell},X^{\mathrm r}_{a,\ell}\},
 \qquad 1\le t_\ell(a)<L_\ell.
 \label{eq:lower-fibres}
\end{equation}
At the central level,
\begin{equation}
 \cF_n(F_{1+2d,n})=\{X^{\mathrm l}_{1+2d,n},X^{\mathrm r}_{1+2d,n}\},
 \qquad 1\le d<h.
 \label{eq:middle-fibres}
\end{equation}
Together with the frozen fibres in \eqref{eq:principal-fibres},
these fibres partition $\ind\mathscr C_n^{\mathrm{br}}$.
\end{proposition}
\begin{proof}
Use the periodic vectors $\epsilon_j$ of
\eqref{eq:antiperiodic-dictionary}, now in the ordered basis of projective
multiplicities. For $1\le\ell<n$ and $1\le t_\ell(a)<L_\ell$,
\eqref{eq:categorical-projective-calculation} gives
\begin{equation}
 m^+(X^{\mathrm l}_{a,\ell})=m^+(X^{\mathrm r}_{a,\ell})
       =\sum_{s=0}^{\ell-1}\epsilon_{a+2s}.
 \label{eq:categorical-positive-weights}
\end{equation}
The starting value is the projective cover calculation for $J_{\ell_{\mathrm l}}$.
Replacing $a$ by $a+2$ changes the sum by
$\epsilon_{a+2\ell}-\epsilon_a$, exactly the change obtained by
substituting $\Phi_Ad$ and its top in $u-C_A^{-1}d$. The calculation
on the $\ell_{\mathrm r}$ orbit gives the first equality.

On the central orbit the calculation is symmetric about $\tau^hO_n^+$.
The positive weight of $X^{\mathrm l}_{1+2d,n}$ and
$X^{\mathrm r}_{1+2d,n}$ is $\lambda_{1+2d,n}$; the positive term of
$X_{1,n}$ is $\bigoplus_{v\text{ odd}}P_v$.
For the principal fibres, $m^+(O_v^+)=e_v$ and
$m^+(\tau^{n-1}O_{\bar v}^+)=e_v$ for odd $v$.
Together with \eqref{eq:categorical-bidegree}, these calculations
place every object in its asserted fibre.

It remains to show that the fibres are full. The initial degrees are
pairwise distinct by \eqref{eq:folded-parameters}--\eqref{eq:folded-degree}
and the index ranges in \eqref{eq:mutable-chamber}. At a fixed level $\ell<n$, equal $r$ allows only $a'=a$ or $a'=2-a$.
For the latter pair, write $a=1+d>1$. If both indices belong to
$I_<(n)$, their lengths are
$\min(\ell,2n-d-\ell)$ and $\ell-d$, which differ because
$\ell<n$. For $(a,n)\in I_=(n)$, the parameter $r=a$ already
distinguishes the indices. The principal and spherical degrees distinguish the frozen
fibres. The displayed fibres are therefore disjoint. There are two for each of the $n(n-2)$ mutable variables,
$3n-1$ in the principal fibres, $n-1$ neutral objects, and one central
object, for a total of
\[
 2n(n-2)+(3n-1)+(n-1)+1=2n^2-1.
\]
This is the number of indecomposables in
\eqref{eq:branching-subcategory}. All displayed objects lie there.
The omitted negative presentations have zero first degree, and the
omitted odd neutral presentations have degree $(\omega_v;0)$ with $v$ odd;
neither is an initial degree. Thus no further objects can enter the
fibres, even when they are computed in the whole $\mathscr C_n$.
\end{proof}

\subsection{Admissible presentations and exchange polynomials}
\label{subsec:categorical-B-identification}
\label{subsec:skew-symmetrizability}
We first give the rule used here and in Section~\ref{sec:exceptional-categories}.
Let $\mathcal S$ be a finite union of full degree fibres in a presentation
category, and put $\mathscr D=\add \mathcal S$. Attach one
independent variable to each fibre, of its degree, and write $\pi_{\cF}(Y)$
for the variable of $Y$. We count both solid arrows and translation
arrows in the original category. Whenever all neighbours of
$X\in\mathcal S$ lie in $\mathcal S$, form
\[
 M_X^{\rm out}=\prod_{X\to Y}\pi_{\cF}(Y),\qquad
 M_X^{\rm in}=\prod_{Y\to X}\pi_{\cF}(Y),
\]
with the arrow multiplicities. In the species case we use the
valuations specified in Section~\ref{subsec:source-f4}.

We call $X$ \emph{admissible} if all its neighbours lie in $\mathcal S$
and these two monomials have the same degree, are distinct and
relatively prime, and neither involves $\pi_{\cF}(X)$. If the admissible
presentations in a fibre give the same unordered pair, its sum defines
the exchange polynomial of that fibre. Such a fibre is mutable; a
fibre without an admissible presentation is frozen. The theorems below
verify this independence and identify the resulting polynomials with
the normalized LP exchanges. The local conditions alone are not
asserted to give an LP seed for an arbitrary degree map.

Restriction to a full subcategory can change irreducible morphisms.
The following lemma permits us to use the original arrows at an
admissible presentation.

\begin{lemma}\label{lem:restriction-marked-star}
Let $\mathscr D$ be a full, direct-summand-closed additive subcategory
of $\mathscr C_n$, and let $S\in\ind\mathscr D$. Suppose that every
source and target of an ambient irreducible morphism incident with
$S$ belongs to $\mathscr D$. Then, for $Y\in\ind\mathscr D$,
\[
 \Irr_{\mathscr D}(S,Y)=\Irr_{\mathscr C_n}(S,Y),\qquad
 \Irr_{\mathscr D}(Y,S)=\Irr_{\mathscr C_n}(Y,S),
\]
under the natural identification of the Hom spaces.
\end{lemma}
\begin{proof}
The middle term $E$ of a minimal left almost split morphism $S\to E$
is the sum of the targets of irreducible morphisms from $S$, with their
multiplicities \cite[Proposition~2.5]{FeiTensor}. Hence $E\in\mathscr D$.
Fullness preserves the almost split property and minimality, so the
irreducible multiplicities are unchanged. Since the radicals agree
and $\rad_{\mathscr D}^2\subseteq\rad_{\mathscr C_n}^2$, the quotient
spaces agree as well. A minimal right almost split morphism proves
the incoming assertion.
\end{proof}

For $\mathcal S=\ind\mathscr C_n^{\mathrm{br}}$, write
$\pi_n(X)=x$ when $X\in\cF_n(x)$. Let $T(U,V)=1$ if the original
diagram contains $U\dashrightarrow V$, and $0$ otherwise.

\begin{theorem}\label{prop:categorical-B}\label{thm:initial-star-rows}
The fibres containing admissible presentations are precisely the mutable
fibres of Part~I. Both $X^{\mathrm l}_{a,\ell}$ and $X^{\mathrm r}_{a,\ell}$
are admissible, except that the following fibres have only the indicated member:
\[
\begin{array}{c|c}
 \ell<n\text{ even},\ a=3-\ell & X^{\mathrm l}_{a,\ell}\\
 \ell<n\text{ even},\ a=2n-\ell-1 & X^{\mathrm r}_{a,\ell}\\
 \ell=n\text{ even},\ a=n-1 & X^{\mathrm l}_{a,n}.
\end{array}
\]
There are no other admissible presentations in $\mathscr C_n^{\mathrm{br}}$.
Whenever both members are admissible, their incoming and outgoing
monomials are interchanged.
Thus $E_x=M_X^{\rm out}+M_X^{\rm in}$ is independent of the admissible
$X\in\cF_n(x)$ and equals the exchange polynomial of Part~I.

For any admissible representative $\sigma_x$, put
$M_x^{\rm out}=M_{\sigma_x}^{\rm out}$ and
$M_x^{\rm in}=M_{\sigma_x}^{\rm in}$. Let $\chi_x\in\{1,-1\}$ record
the exchange-term order of \eqref{eq:row-sign-convention}, so that
\begin{equation}
 (M_{x,+},M_{x,-})=
 \begin{cases}
 (M_x^{\rm out},M_x^{\rm in}),&\chi_x=1,\\
 (M_x^{\rm in},M_x^{\rm out}),&\chi_x=-1.
 \end{cases}
 \label{eq:categorical-ordered-monomials}
\end{equation}
Restriction preserves the irreducible morphism spaces at $\sigma_x$,
and the oriented exchange matrix satisfies
\begin{equation}
 \displaystyle
 b_{xy}=\chi_x\sum_{Y\in\cF_n(y)}
 \left(\dim_\C\Irr_{\mathscr C_n^{\mathrm{br}}}(\sigma_x,Y)
       -\dim_\C\Irr_{\mathscr C_n^{\mathrm{br}}}(Y,\sigma_x)
       +T(\sigma_x,Y)-T(Y,\sigma_x)\right).
 \label{eq:categorical-B-formula}
\end{equation}
\end{theorem}
\begin{proof}
By \eqref{eq:ambient-solid-arrows}--\eqref{eq:ambient-translation-arrows},
the objects with all neighbours in $\mathscr C_n^{\mathrm{br}}$ are
$\tau^tO_v^+$ with $0\le t\le n-2$ and the even neutral presentations.
At $O_v^+$ the two products have different right degrees; at the neutral
presentations they have different left degrees. At $X_{1,n}$ the two
products coincide. Formula~\eqref{eq:original-ar-coordinates} therefore
leaves precisely the stated candidates.

Let $\jmath$ interchange $X^{\mathrm l}_{a,\ell}$ and
$X^{\mathrm r}_{a,\ell}$. On translation orbits it is
\[
 \jmath(\tau^tO_v^+)=\tau^{n-(\ell(v)\bmod2)-t}O_{\bar v}^+.
\]
For a mutable fibre with two admissible members, the fibre formulas give
$\pi_n(\jmath Y)=\pi_n(Y)$ on their complete neighbourhoods. The arrow
formulas show that $\jmath$ reverses every solid and translation arrow
there. Hence the projected monomials are interchanged.
At each stated even boundary, the pair instead has translation indices
$1,n-1$. The latter object has a negative presentation as translation
successor, so only the former can be admissible.

To identify the ordered products, choose an admissible presentation in each
mutable fibre as follows. For $\ell<n$, take
\begin{equation}
 \sigma_{F_{a,\ell}}=\begin{cases}
 X^{\mathrm r}_{a,\ell},&\ell\text{ even and }a=2n-\ell-1,\\
 X^{\mathrm l}_{a,\ell},&\text{otherwise}.
 \end{cases}
 \label{eq:lower-marked-section}
\end{equation}
At the central level take
\begin{equation}
 \sigma_{F_{a,n}}=\begin{cases}
 X^{\mathrm l}_{a,n},&n\text{ even and }a=n-1,\\
 X^{\mathrm r}_{a,n},&\text{otherwise}.
 \end{cases}
 \label{eq:middle-marked-section}
\end{equation}
At these presentations, away from the even boundaries, the fibre and
arrow formulas give
\begin{align*}
 M_x^{\rm out}&=F_{a-2,\ell+1}F_{a,\ell-1}F_{a+2,\ell},&
 M_x^{\rm in}&=F_{a-2,\ell}F_{a,\ell+1}F_{a+2,\ell-1}
 &&(\ell<n),\\
 M_x^{\rm out}&=F_{a+2,n}F_{a,n-1}F_{4-a,n-1},&
 M_x^{\rm in}&=F_{a-2,n}F_{a+2,n-1}F_{2-a,n-1}
 &&(\ell=n).
\end{align*}
Here neighbouring indices use the boundary conventions of Part~I,
and a nonexistent level-zero neighbour contributes $1$.
These are \eqref{eq:lower-stencil} and \eqref{eq:middle-stencil}.

At an even boundary the admissible member is $\tau O_v^+$ for an even
vertex $v$. Its outgoing neighbours are $\Id_v$, $\tau^2O_v^+$, and
$O_w^+$ for $w\sim v$; its incoming neighbours are $O_v^+$ and
$\tau O_w^+$ for $w\sim v$. Thus
\[
\begin{aligned}
 M_x^{\rm in}=p_{\iota(v)}\prod_{w\sim v}\pi_n(\tau O_w^+),\qquad  
 M_x^{\rm out}=s_{\iota(v)}\pi_n(\tau^2O_v^+)
                    \prod_{w\sim v}p_{\iota(w)}.
\end{aligned}
\]
They give \eqref{eq:left-even-exchange}, \eqref{eq:right-even-exchange},
and \eqref{eq:corrected-middle-boundary}. All the displayed products
satisfy the admissibility conditions. For these representatives the
prescribed order has
\begin{equation}
 \chi_x=\begin{cases}
 -1,&x=F_{2n-\ell-1,\ell},\ \ell<n\text{ even},\\
 1,&\text{otherwise}.
 \end{cases}
 \label{eq:categorical-row-sign}
\end{equation}
This proves the classification and the exchange-polynomial assertion.
Lemma~\ref{lem:restriction-marked-star} identifies the original and
restricted irreducible spaces. Taking ordered exponent differences
gives \eqref{eq:categorical-B-formula}.
\end{proof}

The same rule can start with all indecomposables of dominant degree
and nonzero left degree. The only additional objects are the odd
neutral presentations, each adjacent to a negative presentation, so
the admissible set is unchanged. The admissible presentations and
all their neighbours are exactly $\ind\mathscr C_n^{\mathrm{br}}$.
Thus this subcategory and the unordered exchanges can be recovered
from the degree, without specifying marks or seed degrees in advance.

\subsubsection*{Matrix form and the role of the marking.}
Let $\varepsilon(U,V)$ be the signed number of solid and dashed
arrows between indecomposables of the original $\mathscr C_n$,
positive from $U$ to $V$. Let $E=(\varepsilon(U,V))$, let
$P_{U,y}=1$ when $U\in\cF_n(y)$ and zero otherwise, and let
$R_{x,U}=1$ when $U=\sigma_x$ and zero otherwise. Then
Theorem~\ref{prop:categorical-B} says
\begin{equation} 
 B=D_\chi R E P,\qquad D_\chi=\diag(\chi_x).
 \label{eq:marked-fibre-matrix-identity}
\end{equation}
Here $P$ collects whole fibres and $R$ selects admissible presentations.
Changing to the other admissible presentation reverses the corresponding
row of $REP$; changing $\chi_x$ at the same time leaves $B$ unchanged.
Without the exchange-term order, the LP seed is already determined
by the unordered products. Neither incidence between the representatives
nor $P^{\mathsf t}EP$ gives this matrix. The equality $BW=0$ follows
from the exchange degrees. This construction concerns the initial
seed, not a categorical realization of LP mutation.

Although $\varepsilon(X,Y)=-\varepsilon(Y,X)$, summing over only
one fibre need not preserve skew-symmetry. For mutable fibres
$\mathcal F_i$ and marked presentations $\sigma_i$, set
\[ c_{ij}=\sum_{Y\in\mathcal F_j}\varepsilon(\sigma_i,Y),\qquad b_{ij}=\chi_i c_{ij}. \]

\begin{proposition}\label{prop:fibre-arrow-counts}
Put $\mathcal F_i^\circ=\mathcal F_i\setminus\{\sigma_i\}$. Then
\[ c_{ij}+c_{ji}=\sum_{Y\in\mathcal F_j^\circ}\varepsilon(\sigma_i,Y)+\sum_{X\in\mathcal F_i^\circ}\varepsilon(\sigma_j,X). \]
If, for each $i,j$, the sum
$\sum_{Y\in\mathcal F_j}\varepsilon(X,Y)$ is independent of
$X\in\mathcal F_i$, then
$|\mathcal F_i|c_{ij}=-|\mathcal F_j|c_{ji}$.
In that case $(c_{ij})$ is skew-symmetrizable by the positive diagonal
matrix with entries $|\mathcal F_i|$.
\end{proposition}
\begin{proof}
In $c_{ij}+c_{ji}$, the two terms involving only $\sigma_i,\sigma_j$
cancel. This proves the first formula. Under the independence
hypothesis, sum over both fibres and use skew-symmetry of
$\varepsilon$:
\[ |\mathcal F_i|c_{ij}=\sum_{X\in\mathcal F_i,\,Y\in\mathcal F_j}\varepsilon(X,Y)=-|\mathcal F_j|c_{ji}. \]
\end{proof}

The independence hypothesis holds for full orbits of a group of
quiver automorphisms. This is the counting argument behind
skew-symmetrizability in equivariant constructions
\cite[Section~1.2]{Demonet}. Our fibres are defined by equality of
branching degrees; they need not be such orbits. Summing over both
fibres would give a skew-symmetric matrix, but would replace the
exchange monomials by different monomials.

The cluster condition is skew-symmetrizability of the mutable
principal matrix $B_0$ \cite[Definition~4.4]{FominZelevinsky}.
Interchanging the two terms of an LP exchange changes the sign of
the corresponding row of $B$. Even allowing all such choices does
not give a cluster matrix here.

\begin{corollary}\label{cor:symplectic-noncluster}
No choice of row signs makes the principal matrix of the initial
$A_{2n-1}\downarrow C_n$ seed skew-symmetrizable for $n\ge4$.
\end{corollary}
\begin{proof}
For the initial $A_{2n-1}\downarrow C_n$ seed, restrict to
$(F_{1,n-1},F_{3,n-1},F_{3,n})$. The exchange formulas
\eqref{eq:lower-stencil}, \eqref{eq:middle-stencil}, and
\eqref{eq:corrected-middle-boundary} give
\[
M_n=\begin{pmatrix}0&1&1\\-1&0&-1\\ \epsilon_n&\epsilon_n&0\end{pmatrix},\qquad \epsilon_4=-1,\quad\epsilon_n=1\ (n\ge5).
\]
Its determinant is $-2\epsilon_n$. An invertible diagonal multiple
of this matrix cannot be skew-symmetric, since a skew-symmetric
matrix of odd size is singular.

\end{proof}

Corollary~\ref{cor:symplectic-noncluster} concerns these seeds and their
exchange polynomials, not the existence of another cluster algebra
structure on the underlying branching algebra.

\begin{example}[The $A_9\downarrow C_5$ seed]
\label{ex:a9c5-full-diagram}
For $n=5$, Proposition~\ref{prop:categorical-fibres} groups the $49$
indecomposables of $\mathscr C_5^{\mathrm{br}}$ into $29$ initial
fibres, of which $15$ are mutable. Figure~\ref{fig:a9c5-fibres} 
shows the whole ambient category: its $63$ indecomposables, $106$
solid arrows, and $45$ translations. The additional $14$ objects are
the nine negative and five odd neutral presentations. For the oriented
matrix, the reflected boundary representatives are
$X^{\mathrm r}_{7,2},X^{\mathrm r}_{5,4}$ with $\chi=-1$; the central
representative is $X^{\mathrm r}_{3,5}$ with $\chi=1$.

The blue vertex in Figure~\ref{fig:a9c5-seed} is $F_{-1,4}$. Its fibre
and unique admissible presentation are
\[
 \cF_5(F_{-1,4})=\{X^{\mathrm l}_{-1,4},X^{\mathrm r}_{-1,4}\},
 \qquad \sigma_{F_{-1,4}}=X^{\mathrm l}_{-1,4}.
\]
The other member has translation successor
$X^{\mathrm r}_{-3,4}=O_{4_{\mathrm l}}^-$, outside
$\mathscr C_5^{\mathrm{br}}$. The incoming neighbours of
$X^{\mathrm l}_{-1,4}$ are
$X^{\mathrm l}_{1,3},X^{\mathrm l}_{3,5},O_{4_{\mathrm l}}^+$;
its outgoing neighbours are
$X^{\mathrm l}_{1,4},\Id_{4_{\mathrm l}},O_{3_{\mathrm l}}^+,O_5^+$.
Their fibre variables give
\[
 M^{\rm out}=F_{1,4}s_4p_3p_5,\qquad
 M^{\rm in}=F_{1,3}F_{3,5}p_4.
\]
Their sum is \eqref{eq:a9-running-exchange}. The LP mutation replaces
$F_{-1,4}$ by
$ (F_{1,4}s_4p_3p_5+F_{1,3}F_{3,5}p_4)/F_{-1,4}$,
giving the geometric seed of Section~\ref{subsec:a9-geometric}.

The same diagram illustrates Corollary~\ref{cor:symplectic-noncluster}.
Put $x=F_{1,4}$, $y=F_{3,4}$, and $z=F_{3,5}$. Their representatives
are $X^{\mathrm l}_{1,4},X^{\mathrm l}_{3,4},X^{\mathrm r}_{3,5}$.
The arrows
$X^{\mathrm l}_{1,4}\dashrightarrow X^{\mathrm l}_{3,4}$ and
$X^{\mathrm r}_{3,5}\longrightarrow X^{\mathrm l}_{3,4}$
give the skew pairs for $x,y$ and $y,z$. In contrast,
$X^{\mathrm l}_{1,4}\longrightarrow X^{\mathrm l}_{3,5}$ and
$X^{\mathrm r}_{3,5}\longrightarrow X^{\mathrm r}_{1,4}$
give $b_{xz}=b_{zx}=1$. These are different irreducible arrows,
each ending at the unchosen member of a fibre, rather than opposite
arrows between the same two objects. Proposition~\ref{prop:fibre-arrow-counts}
accounts for their contribution. 
\end{example}

\begin{figure}[htbp]
\centering
\resizebox{.8\textwidth}{!}{\AnineCfiveFibres}
\label{fig:a9c5-fibres}
\end{figure}

\section{Categorical constructions for the remaining branchings} \label{sec:exceptional-categories}
We first construct the three non-skew-symmetrizable seeds and then
give two hypersurface cases. We use the admissibility rule of
Section~\ref{subsec:categorical-B-identification}. For each branching
$\Gamma=D_4\downarrow G_2,F_4\downarrow B_4,E_6\downarrow F_4$, the
presentation category, additive degree, and selected union
$\mathcal S_\Gamma$ of full degree fibres are specified below. Put
\[
 \mathscr C_\Gamma^{\rm br}=\add \mathcal S_\Gamma,
 \qquad
 \cF_\Gamma(j)=\{X\in\mathcal S_\Gamma:
                         \bwt_\Gamma(X)=\wt(x_j)\}.
\]
No objects in a fibre are identified; the morphisms are the original
commutative squares. In the simply laced cases, solid-arrow
multiplicity is $\dim_\C\Irr(X,Y)$. In type $F_4$ we use the species
valuations of Section~\ref{subsec:source-f4}. Translation arrows have
multiplicity one. Let $\varepsilon_\Gamma(X,Y)$ denote outgoing minus
incoming categorical multiplicity, including translation arrows.

\begin{theorem}\label{prop:exceptional-fibre-B}
For each of the three constructions below, the mutable fibres are
exactly the fibres containing an admissible presentation. Each mutable
fibre contains a unique admissible presentation $\sigma_i$, and no
frozen fibre contains one. The unordered incoming and outgoing
monomials at $\sigma_i$ are the initial normalized LP exchange
polynomial of that fibre. Thus, once the degree fibres are fixed, no
choice of marked representatives is part of the seed construction.

Restriction to $\mathscr C_\Gamma^{\rm br}$ preserves the irreducible
morphism spaces at every $\sigma_i$. After choosing the order of the
two exchange monomials, the oriented exchange matrix is
\begin{equation}
 b_{ij}=\chi_i\sum_{Y\in\cF_\Gamma(j)}
                    \varepsilon_\Gamma(\sigma_i,Y),\qquad
 B=D_\chi R E_\Gamma P.
 \label{eq:exceptional-fibre-B}
\end{equation}
Here $\chi_i=1$ for outgoing minus incoming and $-1$ for the reverse
order, $R$ selects the unique admissible presentations, and $P$
collects objects by full degree fibre.
\end{theorem}
\begin{proof}
Propositions~\ref{thm:d4-fibres}, \ref{prop:d4-admissible-selection},
\ref{thm:f4-fibres}, and \ref{thm:e6-marked-realization} establish the
case-by-case assertions. Lemma~\ref{lem:restriction-marked-star}
applies over the vertex division fields as well, so restriction
preserves the irreducible spaces and their valuations. Taking ordered
exponent differences gives \eqref{eq:exceptional-fibre-B}.
\end{proof}

For triality the frozen-degree sum is an additional selection
condition: dominance alone leaves one extra local candidate. The
theorem concerns the initial seeds, not arbitrary LP mutations.

\subsection{The \texorpdfstring{$D_4\downarrow G_2$}{D4 to G2} seed}
\label{subsec:source-d4}
The frozen-degree sum in Proposition~\ref{prop:d4-admissible-selection}
selects the branching subcategory for this degree. The corresponding
branching functions are defined in Section~\ref{sec:d4-g2-seed}.

Let
\[
 Q_D:\quad
 \vcenter{\hbox{$\xymatrix@C=2.2em@R=1.3em{
   & 3 \\
   1 & 2\ar[l]\ar[u] & 4\ar[l]
 }$}}
 \qquad
 \mathscr C_{Q_D}=\Ch_2(\proj\C Q_D).
\]
We use the notation of Section~\ref{subsec:original-presentation-category}:
$X_{i,0}=O_i^+$, $X_{i,t}=f(\tau^{t-1}J_i)$ for $1\le t\le3$,
and $\Id_i$; here $X_{i,3}=O_i^-$.

\subsubsection{An integral degree}
Let $e(X_{i,t})$ be the $i$th standard vector and $e(\Id_i)=0$.
Let $p^\pm_j(X)$ be the multiplicity of $P_j$ in $P_\pm(X)$, and put
\[ m_j(X)=\dim(\coker X)_j=\dim\Hom_{\C Q_D}(\coker X,J_j),\qquad j=1,2. \]
These quantities are additive under direct sums. Define
\begin{equation}\label{eq:d4-degree}
 \bwt_D(X)=\tfrac12\bigl(e(X),p^+(X),p^-(X),m_1(X),m_2(X)\bigr)N_D,
\end{equation}
where $N_D$ is displayed in Appendix~\ref{app:categorical-degrees}.

Although the formula contains a factor $1/2$, the degree is integral.
Indeed, if $A_D$ is the symmetric $D_4$ Cartan matrix, presentation
vectors satisfy $e-p^++p^-\in A_D\Z^4$.
This follows from the projective--injective generators of the
presentation lattice: $O_i^+$ and $\Id_i$ map to zero, while $f(J_i)$
map to the columns of $A_DC_{Q_D}^{\mathsf t}$; Coxeter translation
preserves this congruence. In particular,
\[ e_1+e_3-p_1^+-p_3^++p_1^-+p_3^-\equiv0\pmod2. \]
This is precisely the parity needed for the third and sixth
coordinates in \eqref{eq:d4-degree}; all others are integral without
it.

\begin{proposition}\label{thm:d4-fibres}
For each variable $z$ of \eqref{eq:d4-rank-four-order}, let
$\cF(z)=\{X\in\ind\mathscr C_{Q_D}:\bwt_D(X)=\wt(z)\}$.
The full fibres and their admissible presentations are
\[
\begin{array}{c|l|c}
\text{variable}&\text{full fibre}&\text{admissible presentation}\\\hline
x_1=p_4&X_{1,0},X_{1,3}&X_{1,0}\\
x_2=f_9&X_{2,3}&X_{2,3}\\
x_3=h_1&X_{4,2},\Id_1&X_{4,2}\\
x_4=u_D&X_{1,2},X_{2,1},X_{3,2}&X_{1,2}\\\hline
c_1=p_1&X_{4,1}&--\\
c_2=-p_2&X_{4,3}&--\\
c_3=p_3&X_{2,0},X_{3,3}&--\\
c_4=-s_1&\Id_2&--\\
c_5=s_4&X_{1,1}&--\\
c_6=v_D&X_{2,2}&--
\end{array}
\]
These are the only admissible presentations in the union of the
displayed fibres, one per mutable fibre. Their incoming and outgoing
monomials sum to the exchanges \eqref{eq:d4-four-exchanges}.
\end{proposition}

\begin{proof}
The projective dimension vectors are the columns of
\[
C_{Q_D}=\begin{pmatrix}1&1&0&1\\0&1&0&1\\0&1&1&1\\0&0&0&1\end{pmatrix},
 \qquad \underline{\dim}\tau M=-C_{Q_D}^{\mathsf t}C_{Q_D}^{-1}\underline{\dim}M.
\]
Projective covers and these twelve module dimensions determine the
source and target of every presentation. Substitution in
\eqref{eq:d4-degree} gives the displayed fibres. The remaining indecomposables have degrees
\[
\begin{array}{c|rrrr|rr}
X_{3,0}&0&-1&1&1&2&-1\\
X_{3,1}&0&0&2&1&2&0\\
X_{4,0}&0&0&0&-1&0&0\\
\Id_3&0&1&-1&-1&-1&0\\
\Id_4&1&-1&0&0&-1&0
\end{array}
\]
and none has a degree in the table. The only additional dominant
object is $X_{3,1}$; Proposition~\ref{prop:d4-admissible-selection}
explains its exclusion. The four admissible presentations are
\[
\begin{array}{c|c}
X_{1,0}&P_1\to0\\
X_{2,3}&0\to P_2\\
X_{4,2}&P_1\oplus P_3\xrightarrow{(\alpha_{12},\alpha_{32})}P_2\\
X_{1,2}&P_1\xrightarrow{\alpha_{12}}P_2.
\end{array}
\]
Here $\alpha_{12},\alpha_{32}$ denote the maps induced by $2\to1$
and $2\to3$. The incoming and outgoing neighbours are
\[
\begin{array}{c|l|l}
&\text{incoming}&\text{outgoing}\\\hline
X_{1,0}&X_{2,1}&X_{2,0},X_{1,1}\\
X_{2,3}&X_{1,3},X_{3,3},X_{2,2}&X_{1,2},X_{3,2},X_{4,3}\\
X_{4,2}&X_{2,2},X_{4,1}&X_{2,1},\Id_2,X_{4,3}\\
X_{1,2}&X_{2,3},\Id_1,X_{1,1}&X_{2,2},X_{1,3}.
\end{array}
\]
The translations here are $X_{1,0}\dashrightarrow X_{1,1}$,
$X_{2,2}\dashrightarrow X_{2,3}$,
$X_{4,1}\dashrightarrow X_{4,2}\dashrightarrow X_{4,3}$,
and $X_{1,1}\dashrightarrow X_{1,2}\dashrightarrow X_{1,3}$.
All other listed arrows are irreducible morphisms. The resulting
products are
\[
\begin{array}{c|c|c}
&\text{incoming product}&\text{outgoing product}\\\hline
x_1&x_4&c_3c_5\\
x_2&c_3c_6x_1&c_2x_4^2\\
x_3&c_1c_6&c_2c_4x_4\\
x_4&c_5x_2x_3&c_6x_1.
\end{array}
\]
The square
$x_4^2$ is contributed by the two different objects
$X_{1,2},X_{3,2}\in\cF(u_D)$. Thus no frozen fibre contains an admissible presentation in the selected
subcategory.
\end{proof}

\subsubsection{Selection by the frozen-degree sum}
Let $\mathcal D_D$ consist of all indecomposables of dominant degree
with nonzero left component. Here the local admissibility conditions
alone do not select the seed: a frozen-degree sum is also needed.

\begin{proposition}\label{prop:d4-admissible-selection}
There is exactly one choice of four admissible presentations in
$\mathcal D_D$ and ten degree classes containing them and all their
neighbours for which the other six classes have degree sum
$(2,2,2,2;2,2)$. Their full fibres form
$\mathcal S_D=\mathcal D_D\setminus\{X_{3,1}\}$, and are precisely
those in Proposition~\ref{thm:d4-fibres}.
\end{proposition}
\begin{proof}
There are sixteen objects in $\mathcal D_D$, in eleven degree
classes. The admissible presentations relative to $\mathcal D_D$ are
\[ X_{1,0},\quad X_{1,2},\quad X_{2,2},\quad X_{2,3},\quad X_{4,2}. \]
If we omit $X_{2,2}$, the four remaining presentations and their
neighbours use exactly the ten displayed degrees and give the required
frozen sum. Omitting $X_{1,0},X_{1,2}$, or
$X_{2,3}$ instead uses eleven degrees. Omitting $X_{4,2}$ uses nine
degrees, but the degree of the third source fundamental weight in
the unmarked sum is already $3$, which cannot be reduced by adding
a dominant degree. These are all five possibilities.
\end{proof}

\subsubsection{Why the cokernel dimensions occur}
Put $T(X)=(e(X),p^+(X),p^-(X))$. The degree cannot be a linear
function of $T(X)$ for these fibres. Those vectors satisfy
\[ T(X_{1,0})+T(X_{2,2})=T(X_{1,2})+T(X_{2,1}), \]
whereas their required degrees differ by
\[ \wt(p_4)+\wt(v_D)-2\wt(u_D)=(\omega_2-\omega_3;\varpi_2-\varpi_1)\ne0. \]
There is a second relation
\[ T(X_{1,2})+T(X_{2,0})+T(X_{4,3})=T(X_{1,0})+T(X_{2,3})+T(X_{4,1}), \]
whose degree difference is
$(-\omega_1+\omega_2+\omega_3-\omega_4;0)$. 
The two differences are linearly independent. Thus at least two
additional scalar coordinates are necessary if the degree is to be
linear in $T$ and those coordinates, for these fixed fibres. The two
cokernel dimensions in \eqref{eq:d4-degree} suffice. These coordinates
are Hom dimensions against $J_1$ and $J_2$.

The degree is additive on direct sums. The coefficient of $p_4^+$
only affects $X_{4,0},\Id_4$, neither of which belongs to a seed fibre.

Permuting the three outer vertices and applying projective duality
transports the construction to all six mixed orientations of $D_4$.
Duality takes $X_{i,t}$ to $X_{i,3-t}$ and interchanges the incoming and
outgoing products. In the source and sink orientations, no presentation
has three incoming and three outgoing arrows, as required by the second
exchange. Thus these two orientations do not give the same seed by this rule.

\subsection{The \texorpdfstring{$F_4\downarrow B_4$}{F4 to B4} seed}
\label{subsec:source-f4}
We use the Cartan matrix and symmetrizer
\[
A_F=\begin{pmatrix}2&-1&0&0\\-1&2&-1&0\\0&-2&2&-1\\0&0&-1&2\end{pmatrix},\qquad D_F=\diag(2,2,1,1),
\]
with orientation
\[
 Q_F:\quad
 \xymatrix@C=2.6em{
   1\ar[r] & 2 & 3\ar[l]_{(2,1)} & 4\ar[l]
 }.
\]
Here a \emph{species} has division algebras at its vertices and bimodules at its arrows. Take $k=\Q$, $k'=\Q(\sqrt2)$, vertex fields $(k',k',k,k)$, and arrow bimodules
\[ {}_{k'}k'_{k'}\quad(1\to2),\qquad {}_{k'}k'_k\quad(3\to2),\qquad {}_kk_k\quad(4\to3). \]
The valuation records the dimensions over the source and
target fields, in that order; the other arrows have valuation $(1,1)$.
Let $\mathbb S_{Q_F}$ be their tensor algebra. It is a hereditary algebra of valued type $F_4$; see \cite[Sections~1.1--1.2]{GLSFoundations}. The presentation category is defined over $k$, whereas the branching algebra is defined over $\C$. The degree map below takes values in the integral branching-weight lattice.

Put $\mathscr P_{Q_F}=\Ch_2(\proj\mathbb S_{Q_F})$, with the degreewise split exact structure. We write $P_i,J_i$ for the indecomposable projectives and injectives, and
\[ X_{i,0}=(P_i\to0),\qquad X_{i,t}=f(\tau^{t-1}J_i)\ (1\le t\le6),\qquad \Id_i=(P_i\xrightarrow1P_i). \]
Here $f(M)$ denotes the minimal projective presentation of $M$. These are the 32 indecomposable presentations: 24 minimal module presentations, four positive presentations, and four identity presentations. In particular $X_{i,6}=(0\to P_i)$.

We use the \emph{valuations} of irreducible morphisms. For distinct indecomposable objects let
\[ d_X=\dim_k\op{End}(X),\qquad a_X(Y)=\frac{\dim_k\Irr(X,Y)}{d_Y},\qquad b_X(Y)=\frac{\dim_k\Irr(Y,X)}{d_Y}. \]
These are integers. Thus $a_X(Y)$ is the exponent of the variable of $Y$ in the outgoing product at $X$, and $b_X(Y)$ is its exponent in the incoming product. We also use the translation arrows $X_{i,t}\dashrightarrow X_{i,t+1}$, each with multiplicity one. They are not irreducible morphisms.

For each arrow $i\to j$ of $Q_F$, the irreducible arrows are
\[ X_{j,t}\longrightarrow X_{i,t}\quad(0\le t\le6),\qquad X_{i,t+1}\longrightarrow X_{j,t}\quad(0\le t<6). \]
The eight additional arrows through identity presentations are
\[
\begin{gathered}
X_{1,1}\to \Id_1\to X_{1,0},\qquad X_{2,6}\to \Id_2\to X_{2,5},\\
X_{4,2}\to \Id_3\to X_{4,1},\qquad X_{4,1}\to \Id_4\to X_{4,0}.
\end{gathered}
\]
Write $A_F=(c_{ij})$ and $D_F=\diag(d_1,\ldots,d_4)$. An irreducible bimodule between different orbits $i,j$ has $k$-dimension $d_i|c_{ij}|=d_j|c_{ji}|$, where $d=(2,2,1,1)$. The two valuations need not agree. Arrows through $\Id_i$ have $k$-dimension $d_i$. Altogether there are 47 irreducible bimodules and 24 translation arrows.

\begin{lemma}\label{lem:f4-species-category}
The arrows and valuations above are those of $\mathscr P_{Q_F}$.
\end{lemma}
\begin{proof}
The projective rank matrix, injective rank matrix, and Coxeter transformation are obtained from
\[
R_{Q_F}=\begin{pmatrix}1&0&0&0\\1&1&1&1\\0&0&1&1\\0&0&0&1\end{pmatrix},\qquad
J_{Q_F}=D_F^{-1}R_{Q_F}^{\mathsf t}D_F,\qquad E_{Q_F}=R_{Q_F}^{-\mathsf t}D_F,\qquad \Phi_{Q_F}=-E_{Q_F}^{-1}E_{Q_F}^{\mathsf t}.
\]
The 24 vectors $\Phi_{Q_F}^{t-1}(J_{Q_F})_i$, $1\le t\le6$, are the distinct positive roots of $A_F$, ending in $(R_{Q_F})_i$. Reflection functors for species give the corresponding indecomposables and their Auslander--Reiten meshes. The arrows between minimal presentations, and the additional arrows involving positive and identity presentations, are computed by lifting the module maps to projective covers, as in the presentation-category construction of \cite[Section 3]{FeiTensor}.

For completeness, the accompanying calculation constructs all these presentations over the species with vertex fields $\Q$ and $\Q(\sqrt2)$. A map between projectives is a linear combination of directed paths with coefficients in the relevant vertex field. Morphisms of presentations are commutative squares. We compute their compositions and quotient the radical by its square. The endomorphism algebras are $k'$ for orbits 1 and 2, and $k$ for orbits 3 and 4; the same rule holds for $\Id_i$. The resulting 47 irreducible bimodules give exactly the arrows and valuations above. 
\end{proof}

\subsubsection{The degree and its fibres}
For a presentation $X$, let $p^+(X),p^-(X)$ record the multiplicities of the indecomposable projectives in its domain and codomain. Put $e(X_{i,t})=e_i$, $e(\Id_i)=0$, and extend by direct sums. We also use
\[ m_3(X)=\dim_k\Hom_{\mathbb S_{Q_F}}(\coker X,J_3)=\dim_k(\coker X)_3. \]
The last equality is the usual injective-evaluation adjunction; the vertex field at 3 is $k$.

Define an additive integral degree by
\begin{equation}\label{eq:f4-degree}
\bwt_F(X)=\bigl(e(X),p^+(X),p^-(X),m_3(X)\bigr)N_F,
\end{equation}
where the target coordinates are
$(\lambda_1,\ldots,\lambda_4;\mu_1,\ldots,\mu_4)$.
The matrix $N_F$ is displayed in Appendix~\ref{app:categorical-degrees}.
No division is used in this degree. Grouping the indecomposables
of dominant degree and nonzero left component gives the full fibres
below. The last column records the admissible presentations; their
uniqueness is proved in Proposition~\ref{thm:f4-fibres}.
\begin{center}
\begin{tabular}{@{}lll@{}}
\toprule
Function&Full fibre&Admissible presentation\\\midrule
$x_1=p_2$&$X_{3,0}$&$X_{3,0}$\\
$x_2=-\ell_F$&$X_{1,0},X_{4,1}$&$X_{4,1}$\\
$x_3=-v_F$&$X_{4,2}$&$X_{4,2}$\\
$x_4=-w_F$&$X_{1,1},X_{4,0}$&$X_{1,1}$\\\midrule
$c_1=p_1$&$\Id_1$&\\
$c_2=p_4$&$\Id_3$&\\
$c_3=q_1$&$X_{1,2},\Id_4$&\\
$c_4=-q_4$&$X_{4,3}$&\\
$c_5=q_7$&$X_{3,1}$&\\
$c_6=q_8$&$X_{3,2}$&\\
$c_7=r$&$X_{2,1}$&\\
$c_8=t$&$X_{2,0}$&\\\bottomrule
\end{tabular}
\end{center}
The four mutable weights are
\[
\begin{gathered}
\wt(x_1)=(\omega_2;\varpi_1+\varpi_3),\qquad
\wt(x_2)=(\omega_1+\omega_3;\varpi_1+\varpi_3),\\
\wt(x_3)=(\omega_1+\omega_4;\varpi_3),\qquad
\wt(x_4)=(\omega_1+\omega_2;\varpi_1+\varpi_3).
\end{gathered}
\]
The frozen weights are given by the functions in \eqref{eq:f4-seed-labels}. Substituting the 32 presentation vectors in \eqref{eq:f4-degree} gives exactly the fifteen objects in the table as the objects of dominant nonzero source degree. Each of the other seventeen has a negative Dynkin coordinate.

\begin{proposition}\label{prop:f4-hom-necessary}
For these fibres, one cannot replace \eqref{eq:f4-degree} by a degree linear only in $T=(e,p^+,p^-)$. The single additional Hom dimension $m_3$ suffices.
\end{proposition}
\begin{proof}
The presentation vectors satisfy
\[ T(X_{4,2})-T(X_{3,2})-T(X_{4,1})+T(X_{2,1})-T(X_{2,0})+T(X_{3,0})=0. \]
The corresponding degree combination is
\[ \wt(x_3)-\wt(c_6)-\wt(x_2)+\wt(c_7)-\wt(c_8)+\wt(x_1)=(\omega_2-\omega_3;\varpi_2-\varpi_4)\ne0. \]
The same combination of the $m_3$ values is one. The last row of $N_F$ accounts for precisely this difference. Substitution in \eqref{eq:f4-degree} verifies sufficiency for all fibres.
\end{proof}

\subsubsection{The exchange products}
Let $\alpha_{ji}:P_j\to P_i$ denote the path map induced by an arrow $i\to j$. The four admissible presentations are
\[
\begin{array}{c|c}
x_1& P_3\longrightarrow0\\
x_2& P_3\xrightarrow{\alpha_{34}}P_4\\
x_3& P_2\xrightarrow{\alpha_{23}}P_3\\
x_4& P_2\xrightarrow{\alpha_{21}}P_1.
\end{array}
\]
The last three are the minimal presentations of the simple modules at vertices 4, 3, and 1. The cokernel of the first is zero; we keep its positive projective term.

\begin{proposition}\label{thm:f4-fibres}
For the degree \eqref{eq:f4-degree}, exactly four presentations are
admissible in $\mathcal D_F$. Each belongs to a different mutable
fibre in the table. Their incoming and outgoing products are
\[
\begin{array}{c|c|c}
&\text{incoming}&\text{outgoing}\\\hline
X_{3,0}&c_8x_2&c_5x_4\\
X_{4,1}&c_2c_5x_4&c_3x_1x_3\\
X_{4,2}&c_6x_2&c_2c_4c_5\\
X_{1,1}&c_7x_2&c_1c_3c_8.
\end{array}
\]
Their sums are \eqref{eq:f4-exchanges}. Thus these exchanges are
determined by the degree fibres without a choice of representatives.
\end{proposition}
\begin{proof}
All neighbours occur in the twelve fibres. More explicitly, the incoming and outgoing lists, including translations, are
\[
\begin{array}{c|c|c}
X_{3,0}&X_{2,0},X_{4,1}&X_{4,0},X_{3,1}\\
X_{4,1}&X_{3,1},\Id_3,X_{4,0}&X_{3,0},\Id_4,X_{4,2}\\
X_{4,2}&X_{3,2},X_{4,1}&X_{3,1},\Id_3,X_{4,3}\\
X_{1,1}&X_{2,1},X_{1,0}&X_{2,0},\Id_1,X_{1,2}.
\end{array}
\]
Every exponent at these particular representatives is one, after taking the appropriate valuation of the bimodule. For instance $\Irr(X_{2,0},X_{3,0})$ has $k$-dimension two, and the incoming multiplicity at $X_{3,0}$ is $2/d_{X_{2,0}}=1$.
Replacing each neighbour by its fibre variable gives the four
products, which satisfy the admissibility conditions. The other eleven
objects of $\mathcal D_F$ fail these conditions, proving uniqueness. 
\end{proof}

With outgoing minus incoming as the row convention, the principal matrix is
\[
B_0=\begin{pmatrix}0&-1&0&1\\1&0&1&-1\\0&-1&0&0\\0&-1&0&0\end{pmatrix}.
\]
It is not skew-symmetrizable for any row-sign choice, because $b_{14}=1$ while $b_{41}=0$. Categorically, the arrow $X_{3,0}\to X_{4,0}$ contributes to the exchange at $p_2$, but $X_{4,0}$ is the nonadmissible member of the $-w_F$ fibre. Its admissible representative $X_{1,1}$ has no incidence with $X_{3,0}$.

For the geometric comparison we reverse rows $3$ and $4$. Then
\[ a^{\mathsf t}\diag(1,1,-1,-1)B_0a=a_1a_4+2a_2a_3\ge0\qquad(a\ge0). \]

Projective duality gives the same construction for the opposite valued
quiver $Q_F^{\mathrm{op}}$, with $X_{i,t}$ corresponding to $X_{i,6-t}$.
It interchanges the two products at each admissible presentation.

\subsection{The \texorpdfstring{$E_6\downarrow F_4$}{E6 to F4} seed}
\label{sec:e6-categorical-reconstruction}
We use the presentation category of
\[
 Q_E:\quad
 \vcenter{\hbox{$\xymatrix@C=2.2em@R=1.3em{
   & & 2 \\
   1\ar[r] & 3 & 4\ar[l]\ar[u]\ar[r] & 5 & 6\ar[l]
 }$}}.
\]
Put $M_{i,t}=\tau^{t-1}J_i$ and $X_{i,t}=f(M_{i,t})$ for
$1\le t\le6$, together with $X_{i,0}=(P_i\to0)$ and
$\Id_i=(P_i\xrightarrow1P_i)$. These are the forty-eight
indecomposable presentations. Write $T=(e,p^+,p^-)$, where $e(X_{i,t})=e_i$, $e(\Id_i)=0$,
and $p^+,p^-$ are the multiplicities in the domain and codomain.
We take the following twelve modules, in order:
\[
\begin{array}{lll}
Z_1=M_{1,4},&Z_2=M_{1,5}\oplus M_{1,6},&Z_3=M_{2,2},\\
Z_4=M_{2,3},&Z_5=M_{3,4},&Z_6=M_{3,5},\\
Z_7=M_{4,4},&Z_8=M_{4,5},&Z_9=M_{5,1},\\
Z_{10}=M_{5,2},&Z_{11}=M_{6,3},&Z_{12}=M_{6,4}.
\end{array}
\]
The degree is
\begin{equation} \bwt_E(X)=\tfrac13\bigl(T(X),h_1(X),\ldots,h_{12}(X)\bigr)N_E, \qquad h_j(X)=\dim\Hom_{\C Q_E}(\coker X,Z_j). \label{eq:e6-hom-degree} \end{equation}
The integer $30\times10$ matrix $N_E$ is given in
\cite[A5, Appendix~A]{ElectronicSupplement}. The modules $Z_j$ and $N_E$ specify the degree on the full
presentation category.

Let $\mathcal D_E$ consist of the indecomposables of dominant degree
with nonzero left component. We apply the admissibility rule to this
set; no list of representatives is needed as input.
\begin{proposition}\label{thm:e6-marked-realization}
The degree \eqref{eq:e6-hom-degree} is integral and additive on
$K_0^{\mathrm{split}}(\mathscr C_{Q_E})$. The set $\mathcal D_E$ consists
of thirty-five objects in eighteen full degree fibres. Exactly eight
objects are admissible, one in each of the eight mutable fibres of
Table~\ref{tab:e6-short-marks}; the other ten fibres are frozen.
Their unordered products give the exchanges
\eqref{eq:e6-formal-exchanges}, without a choice of representatives.
\end{proposition}
\begin{proof}
Let $C_Q$ be the projective Cartan matrix, $\Psi=-C_Q^{-1}C_Q^{\mathsf t}$,
and $\chi=(1,0,-1,0,1,-1)$. We have
$p^+(X_{i,t})-p^-(X_{i,t})=\Psi^te_i$ and
$\chi\Psi-\chi=(-3,0,0,0,0,3)$. 
The matrix satisfies $N_E\equiv qr\pmod3$, where
\[ q=(\chi,-\chi,\chi,0^{12})^{\mathsf t},\qquad r=(0,1,0,0,-1,0;0,1,1,1). \]
Thus the degree is integral. Additivity follows from direct sums.

The projective-cover and cokernel Hom calculations give the full
fibres in Table~\ref{tab:e6-short-marks}; the other thirteen objects
have degree zero or a negative Dynkin coordinate
\cite[A5, Check~V22]{ElectronicSupplement}.
For each arrow $i\to j$ of $Q_E$, the solid arrows are
$X_{j,t}\to X_{i,t}$ and $X_{i,t+1}\to X_{j,t}$, wherever defined.
Put $\eta=(1\,6)(3\,5)$. The additional solid arrows are
\[
 X_{i,1}\to\Id_i\to X_{i,0}\quad(i=1,4,6),\qquad
 X_{j,6}\to\Id_{\eta(j)}\to X_{j,5}\quad(j=2,3,5).
\]
Together with $X_{i,t}\dashrightarrow X_{i,t+1}$, these give all
seventy-seven solid and thirty-six translation arrows.
Of the thirty-five objects in the table, twelve have a neighbour
outside $\mathcal D_E$; among the other twenty-three, fifteen have
unequal incoming and outgoing degrees. The remaining eight give
coprime, distinct products omitting their own fibre. They are exactly
the admissible presentations recorded in the last column.

For example, the incoming neighbours of $X_{5,3}$ are
$X_{4,4},X_{6,4},X_{5,2}$, and the outgoing neighbours are
$X_{4,3},X_{6,3},X_{5,4}$. Their fibres give
$z_7z_{33}z_{13}$ and $u_{22}z_8u_{25}$, respectively, hence $E_{23}$.
The same calculation gives the other seven exchanges without omitting any
neighbours. Exact LP substitution gives $\widehat E_i=E_i$.
\end{proof}

\begin{table}[htbp]\centering\small
\setlength{\tabcolsep}{4pt}
\begin{tabular}{@{}llll@{}}\toprule
Function&Branching degree&Full fibre&Admissible presentation\\\midrule
$u_{12}$&$(\omega_2;\varpi_4)$&$X_{1,2},X_{6,2}$&$X_{1,2}$\\
$u_{4}$&$(\omega_5;\varpi_4)$&$X_{6,1},X_{1,3}$&$X_{6,1}$\\
$u_{9}$&$(\omega_3;\varpi_4)$&$X_{1,1}$&$X_{1,1}$\\
$u_{21}$&$(\omega_2+\omega_5;\varpi_3)$&$X_{5,1},X_{3,4},X_{3,2}$&$X_{3,2}$\\
$u_{22}$&$(\omega_2+\omega_5;\varpi_2)$&$X_{1,5},X_{4,3}$&$X_{1,5}$\\
$u_{23}$&$(\omega_2+\omega_4;\varpi_2)$&$X_{1,4},X_{4,5},X_{5,3}$&$X_{5,3}$\\
$u_{24}$&$(\omega_2+\omega_3;\varpi_3)$&$X_{3,1},X_{2,5}$&$X_{3,1}$\\
$u_{25}$&$(\omega_2+\omega_3;\varpi_2)$&$X_{2,4},X_{4,1},X_{5,4}$&$X_{2,4}$\\
\midrule
$z_{1}$&$(\omega_6;0)$&$\Id_{6}$&---\\
$z_{2}$&$(\omega_1;0)$&$\Id_{1}$&---\\
$z_{3}$&$(\omega_6;\varpi_4)$&$X_{6,0}$&---\\
$z_{5}$&$(\omega_5;\varpi_3)$&$X_{5,0},X_{3,5},X_{3,3}$&---\\
$z_{7}$&$(\omega_4;\varpi_2)$&$X_{4,0},X_{1,6},X_{4,4}$&---\\
$z_{8}$&$(\omega_4;\varpi_1)$&$X_{6,3}$&---\\
$z_{10}$&$(\omega_3;\varpi_3)$&$X_{3,0},X_{2,3}$&---\\
$z_{13}$&$(\omega_2;\varpi_1)$&$X_{2,0},X_{5,2}$&---\\
$z_{14}$&$(\omega_1;\varpi_4)$&$X_{1,0}$&---\\
$z_{33}$&$(\omega_2+\omega_3+\omega_5;\varpi_2)$&$X_{4,2},X_{6,4}$&---\\
\bottomrule\end{tabular}
\caption{The eighteen full degree fibres in the $E_6$ presentation category.}
\label{tab:e6-short-marks}
\end{table}
We use the mutable and frozen orders
\[
\begin{gathered}
(u_{12},u_4,u_9,u_{21},u_{22},u_{23},u_{24},u_{25}),\quad
(z_1,z_2,z_3,z_5,z_7,z_8,z_{10},z_{13},z_{14},z_{33}).
\end{gathered}
\]
The six principal variables are
\[ (p_1,p_2,p_3,p_4,p_5,p_6)=(z_{14},z_{13},z_{10},z_7,z_5,z_3), \]
with degrees $(\omega_i;\pi\omega_i)$, where
$\pi(a_1,\ldots,a_6)=(a_2,a_4,a_3+a_5,a_1+a_6)$.
The unordered exchanges determined by the admissible presentations are
\begin{align}
 E_{12}&=u_4u_{24}+u_9u_{21},&
 E_4&=u_{21}z_3+u_{12}z_1z_5,\notag\\
 E_9&=u_{12}z_2z_{10}+u_{24}z_{14},&
 E_{21}&=u_{12}z_5z_{33}+u_4u_{22}u_{24},\notag\\
 E_{22}&=u_{21}z_7+u_{23}z_5,&
 E_{23}&=u_{22}u_{25}z_8+z_7z_{13}z_{33},\notag\\
 E_{24}&=u_9u_{21}u_{25}+u_{12}z_{10}z_{33},&
 E_{25}&=u_{24}z_7+u_{23}z_{10}.
 \label{eq:e6-formal-exchanges}
\end{align}
The first monomial fixes the sign of the corresponding row of $B$. Its order is outgoing first
except at $u_4$. The same-sign pair between
$u_{24}$ and $u_{25}$ has two different lifts,
$X_{3,1}\to X_{4,1}$ and $X_{2,4}\dashrightarrow X_{2,5}$;
it does not arise from opposite arrows between two fixed objects.

\begin{corollary}\label{cor:noncluster-seeds}
No choice of row signs makes the principal matrix of any of the three
exceptional seeds skew-symmetrizable.
\end{corollary}
\begin{proof}
For $D_4\downarrow G_2$, the exchange polynomial at $p_4$ is independent of
$f_9$, whereas the exchange polynomial at $f_9$ contains $p_4$.
For $F_4\downarrow B_4$, the exchange polynomial at $-w_F$ is independent of
$p_2$, whereas the exchange polynomial at $p_2$ contains $-w_F$.
Thus in each case one principal entry is zero and the transposed
entry is nonzero. Neither row signs nor nonzero diagonal factors
can change this. Finally, in the $E_6\downarrow F_4$ seed, the principal
submatrix on $(u_{21},u_{22},u_{23},u_{25},u_{24})$ is
\[
\begin{pmatrix}0&-1&0&0&-1\\1&0&-1&0&0\\0&1&0&1&0\\0&0&-1&0&1\\1&0&0&1&0\end{pmatrix}.
\]
Its determinant is $2$, so the same odd-size argument applies.
\end{proof}

\subsection{Two hypersurface cases \texorpdfstring{$B_3\downarrow G_2$}{B3 to G2} and \texorpdfstring{$G_2\downarrow A_2$}{G2 to A2}} \label{subsec:hypersurface-cases}
In this section, we record two hypersurface cases. Their branching algebras are cluster algebras of type $A_1$.
\subsubsection*{The case \texorpdfstring{$B_3\downarrow G_2$}{B3 to G2}.}
Number $B_3$ so that vertex $3$ is short, and let $\varpi_1$ be the
seven-dimensional $G_2$ weight. Then
$\pi\omega_1=\pi\omega_3=\varpi_1$ and $\pi\omega_2=\varpi_2$.
\begin{proposition}\label{prop:b3-g2-hypersurface}
There are homogeneous functions of degrees
\[
\begin{array}{c|ccccccc}
 &p_1&p_2&p_3&s&h&q&r\\\hline
 \lambda&\omega_1&\omega_2&\omega_3&\omega_3&\omega_1+\omega_2&\omega_2&\omega_1+\omega_3\\
 \mu&\varpi_1&\varpi_2&\varpi_1&0&\varpi_2&\varpi_1&\varpi_2
\end{array}
\]
such that
\[ \B(\Spin_7,G_2)=
 \C[p_1,p_2,p_3,s,h,q,r]/(qr-p_1p_2s-p_3h)
 =\ULP(q;p_1,p_2,p_3,s,h).
\]
The exchange is $E_q=p_1p_2s+p_3h$, with mutation $q'=r$ and
polynomial frozen coefficients. 
\end{proposition}
\begin{proof}
On $U_{B_3}/U_{G_2}\simeq\mathbb A^3_{a,b,c}$, take operators
$-\partial_a$, $a\partial_b$, $\partial_a+b\partial_c$, and set
$q=p_2b$, $r=p_1p_3a$, $s=p_3c$, $h=p_1p_2(ab-c)$.
Lemma~\ref{lem:principal-open-regularity} proves their regularity.
The normal monomials $a^ib^jc^k(ab-c)^\ell$, $\min(i,j)=0$, have
noncancelling derivative orders $(i+\ell,j+\ell,i+k)$.
Homogenizing them proves generation and the sole relation.
Writing $S=\C[p_1,p_2,p_3,s,h]$, the identity
$S[q^{\pm1}]\cap S[(E_q/q)^{\pm1}]=S[q,E_q/q]$ proves the upper-algebra
equality. The quotient construction and the all-weight normal-form
argument are given in \cite[E2, Sections~2--3]{ElectronicSupplement}.
\end{proof}

For the categorical construction, use the hereditary species
\[ Q:\quad\xymatrix@C=2.4em{1&2\ar[l]&3\ar[l]_{(2,1)}},
 \qquad (k_1,k_2,k_3)=(L,L,k),\quad
 k=\Q,\ L=\Q(\sqrt2),
\]
with arrow bimodules ${}_LL_L$ and ${}_LL_k$, and its presentation
category. Put $X_{i,0}=(P_i\to0)$,
$X_{i,j}=f(\tau^{j-1}J_i)$ for $1\le j\le3$, and define
\[ \bwt(X_{i,j})=
 \left(\sum_v m_v^+(X_{i,j})\omega_v;\pi\omega_i\right),
 \qquad\bwt(\Id_i)=(\omega_i;0).
\]
Among all dominant, nonzero-left-degree objects, the unique admissible
presentation is $X_{3,1}=(P_2\to P_3)$. The smallest union of full
fibres containing it and its complete neighbourhood is
\[
\begin{array}{c|cccccc}
 x&q&p_1&p_2&p_3&s&h\\\hline
 \cF(x)&\{X_{1,1},X_{3,1}\}&\{X_{1,0},X_{1,2},X_{3,2}\}&
 \{X_{2,0}\}&\{X_{3,0}\}&\{\Id_3\}&\{X_{2,1}\}.
\end{array}
\]
At $X_{3,1}$ the incoming neighbours are $X_{2,1},X_{3,0}$ and the
outgoing neighbours are $X_{2,0},\Id_3,X_{3,2}$.
The two-dimensional irreducible spaces contribute exponent $2/2=1$
after division by the neighbouring endomorphism dimensions. Thus
$M^{\rm in}=p_3h$, $M^{\rm out}=p_1p_2s$.
The complete presentation calculation and uniqueness check are in
\cite[E2, Sections~4--5]{ElectronicSupplement}.
The frozen degrees sum to $(2,2,2;2,2)$.

\subsubsection*{The case \texorpdfstring{$G_2\downarrow A_2$}{G2 to A2}.}
Let $A_2$ be the long-root subgroup $\SL_3\subset G_2$.
For $\alpha_1$ short, take its simple roots to be
$\alpha_2,3\alpha_1+\alpha_2$, so
$\pi\omega_1=\varpi_2$ and $\pi\omega_2=\varpi_1+\varpi_2$.
This is the case of \cite[Section~4]{PasquierRessayre}.
\begin{proposition}\label{prop:g2-a2-hypersurface}
There are homogeneous functions of degrees
\[
\begin{array}{c|cccccc}
 &p_1&p_2&q&r&s&t\\\hline
 \lambda&\omega_1&\omega_2&\omega_1&\omega_2&\omega_1&\omega_2\\
 \mu&\varpi_2&\varpi_1+\varpi_2&\varpi_1&\varpi_2&0&\varpi_1
\end{array}
\]
such that
\[ \B(G_2,\SL_3)=\C[p_1,p_2,q,r,s,t]/(qr-p_1t-p_2s)
 =\ULP(q;p_1,p_2,s,t).
\]
The exchange is $E_q=p_1t+p_2s$, with mutation $q'=r$ and polynomial
frozen coefficients. 
\end{proposition}
\begin{proof}
Use $U_{G_2}/U_{\SL_3}\simeq\mathbb A^3_{a,b,c}$ with operators
$\partial_a+b\partial_c$, $a\partial_b$, and put
$q=p_1a$, $r=p_2b$, $s=p_1c$, $t=p_2(ab-c)$.
The same normal monomials have noncancelling derivative orders
$(i+k,j+\ell)$. Regularity, generation, and the sole relation follow
by homogenization; the same two-chart intersection proves the upper
LP equality. See \cite[E3, Sections~2--3]{ElectronicSupplement} for
the quotient construction and the full normal-form proof.
\end{proof}

Use the species
\[
 Q:\quad\xymatrix@C=2.6em{1\ar[r]^{(3,1)}&2},
 \qquad k_1=k=\Q,\quad k_2=L=\Q(u),\quad u^3=2,
\]
with bimodule ${}_LL_k$ and its presentation category. Put
$X_{i,0}=(P_i\to0)$ and $X_{i,j}=f(\tau^{j-1}J_i)$ for $1\le j\le3$.
Let $e(X_{i,j})=e_i$, $e(\Id_i)=0$, and let $m_1,m_2$ be the
cokernel ranks over $k,L$. With $\kappa=e_1+2e_2$, define
\[
 \bwt(X)=\bigl(e_1\omega_1+e_2\omega_2;
 (m_1-\kappa)\varpi_1+(\kappa-m_1+m_2)\varpi_2\bigr).
\]
The dominant, nonzero-left-degree set consists of five singleton
fibres, assigned respectively by
\[
 (q,p_1,p_2,s,t)\longleftrightarrow
 (X_{1,2},X_{1,3},X_{2,2},X_{1,1},X_{2,1}).
\]
Only $X_{1,2}$ is admissible. Its incoming neighbours are
$X_{2,2},X_{1,1}$ and its outgoing neighbours are $X_{2,1},X_{1,3}$.
The cross-orbit irreducible spaces contribute exponent $3/3=1$;
the other two arrows are translations. Hence
$M^{\rm in}=p_2s$, $M^{\rm out}=p_1t$. The presentation maps and
radical quotients are given in \cite[E3, Sections~4--5]{ElectronicSupplement}.
The frozen degrees sum to $(2,2;2,2)$.

\section{The exceptional branching algebras}
\label{sec:verification}
We identify the three non-skew-symmetrizable categorical seeds of
Section~\ref{sec:exceptional-categories} with upper LP algebras of
branching functions. Highest-weight differential equations prove
regularity of polynomial quotient-coordinate functions, and rational
inverses prove full fraction-field generation. The common
codimension-one comparison is proved in
Section~\ref{subsec:common-upper-comparison}. The triality construction
is given entirely in the text; the additional coefficient calculations
for $E_6\downarrow F_4$ and $F_4\downarrow B_4$ are in
\cite[E1]{ElectronicSupplement}.
As in Section~\ref{sec:algebra-equality}, the \emph{Gauss chart} for
$A=\B(G,K)$ is the principal open $\Spec A[P^{-1}]$, where
$P=\prod_i p_i$. Its coordinates are obtained by quotienting the
factorization $g=u_-tu_+$ on the big cell; in each case below we
specify a polynomial section of $U_G\to U_G/U_K$. For a homogeneous
function $f$ of left degree $\lambda$, write $f^\circ=p^{-\lambda}f$
on this chart.

\subsection{The \texorpdfstring{$D_4\downarrow G_2$}{D4 to G2} branching algebra}
\label{sec:d4-g2-seed}
Put $A=\B(\Spin_8,G_2)$, with central $D_4$ vertex $2$ and
$E_\alpha=E_1+E_3+E_4$, $E_\beta=E_2$ for the short and long
simple roots of $G_2$. Thus
$\pi\lambda=(\lambda_1+\lambda_3+\lambda_4,\lambda_2)$.
We use the degree fibres of Proposition~\ref{thm:d4-fibres}.

\subsubsection*{Quotient coordinates and branching functions.}
\begin{proposition}\label{prop:d4-polynomial-chart}
There are polynomial coordinates $a,b,c,d,e,f$ on
$U_{D_4}/U_{G_2}$ in which left differentiation is given by
\begin{equation}\label{eq:d4-differential-operators}
\begin{aligned}
 D_1&=-\partial_b,&
 D_2&=-b\partial_c-(a-b)\partial_d,\\
 D_3&=\partial_a+\partial_b+d\partial_f,&
 D_4&=-\partial_a-c\partial_e.
\end{aligned}
\end{equation}
The coordinates have right weights
$-\alpha,-\alpha,-(\alpha+\beta),-(\alpha+\beta),
-(2\alpha+\beta),-(2\alpha+\beta)$, respectively. For $P=p_1p_2p_3p_4$,
\begin{equation}\label{eq:d4-abstract-gauss}
 A[P^{-1}]=\C[p_1^{\pm1},p_2^{\pm1},p_3^{\pm1},p_4^{\pm1},a,b,c,d,e,f].
\end{equation}
\end{proposition}
\begin{proof}
Give $a,b,c,d,e,f$ weights $1,1,2,2,3,3$. The operators satisfy the
$D_4$ Serre relations and lower this grading. Their negatives therefore
integrate to the action on functions by inverse left multiplication.
The vector fields
\[
 D_1,\ D_3,\ [D_1,D_2],\ [D_2,D_3],\
 [D_1,[D_2,D_3]],\ [D_1,[D_2,D_4]]
\]
have determinant $-1$, so every orbit is open and the action is
transitive. Both $D_2$ and $D_1+D_3+D_4=d\partial_f-c\partial_e$
vanish at the origin. Its stabilizer contains $U_{G_2}$ and has the
same dimension $12-6=6$; unipotent subgroups in characteristic zero
are connected, so the stabilizer is $U_{G_2}$. The indicated weights
make this identification $T_{G_2}$-equivariant. Gauss decomposition
then gives \eqref{eq:d4-abstract-gauss}.
\end{proof}

Set $S=e-ac+b(c+d)-f$ and $V=(f-e)^2-a(ce+df)$. Besides the
principal functions, define
\begin{equation}\label{eq:d4-polynomial-dictionary}
\begin{array}{c|c|c}
 \text{function}&\text{polynomial on the Gauss chart}&\text{degree}\\\hline
 s_1&p_1S&(\omega_1;0)\\
 s_3&p_3f&(\omega_3;0)\\
 s_4&p_4e&(\omega_4;0)\\
 h_1&p_2(c+d)&(\omega_2;\varpi_1)\\
 f_9&p_3p_4a&(\omega_3+\omega_4;\varpi_2)\\
 u_D&p_3p_4(f-e)&(\omega_3+\omega_4;\varpi_1)\\
 v_D&p_2p_3p_4V&(\omega_2+\omega_3+\omega_4;\varpi_2).
\end{array}
\end{equation}
These formulas fix the functions and their normalization. Their global
regularity follows from the highest-weight presentation, as used below.

\subsubsection*{The seed and its first mutations.}
Take
\begin{equation}\label{eq:d4-rank-four-order}
 (x_1,x_2,x_3,x_4;c_1,\ldots,c_6)
 =(p_4,f_9,h_1,u_D;p_1,-p_2,p_3,-s_1,s_4,v_D).
\end{equation}
The categorical exchanges are
\begin{equation}\label{eq:d4-four-exchanges}
\begin{aligned}
 E_1&=x_4+c_3c_5,&E_2&=c_3c_6x_1+c_2x_4^2,\\
 E_3&=c_1c_6+c_2c_4x_4,&E_4&=c_6x_1+c_5x_2x_3.
\end{aligned}
\end{equation}
\begin{theorem}[The $D_4\downarrow G_2$ seed]
\label{thm:d4-regular-seed}
The ten functions \eqref{eq:d4-rank-four-order} are algebraically
independent, generate $\Frac A$, and form a normalized LP seed
$\Sigma_D$ with exchanges \eqref{eq:d4-four-exchanges} and
$\widehat E_i=E_i$. They are pairwise nonassociate primes. Their first
mutations are the regular functions \eqref{eq:d4-polynomial-companions},
each coprime to its parent. The six frozen degrees sum to
$(2\rho_{D_4};2\rho_{G_2})$.
\end{theorem}
\begin{proof}
The highest-weight presentation gives, for dominant $\lambda$,
\[
 p^\lambda F\in A\quad\Longleftrightarrow\quad
 D_i^{\lambda_i+1}F=0\quad(1\le i\le4).
\]
The coefficient space is the annihilator of
$\sum_iE_i^{\lambda_i+1}\mathcal U(\mathfrak u_{D_4})$
under the differential pairing; take right-$U_{G_2}$-invariants. Applying the same criterion to
$\lambda-\omega_i$ tests divisibility by $p_i$ when $\lambda_i>0$.
For each polynomial in \eqref{eq:d4-polynomial-dictionary}, the
maximal nonzero derivative orders are exactly its displayed left
Dynkin coefficients. Thus its homogenization is regular and has no
principal factor. The invariant-ring argument of
Proposition~\ref{prop:branching-ufd} also proves that $A$ is factorial;
the principal
functions are prime by their fundamental left degrees.
Each nonprincipal seed polynomial is primitive and linear in one
coordinate: use $b,e,c,a,f,c$ for $S,e,c+d,a,f-e,V$, respectively.
It is therefore irreducible in $\C[a,b,c,d,e,f]$ and remains
irreducible on the Gauss chart. Since $A$ is factorial, excluding
principal factors proves that all ten seed functions are prime.
Their distinct degrees make them nonassociate.

Direct multiplication gives $x_ix_i'=E_i$, where
\begin{equation}\label{eq:d4-polynomial-companions}
\begin{aligned}
 x_1'&=p_3f,&
 x_2'&=-p_2p_3p_4(ce+df),\\
 x_3'&=p_1p_3p_4\bigl((b-a)f-be\bigr),&
 x_4'&=p_2p_4(f-e-ad).
\end{aligned}
\end{equation}
Their maximal derivative orders are, respectively,
$(0,0,1,0)$, $(0,1,1,1)$, $(1,0,1,1)$, and $(0,1,0,1)$.
The same criterion proves their global regularity. For $i=1,3,4$,
the companion-minus-parent left degree is not dominant; for $i=2$
the right difference is $2\varpi_1-\varpi_2$. Thus every companion
is coprime to its parent.

For the rational inverse, recover $p_1=c_1$, $p_2=-c_2$, $p_3=c_3$,
$p_4=x_1$, and put
\[
 a=\frac{x_2}{p_3p_4},\quad e=\frac{c_5}{p_4},\quad
 f=e+\frac{x_4}{p_3p_4},\quad H=\frac{x_3}{p_2},\quad
 S=-\frac{c_4}{p_1},\quad V=\frac{c_6}{p_2p_3p_4}.
\]
Then
\begin{equation}\label{eq:d4-seed-inverse}
 c=\frac{(e-f)^2-afH-V}{a(e-f)},\qquad
 d=H-c,\qquad b=\frac{S-e+f+ac}{H}.
\end{equation}
Substitution in both directions proves the inverse identities.
The ten functions are therefore algebraically independent and generate
$\Frac A$. Each binomial in \eqref{eq:d4-four-exchanges} is primitive
and linear in a frozen variable with coprime coefficient and constant
term, hence irreducible. The LP substitution rule gives
$\widehat E_i=E_i$. Finally,
$\sum_{j=1}^6\wt(c_j)=(2,2,2,2;2,2)
 =(2\rho_{D_4};2\rho_{G_2})$.
\end{proof}

In the initial variable order \eqref{eq:d4-rank-four-order}, let $W_D$
list the degrees and let
$B_D$ use outgoing minus incoming, as in Section~\ref{subsec:source-d4}.
Then $B_DW_D=0$. The minors of $B_D$ in columns $(1,2,4,5)$ and of
$W_D$ in rows $(1,2,3,4,5,7)$ both have determinant $-1$. Hence
\begin{equation}\label{eq:d4-integral-grading}
 0\longrightarrow\Z^4\xrightarrow{B_D^\T}\Z^{10}
 \xrightarrow{W_D^\T}P(D_4)\oplus P(G_2)\longrightarrow0
\end{equation}
is exact. The following functions will also be used in
Section~\ref{subsec:d4-geometric}:
\begin{equation}\label{eq:d4-geometric-functions}
 f_{12}=p_2p_4(f-ad),\qquad f_{13}=p_2p_3(e-ac),\qquad
 g_2=p_2s_3-f_{13}=p_2p_3(f-e+ac).
\end{equation}
Their regularity follows from the same differential criterion, and
$x_4'=f_{12}-p_2s_4$.

\subsection{The upper-algebra comparison}
\label{subsec:common-upper-comparison}
We prove the regularity and codimension-one statements used for the
three exceptional algebra identifications.

\begin{lemma}\label{lem:one-dimensional-exchange-divisor}\label{lem:principal-open-regularity}
Let $A=\B(G,K)$, with $G$ simply connected and semisimple, and let
$D_iF(u)=\left.\frac{d}{ds}\right|_{s=0}F(\exp(sE_i)u)$ on
$\C[U_G/U_K]$. For dominant $\lambda=\sum_i\lambda_i\omega_i$,
\[
 p^\lambda F\in A\quad\Longleftrightarrow\quad
 D_i^{\lambda_i+1}F=0\quad\text{for every }i.
\]
For such a regular function and $\lambda_i>0$,
$p_i\mid p^\lambda F$ if and only if $D_i^{\lambda_i}F=0$.
\end{lemma}
\begin{proof}
Under the differential pairing of $\C[U_G]$ with
$\mathcal U(\mathfrak u_G)$, the highest-coordinate coefficient space
of $V_G(\lambda)$ annihilates
$\sum_iE_i^{\lambda_i+1}\mathcal U(\mathfrak u_G)$, by the
highest-weight presentation. These are precisely the displayed
kernel conditions. Taking right $U_K$-invariants proves the first
assertion; applying it to $\lambda-\omega_i$ proves the second.
Using inverse left multiplication replaces all $D_i$ by $-D_i$ and
does not change these kernels.
\end{proof}

\begin{lemma}\label{lem:branching-volume}\label{lem:d4-volume}
For $(G,K)=(\Spin_8,G_2),(E_6,F_4),(F_4,\Spin_9)$, the variety
$X=\Spec\B(G,K)$ has a rational top form $\omega$ of degree
$(2\rho_G;2\rho_K)$, regular and nonzero at every height-one point.
On a polynomial Gauss chart, up to a nonzero constant,
\begin{equation}\label{eq:d4-volume}
 \omega=P\,dp_1\wedge\cdots\wedge dp_l\wedge d\xi_1\wedge\cdots\wedge d\xi_m,
 \qquad l=\rank G,\quad P=\prod_i p_i.
\end{equation}
\end{lemma}
\begin{proof}
Put $Y=U_G^-\backslash G$, $H=U_K$, and $A=\C[Y]^H$.
Factoriality and finite generation follow as in
Proposition~\ref{prop:branching-ufd}. The invariant volume on $Y$
restricts to $P^2d\log p_1\wedge\cdots\wedge d\log p_l\wedge du$
on $T_G\times U_G$. Contract it with the fundamental fields of a
basis of $\mathfrak u_K$. On the geometric quotient of this open
set the fields span the vertical tangent spaces, so their contraction
descends to a nowhere-vanishing top form. In polynomial quotient
coordinates it has the displayed expression: a nowhere-vanishing
polynomial volume coefficient is constant. The left torus character
is $2\rho_G$; contraction contributes the right character $2\rho_K$.

It remains to check the principal divisors. The function $p_i$ is
prime in the base-affine ring $R=\C[Y]$ by its fundamental left
degree. Its reduced Bruhat divisor $\widetilde D_i$ has dense cell
$U_G^-t\dot s_i u$. The right-$U_G$ stabilizer there is
$u^{-1}U_{\alpha_i}u$. We show that its intersection with $H$ is
generically trivial by finding $u$ with
$\Ad(u^{-1})E_i\notin\mathfrak u_K$.
For triality, take $u=1$ at an outer vertex; at vertex $2$ take
$u=\exp(tE_1)$, so the root component of $[E_1,E_2]$ has no triality
partners. For $E_6$, the involution exchanges $1,6$ and $3,5$ and
fixes $2,4$. Take $u=1$ at a moved vertex and $u=\exp(tE_3)$ at
vertex $4$. At vertex $2$ take $u=\exp(tE_4)\exp(sE_3)$; then
\[
 \Ad(u^{-1})E_2=E_2-t[E_4,E_2]+st[E_3,[E_4,E_2]].
\]
For $st\ne0$, the $\alpha_2+\alpha_4+\alpha_3$ component is nonzero
and its involutive partner is absent. For $F_4$, use Bourbaki roots
\[
 \alpha_1=e_2-e_3,\quad\alpha_2=e_3-e_4,\quad\alpha_3=e_4,
 \quad\alpha_4=\tfrac12(e_1-e_2-e_3-e_4).
\]
Taking $u=\prod_{j=i+1}^4\exp(t_jE_j)$ produces a nonzero coefficient
of $\prod_{j=i+1}^4t_j$ in the root space
$\alpha_i+\cdots+\alpha_4$, a half-sum root absent from $B_4$.
The empty product handles $i=4$. In each case trivial stabilizer
holds on a nonempty open subset of $\widetilde D_i$.

Because $p_i$ is invariant, $(p_i)R\cap A=(p_i)A$: division by $p_i$
preserves invariance. Thus $\widetilde D_i\to D_i=V_X(p_i)$ is
dominant, and $p_i$ is a uniformizer at both generic points.
A separating transcendence basis on $D_i$, together with $p_i$,
shows that the quotient map is smooth at these generic points:
the residue differentials remain independent in characteristic zero,
and $dp_i$ is transverse. Since $X$ is normal, it is smooth at the
generic point of $D_i$. The independent action fields consequently
span the vertical tangent space there. Their contraction has order
zero along $\widetilde D_i$, hence the descended form has order zero
along $D_i$. All other height-one points belong to the Gauss chart.
\end{proof}

\begin{proposition}\label{prop:exceptional-upper-criterion}
Let $(G,K)$ be one of the three pairs in
Lemma~\ref{lem:branching-volume}, and put $A=\B(G,K)$.
Suppose regular homogeneous functions $(x_1,\ldots,x_r,c_1,\ldots,c_q)$,
with $r+q=\dim A$, generate $\Frac A$ and form a normalized LP seed
$\Sigma$ with $\widehat E_i=E_i$. Assume the $x_i$ are pairwise
nonassociate primes, their companions $x_i'=E_i/x_i$ are regular
and coprime to their parents, and
$\sum_f\wt(c_f)=(2\rho_G;2\rho_K)$. Then, for
$M=\prod_i x_i$ and $S=\C[c_1,\ldots,c_q]$,
\begin{equation}\label{eq:exceptional-comparison}
 A[M^{-1}]=S[x_1^{\pm1},\ldots,x_r^{\pm1}],\qquad
 A=\UB(\Sigma)=\ULP(\Sigma).
\end{equation}
All frozen functions remain polynomial coefficients.
\end{proposition}
\begin{proof}
Write $d\mathbf x\wedge d\mathbf c=J\omega$, with $\omega$ from
Lemma~\ref{lem:branching-volume}. Normality makes $J$ regular, as
well as the coefficients $J_i'$ after replacing $x_i$ by $x_i'$.
Differentiating $x_ix_i'=E_i$ gives $x_i'J=-x_iJ_i'$.
Coprimality implies $M\mid J$. The frozen-weight sum gives
$\deg(J/M)=0$, so $J=\kappa M$ with $\kappa\in\C^\times$;
nonvanishing follows from algebraic independence. Thus
\begin{equation}\label{eq:d4-global-jacobian}
 d\mathbf x\wedge d\mathbf c=\kappa M\omega,\qquad
 J_i'=-\kappa x_i'\prod_{j\ne i}x_j.
\end{equation}

Put $R_0=S[\mathbf x^{\pm1}]\subset T=A[M^{-1}]$.
This birational inclusion has equal unit groups by factoriality and
primality of the mutable functions. Its map on spectra is \textup{\'{e}}tale
at every height-one point of $\Spec T$, by the Jacobian identity.
For each irreducible nonunit $g\in R_0$, equality of units gives a
height-one prime of $T$ containing $g$. Its contraction is $(g)$ by
\'{e}taleness. The two local rings are DVRs in the same fraction field
and therefore coincide. Hence $T$ has no pole at any height-one prime
of $R_0$. Intersecting these local rings gives $T=R_0$.

Fix $i$, and put $C_i=S[x_j^{\pm1}:j\ne i]$ and
$R_i=A[x_j^{-1}:j\ne i]$. The exchange $E_i$ is prime in $C_i$,
and the adjacent Jacobian in \eqref{eq:d4-global-jacobian} is a unit
at the generic point of $V_{R_i}(x_i)$. Thus
Lemma~\ref{lem:adjacent-chart} gives
$R_i[(x_i')^{-1}]=C_i[(x_i')^{\pm1}]$.
The initial and adjacent principal opens contain every height-one
point of $\Spec A$: at most one mutable prime vanishes there, and
its companion is a unit. Normality gives $A=\UB(\Sigma)$, and
Theorem~\ref{thm:upper-invariance} gives $\UB(\Sigma)=\ULP(\Sigma)$.
\end{proof}

\begin{corollary}\label{thm:d4-upper-equality}\label{cor:d4-frozen-orders}
For the seed of Theorem~\ref{thm:d4-regular-seed}, put
$A=\B(\Spin_8,G_2)$ and $M=x_1x_2x_3x_4$. Then
\[
 A[M^{-1}]=\C[c_1,\ldots,c_6][x_1^{\pm1},\ldots,x_4^{\pm1}],
 \qquad A=\UB(\Sigma_D)=\ULP(\Sigma_D).
\]
All frozen coefficients remain polynomial, and every LP variable is
regular on the branching variety. For each frozen function $c_f$,
its initial-coordinate exponent valuation equals $\ord_{V(c_f)}$
on $\Frac A$.
\end{corollary}
\begin{proof}
Apply Proposition~\ref{prop:exceptional-upper-criterion} using
Theorem~\ref{thm:d4-regular-seed}. The Laurent phenomenon puts every
LP variable in the upper algebra. Each $c_f$ is prime and
nonassociate to the mutable functions, so its generic point lies in
the initial Laurent chart. The local rings, and hence the discrete
valuations, coincide.
\end{proof}

\subsection{The \texorpdfstring{$E_6\downarrow F_4$}{E6 to F4} branching algebra}
\label{sec:e6-branching}
We realize the degrees and exchange polynomials of
Section~\ref{sec:e6-categorical-reconstruction} by normalized highest
coefficients. We use the mutable and frozen orders of
Table~\ref{tab:e6-short-marks}, and use the six principal functions
$(p_1,p_2,p_3,p_4,p_5,p_6)=(z_{14},z_{13},z_{10},z_7,z_5,z_3)$.
The quotient-coordinate formulas and their exact verification are given
in \cite[E1, Section~2]{ElectronicSupplement}.

\begin{proposition}\label{prop:e6-defining-degrees}
The degrees of
$p_1,p_2,p_3,p_4,p_5,p_6,z_2,u_9,u_{12},z_8$
form an integral basis of $P(E_6)\oplus P(F_4)$. The exchange-degree
relations determine the other eight degrees uniquely and integrally.
The four nonprincipal values come from natural highest-weight lines
in the Albert algebra and its exterior powers.
\end{proposition}
\begin{proof}
The four nonprincipal degrees are
$(\omega_1;0)$, $(\omega_3;\varpi_4)$,
$(\omega_2;\varpi_4)$, and $(\omega_4;\varpi_1)$.
Subtract the corresponding principal degrees. Their target components are
$-\varpi_4,\varpi_4-\varpi_3,\varpi_4-\varpi_1$, and $\varpi_1-\varpi_2$,
a basis of $P(F_4)$. In the exchange matrix, the columns of
$u_4,u_{21},u_{22},u_{23},u_{24},u_{25},z_1,z_{33}$ have determinant one.
This proves the integral assertion.

Let $J=V_{E_6}(\omega_1)$ be the complex Albert algebra, with unit $e$
and trace-zero subspace $J_0$. Then
\[ J=\C e\oplus J_0,\qquad \mathfrak e_6=\op{Der}(J)\oplus L_{J_0},\qquad F_4=\op{Aut}(J), \]
where $L_x$ is multiplication by $x$ \cite{Baez}.
The line $\C e$ gives $z_2$, the highest line in $e\wedge J_0$ gives
$u_9$, and the highest line in $L_{J_0}$ gives $u_{12}$.
The trace form identifies $\mathfrak f_4$ with a submodule
$j(\mathfrak f_4)\subset\Lambda^2J_0$; the highest line of
$e\wedge j(\mathfrak f_4)\subset\Lambda^3J$ gives $z_8$.
Here $\Lambda^2J=V_{E_6}(\omega_3)$ and
$\Lambda^3J=V_{E_6}(\omega_4)$. The decompositions
\[
\begin{aligned}
\Lambda^2J_0&=V_{F_4}(\varpi_1)\oplus V_{F_4}(\varpi_3),\\
\Lambda^3J_0&=V_{F_4}(\varpi_2)\oplus V_{F_4}(\varpi_3)
                 \oplus V_{F_4}(\varpi_1+\varpi_4)
\end{aligned}
\]
show that the required lines occur once. The multiplicities are also obtained from the differential calculation
in Proposition~\ref{thm:e6-regular-exchanges}; the integral assertions
use the unit minors above.
\end{proof}

\begin{proposition}\label{thm:e6-regular-exchanges}
The eighteen seed degrees have branching multiplicity one,
and their nonzero functions are prime. These lines admit a
simultaneous normalization for which all eight quotients
$u_i'=E_i/u_i$ are nonzero regular branching functions, coprime to
their parents. The frozen degrees sum to $(2\rho_{E_6};2\rho_{F_4})$.
\end{proposition}
\begin{proof}
Use the twelve quotient coordinates and normalized polynomials in
\cite[E1, Section~2]{ElectronicSupplement}, with all six principal
coefficients. Lemma~\ref{lem:principal-open-regularity} reduces each
branching component to a finite kernel on monomials of its right weight.
The eighteen kernels are one-dimensional; the displayed polynomials
span them. Each nonprincipal polynomial is primitive and linear in one
coordinate, hence irreducible, and satisfies
$D_i^{\lambda_i} f^{\circ}\ne0$ whenever $\lambda_i>0$.
Thus it has no principal factor. Factoriality gives primality in the
branching algebra, not only on the principal open.

The eight polynomial companions in \cite[E1, Equation~(2.5)]{ElectronicSupplement}
satisfy \eqref{eq:e6-formal-exchanges} and the same differential criterion
at their respective degrees. Hence all eight quotients are regular.
Each companion-minus-parent degree has a negative Dynkin coordinate,
proving coprimality. Polynomial LP substitution gives $\widehat E_i=E_i$.
The unit minor in Proposition~\ref{prop:e6-defining-degrees} gives
\[
0\to\Z^8\xrightarrow{B_E^{\mathsf t}}\Z^{18}
 \xrightarrow{W_E^{\mathsf t}}P(E_6)\oplus P(F_4)\to0.
\]
Finally, the frozen degrees sum to $(2\rho_{E_6};2\rho_{F_4})$.
\end{proof}

\begin{theorem}[The $E_6\downarrow F_4$ upper LP algebra]
\label{thm:e6-upper-equality}
The normalized functions above are algebraically independent and generate
$\Frac\B(E_6,F_4)$.
With $M=\prod_{i\in I_{\rm mut}}u_i$, one has
\begin{equation}
\begin{aligned}
 \B(E_6,F_4)[M^{-1}]&=
 \C[z_f:f\in I_{\rm fr}][u_i^{\pm1}:i\in I_{\rm mut}],\\
 \B(E_6,F_4)&=\UB(\Sigma_E)=\ULP(\Sigma_E).
\end{aligned}
\label{eq:e6-upper-equality}
\end{equation}
All ten frozen functions remain polynomial coefficients, and their
coordinate valuations equal their prime-divisor orders.
\end{theorem}
\begin{proof}
Choose one root from each nonfixed positive-root pair of the pinned
involution, ordered and normalized as in \cite[E1, Section~2]{ElectronicSupplement}.
The ordered product of their root subgroups is a polynomial section of
$U_{E_6}/U_{F_4}$: the Baker--Campbell--Hausdorff formula is solved
successively by root height. Thus, for $P=p_1\cdots p_6$,
\[
 \B(E_6,F_4)[P^{-1}]
 =\C[p_1^{\pm1},\ldots,p_6^{\pm1},t_1,\ldots,t_{12}].
\]
For the explicit normalized coefficients, we put
\[
\begin{gathered}
(a,b,c,d,f)=(- u_{12}^{\circ}, u_9^{\circ}, u_{25}^{\circ}, z_2^{\circ},- u_4^{\circ}),\quad
(g,q,r,v)=( u_{21}^{\circ}, u_{22}^{\circ}, u_{23}^{\circ}, u_{24}^{\circ}),
\quad\ell=(r+v)/c.
\end{gathered}
\]
The complete inverse is the following ordered recovery:
\[
\begin{gathered}
t_6=a,\qquad t_{11}=b,\qquad t_9=c,\qquad t_{12}=d,\qquad t_8=\frac{v-da}{b},\qquad t_{10}=\frac{g+r}{q},\\
t_2=\frac{a z_{33}^{\circ}-gbc+fqv}{gv},\qquad
 t_3=\frac{t_2t_8+c-q}{a},\qquad
t_4=\frac{a z_8^{\circ}-\ell f+t_{10}(\ell t_2+b)}{r},\\
 t_5=\frac{t_4t_8-\ell+t_{10}}{a},\qquad
t_7=\frac{f-t_{10}t_2-b+t_4c}{a},\qquad
 t_1=\frac{a z_1^{\circ}-g+ft_8}{af}.
\end{gathered}
\]
Substitution proves both compositions; every denominator is a monomial
in the eight mutable coefficients. Consequently
\[
\B(E_6,F_4)[(PM)^{-1}]
=\C[p_1^{\pm1},\ldots,p_6^{\pm1},u_i^{\pm1},z_1,z_2,z_8,z_{33}],
\]
proving algebraic independence and full fraction-field generation.
Proposition~\ref{thm:e6-regular-exchanges} supplies the other hypotheses
of Proposition~\ref{prop:exceptional-upper-criterion}.
Each frozen prime is nonassociate to every mutable, so its generic
divisor lies in the initial Laurent chart, identifying its coordinate
valuation with its divisor order.
\end{proof}

\subsection{The \texorpdfstring{$F_4\downarrow B_4$}{F4 to B4} branching algebra}
\label{sec:f4-b4-chart}
For the inner embedding $\Spin_9\subset F_4$, the principal degree
restriction is
\[ \pi\omega_1=\varpi_2,\quad \pi\omega_2=\varpi_1+\varpi_3,\quad \pi\omega_3=\varpi_1+\varpi_4,\quad \pi\omega_4=\varpi_1. \]
The maximal parabolic with Levi of type $B_3$ has nilradical
$\mathfrak n_1\oplus\mathfrak n_2$, with $\mathfrak n_1$ the
spin module and $\mathfrak n_2$ its seven-dimensional centre.
Since $U_{B_4}=\exp(\mathfrak n_2)\rtimes U_{B_3}$,
$U_{F_4}/U_{B_4}\simeq\mathfrak n_1\simeq\mathbb A^8$. Hence
\begin{equation} \B(F_4,\Spin_9)[(p_1p_2p_3p_4)^{-1}] =\C[p_1^{\pm1},\ldots,p_4^{\pm1},\xi_1,\ldots,\xi_8]. \label{eq:f4-gauss-chart} \end{equation}
Its root coordinates correspond to
$\frac12(e_1\pm e_2\pm e_3\pm e_4)$ in the order
$+++,++-,+-+,+--,-++,-+-,--+,---$.
The $e_1$-sign is positive; the three displayed signs are those of
$e_2,e_3,e_4$, in that order.
The following formulas fix the normalized functions used in the seed:
\[
\begin{gathered}
q_1=p_3\xi_1,\quad q_2=p_2\xi_2,\quad q_4=p_1\xi_4,\quad
q_6=p_2\xi_6,\quad q_7=p_3\xi_7,\quad q_8=p_4\xi_8,\\
r=p_2(\xi_1\xi_6+\xi_2\xi_5),\quad
 t=p_2(\xi_3\xi_6+\xi_4\xi_5),\quad
s=p_4(\xi_1\xi_8+\xi_2\xi_7+\xi_3\xi_6+\xi_4\xi_5),
\quad h=p_2s-p_4t.
\end{gathered}
\]
The root coordinates are normalized so that the left-translation
operators, with $\partial_j=\partial/\partial\xi_j$, are
\[
\begin{aligned}
D_1=\xi_5\partial_3-\xi_6\partial_4,\quad
D_2=\xi_3\partial_2-\xi_7\partial_6,\quad
D_3=\xi_2\partial_1+\xi_4\partial_3-\xi_6\partial_5-\xi_8\partial_7,\quad
D_4=\partial_8.
\end{aligned}
\]
The first three give the $B_3$ spin action, and the last is translation
in the simple-root direction. For each displayed $p_jF$, one checks
$D_iF=0$ for $i\ne j$ and $D_j^2F=0$.
Lemma~\ref{lem:principal-open-regularity} therefore proves global
regularity. The root action and polynomial checks are given in
\cite[E1, Section~3]{ElectronicSupplement}.

\begin{proposition}\label{prop:f4-fundamental-lines}
Let $E$ and $S$ be the vector and spin modules of $K=\Spin_9$, and
let $J_0$ be the trace-zero Albert algebra. The inclusion gives
\[ J_0|_K=\C\oplus E\oplus S,\qquad \mathfrak f_4=\mathfrak{so}(E)\oplus S. \]
The action and bracket maps have kernels
\[ \ker(\Lambda^2J_0\to\mathfrak f_4)=V_{F_4}(\omega_3),\qquad \ker(\Lambda^2\mathfrak f_4\to\mathfrak f_4)=V_{F_4}(\omega_2). \]
Their highest-weight lines give
\[
\begin{array}{c|c@{\qquad}c|c}
 f&\wt(f)&f&\wt(f)\\\hline
 p_2&(\omega_2;\varpi_1+\varpi_3)&p_1&(\omega_1;\varpi_2)\\
 p_4&(\omega_4;\varpi_1)&q_1&(\omega_3;\varpi_1)\\
 q_4&(\omega_1;\varpi_4)&q_7&(\omega_3;\varpi_3)\\
 q_8&(\omega_4;\varpi_4)&r&(\omega_2;\varpi_2)\\
 q_2&(\omega_2;\varpi_1+\varpi_4)&q_6&(\omega_2;\varpi_2+\varpi_4)\\
s&(\omega_4;0)&t&(\omega_2;\varpi_3).
\end{array}
\]
In particular, the line of $r$ is determined by the bracket and the
subgroup $K$, up to a nonzero scalar.
\end{proposition}
\begin{proof}
The \emph{Peirce decomposition} and the symmetric-pair decomposition are described in
\cite[Section~4.2]{Baez}. Write $\mathfrak k=\mathfrak{so}(E)$.
Highest weights and dimensions give the two kernel identifications:
$2\omega_4-\alpha_4=\omega_3$ has dimension $273=\binom{26}{2}-52$,
and $2\omega_1-\alpha_1=\omega_2$ has dimension
$1274=\binom{52}{2}-52$.
The Clifford decompositions
\[ \Lambda^2S=\mathfrak k\oplus\Lambda^3E,\qquad E\otimes S=S\oplus V_K(\varpi_1+\varpi_4),
\qquad \mathfrak k\otimes S=S\oplus V_K(\varpi_1+\varpi_4)\oplus V_K(\varpi_2+\varpi_4) \]
follow from \emph{Clifford contraction}
\cite[Section~7, Propositions~1--2]{Friedrich}. Also
$\Lambda^2\mathfrak k=\mathfrak k\oplus V_K(\varpi_1+\varpi_3)$.
Expanding the two exterior squares shows that every line in the table
occurs once. The lines of $q_1$ and $q_7$ come respectively from
$\C\wedge E$ and $\Lambda^3E\subset\Lambda^2S$ in the first kernel;
$p_2$ comes from the kernel of the bracket on $\Lambda^2\mathfrak k$.
The two non-spin summands of $\mathfrak k\otimes S$ lie in the bracket
kernel and give the multiplicity-one lines of $q_2,q_6$.
The scalar summand of $J_0$ gives $s$, and
$\Lambda^3E\subset\Lambda^2S$ in the bracket kernel gives $t$.
The other principal and spin lines are already visible in $J_0$ and
$\mathfrak f_4$.

For $r$, the bracket is nonzero on the unique copies of $\mathfrak k$
in $\Lambda^2\mathfrak k$ and $\Lambda^2S$. The latter assertion follows
from simplicity: otherwise $S$ would be an abelian ideal of
$\mathfrak f_4$. Normalize the two embeddings
$j_{\mathfrak k},j_S$ so that each composes with the bracket to the
identity on $\mathfrak k$. Then
$Z\mapsto j_{\mathfrak k}(Z)-j_S(Z)$
embeds $\mathfrak k$ in $V_{F_4}(\omega_2)$. Its highest vector defines
a nonzero matrix coefficient of degree $(\omega_2;\varpi_2)$.
Multiplicity one identifies its line with that of $r$.
\end{proof}

\subsubsection{The seed from the valued quiver}
Set
\begin{equation}
\begin{aligned}
w_F=p_1p_2(\xi_2\xi_3-\xi_1\xi_4),\qquad
\ell_F=p_1p_3\xi_3,\qquad
v_F=p_1p_4(\xi_3\xi_8+\xi_4\xi_7).
\end{aligned}
\label{eq:f4-additional-functions}
\end{equation}
The seed associated with the species in Section~\ref{subsec:source-f4} is
\begin{equation}
\begin{aligned}
(x_1,x_2,x_3,x_4)=(p_2,-\ell_F,-v_F,-w_F),\quad
(c_1,\ldots,c_8)=(p_1,p_4,q_1,-q_4,q_7,q_8,r,t).
\end{aligned}
\label{eq:f4-seed-labels}
\end{equation}
Its incoming and outgoing products in Proposition~\ref{thm:f4-fibres} give
\begin{equation}
\begin{aligned}
 E_1&=c_8x_2+c_5x_4,& E_2&=c_3x_1x_3+c_2c_5x_4,\\
 E_3&=c_6x_2+c_2c_4c_5,& E_4&=c_7x_2+c_1c_3c_8.
\end{aligned}
\label{eq:f4-exchanges}
\end{equation}

\begin{lemma}\label{lem:f4-new-regularity}
The functions $w_F,v_F,\ell_F$ extend regularly from the principal open.
The four first companions of \eqref{eq:f4-seed-labels} are the regular
functions
\[ (x_1',x_2',x_3',x_4')=(z_F,h,p_3,q_5),\qquad q_5=p_3\xi_5, \]
where
\[ z_F=p_1p_3(\xi_1\xi_4\xi_7-\xi_2\xi_3\xi_7-\xi_3^2\xi_6-\xi_3\xi_4\xi_5). \]
They satisfy $x_ix_i'=E_i$.
\end{lemma}
\begin{proof}
Apply Lemma~\ref{lem:principal-open-regularity} to the displayed
polynomials. The operators above give
$D_i^{\lambda_i+1} f^{\circ}=0$ for $w_F,\ell_F,v_F,z_F,q_5$,
with respective left degrees
$\omega_1+\omega_2$, $\omega_1+\omega_3$, $\omega_1+\omega_4$,
$\omega_1+\omega_3$, $\omega_3$. Thus they are regular; $h$ is already
polynomial in regular functions.

The functions $w_F,\ell_F,v_F$ are prime. Their right degrees are
$\varpi_1+\varpi_3$, $\varpi_1+\varpi_3$, and $\varpi_3$.
A factorization would require a factor of left degree $\omega_1$,
whose right degree is $\varpi_2$ or $\varpi_4$ by
Proposition~\ref{prop:f4-fundamental-lines}; neither can occur in these
sums of dominant degrees. Finally, substitution gives all four identities
$x_ix_i'=E_i$, as checked in \cite[E1, Section~3]{ElectronicSupplement}.
\end{proof}

\begin{theorem}[The $F_4\downarrow B_4$ upper LP algebra]\label{thm:f4-explicit-seed}
The functions in \eqref{eq:f4-seed-labels}, with the exchanges
\eqref{eq:f4-exchanges}, form a normalized LP seed $\Sigma_F$.
Its associated exchange Laurent polynomials equal the ordinary ones.
With polynomial frozen coefficients,
\[ \B(F_4,\Spin_9)=\UB(\Sigma_F)=\ULP(\Sigma_F). \]
The frozen weights sum to $(2\rho_{F_4};2\rho_{B_4})$. Each frozen
coordinate valuation equals its divisor order on the branching algebra.
\end{theorem}
\begin{proof}
There is an explicit rational inverse. Recover $p_1=c_1$, $p_2=x_1$,
$p_4=c_2$, $p_3=\frac{c_6x_2+c_2c_4c_5}{x_3}$.
The functions $q_1,q_4,q_7,q_8$ then give
$\xi_1,\xi_4,\xi_7,\xi_8$, and $x_2=-\ell_F$ gives $\xi_3$.
The expression for $x_4=-w_F$ gives $\xi_2$. Finally, $r,t$ determine
$\xi_5,\xi_6$ by a linear system with determinant
\[ \xi_1\xi_4-\xi_2\xi_3=\frac{x_4}{p_1p_2}. \]
Substitution proves both inverse identities. The seed therefore has
the full fraction field and algebraically independent coordinates.
Direct LP substitution gives $\widehat E_i=E_i$ for these four
irreducible binomials.

Lemma~\ref{lem:f4-new-regularity} proves primality of the mutables and
regularity of their companions. Each companion-minus-parent left degree
has a negative Dynkin coordinate, so parent and companion are coprime.
The frozen sum follows from the highest weights. Let $B_F$ be the full outgoing-minus-incoming exchange matrix and $W_F$
the matrix of branching weights in the order \eqref{eq:f4-seed-labels}.
Integral elimination gives
\[ 0\longrightarrow\Z^4\xrightarrow{B_F^\T}\Z^{12}\xrightarrow{W_F^{\mathsf t}}P(F_4)\oplus P(B_4)\longrightarrow0. \]
Proposition~\ref{prop:exceptional-upper-criterion} now applies.
In the coordinate order of \eqref{eq:f4-seed-labels}, the Jacobian is
\[ \det\frac{\partial(\mathbf x,\mathbf c)}{\partial(\mathbf p,\boldsymbol\xi)}=P x_1x_2x_3x_4. \]
The frozen functions have fundamental left degree and are prime.
They are nonassociate to the mutables; hence their generic divisors
lie in the initial Laurent chart, which identifies their valuations.
\end{proof}

\section{Theta bases and polyhedral branching formulas}
\label{sec:applications}\label{sec:polyhedral}\label{subsec:exceptional-matrices}\label{app:inequality-matrices}
For the matrices in Section~\ref{app:seeds}, put
\[
\begin{gathered}
 U(B)=\UB_{\C[f_1,\ldots,f_s]}(B),\qquad
 R(B)=U(B)[(f_1\cdots f_s)^{-1}],\\
 V(B)=\Spec\UB_{\C[c_1^{\pm1},\ldots,c_s^{\pm1}]}(B^\vee).
\end{gathered}
\]
Here $B^\vee$ is the transposed-seed matrix defined in
Section~\ref{sec:binomial}. The $c_j$ are its frozen variables;
$U(B)$ is the finite intersection of the initial and adjacent
Laurent rings, with polynomial frozen coefficients. The following subsections specify the functions and their order.
Section~\ref{app:seeds} gives the ordered exchange monomials and the
integral completions $B_p$.

\subsubsection*{The original branching realizations.}
Theorems~\ref{thm:a2c-algebra-equality}, \ref{thm:e6-upper-equality},
and~\ref{thm:f4-explicit-seed}, together with
Corollary~\ref{thm:d4-upper-equality}, identify the original branching
algebras and their full fraction fields.
The geometric seeds below are obtained by LP mutations and choices
of exchange-term order; permutations only reorder their coordinates.
The upper-algebra comparison of Section~\ref{sec:setup} therefore gives,
for each of these choices of coordinates,
\begin{equation} \B(G,K)\simeq U(B). \label{eq:branching-realization} \end{equation}

The relation with the initial seeds in Parts~I and II is summarized
below. Row reversals exchange the two ordered monomials of a binomial;
they do not change its sum or the original LP functions.
\begin{center}
\begin{tabular}{@{}ll@{}}
\toprule
Branching & Passage from the initial seed\\
\midrule
$A_7\downarrow C_4$ & reverse row $8$\\
$A_9\downarrow C_5$ & mutation at $F_{-1,4}$\\
$D_4\downarrow G_2$ & mutate $(4,1)$, then reverse row $2$\\
$E_6\downarrow F_4$ & original term order\\
$F_4\downarrow B_4$ & reverse rows $3$ and $4$\\
\bottomrule
\end{tabular}
\end{center}
Section~\ref{app:seeds} specifies the orientations and coefficient
monomials of the geometric seeds.

\subsubsection*{Weights.}
In every case, the row $W_i$ is the branching weight of the function
in the extended orders below. The original weight formulas and
tables in Parts~I and II therefore determine $W$ without a new
branching calculation. For a mutation to the geometric seed use
\[ W'_k=\sum_j [B_{kj}]_+W_j-W_k,\qquad W'_i=W_i\quad(i\ne k). \]
All five matrices satisfy $BW=0$ and $\rank W=s$.

\subsubsection*{Regular functions.}
For each of the five copositive cases, put
\[ h_i=-B_p^{-1}e_i\in L^*,\qquad \psi_i=y^{h_i}T_i(y_1,\ldots,y_r), \qquad 1\le i\le r. \]
Section~\ref{app:bounds} constructs $T_i$ with nonnegative integer coefficients, constant term
one, and nonnegative exponents. If its coefficient polynomials satisfy
Lemma~\ref{lem:center}, $\psi_i$ is regular. Its unique
positive-chamber minimum is $h_i$, and the gradient matrix satisfies
\[ J_{\rm bnd}=-\bigl(B_p^{-1}e_1,\ldots,B_p^{-1}e_r\bigr)^{\mathsf t},\qquad J_{\rm bnd}B^\T=-I_r. \]
This proves \eqref{eq:finite-gradient}.

\begin{theorem}[Geometric models and boundary bases]\label{thm:geometric-models}
For each matrix $B$ in Section~\ref{app:seeds}, $V(B)$ is a smooth
affine log Calabi--Yau variety with the specified open subset obtained
by the toric blowups of Section~\ref{sec:binomial}. The geometric
comparison gives
\[
 \Ageom(V(B))_1\simeq R(B),\qquad
 U(B)=\bigoplus_{\substack{m\in\Z^{r+s}\\Hm\ge0}}\C\Theta_m,
\]
where $H$ is the matrix of boundary inequalities constructed in
Section~\ref{app:boundary-support-construction}.
Every valuation-bound space $R(\mathbf d)$ and every frozen valuation
minimum is given by \eqref{eq:all-thresholds}.
\end{theorem}
\begin{proof}
Section~\ref{app:affine} proves the affine models and the original
Laurent-chart containments by explicit complete intersections and finite
unit-ideal covers. Section~\ref{app:bounds} constructs regular functions
with $J_{\rm bnd}B^\T=-I_r$. Section~\ref{app:boundary-data} supplies every
boundary support, the required mutation sequences, and the positive balance.
The quadratic forms below are copositive. Corollary~\ref{cor:application}
therefore applies. Polynomial frozen coefficients are restored by the
finite-intersection argument in Section~\ref{sec:boundary}.
\end{proof}

\begin{corollary}[Branching theta bases and counting formulas]
\label{cor:geometric-branching-formulas}
The basis in Theorem~\ref{thm:geometric-models} is a homogeneous theta
basis of the corresponding branching algebra under the identification
\eqref{eq:branching-realization}. If $W$ has the function weights
as rows in the geometric coordinate order, then
\[
 \wt(\Theta_m)=W^{\mathsf t}m,\qquad
 \dim\B(G,K)_{\lambda,\mu}
 =\#\{m\in\Z^{r+s}:Hm\ge0,\ W^{\mathsf t}m=(\lambda;\mu)\}.
\]
The real weight fibres are bounded. The weight lattice here is the
image of $W^{\mathsf t}$, not a larger ambient lattice.
\end{corollary}
\begin{proof}
Proposition~\ref{prop:binomial} writes every exponent of $\Theta_m$
as $m+B^\T a$. Since $BW=0$, all its terms have weight
$W^{\mathsf t}m$. The basis in Theorem~\ref{thm:geometric-models}
therefore gives the counting formula.
The matrices $W$ have rank $s$, so
$\ker_{\R}W^{\mathsf t}=\op{im}_{\R}B^\T$.
The positive-balance relation \eqref{eq:balance} bounds every real
branching-weight fibre, by Proposition~\ref{prop:balance}.
\end{proof}

Under an integral change of lattice coordinates $m=Tg$, the matrices
transform as
\[ H_g=H_mT,\qquad W_g=T^{\mathsf t}W_m. \]

\subsection{The family \texorpdfstring{$A_{2n-1}\downarrow C_n$}{A2n-1 to Cn}}
\label{subsec:a9-geometric}\label{subsec:symplectic-models}
\subsubsection*{The case $n=2$.}
The seed has no mutable variables. The localized variety is a torus,
its mirror functions are characters, and restoring the frozen divisors
 gives $\C[p_1,p_2,p_3,s_2,c_2]$ with its monomial basis.

\subsubsection*{The cluster case $n=3$.}\label{subsec:a5c3-boundary-model}

The seed \eqref{eq:rank-three-seed} is of cluster type $A_3$.
The columns of its exchange matrix indexed by $(p_4,p_1,s_2)$ form
$-I_3$. Hence the rows generate a saturated rank-three sublattice.
For the frozen order
$(p_1,p_2,p_3,p_4,p_5,s_2,s_4,c_3)$,
optimizing words are, respectively,
\[ ((2,3),\varnothing,(3),(1,2,3),\varnothing,(3),\varnothing,(1,2)). \]

\begin{corollary}\label{thm:a5c3-canonical-model}
The branching algebra $\B(\SL_6,\Sp_6)$ has a homogeneous cluster theta
basis. Its elements are cluster monomials with polynomial frozen factors,
and every frozen valuation takes the minimum over the nonzero terms of a
finite basis expansion. This gives the case $n=3$ of the
comparison theorem.
\end{corollary}
\begin{proof}
Invert the eight frozen variables in
Theorem~\ref{thm:a2c-algebra-equality}. The minor $-I_3$ makes the
exchange-lattice map surjective, so \cite[Lemma~B.7]{GHKK} splits the
principal-coefficient family as the specified coefficient fibre times
a torus. Finite type gives the full principal-coefficient theta basis
\cite[Proposition~8.28 and Corollary~8.30]{GHKK}. The cluster chambers
cover its index space, so its elements are cluster monomials
\cite[Theorem~0.3]{GHKK}. Torus equivariance therefore descends this
basis to mutable cluster monomials times Laurent frozen characters
on the specified fibre, not merely a generic one.

For each frozen variable the displayed mutation word makes its exchange
column nonnegative. Hence the seed is optimized for its divisor
\cite[Lemma~9.2]{GHKK}. Proposition~9.7 and Lemma~9.10(2) there imply
that regularity along each frozen divisor is termwise in the theta
basis. Multiplying by a frozen character gives the same assertion at
every integral valuation threshold. Lemma~\ref{lem:frozen-descent}
shows that mutable cluster variables have zero frozen valuations.
Thus imposing all nonnegative thresholds selects precisely the cluster
monomials with polynomial frozen factors. Their span is the unlocalized
upper algebra by that lemma, so no polynomial-coefficient
cluster/upper equality is assumed. Proposition~\ref{prop:actual-frozen-valuations}
identifies the valuations with branching divisors, and the seed grading
gives homogeneity.
\end{proof}

\subsubsection*{The case $n=4$.}
Use the mutable order
$(x_1,\ldots,x_8)=(F_{3,1},F_{5,1},F_{1,2},F_{3,2},F_{5,2},F_{1,3},F_{3,3},F_{3,4})$
and the frozen order $(p_1,\ldots,p_7,s_2,s_4,s_6,c_4)$.
Reverse only the order of the two monomials in exchange $8$.
The functions do not change, and
$a^{\mathsf t}B_0a=2a_6a_8$.
The sequence $(4,1,2,3)$ gives complete intersections satisfying
Lemma~\ref{lem:independent-zero} on both the original and dual sides.
The eight regular functions satisfy $J_{\rm bnd}B^\T=-I_8$.
The eleven boundary supports give $44$ normals, obtained from
$23$ mutation sequences in the original coordinates.
Appendix~\ref{app:part-II} contains these calculations, so
Theorem~\ref{thm:geometric-models} applies.

\subsubsection*{The case $n=5$.}
We continue the example of Figures~\ref{fig:a9c5-seed} and
\ref{fig:a9c5-fibres}. The geometric seed is obtained by mutating the
initial function $F_{-1,4}$, shown in blue in Figure~\ref{fig:a9c5-seed}.
To match the matrices in
Section~\ref{app:seeds}, enumerate the initial mutable functions as
\[
\begin{split}
(v_1,\ldots,v_{15})={}(F_{7,1},F_{5,1},F_{3,1},F_{7,2},F_{5,2},F_{3,2},F_{1,2},
F_{5,3},F_{3,3},F_{1,3},F_{3,4},F_{1,4},F_{5,4},F_{-1,4},F_{3,5}),
\end{split}
\]
and use the frozen order
$(f_1,\ldots,f_{14})=(s_8,s_6,s_4,s_2,p_9,p_8,p_7,p_6,p_5,p_4,p_3,p_2,p_1,c_5)$.
This enumeration changes no functions and is not required for the
mutation. In this order $v_{14}=F_{-1,4}$, so mutation at that function
gives $x_i=v_i$ for $i\ne14$ and
$x_{14}=(s_4p_5p_3+F_{1,3}F_{3,5}p_4F_{-1,4})/F_{1,4}$.
Relative to the reordered initial seed, reverse row $15$ before this
mutation and again afterwards. The ordered monomials in
Section~\ref{app:seeds} give the resulting orientation. Their polynomial
substitutions preserve all coefficient factors.
The weight of $x_{14}$, in the order
$(\lambda_1,\ldots,\lambda_9;\mu_1,\ldots,\mu_5)$, is
\[ W_{14}=(1,0,1,1,1,0,1,0,0;0,0,1,0,1). \]
The marked fibre of $F_{-1,4}$ in Figure~\ref{fig:a9c5-fibres}
recovers \eqref{eq:a9-running-exchange}; subsequent exchanges follow
by LP mutation.
The geometric matrix satisfies $a^{\mathsf t}B_0a=2a_{14}a_{15}$.
The two affine varieties have the covers by four principal opens in
Section~\ref{app:affine}. The fifteen regular functions satisfy
$J_{\rm bnd}B^\T=-I_{15}$. The $109$ boundary terms, $66$ original
mutation sequences, and the positive balance prove this case of
Theorem~\ref{thm:geometric-models}.

These constructions give homogeneous branching theta bases for
$A_{2n-1}\downarrow C_n$ throughout $2\le n\le5$.

\begin{remark}[The cases $n\ge6$]\label{rem:higher-rank-scope}
Theorem~\ref{thm:a2c-algebra-equality} identifies the
$A_{2n-1}\downarrow C_n$ branching algebra with its upper LP algebra
for every $n\ge2$. The geometric comparison and theta-basis construction
are established here only for $2\le n\le5$; their extension to $n\ge6$
remains open.
\end{remark}

\subsection{\texorpdfstring{$D_4\downarrow G_2$}{D4 to G2}}
\label{subsec:d4-geometric}
Yan \cite{YanBranching} records a cluster algebra structure on
$\B(\Spin_8,G_2)$. We use the LP seed of
Sections~\ref{sec:exceptional-categories} and~\ref{sec:d4-g2-seed}
for the comparison with its mirror algebra.

Write $\mathbf x$ for the initial tuple in \eqref{eq:d4-rank-four-order}.
The sequence $(4,1)$ gives
\[
\begin{aligned}
 \widetilde x_4=\frac{c_6x_1+c_5x_2x_3}{x_4}=f_{12}-p_2s_4,\qquad
 \widetilde x_1=\frac{x_2x_3+c_3\widetilde x_4}{x_1}=g_2.
\end{aligned}
\]
The second identity follows from $g_2=p_2s_3-f_{13}$ and the quotient
coordinates of Section~\ref{sec:d4-g2-seed}. We take
\begin{equation}
 (\widetilde x_1,\widetilde x_2,\widetilde x_3,\widetilde x_4;f_1,\ldots,f_6)
 =(g_2,f_9,h_1,f_{12}-p_2s_4;p_1,-p_2,p_3,-s_1,s_4,v_D).
 \label{eq:d4-geometric-order}
\end{equation}
Here $\widetilde x_2=x_2$, $\widetilde x_3=x_3$, and $f_j=c_j$.
Both mutations use the associated exchange Laurent polynomials.
Afterwards, interchange the two monomials in exchange $2$.
Denote the resulting exchange matrix by $B_D^{\rm geom}$ and the
weight matrix in the order \eqref{eq:d4-geometric-order} by
$W_D^{\rm geom}$. These differ from the initial matrices $B_D,W_D$
of Section~\ref{sec:d4-g2-seed}; below in this subsection,
$B=B_D^{\rm geom}$ and $W=W_D^{\rm geom}$.
In Section~\ref{app:seeds}, these mutable functions are denoted by
$x_i$, as for the other geometric seeds.

The resulting orientation has
$a^{\mathsf t}B_0a=2a_1a_2+a_2a_4$.
Both the original and dual affine varieties are complete intersections
satisfying Lemma~\ref{lem:independent-zero} in their initial coordinates.
For the regular functions $\psi_i=y^{-B_p^{-1}e_i}T_i$ we use
\[ (T_1,T_2,T_3,T_4)=(1+y_1+y_1y_3,1+y_2,1+y_3,1+y_4+y_1y_4+y_1y_3y_4). \]
They satisfy the divisibility conditions in Lemma~\ref{lem:center}
and $J_{\rm bnd}B^\T=-I_4$.

The boundary functions are $\Psi_j=y_{4+j}P_j$, where
\[
\begin{gathered}
(P_1,P_2,P_3,P_4,P_5,P_6)=(1+y_3,1+y_2,1+y_1,1,1,1+y_2+y_4+y_2y_4+y_1y_4+y_1y_3y_4).
\end{gathered}
\]
Eight original mutation sequences in Section~\ref{app:boundary-data}
give all fourteen normals. Together with the positive balance, these
prove Theorem~\ref{thm:geometric-models} for this case.
In the order \eqref{eq:d4-geometric-order}, the inequality matrix is
\begingroup\small\setlength{\arraycolsep}{3.5pt}\renewcommand{\arraystretch}{1.02}
\begin{equation}
H_D=\left(\begin{array}{rrrr|rrrrrr}
0&0&0&0&1&0&0&0&0&0\\
0&0&1&0&1&0&0&0&0&0\\
0&0&0&0&0&1&0&0&0&0\\
0&1&0&0&0&1&0&0&0&0\\
0&0&0&0&0&0&1&0&0&0\\
1&0&0&0&0&0&1&0&0&0\\
0&0&0&0&0&0&0&1&0&0\\
0&0&0&0&0&0&0&0&1&0\\
0&0&0&0&0&0&0&0&0&1\\
0&0&0&1&0&0&0&0&0&1\\
0&1&0&0&0&0&0&0&0&1\\
0&1&0&1&0&0&0&0&0&1\\
1&0&0&1&0&0&0&0&0&1\\
1&0&1&1&0&0&0&0&0&1
\end{array}\right).
\label{eq:triality-geometric-H}
\end{equation}
\endgroup
Its rows are the exponents of the six boundary functions $\Psi_j$,
grouped by frozen index. The first four columns are mutable and the
last six are frozen. The inequalities $H_Dm\ge0$ give the branching formula of
Corollary~\ref{cor:geometric-branching-formulas}; all fourteen define
facets by the binary-support argument in Section~\ref{sec:boundary}.

\subsection{\texorpdfstring{$E_6\downarrow F_4$}{E6 to F4}}
\label{subsec:e6-geometric}
We use the original function order
\[
\begin{gathered}
(x_1,\ldots,x_8)=(u_{12},u_4,u_9,u_{21},u_{22},u_{23},u_{24},u_{25}),\\
(f_1,\ldots,f_{10})=(z_1,z_2,z_3,z_5,z_7,z_8,z_{10},z_{13},z_{14},z_{33}).
\end{gathered}
\]
The ordered terms of \eqref{eq:e6-formal-exchanges} define $B$ without
mutation, and $a^{\mathsf t}B_0a=2a_7a_8$.
The original and dual mutation sequences $(1,2)$ and $(8,1,2)$ give
complete intersections satisfying Lemma~\ref{lem:independent-zero}.
The eight regular functions of Section~\ref{app:bounds} satisfy
$J_{\rm bnd}B^\T=-I_8$. The ten boundary functions have $44$ terms;
$28$ original mutation sequences give their normals. The positive
balance in Section~\ref{app:boundary-data} completes the application
of Theorem~\ref{thm:geometric-models}.
In this order, the inequality matrix is
\begingroup\scriptsize\setlength{\arraycolsep}{2.9pt}\renewcommand{\arraystretch}{1.08}
\begin{equation}
H_E=\left(\begin{array}{rrrrrrrr|rrrrrrrrrr}
0&0&0&0&0&0&0&0&1&0&0&0&0&0&0&0&0&0\\
0&1&0&0&0&0&0&0&1&0&0&0&0&0&0&0&0&0\\
0&1&0&1&0&0&0&0&1&0&0&0&0&0&0&0&0&0\\
0&0&0&0&0&0&0&0&0&1&0&0&0&0&0&0&0&0\\
0&0&0&0&0&0&0&0&0&0&1&0&0&0&0&0&0&0\\
0&0&0&0&0&0&0&0&0&0&0&1&0&0&0&0&0&0\\
0&0&0&0&1&0&0&0&0&0&0&1&0&0&0&0&0&0\\
0&1&0&0&0&0&0&0&0&0&0&1&0&0&0&0&0&0\\
0&1&0&0&1&0&0&0&0&0&0&1&0&0&0&0&0&0\\
0&1&0&1&1&0&0&0&0&0&0&1&0&0&0&0&0&0\\
0&0&0&0&0&0&0&0&0&0&0&0&1&0&0&0&0&0\\
0&0&0&0&0&1&0&0&0&0&0&0&1&0&0&0&0&0\\
0&0&0&0&0&0&0&0&0&0&0&0&0&1&0&0&0&0\\
0&0&0&0&0&0&0&0&0&0&0&0&0&0&1&0&0&0\\
0&0&0&0&0&0&0&1&0&0&0&0&0&0&1&0&0&0\\
0&0&0&0&0&0&1&0&0&0&0&0&0&0&1&0&0&0\\
0&0&0&1&0&0&1&0&0&0&0&0&0&0&1&0&0&0\\
0&0&0&0&0&0&0&0&0&0&0&0&0&0&0&1&0&0\\
0&0&0&0&0&1&0&0&0&0&0&0&0&0&0&1&0&0\\
0&0&0&0&0&1&0&1&0&0&0&0&0&0&0&1&0&0\\
0&0&0&0&1&1&0&0&0&0&0&0&0&0&0&1&0&0\\
0&0&0&0&1&1&0&1&0&0&0&0&0&0&0&1&0&0\\
0&0&0&1&1&1&0&0&0&0&0&0&0&0&0&1&0&0\\
0&0&0&1&1&1&0&1&0&0&0&0&0&0&0&1&0&0\\
1&0&0&1&1&1&0&0&0&0&0&0&0&0&0&1&0&0\\
1&0&0&1&1&1&0&1&0&0&0&0&0&0&0&1&0&0\\
0&0&0&0&0&0&0&0&0&0&0&0&0&0&0&0&1&0\\
0&0&1&0&0&0&0&0&0&0&0&0&0&0&0&0&1&0\\
1&0&1&0&0&0&0&0&0&0&0&0&0&0&0&0&1&0\\
1&1&1&0&0&0&0&0&0&0&0&0&0&0&0&0&1&0\\
0&0&0&0&0&0&0&0&0&0&0&0&0&0&0&0&0&1\\
0&0&0&0&0&0&1&0&0&0&0&0&0&0&0&0&0&1\\
0&0&0&0&0&1&0&0&0&0&0&0&0&0&0&0&0&1\\
0&0&0&0&0&1&0&1&0&0&0&0&0&0&0&0&0&1\\
0&0&0&0&0&1&1&0&0&0&0&0&0&0&0&0&0&1\\
0&0&0&0&1&1&0&0&0&0&0&0&0&0&0&0&0&1\\
0&0&0&0&1&1&0&1&0&0&0&0&0&0&0&0&0&1\\
0&0&0&0&1&1&1&0&0&0&0&0&0&0&0&0&0&1\\
0&0&0&1&1&1&1&0&0&0&0&0&0&0&0&0&0&1\\
0&0&1&0&0&0&1&0&0&0&0&0&0&0&0&0&0&1\\
0&0&1&0&0&1&1&0&0&0&0&0&0&0&0&0&0&1\\
0&0&1&0&1&1&1&0&0&0&0&0&0&0&0&0&0&1\\
0&0&1&1&1&1&1&0&0&0&0&0&0&0&0&0&0&1\\
1&0&1&1&1&1&1&0&0&0&0&0&0&0&0&0&0&1
\end{array}\right).
\label{eq:e6-geometric-H}
\end{equation}
\endgroup
The first eight columns are mutable and the last ten are frozen.
The rows are grouped by frozen index. The inequalities $H_Em\ge0$
give the branching formula of
Corollary~\ref{cor:geometric-branching-formulas}.

\subsection{\texorpdfstring{$F_4\downarrow B_4$}{F4 to B4}}
\label{subsec:f4-geometric}
We use the functions of \eqref{eq:f4-seed-labels}, in the order
\[
\begin{gathered}
(x_1,x_2,x_3,x_4)=(p_2,-\ell_F,-v_F,-w_F),\quad
(f_1,\ldots,f_8)=(p_1,p_4,q_1,-q_4,q_7,q_8,r,t).
\end{gathered}
\]
No function changes. Reverse rows $3$ and $4$ of the categorical
outgoing-minus-incoming matrix to obtain the orientation in
Section~\ref{app:seeds}. Its quadratic form is
$a_1a_4+2a_2a_3$.
The original and dual affine varieties are complete intersections
satisfying Lemma~\ref{lem:independent-zero}.
For $\psi_i=y^{-B_p^{-1}e_i}T_i$, take
\[ (T_1,T_2,T_3,T_4)=(1+y_1,1+y_2+y_1y_2,1+y_3,1+y_4+y_2y_4). \]
These regular functions satisfy $J_{\rm bnd}B^\T=-I_4$.

The boundary functions are $\Psi_j=y_{4+j}P_j$, where
\[
\begin{gathered}
P_1=1+y_4+y_2y_4,\qquad P_2=1+y_3+y_2+y_1y_2,\qquad P_3=1+y_4,\quad P_4=1+y_3,\\
P_5=1+y_3+y_2,\qquad P_6=P_7=1,\qquad P_8=1+y_4+y_2y_4+y_1+y_1y_4+y_1y_2y_4.
\end{gathered}
\]
Nine original mutation sequences give their 22 normals, and
the positive balance proves the remaining hypothesis of
Theorem~\ref{thm:geometric-models}. In the preceding function order,
the inequality matrix is
\begingroup\small\setlength{\arraycolsep}{3.5pt}\renewcommand{\arraystretch}{1.02}
\begin{equation}
H_F=\left(\begin{array}{rrrr|rrrrrrrr}
0&0&0&0&1&0&0&0&0&0&0&0\\
0&0&0&1&1&0&0&0&0&0&0&0\\
0&1&0&1&1&0&0&0&0&0&0&0\\
0&0&0&0&0&1&0&0&0&0&0&0\\
0&0&1&0&0&1&0&0&0&0&0&0\\
0&1&0&0&0&1&0&0&0&0&0&0\\
1&1&0&0&0&1&0&0&0&0&0&0\\
0&0&0&0&0&0&1&0&0&0&0&0\\
0&0&0&1&0&0&1&0&0&0&0&0\\
0&0&0&0&0&0&0&1&0&0&0&0\\
0&0&1&0&0&0&0&1&0&0&0&0\\
0&0&0&0&0&0&0&0&1&0&0&0\\
0&0&1&0&0&0&0&0&1&0&0&0\\
0&1&0&0&0&0&0&0&1&0&0&0\\
0&0&0&0&0&0&0&0&0&1&0&0\\
0&0&0&0&0&0&0&0&0&0&1&0\\
0&0&0&0&0&0&0&0&0&0&0&1\\
0&0&0&1&0&0&0&0&0&0&0&1\\
0&1&0&1&0&0&0&0&0&0&0&1\\
1&0&0&0&0&0&0&0&0&0&0&1\\
1&0&0&1&0&0&0&0&0&0&0&1\\
1&1&0&1&0&0&0&0&0&0&0&1
\end{array}\right).
\label{eq:f4-geometric-H}
\end{equation}
\endgroup
Its first four columns are mutable and its last eight are frozen.
The inequalities $H_Fm\ge0$ give the branching formula of Corollary~\ref{cor:geometric-branching-formulas}; all
twenty-two inequalities define facets by the binary-support argument
in Section~\ref{sec:boundary}.

\subsection{\texorpdfstring{$B_3\downarrow G_2$}{B3 to G2} and \texorpdfstring{$G_2\downarrow A_2$}{G2 to A2}}
\label{subsec:hypersurface-models}
The normal-form proofs in Section~\ref{subsec:hypersurface-cases} give
homogeneous bases
\begin{equation}\label{eq:hypersurface-theta}
 \{\mathbf f^{\mathbf v}q^ir^j:
       \mathbf v\in\N^{I_{\rm fr}},\ i,j\in\N,\ ij=0\},
\end{equation}
where $\mathbf f=(p_1,p_2,p_3,s,h)$ for $B_3\downarrow G_2$ and
$\mathbf f=(p_1,p_2,s,t)$ for $G_2\downarrow A_2$.
These are the cluster theta bases: each exchange row has a unit entry,
so the split coefficient-fibre argument of
Corollary~\ref{thm:a5c3-canonical-model} applies, and the two type-$A_1$
cluster chambers give the monomials avoiding $qr$.
Counting these monomials gives the following closed formulas.
Write $[x]_+=\max(x,0)$; all Dynkin coefficients below are nonnegative
integers, with the fundamental-weight conventions of
Section~\ref{subsec:hypersurface-cases}.

\subsubsection*{The case $B_3\downarrow G_2$.}
For $\lambda=A\omega_1+B\omega_2+C\omega_3$ and
$\mu=D\varpi_1+E\varpi_2$, we have
\begin{equation}\label{eq:b3g2-closed-formula}
\begin{aligned}
 \dim\B(\Spin_7,G_2)_{\lambda,\mu}
 =\bigl[1&+\min\{A,E,A+B+C-D-E\}\\
         &-\max\{0,A-D,E-B,A+B-D-E\}\bigr]_+.
\end{aligned}
\end{equation}
Indeed, the weight equations for \eqref{eq:hypersurface-theta} give
$i=[B-E]_+$ and $j=[E-B]_+$. Put $z=v_h+j$, where $v_h$ is the
exponent of $h$. The five frozen exponents are then
\[
 (A-z,\ E-z,\ D-A+z-i,\ A+B+C-D-E-z,\ z-j).
\]
Their nonnegativity is equivalent to
\[
 \max\{0,A-D,E-B,A+B-D-E\}\le z
 \le\min\{A,E,A+B+C-D-E\}.
\]
Every integer in this interval gives exactly one basis monomial,
proving \eqref{eq:b3g2-closed-formula}.

\subsubsection*{The case $G_2\downarrow A_2$.}
For $\lambda=A\omega_1+B\omega_2$ and
$\mu=C\varpi_1+D\varpi_2$, we have
\begin{equation}\label{eq:g2a2-closed-formula}
\begin{aligned}
 \dim\B(G_2,\SL_3)_{\lambda,\mu}
 =\bigl[1+\min\{B,D\}-\max\{0,B-C,D-A,C+D-A-B\}\bigr]_+.
\end{aligned}
\end{equation}
Now $i=[C-B]_+$ and $j=[B-C]_+$. With $z=v_{p_2}+j$, the frozen
exponents are
\[
 (D-z,\ z-j,\ A-D+z-i,\ B-z).
\]
Thus $\max\{0,B-C,D-A,C+D-A-B\}\le z\le\min\{B,D\}$,
and counting the integers proves \eqref{eq:g2a2-closed-formula}.

Both bases respect every frozen valuation. Modulo a frozen variable
$f_k$, the exchange becomes $qr$ equal to a nonzero monomial in the
remaining frozen variables. This quotient is a domain, and its normal
monomials avoiding $qr$ are linearly independent. Consequently $f_k$
is prime and, for every nonzero finite expansion,
$\nu_{f_k}\!\left(\sum_{i,j,\mathbf v}
       a_{i,j,\mathbf v}\mathbf f^{\mathbf v}q^ir^j\right)
   =\min_{a_{i,j,\mathbf v}\ne0}v_k.
$
Multiplication by frozen characters gives the same assertion at every
integral valuation threshold.

\subsection{The family \texorpdfstring{$D_n\downarrow B_{n-1}$}{Dn to Bn-1}}
\begin{proposition}\label{prop:orthogonal-polynomial}
For $n\ge3$, $\B(\Spin_{2n},\Spin_{2n-1})$ is a polynomial algebra
on $2n-1$ homogeneous generators, all frozen in its LP seed with no mutable variables.
\end{proposition}
\begin{proof}
The multiplicity-free orthogonal restriction rule is
\[ \lambda_1\ge\mu_1\ge\lambda_2\ge\cdots\ge \lambda_{n-1}\ge\mu_{n-1}\ge|\lambda_n|, \]
with all coordinates integral or all half-integral
\cite[the $\Spin_{2n-1}\subset\Spin_{2n}$ example in Section~2]{PasquierRessayre}. The consecutive differences down
to $\lambda_{n-1}-\mu_{n-1}$, followed by
$\mu_{n-1}-\lambda_n$ and $\mu_{n-1}+\lambda_n$, identify its
branching-weight monoid additively with $\N^{2n-1}$. Indeed the last
two differences recover $\mu_{n-1},\lambda_n$ by their half-sum and
half-difference, and the rest are recovered successively. Nonzero
generators of the primitive degrees therefore freely generate the
multiplicity-free branching algebra. The simply connected weight
lattices are essential to this assertion.
\end{proof}

Its polyhedral model is the nonnegative orthant. The degree map is the
interlacing-coordinate map in the proof, so the lattice points of each
weight fibre give the corresponding branching multiplicity.

\subsection{Saturation}
\label{subsec:saturation}
A submonoid $S$ of a lattice $\Gamma$ is \emph{saturated} if
$N\gamma\in S$ with $N\in\Z_{>0}$ and $\gamma\in\Gamma$ implies
$\gamma\in S$. For a model in
Corollary~\ref{cor:geometric-branching-formulas}, put $d=r+s$ and
\[
C=\{m\in\R^d:Hm\ge0\},\qquad q=W^\T,\qquad
\Gamma=q(\Z^d),\qquad P_\gamma=C\cap q^{-1}(\gamma).
\]
The branching semigroup is $q(C\cap\Z^d)$. Since
$P_{N\gamma}=NP_\gamma$ and every nonempty rational fibre has a rational
point, saturation in $\Gamma$ is equivalent to
\[
P_\gamma\ne\varnothing\quad\Longrightarrow\quad
P_\gamma\cap\Z^d\ne\varnothing\qquad(\gamma\in\Gamma).
\]
It suffices to find an integral lift in $C$ of each element of the
Hilbert basis of $q(C)\cap\Gamma$, its minimal additive generating
set: summing these lifts gives every
integral weight in the projected cone. This is a finite criterion on
$(H,W)$, not a test on weights below a cutoff.

For $D_n\downarrow B_{n-1}$, saturation follows from the interlacing
coordinates of Proposition~\ref{prop:orthogonal-polynomial}. For
$B_3\downarrow G_2$ and $G_2\downarrow A_2$, it follows from
\eqref{eq:b3g2-closed-formula}--\eqref{eq:g2a2-closed-formula}, whose
interval endpoints are homogeneous integral functions of the weights.
These recover the corresponding results of
\cite[Theorem~5]{PasquierRessayre}. That theorem also proves saturation
for $F_4\downarrow B_4$, $E_6\downarrow F_4$, and
$A_{2n-1}\downarrow C_n$ for $2\le n\le5$. The triality model gives
the following additional result. A polytope is \emph{integral} if
all its vertices are lattice points.

\begin{proposition}\label{prop:triality-saturation}
	Every nonempty integral-weight fibre of the model
	\eqref{eq:triality-geometric-H} is an integral polytope.
	In particular, the branching semigroup for $\Spin_8\downarrow G_2$
	is saturated.
\end{proposition}
\begin{proof}
	Choose $m_0\in\Z^{10}$ with $(W_D^{\rm geom})^\T m_0=\gamma$. Since
	$\ker_\R (W_D^{\rm geom})^\T=\op{im}_\R (B_D^{\rm geom})^\T$, write
	$m=m_0+(B_D^{\rm geom})^\T a$ and put $A=H_D(B_D^{\rm geom})^\T$. At a vertex, four independent
	active equations give $A_Ia=-(H_D)_Im_0$. The nonsingular four-row
	minors of $A$ have absolute value one, except for sixteen of absolute
	value two. The exhaustive determinant calculation is recorded in
	\cite[S1]{ElectronicSupplement}. For $s=H_Dm\ge0$, the identities
	\[
	\begin{gathered}
		s_1+s_{14}=s_2+s_{13},\qquad s_3+s_{11}=s_4+s_9,\qquad
		s_3+s_{12}=s_4+s_{10},\\
		s_5+s_{13}=s_6+s_{10},\qquad s_9+s_{12}=s_{10}+s_{11}
	\end{gathered}
	\]
	force each exceptional active set to extend to one containing a
	unimodular basis; all sixteen replacements are listed there.
	For example, vanishing of $s_1,s_3,s_6,s_{10}$ forces $s_5=s_{13}=0$,
	and the rows $(1,3,5,6)$ of $A$ have determinant one.
	Thus every vertex has integral $a$, and hence integral $m$.
	The fibres are bounded, so a nonempty fibre has an integral vertex.
	The saturation criterion above completes the proof.
\end{proof}

\appendix
\section{Auxiliary geometric arguments}
\label{app:part-I}
\subsection{Smooth affine models}\label{app:affine-lemmas}
We use the following two lemmas for the affine models in
Section~\ref{app:affine}.

\begin{lemma}[The independent-zero test]\label{lem:independent-zero}
Let $S=\C[c_1^{\pm1},\ldots,c_s^{\pm1}]$ and let $G_i$ be primitive
binomials with linearly independent exponent-difference vectors, independent
of $u_i$. In
\[ Z=\Spec S[u_1,\ldots,u_r,v_1,\ldots,v_r]/(u_iv_i-G_i)_{i=1}^r, \]
suppose every possible zero set $I=\{i:u_i=0\}$ has no dependency inside
$I$: each $G_i$, $i\in I$, is independent of all $u_j$, $j\in I$.
Then $Z$ is a smooth integral complete intersection of dimension $r+s$.
The union of its initial and adjacent torus charts has complement of
codimension at least two, and its coordinate ring is the finite initial
upper intersection.
\end{lemma}
\begin{proof}
At a point with zero set $I$, invert the other $u_j$ and eliminate their
companions. A nonzero $v_i$ gives a Jacobian pivot in $u_i$.
For the remaining indices, both monomials $A_i,B_i$ of $G_i$ are units,
and $G_i=0$ gives
$ dG_i=A_i\,d\log(A_i/B_i)$.
These logarithmic rows are independent even after deleting the columns
in $I$, since those columns are zero in the selected rows. The Jacobian
therefore has full rank. On the stratum with zero set $I$, there are
$|I|$ independent character equations on the remaining torus and
$|I|$ free companions. Its dimension is $r+s-|I|$. No irreducible
component can be confined to the union of nonempty zero strata, because
$r$ equations in a smooth $(2r+s)$-dimensional space have components of
dimension at least $r+s$. Thus every component meets the single initial
torus; $Z$ is integral. Strata with $|I|\ge2$ have codimension at least two.

On the open where only $u_i$ is allowed to vanish, the ring is the
hypersurface $u_iv_i=G_i$. Its two Laurent tori omit only
$u_i=v_i=0$, of codimension two. Normality identifies its regular
functions with their Laurent intersection, and then identifies the
ring of $Z$ with the full finite intersection.
\end{proof}

There is a particularly elementary way to check the hypothesis. For a
candidate set $I$, inspect the two monomials of each $G_i$, $i\in I$.
Exactly one vanishing monomial makes $I$ impossible. For each surviving
candidate, verify that neither monomial vanishes. This checks the hypothesis directly from the monomial supports.

\begin{lemma}[Finite unit-ideal covers]\label{lem:unit-cover}
Work over a Laurent coefficient ring. Suppose the active exchange vectors are linearly independent and
generate a saturated lattice at each step of a finite binomial mutation calculation. Allow inversion of an
active variable and proper LP freezing. Suppose every leaf has the
independent-zero property and each split is justified by
\[ u_i u_i'=A+u_j^e C,\qquad e>0,\quad A\text{ a coefficient unit}. \]
Then the initial LP algebra and its finite upper bound agree. The common algebra is finitely generated. Its spectrum is smooth,
and the initial and adjacent torus charts have complement of codimension
at least two.
\end{lemma}
\begin{proof}
Lattice saturation makes all active exchange vectors primitive, hence
all active binomials irreducible. At every mutation we use the associated
exchange Laurent polynomial. Theorem~\ref{thm:upper-invariance} identifies
the finite upper bounds throughout the calculation; no equality between
ordinary and associated exchange polynomials is needed.

The equation gives $1\in(u_i,u_j)$, so the two indicated
principal opens cover. Localization commutes with the finite Laurent
intersection. The adjacent condition at the inverted variable becomes
redundant, and the result is the properly frozen upper bound. Proper
freezing preserves coefficient units and is compatible with mutation
\cite[Section~3.5]{LamPylyavskyy}.
At a leaf, Lemma~\ref{lem:independent-zero} identifies the upper bound
with the algebra generated by the coordinates and their companions.
These generators are LP elements on the corresponding localization.
Induction over the unit-ideal cover therefore gives equality of LP and
upper algebras: for comaximal $a,b$ in a domain $A$,
$A[a^{-1}]\cap A[b^{-1}]=A$. The finite cover also gives finite generation
and smoothness; generators and the finitely many unit-ideal relations can
be cleared of denominators to prove the finite-generation assertion.

In the initial upper bound $U$, the ideal $(u_i)$ is the contraction of
its prime ideal in the $i$th adjacent Laurent ring. Indeed $u_i$ is a
unit in all other rings of the finite Laurent intersection. The different $u_i$ are
nonassociate, so no height-one prime contains two of them. This proves
the assertion about the complement of the torus charts.
\end{proof}

The residue volume on each leaf is
$\bigwedge_i d\log u_i\wedge\bigwedge_jd\log c_j$ on its torus.
A mutation preserves this form up to sign, so it gives a nowhere-zero
logarithmic volume on the glued smooth affine variety. The initial
charts with one mutable variable, after normalization by a coefficient unit, are
$u_i v_i=1+y_i$. Since the $u$-exponent of $y_j$ is $(B_p)_{ij}=B_{ij}$,
their rays are $b_i=B^\T e_i$ and their centres are exactly
$1+y_i=0$. These charts form a big open and hence give the required
immersion $V^\circ\subset V$. The ambient-codimension argument in
Section~\ref{sec:binomial} then supplies the comparison along skeletal
curves.

\subsection{Boundary supports}\label{app:boundary-support-construction}
For each frozen position $f$, we construct
$K_f\subseteq\{0,1\}^r$ and put
\[ \Psi_f=y_f\sum_{k\in K_f}\prod_{i=1}^r y_i^{k_i},\qquad h_{f,k}=e_f^*+\sum_{i=1}^r k_i e_i^*. \]
Start with $K_f=\{0\}$ and apply the divisibility test of
Lemma~\ref{lem:center} in the order $i=1,\ldots,r$.
For fixed exponents other than the $i$th, the coefficient polynomial
is $A+B'T$. There is no condition if $\delta_{i,\widehat k}\le0$;
if $\delta_{i,\widehat k}=1$, divisibility requires $A=B'$;
if $\delta_{i,\widehat k}\ge2$, the polynomial must vanish.
Whenever the required power is one and only one of
$k$ and $k+(1-2k_i)e_i$ is present, adjoin the other. Repeat until
no new exponent is added.

For a given seed, one must then check every coefficient polynomial
in Lemma~\ref{lem:center}, including the conditions whose required
power is greater than one. When these checks hold, the resulting
$\Psi_f$ are regular and their supports give the boundary inequalities.
Section~\ref{app:boundary-data} applies this construction to the five
geometric seeds, and gives the corresponding mutation sequences and
positive balances.

\section{Data for the branching models}\label{app:part-II}
\label{app:geometric-input}
\subsection{Categorical degree matrices}
\label{app:categorical-degrees}
The degree formulas and their categorical coordinates are defined in
Section~\ref{sec:exceptional-categories}. We record their coefficient
matrices here, in the coordinate orders used there.

\subsubsection*{The degree for $D_4\downarrow G_2$.}
For \eqref{eq:d4-degree}, the rows correspond to
$(e,p^+,p^-,m_1,m_2)$ and the columns to
$(\lambda_1,\ldots,\lambda_4;\mu_1,\mu_2)$. The formula includes the
factor $1/2$:
\[
N_D=\left(\begin{array}{rrrrrr}
0&-2&1&2&2&-1\\0&-2&2&2&2&0\\0&-2&3&2&4&-1\\0&0&0&0&2&0\\\hline
0&2&-1&0&0&1\\0&2&0&-2&0&0\\0&0&-1&0&0&-1\\0&0&0&-2&-2&0\\\hline
0&0&1&0&2&-1\\2&-2&0&2&0&0\\0&2&-1&-2&-2&1\\2&-2&0&2&0&0\\\hline
0&2&-2&0&-2&2\\-2&2&2&-2&0&0
\end{array}\right).
\]

\subsubsection*{The degree for $F_4\downarrow B_4$.}
For \eqref{eq:f4-degree}, the rows correspond to $(e,p^+,p^-,m_3)$
and the columns to $(\lambda_1,\ldots,\lambda_4;\mu_1,\ldots,\mu_4)$.
No scalar denominator occurs:
\[
N_F=\left(\begin{array}{rrrr|rrrr}
-1&-1&0&0&1&-3&-2&4\\
-3&-2&-1&0&0&-5&-5&8\\
-2&-2&0&0&0&-4&-3&6\\
0&-1&0&0&0&-2&-1&3\\\hline
2&1&1&0&0&3&3&-4\\
3&3&1&0&0&5&6&-8\\
2&3&0&0&1&4&4&-6\\
1&2&0&0&1&2&2&-3\\\hline
-1&-1&-1&0&0&-2&-3&4\\
0&0&0&0&0&0&0&0\\
-2&-3&0&1&0&-4&-4&6\\
-1&-2&1&0&0&-2&-2&3\\\hline
0&1&-1&0&0&1&0&-1
\end{array}\right).
\]

For the matrix $N_E$ in \eqref{eq:e6-hom-degree}, see
\cite[A5, Appendix~A]{ElectronicSupplement}. Its ordered Hom coordinates
and the proof of integrality are in Section~\ref{sec:e6-categorical-reconstruction}.

\subsection{Geometric seed data and integral completions}
\label{app:seeds}\label{app:normalization}
The functions and their orders are specified in
Sections~\ref{subsec:symplectic-models}--\ref{subsec:f4-geometric}.
For $A_7\downarrow C_4$, use the exchange formulas of
Section~\ref{sec:folded-seed} with row $8$ reversed. For $E_6\downarrow F_4$,
use the ordered terms of \eqref{eq:e6-formal-exchanges}.
The following table determines the remaining three oriented exchange
matrices. The unimodular completions used for the dual coordinates are
supplied in the electronic data.

Write $z=(x_1,\ldots,x_r,f_1,\ldots,f_s)$ and
$E_i=M_i^++M_i^-$. For the seeds with four mutable variables and the $A_9\downarrow C_5$ seed, the table fixes the orientation
$B_{ij}=\exp_{z_j}M_i^+-\exp_{z_j}M_i^-$.
In particular, exchanging the two entries of a row is not implicit.
All ordinary and associated exchanges agree at these starting seeds:
the binomials have primitive, linearly independent exponent-difference
vectors and disjoint monomial supports.
Subsequent mutations use the associated exchange Laurent polynomials.

{\small
\begin{longtable}{@{}rll@{}}
\toprule
$i$ & $M_i^+$ & $M_i^-$\\
\midrule
\endhead
\multicolumn{3}{l}{\textbf{$A_9\downarrow C_5$}: $r=15$, $s=14$}\\[2pt]
1&$x_{5}f_{5}$&$x_{2}x_{4}$\\
2&$x_{1}x_{6}$&$x_{3}x_{5}$\\
3&$x_{2}x_{7}$&$x_{6}f_{13}$\\
4&$x_{1}x_{8}f_{6}$&$x_{5}f_{1}f_{5}f_{7}$\\
5&$x_{2}x_{4}x_{9}$&$x_{1}x_{6}x_{8}$\\
6&$x_{3}x_{5}x_{10}$&$x_{2}x_{7}x_{9}$\\
7&$x_{6}f_{4}f_{11}f_{13}$&$x_{3}x_{10}f_{12}$\\
8&$x_{5}x_{11}f_{7}$&$x_{4}x_{9}x_{13}$\\
9&$x_{6}x_{8}x_{12}$&$x_{5}x_{10}x_{11}$\\
10&$x_{7}x_{9}f_{3}f_{9}$&$x_{6}x_{14}$\\
11&$x_{9}x_{13}f_{14}$&$x_{8}x_{12}x_{15}$\\
12&$x_{11}x_{14}$&$x_{9}f_{10}f_{14}$\\
13&$x_{8}x_{15}f_{8}$&$x_{11}f_{2}f_{7}f_{9}$\\
14&$x_{10}x_{15}f_{10}$&$x_{12}f_{3}f_{9}f_{11}$\\
15&$x_{11}x_{14}$&$x_{13}f_{3}f_{11}f_{14}$\\
\addlinespace
\multicolumn{3}{l}{\textbf{$D_4\downarrow G_2$}: $r=4$, $s=6$}\\[2pt]
1&$x_{2}x_{3}$&$x_{4}f_{3}$\\
2&$x_{1}x_{4}$&$f_{2}f_{6}$\\
3&$f_{2}f_{3}f_{4}$&$x_{1}f_{1}$\\
4&$x_{1}f_{5}$&$f_{6}$\\
\addlinespace
\multicolumn{3}{l}{\textbf{$F_4\downarrow B_4$}: $r=4$, $s=8$}\\[2pt]
1&$x_{4}f_{5}$&$x_{2}f_{8}$\\
2&$x_{1}x_{3}f_{3}$&$x_{4}f_{2}f_{5}$\\
3&$x_{2}f_{6}$&$f_{2}f_{4}f_{5}$\\
4&$x_{2}f_{7}$&$f_{1}f_{3}f_{8}$\\
\addlinespace
\bottomrule
\end{longtable}
}
For each seed choose a unimodular completion $B_p$ by adjoining
standard basis rows to $B$. The completions used in the computations
are the electronic matrices named \texttt{P}; all formulas below
involving $B_p$ refer to these choices.

\subsection{Affine models for the branching pairs}\label{app:affine}
We apply Lemmas~\ref{lem:independent-zero} and~\ref{lem:unit-cover}
to the seeds in Section~\ref{app:seeds}. Polynomial frozen coefficients
are restored by the argument of Section~\ref{app:polynomial-frozen}.
\subsubsection*{The complete-intersection cases.}
For $A_7\downarrow C_4$, the sequence $(4,1,2,3)$ gives the independent-zero
property on both the original and dual sides. For $E_6\downarrow F_4$, use
$(1,2)$ and $(8,1,2)$, respectively. The cases with four mutable variables already have
the property in their geometric coordinates. Their dual exchanges,
with mutable variables $u_i$ and invertible coefficient $d_1$, are
\[
\begin{array}{c|llll}
&D_1&D_2&D_3&D_4\\\hline
D_4\downarrow G_2&u_2u_4+u_3&1+u_1&1+d_1u_1&u_1+u_2\\
F_4\downarrow B_4&1+u_2&u_1+u_3u_4&1+d_1u_2&u_1+u_2.
\end{array}
\]
In the $D_4\downarrow G_2$ case the only possible zero sets of mutable coordinates
are the empty set, singletons, and $\{2,3\},\{3,4\}$, on both sides.
For $F_4\downarrow B_4$ the two pairs are $\{1,3\},\{3,4\}$. Substitution in
the binomials verifies these lists: if precisely one monomial vanishes,
the corresponding exchange equation is impossible. For every remaining
zero set, the exchanges indexed by that set are independent of its
coordinates. Lemma~\ref{lem:independent-zero} therefore applies.

\subsubsection*{The two $A_9\downarrow C_5$ covers.}
Each table describes three covers by pairs of principal open subsets
and four terminal complete-intersection presentations. The final column gives the number of remaining mutable variables.
Every active row lattice in both tables is saturated; every used
binomial update agrees with independent monomial substitution.
There are $33$ original and $53$ dual mutation steps. The four final
charts have $12,12,13,14$ mutable variables, respectively, and the
numbers of possible zero sets are $322,470,670,1160$ on both sides.
Testing the $32768$ subsets across these charts gives the
independent-zero property. In the fourth chart, dependence of one
exchange polynomial on another mutable coordinate is not symmetric
in the two indices. The independent-zero test does not require this
symmetry.
\subsubsection*{The $A_9\downarrow C_5$ original cover.}
At each child, invert the indicated parent coordinate and then apply
the additional mutation sequence.
\begin{center}\small
\begin{tabular}{@{}rlll@{}}
\toprule
Node & Parent/inversion & Additional word & Conclusion\\
\midrule
1&--&$\begin{gathered}(6,5,7,13,9,11,3,1,12,6,10,\\1,12,8,11,9)\end{gathered}$&split $(1,12)$\\[3pt]
2&1 / 1&$(12,9,10,2)$&split $(8,11)$\\[3pt]
3&2 / 8&$(13,3,11)$&split $(3,11)$\\[3pt]
4&3 / 3&$(11,5,7,4)$&12 mutable variables\\[3pt]
5&3 / 11&$\varnothing$&12 mutable variables\\[3pt]
6&2 / 11&$\varnothing$&13 mutable variables\\[3pt]
7&1 / 12&$(3,5,13,6,15,14)$&14 mutable variables\\[3pt]
\bottomrule
\end{tabular}
\end{center}
\subsubsection*{The $A_9\downarrow C_5$ dual cover.}
Use the same convention for the dual cover.
\begin{center}\small
\begin{tabular}{@{}rlll@{}}
\toprule
Node & Parent/inversion & Additional word & Conclusion\\
\midrule
1&--&$\begin{gathered}(6,5,7,11,9,3,6,2,12,1,11,\\1,4,8,15,7,3,4,14,11,1,6)\end{gathered}$&split $(8,2)$\\[3pt]
2&1 / 8&$(12,13,2,1,13,12)$&split $(10,6)$\\[3pt]
3&2 / 10&$(6,11)$&split $(6,11)$\\[3pt]
4&3 / 6&$(11,5,7,4)$&12 mutable variables\\[3pt]
5&3 / 11&$\varnothing$&12 mutable variables\\[3pt]
6&2 / 6&$\varnothing$&13 mutable variables\\[3pt]
7&1 / 2&$\begin{gathered}(7,4,5,11,15,13,14,9,10,15,13,\\4,11,1,12,10,13,3,9)\end{gathered}$&14 mutable variables\\[3pt]
\bottomrule
\end{tabular}
\end{center}

\subsection{Regular functions with negative coordinate gradients}
\label{app:bounds}\label{app:comparison-certificates}
We compute the polynomials $T_i$ in the functions
$\psi_i=y^{-B_p^{-1}e_i}T_i$ introduced in
Section~\ref{sec:applications}. Their regularity is tested by
Lemma~\ref{lem:center}.

\subsubsection*{The polynomials.}
The four polynomials for $D_4\downarrow G_2$ and $F_4\downarrow B_4$ are given in
Sections~\ref{subsec:d4-geometric} and~\ref{subsec:f4-geometric}.
Their coefficient polynomials satisfy the required divisibility directly.
For $A_7\downarrow C_4$, $A_9\downarrow C_5$, and $E_6\downarrow F_4$, use the following finite
construction. Start with the exponent set $S=\{0\}$ and repeat the
steps in the order $j=1,\ldots,r$, six times. For the coefficient polynomial with exponents other than the $j$th fixed at $k$, put
\[ d=\delta_{ij}-\sum_{t\ne j}B_{jt}k_t. \]
In these calculations $d\le1$. The exponent set is unchanged when $d\le0$.
When $d=1$, pair its least remaining exponent $a$ with $a+1$,
remove this pair from the remaining set, and continue. The union of
the pairs replaces this exponent set and contains every previous exponent. After the six repetitions, set $T_i=\sum_{k\in S}y^k$.
Division of each resulting coefficient polynomial by $(1+y_j)^{[d]_+}$ verifies the
construction.

The numbers of terms are
\[
\begin{array}{c|l}
 A_7\downarrow C_4&(6,4,4,20,6,6,4,5)\\
 A_9\downarrow C_5&(6,11,10,10,70,85,5,10,65,21,15,23,15,11,6)\\
 D_4\downarrow G_2&(3,2,2,4)\\
 E_6\downarrow F_4&(6,3,4,3,4,9,6,2)\\
 F_4\downarrow B_4&(2,3,2,3).
\end{array}
\]
Every constant coefficient is one. Lemma~\ref{lem:center}, applied with
$-\pair{h_i}{b_j}=\delta_{ij}$, proves regularity of all these functions.

\subsection{Boundary functions and mutation sequences}\label{app:boundary-data}
Apply the construction in Section~\ref{app:boundary-support-construction}
to each frozen variable, in the mutable order $1,\ldots,r$.
Every resulting coefficient polynomial satisfies Lemma~\ref{lem:center};
no nonzero coefficient polynomial requires a power greater than one.
Thus
\[ \Psi_f=y_f\sum_{k\in K_f}\prod_{i=1}^r y_i^{k_i},\qquad h_{f,k}=e_f^*+\sum_{i=1}^r k_i e_i^*. \]
The numbers of terms are $44,109,14,44,22$, in the order
$A_7\downarrow C_4,A_9\downarrow C_5,D_4\downarrow G_2,E_6\downarrow F_4,F_4\downarrow B_4$.
Sort the normals first by frozen position and then lexicographically
by their mutable exponents.
\subsubsection*{Positive balances.} Set $\zeta_j=1$ except for the following entries:
\[
\begin{gathered}
\text{$A_7\downarrow C_4$:}\\
\zeta_{6}=5,\quad \zeta_{15}=16,\quad \zeta_{16}=2,\quad \zeta_{25}=13,\quad \zeta_{26}=5\\
\zeta_{34}=2,\quad \zeta_{41}=9,\quad \zeta_{43}=8
\end{gathered}
\]
\[
\begin{gathered}
\text{$A_9\downarrow C_5$:}\\
\zeta_{25}=3, \quad \zeta_{32}=5, \quad \zeta_{38}=38, \quad \zeta_{39}=32, \quad \zeta_{58}=56, \quad \zeta_{60}=23, \quad \zeta_{79}=40\\
\zeta_{81}=8, \quad \zeta_{89}=34, \quad \zeta_{93}=7, \quad \zeta_{95}=10, \quad \zeta_{96}=10, \quad \zeta_{99}=2, \quad \zeta_{106}=26\\
\zeta_{108}=6
\end{gathered}
\]
\[
\begin{gathered}
\text{$D_4\downarrow G_2$:}\\
\zeta_{2}=3,\quad \zeta_{4}=2,\quad \zeta_{6}=3,\quad \zeta_{7}=2
\end{gathered}
\]
\[
\begin{gathered}
\text{$E_6\downarrow F_4$:}\\
\zeta_{4}=4, \quad \zeta_{5}=3, \quad \zeta_{11}=22, \quad \zeta_{13}=22, \quad \zeta_{15}=4, \quad \zeta_{31}=2, \quad \zeta_{32}=2\\
\zeta_{34}=9
\end{gathered}
\]
\[
\begin{gathered}
\text{$F_4\downarrow B_4$:}\\
\zeta_{11}=6,\quad \zeta_{13}=3,\quad \zeta_{15}=10,\quad \zeta_{16}=5
\end{gathered}
\]
These vectors satisfy $(HB^\T)^\T\zeta=0$, with $\zeta>0$ and $\rank(HB^\T)=r$. 
\subsubsection*{Original mutation words.} The following finite lists suffice. The empty word is included. Compute the associated exchange Laurent polynomials along each word and its leading matrix $\mathsf L_w$ in the order $\op{lex}(B_p^{-\mathsf t}m)$. For each constructed normal $h_{f,k}$, at least one listed word satisfies $\mathsf L_wh_{f,k}=e_f$. These calculations take place in the original LP pattern and verify the criterion of Lemma~\ref{lem:normal} directly.
\noindent\textbf{$A_7\downarrow C_4$}.\quad 23 words:
{\small
\[
\begin{gathered}
\varnothing,\;(1),\;(3),\;(5),\;(6),\;(8),\;(1,2),\\[2pt]
(2,5),\;(3,4),\;(3,6),\;(5,7),\;(5,8),\;(6,7),\;(7,8),\\[2pt]
(2,5,7),\;(2,5,8),\;(3,4,5),\;(4,7,8),\;(5,7,8),\;(1,4,7,8),\;(2,4,5,7),\\[2pt]
(2,5,7,8),\;(2,4,5,7,8)
\end{gathered}
\]
}
\noindent\textbf{$A_9\downarrow C_5$}.\quad 66 words:
{\small
\[
\begin{gathered}
\varnothing,\;(3),\;(4),\;(7),\;(12),\;(13),\;(14),\;(15),\;(1,4),\;(2,3),\\[2pt]
(4,8),\;(4,13),\;(6,7),\;(8,13),\;(10,14),\;(11,12),\;(11,15),\;(12,14),\\[2pt]
(12,15),\;(13,14),\;(13,15),\;(1,2,3),\;(1,4,8),\;(1,4,13),\;(4,8,11),\\[2pt]
(4,8,13),\;(4,13,15),\;(5,6,7),\;(5,8,13),\;(8,13,14),\;(8,13,15),\;(9,10,14),\\[2pt]
(9,11,15),\;(10,12,14),\;(11,12,13),\;(12,11,15),\;(1,4,5,8),\;(1,4,8,11),\\[2pt]
(1,4,8,13),\;(1,4,13,15),\;(2,5,8,13),\;(4,5,6,7),\;(4,8,13,15),\;(5,8,13,14),\\[2pt]
(5,8,13,15),\;(6,9,11,15),\;(7,10,12,14),\;(8,9,10,14),\;(8,11,13,15),\\[2pt]
(1,4,5,8,11),\;(1,4,5,8,13),\;(1,4,8,13,15),\;(2,5,8,13,14),\;(2,5,8,13,15),\\[2pt]
(3,6,9,11,15),\;(4,8,11,13,15),\;(5,8,11,13,15),\;(1,4,5,8,9,11),\\[2pt]
(1,4,5,8,13,15),\;(1,4,8,11,13,15),\;(2,5,8,11,13,15),\;(5,8,9,11,13,15),\\[2pt]
(1,4,5,8,11,13,15),\;(2,5,8,9,11,13,15),\;(1,4,5,8,9,11,13,15),\\[2pt]
(2,5,6,8,9,11,13,15)
\end{gathered}
\]
}
\noindent\textbf{$D_4\downarrow G_2$}.\quad 8 words:
{\small \[ \varnothing,\;(1),\;(2),\;(3),\;(4),\;(1,4),\;(2,4),\;(1,3,4) \]}
\noindent\textbf{$E_6\downarrow F_4$}.\quad 28 words:
{\small
\[
\begin{gathered}
\varnothing,\;(2),\;(3),\;(5),\;(6),\;(7),\;(8),\;(1,3),\;(2,4),\;(2,5),\\[2pt]
(3,7),\;(4,7),\;(5,6),\;(6,7),\;(6,8),\;(1,2,3),\;(2,4,5),\;(3,6,7),\;(4,5,6),\\[2pt]
(5,6,7),\;(5,6,8),\;(1,4,5,6),\;(3,5,6,7),\;(4,5,6,8),\;(5,6,4,7),\\[2pt]
(1,4,5,6,8),\;(3,5,6,4,7),\;(3,1,5,6,4,7)
\end{gathered}
\]
}
\noindent\textbf{$F_4\downarrow B_4$}.\quad 9 words:
{\small \[ \varnothing,\;(1),\;(2),\;(3),\;(4),\;(1,2),\;(1,4),\;(2,4),\;(2,1,4) \]}

\subsection{Supplementary proofs and data}
\label{app:verification}\label{app:data}
The mathematical supplement \cite{ElectronicSupplement} is organized
by component identifiers; local section and equation numbers restart
within each component. Thus ``E1, Equation (2.5)'' refers to the
companion formulas in component E1.

E1 contains the explicit coefficient polynomials, differential
regularity calculations, and rational inverses for
$E_6\downarrow F_4$ and $F_4\downarrow B_4$. E2 and E3 give the
quotient constructions, all-weight normal-form proofs, and valued
presentation calculations for $B_3\downarrow G_2$ and
$G_2\downarrow A_2$, respectively. S1 gives the geometric triality
matrices, the exhaustive determinant calculation, and all sixteen
zero-slack replacements used in Proposition~\ref{prop:triality-saturation}.
This finite certificate is separate from the computer-free triality
seed and algebra-identification proofs in Section~\ref{sec:verification}.

Component A5 supplies the categorical data cited in
Sections~\ref{sec:e6-categorical-reconstruction} and
\ref{app:categorical-degrees}: $N_E$, its ordered Hom inputs,
and the degrees of all indecomposable presentations. The exceptional
inequality matrices $H_D,H_E,H_F$ remain in
Sections~\ref{subsec:d4-geometric}--\ref{subsec:f4-geometric}.
The inequalities use $Hm\ge0$, with coordinates specified there.

\section*{Declaration of generative AI and AI-assisted technologies}
Section~\ref{sec:mechanism} grew mainly out of prolonged mathematical
discussions with ChatGPT from June to August 2026. The author takes
full responsibility for the content of the manuscript.
The electronic supplement \cite{ElectronicSupplement} was prepared with the help of ChatGPT as well.

\end{document}


\maketitle\vspace{-1.1em}
This supplement contains the exceptional coefficient formulas and rational
inverses, the two hypersurface constructions, the $E_6$ categorical degree, and the finite triality
saturation certificate cited in the article. All branching degrees use
fundamental-weight coordinates, and frozen coefficients remain polynomial.
We write $P(G)$ for the weight lattice and $\rho_G$ for the half-sum
of positive roots; $\pi$ denotes restriction to the subgroup.
The component identifiers E1, E2, E3, A5, and S1 agree with the article.
Local section and equation numbers restart in each component; in particular,
``E1, Equation (2.5)'' refers to the eight exceptional companions.

\begingroup
\titlespacing*{\section}{0pt}{1.4ex}{.5ex}
\tableofcontents
\endgroup
\vspace{1.0em}
\startcomponent{E1}{Exceptional coefficients and inverses}{Exceptional coefficients and rational inverses}
\counterwithin{equation}{section}\counterwithin{theorem}{section}
Global regularity is tested by highest-weight differential equations;
the rational inverses below give full fraction-field generation.
All root actions and coefficient calculations use rational coefficients.

\section{Highest-weight equations and global regularity}
Let $G$ be simply connected and semisimple, with compatible maximal
unipotent subgroups $U_K\subset U=U_G$. Write
$A=\C[G]^{U_G^-\times U_K}$. The principal coefficients $p_i$ have degrees
$(\omega_i;\pi\omega_i)$ and restrict to the torus characters on the
Gauss decomposition. For $\lambda=\sum_i\lambda_i\omega_i$, set
$p^\lambda=\prod_i p_i^{\lambda_i}$. For $P=\prod_i p_i$,
\[
 A[P^{-1}]=\C[p_i^{\pm1}]\otimes\C[U/U_K].
\]
For a simple positive root vector $E_i$, put
\[
 D_iF(u)=\left.\frac{d}{ds}\right|_{s=0}F(\exp(sE_i)u).
\]
These operators preserve right $U_K$-invariance. Pullback of the left
action reverses brackets; this has no effect on the equations below.

\begin{lemma}[Highest-weight equations]\label{E1:lem:indicator}
Let $\lambda=\sum_i\lambda_i\omega_i$ be dominant and
$F\in\C[U/U_K]$. Then
\begin{equation}
 p^\lambda F\in A
 \quad\Longleftrightarrow\quad
 D_i^{\lambda_i+1}F=0\quad\hbox{for every }i.
 \label{E1:eq:indicator}
\end{equation}
If $p^\lambda F$ is regular and $\lambda_i>0$, then
\begin{equation}
 p_i\mid p^\lambda F\text{ in }A
 \quad\Longleftrightarrow\quad D_i^{\lambda_i}F=0.
 \label{E1:eq:divisibility}
\end{equation}
\end{lemma}
\begin{proof}
Pair $\C[U]$ with the enveloping algebra $\U(\mathfrak u)$ by
differentiation at the identity. The highest-coordinate covector
$\ell_\lambda$ of $V_G(\lambda)$ is cyclic as a right
$\U(\mathfrak u)$-module. The highest-weight presentation identifies it
with
\[
 \U(\mathfrak u)\big/
 \sum_i E_i^{\lambda_i+1}\U(\mathfrak u).
\]
Taking the dual identifies the corresponding matrix-coefficient space
on $U$ with the polynomials annihilated by all
$D_i^{\lambda_i+1}$. Indeed, vanishing of $D_i^{\lambda_i+1}F$ on $U$
is equivalent to annihilating
$E_i^{\lambda_i+1}\U(\mathfrak u)$ under the pairing.
These are exactly the restrictions of the left degree-$\lambda$
component of $\C[U_G^-\backslash G]$. Taking right $U_K$-invariants
proves \eqref{E1:eq:indicator}. To divide by $p_i$, apply the same result
to $\lambda-\omega_i$; the polynomial $F$ on the quotient is unchanged.
\end{proof}

If $F$ is a torus eigenfunction, its right weight is read directly from
the root coordinates. In particular, if the coordinate $t_j$ corresponds
to a root $\beta_j$, then
\[
 \wt(p^\lambda t_1^{a_1}\cdots t_d^{a_d})
   =\bigl(\lambda;\pi\lambda-\sum_j a_j\pi\beta_j\bigr).
\]
The finite homogeneous spaces in \eqref{E1:eq:indicator} can therefore be
computed by enumerating monomials of a fixed root weight and taking a
kernel over $\mathbb Q$.

We also use the following elementary consequence. The rings $A$ here
are factorial. If $F$ is irreducible in $\C[U/U_K]$, and
$D_i^{\lambda_i}F\ne0$ whenever $\lambda_i>0$, then $p^\lambda F$ is
prime in $A$. In fact, it is irreducible after inverting $P$, and
\eqref{E1:eq:divisibility} excludes every possible principal factor.
The principal coefficients themselves are prime by their fundamental
left degrees. All units in $A$ are constant.

\section{The \texorpdfstring{$E_6\downarrow F_4$}{E6 to F4} coefficients}
\subsection{The quotient and its differential operators}
Use the edges $(1,3),(3,4),(4,5),(5,6),(2,4)$ and the pinned involution
$\sigma=(1\,6)(3\,5)$. The following ordered roots are written in simple-root
coordinates:
\begin{equation}\label{E1:eq:roots}
\begin{array}{c|rrrrrr@{\qquad}c|rrrrrr}
j&\multicolumn{6}{c}{\beta_j}&j&\multicolumn{6}{c}{\beta_j}\\\hline
1&0&0&0&0&0&1&7&0&0&1&1&1&1\\
2&0&0&0&0&1&0&8&0&1&0&1&1&1\\
3&0&0&0&0&1&1&9&0&1&1&1&1&1\\
4&0&0&0&1&1&0&10&0&1&1&2&1&1\\
5&0&0&0&1&1&1&11&0&1&1&2&2&1\\
6&0&1&0&1&1&0&12&1&1&1&2&2&1
\end{array}
\end{equation}
They are one root from each nonfixed positive-root pair.
Fix the Chevalley generators by the minuscule weight basis of
$V_{E_6}(\omega_1)$: $E_iv_w=v_{w+\alpha_i}$ when
$\langle w,\alpha_i^\vee\rangle=-1$, and it is zero otherwise.
For each positive nonsimple root $\beta$, in increasing height order,
set $E_\beta=[E_i,E_{\beta-\alpha_i}]$ for the first $i=1,\ldots,6$
with a nonzero bracket. This specifies every sign used here
\cite[Definition~2.2]{Geck}.

The map
\[
 (t,v)\longmapsto
 u(t)v,\qquad
 u(t)=\prod_{j=1}^{12}\exp(t_jE_{\beta_j}),\quad v\in U_{F_4},
\]
is a polynomial isomorphism onto $U_{E_6}$.
To see this, order the Lie algebra by positive-root height. At every
height, the chosen root spaces complement the fixed subspace.
The linear terms of the Baker--Campbell--Hausdorff formula give this
direct sum, and its other terms use smaller heights. Solving successively
therefore gives a polynomial inverse. Hence
$U_{E_6}/U_{F_4}\simeq\mathbb A^{12}$ with coordinates $t_1,\ldots,t_{12}$.

With $\partial_j=\partial/\partial t_j$, left translation gives
\begin{equation}\label{E1:eq:e6-operators}
\begin{aligned}
D_1&=-\partial_1-t_2\partial_3-t_4\partial_5-t_6\partial_8
                              +t_{11}\partial_{12},\\
D_2&=t_4\partial_6+t_5\partial_8+t_7\partial_9,\\
D_3&=-\partial_2+t_5\partial_7+t_8\partial_9
 +(t_{10}+t_4t_8-t_5t_6)\partial_{11},\\
D_4&=t_2\partial_4+t_3\partial_5+t_9\partial_{10},\\
D_5&=\partial_2+t_1\partial_3,\qquad D_6=\partial_1.
\end{aligned}
\end{equation}
For completeness, these formulas can be derived without differentiating
any branching coefficient. Project
$\operatorname{Ad}(u(t)^{-1})E_i$ to
$\mathfrak u_{E_6}/\mathfrak u_{F_4}$ and express it in the columns
$u(t)^{-1}\partial_j u(t)$. This is a triangular linear system with ones on its diagonal,
computed in the specified integer minuscule matrices.

\subsection{The normalized polynomials}
For a function $f$ of left degree $\lambda$, write
$ f^{\circ}=p^{-\lambda}f$. Put
\begin{equation}\label{E1:eq:abbreviations}
\begin{aligned}
 F&=t_{10}t_2+t_{11}-t_4t_9+t_6t_7,&
 L&=t_{10}+t_4t_8-t_5t_6,\\
 V&=t_{11}t_8+t_{12}t_6,&
 H&=t_{12}-t_{10}t_3+t_9t_5-t_7t_8,\\
 Q&=t_2t_8-t_3t_6+t_9.
\end{aligned}
\end{equation}
The twelve nonprincipal functions are defined by the following table.
The degree column records $(\lambda;\mu)$; thus the global function is
$p^\lambda$ times the displayed polynomial.
\begin{equation}\label{E1:eq:e6-table}
\begin{array}{c|l|l}
f& f^{\circ}&(\lambda;\mu)\\\hline
u_{12}&-t_6&(\omega_2;\varpi_4)\\
u_4&-F&(\omega_5;\varpi_4)\\
u_9&t_{11}&(\omega_3;\varpi_4)\\
u_{21}&Ft_8+t_6H&(\omega_2+\omega_5;\varpi_3)\\
u_{22}&Q&(\omega_2+\omega_5;\varpi_2)\\
u_{23}&t_9L-V&(\omega_2+\omega_4;\varpi_2)\\
u_{24}&V&(\omega_2+\omega_3;\varpi_3)\\
u_{25}&t_9&(\omega_2+\omega_3;\varpi_2)\\
z_1&t_1F+H&(\omega_6;0)\\
z_2&t_{12}&(\omega_1;0)\\
z_8&t_7L-t_{11}t_5-t_{12}t_4&(\omega_4;\varpi_1)\\
z_{33}&H(t_2V+t_{11}t_9)+F(t_3V-t_{12}t_9)
 &(\omega_2+\omega_3+\omega_5;\varpi_2)
\end{array}
\end{equation}
The other six functions are
\[
 (p_1,p_2,p_3,p_4,p_5,p_6)=(z_{14},z_{13},z_{10},z_7,z_5,z_3),
 \qquad p_i^{\circ}=1,
\]
with degree $(\omega_i;\pi\omega_i)$, where
$\pi(a_1,\ldots,a_6)=(a_2,a_4,a_3+a_5,a_1+a_6)$.

\begin{proposition}\label{E1:prop:e6-functions}
All eighteen displayed functions are regular, prime, and span their
branching components. The eight exchange polynomials in the manuscript
have regular nonzero quotients by the corresponding mutable functions,
and these quotients are coprime to their parents.
\end{proposition}
\begin{proof}
Apply Lemma~\ref{E1:lem:indicator} with \eqref{E1:eq:e6-operators}.
Here is the complete finite kernel calculation. The second column counts
all monomials of the prescribed right weight, the third is the rank of
the differential equations, and the last gives a variable in which the
polynomial in \eqref{E1:eq:e6-table} is primitive and linear.
\[
\begin{array}{c|rr|c@{\qquad}c|rr|c}
f&\#&\mathrm{rank}&t_j&f&\#&\mathrm{rank}&t_j\\\hline
u_{12}&1&0&t_6&u_{24}&20&19&t_6\\
u_4&8&7&t_2&u_{25}&4&3&t_9\\
u_9&8&7&t_{11}&z_1&15&14&t_1\\
u_{21}&20&19&t_2&z_2&15&14&t_{12}\\
u_{22}&4&3&t_2&z_8&20&19&t_4\\
u_{23}&20&19&t_4&z_{33}&201&200&t_2
\end{array}
\]
For each principal function the count is one and the rank is zero.
More explicitly, if $C$ is the $F_4$ Cartan matrix with columns in
fundamental-weight coordinates, the monomials are exactly
\[
 \sum_j a_j\bar\beta_j=C^{-1}(\pi\lambda-\mu),\qquad a_j\ge0,
 \quad
 \bar\beta_j=((\beta_j)_2,(\beta_j)_4,
                    (\beta_j)_3+(\beta_j)_5,(\beta_j)_1+(\beta_j)_6).
\]
Substituting each monomial into the six equations in
\eqref{E1:eq:indicator} gives the stated ranks by rational elimination.
Its nonzero kernel vectors are the polynomials in
\eqref{E1:eq:e6-table}. Thus the components are one-dimensional,
over $\C$ as well as $\mathbb Q$.

In each indicated variable, the leading and constant coefficients have
gcd one. A primitive linear polynomial is irreducible over
$\C[t_1,\ldots,t_{12}]$. Direct differentiation also gives
$D_i^{\lambda_i} f^{\circ}\ne0$ for every $i$ with $\lambda_i>0$.
The consequence of \eqref{E1:eq:divisibility} noted above proves primality.

In mutable order $(12,4,9,21,22,23,24,25)$ the normalized companions are
\begin{equation}\label{E1:eq:e6-companions}
\begin{array}{c|l@{\qquad}c|l}
i& (u_i')^{\circ}&i& (u_i')^{\circ}\\\hline
12&t_{12}F-t_{11}H&22&t_{10}\\
4&t_6t_1-t_8&23&Qt_7-t_{11}t_3-t_{12}t_2\\
9&t_8&24&FQ- u_{21}^{\circ}t_2\\
21&-(t_2V+t_{11}t_9)&25&L
\end{array}
\end{equation}
The bar on a companion uses its own left degree. Polynomial substitution
proves all eight identities
\begin{align*}
u_{12}u_{12}'&=u_4u_{24}+u_9u_{21},&
u_4u_4'&=u_{21}z_3+u_{12}z_1z_5,\\
u_9u_9'&=u_{12}z_2z_{10}+u_{24}z_{14},&
u_{21}u_{21}'&=u_{12}z_5z_{33}+u_4u_{22}u_{24},\\
u_{22}u_{22}'&=u_{21}z_7+u_{23}z_5,&
u_{23}u_{23}'&=u_{22}u_{25}z_8+z_7z_{13}z_{33},\\
u_{24}u_{24}'&=u_9u_{21}u_{25}+u_{12}z_{10}z_{33},&
u_{25}u_{25}'&=u_{24}z_7+u_{23}z_{10}.
\end{align*}
Each polynomial in \eqref{E1:eq:e6-companions} satisfies
\eqref{E1:eq:indicator} at its companion degree and is nonzero.
In each case the companion degree minus the parent degree has a negative
Dynkin coordinate, so primality of the parent excludes a common factor.
\end{proof}

The binomials are irreducible in the polynomial coefficient ring.
For each ordered pair $i\ne j$, substituting $u_j=E_j/u_j$ into $E_i$
and taking a numerator relatively prime to its monomial denominator
gives a polynomial coprime to $E_j$. Thus the associated exchange
Laurent polynomials equal the ordinary ones. These are finite polynomial gcd tests with the displayed exchanges.
The exchange matrix has a unit minor on columns
$(u_4,u_{21},u_{22},u_{23},u_{24},u_{25},z_1,z_{33})$.
Together with the ten defining degrees in the manuscript, this gives
\[
0\longrightarrow\Z^8\xrightarrow{B_E^{\mathsf t}}\Z^{18}
\xrightarrow{W_E^{\mathsf t}}P(E_6)\oplus P(F_4)\longrightarrow0.
\]
The frozen degrees sum to $(2\rho_{E_6};2\rho_{F_4})$.

\subsection{The full rational inverse}\label{E1:subsec:e6-inverse}
Regard the twelve normalized values in \eqref{E1:eq:e6-table} as independent
coordinates. Set
\[
\begin{gathered}
 a=- u_{12}^{\circ},\quad b= u_9^{\circ},\quad c= u_{25}^{\circ},\quad
 d= z_2^{\circ},\quad f=- u_4^{\circ},\\
 g= u_{21}^{\circ},\quad q= u_{22}^{\circ},\quad r= u_{23}^{\circ},\quad
 v= u_{24}^{\circ},\qquad \ell=(r+v)/c.
\end{gathered}
\]
The inverse is the following ordered recovery:
\begin{equation}\label{E1:eq:e6-inverse}
\begin{aligned}
t_6&=a,&t_{11}&=b,&t_9&=c,&t_{12}&=d,\\
t_8&=\frac{v-da}{b},&
t_{10}&=\frac{g+r}{q},\\
t_2&=\frac{a z_{33}^{\circ}-gbc+fqv}{gv},&
t_3&=\frac{t_2t_8+c-q}{a},\\
t_4&=\frac{a z_8^{\circ}-\ell f+t_{10}(\ell t_2+b)}{r},&
t_5&=\frac{t_4t_8-\ell+t_{10}}{a},\\
t_7&=\frac{f-t_{10}t_2-b+t_4c}{a},&
t_1&=\frac{a z_1^{\circ}-g+ft_8}{af}.
\end{aligned}
\end{equation}
Every denominator is a monomial in the eight mutable coordinates.
Here is a direct elimination proof. From \eqref{E1:eq:abbreviations},
\[
 g+r=t_{10}q,\qquad
 a z_{33}^{\circ}=gv t_2+gbc-fqv,\qquad
 a z_8^{\circ}=r t_4+\ell f-t_{10}(\ell t_2+b).
\]
These determine $t_{10},t_2,t_4$. The remaining coordinates follow
from the definitions of $V,Q,F,L,H,z_1$.
Conversely, inserting \eqref{E1:eq:e6-inverse} in \eqref{E1:eq:e6-table}
recovers all twelve independent coordinates: the same three identities
recover $g, z_{33}^{\circ}, z_8^{\circ}$, and the defining equations recover the
others. 

It follows, without a Jacobian degree argument, that the seed generates
$\Frac\B(E_6,F_4)$. More precisely, with $M=\prod_i u_i$,
\[
 \B(E_6,F_4)[(PM)^{-1}]
 =\C[p_1^{\pm1},\ldots,p_6^{\pm1},u_i^{\pm1},z_1,z_2,z_8,z_{33}].
\]
These statements verify all hypotheses of the exceptional upper-algebra
comparison in the manuscript. It proves the asserted equality with the upper
LP algebra, keeping the frozen coefficients polynomial.

\section{The \texorpdfstring{$F_4\downarrow\Spin_9$}{F4 to Spin9} coefficients}
\subsection{The quotient action}
Use
\[
 \alpha_1=e_2-e_3,\quad\alpha_2=e_3-e_4,\quad\alpha_3=e_4,
 \quad\alpha_4=\tfrac12(e_1-e_2-e_3-e_4).
\]
The positive $B_4$ roots are the integral roots. The maximal parabolic
with Levi of type $B_3$ on vertices $1,2,3$ has nilradical
$\mathfrak n_1\oplus\mathfrak n_2$: $\mathfrak n_1$ has the eight roots
$\frac12(e_1\pm e_2\pm e_3\pm e_4)$ and $\mathfrak n_2$ is its
seven-dimensional centre. Thus
\[
 U_{F_4}=\exp(\mathfrak n_1\oplus\mathfrak n_2)\rtimes U_{B_3},
 \qquad U_{B_4}=\exp(\mathfrak n_2)\rtimes U_{B_3}.
\]
The quotient is the eight-dimensional vector space $\mathfrak n_1$.
Order its root coordinates $\xi_1,\ldots,\xi_8$ by
$+++,++-,+-+,+--,-++,-+-,--+,---$.
Choose their scalars so that
\begin{equation}\label{E1:eq:f4-operators}
\begin{aligned}
D_1&=\xi_5\partial_3-\xi_6\partial_4,&
D_2&=\xi_3\partial_2-\xi_7\partial_6,\\
D_3&=\xi_2\partial_1+\xi_4\partial_3-\xi_6\partial_5-\xi_8\partial_7,&
D_4&=\partial_8.
\end{aligned}
\end{equation}
Here $D_1,D_2,D_3$ are the $B_3$ spin action and $D_4$ is translation
in the simple-root direction.
These operators also give an independent finite description of the
quotient: they satisfy the $F_4$ Serre relations, and their root brackets
have eight independent constant translations, nine nonzero $B_3$ linear
fields, and seven zero central fields. The stabilizer at the origin is
exactly $\mathfrak u_{B_4}$. Their nilpotent affine action integrates to
$U_{F_4}$, with orbit all of $\mathbb A^8$, so its homogeneous space is
$U_{F_4}/U_{B_4}$. The root brackets and stabilizer are obtained from the four displayed operators.

\subsection{Global regularity and the seed}
Write $p_i$ for the principal coefficients. The polynomials are
\begin{equation}\label{E1:eq:f4-coefficients}
\begin{gathered}
 q_1=p_3\xi_1,\quad q_2=p_2\xi_2,\quad q_4=p_1\xi_4,
 \quad q_6=p_2\xi_6,\quad q_7=p_3\xi_7,\quad q_8=p_4\xi_8,\\
 r=p_2(\xi_1\xi_6+\xi_2\xi_5),\quad
 t=p_2(\xi_3\xi_6+\xi_4\xi_5),\\
 s=p_4(\xi_1\xi_8+\xi_2\xi_7+\xi_3\xi_6+\xi_4\xi_5),
 \qquad h=p_2s-p_4t.
\end{gathered}
\end{equation}
For each of the nine functions $p_jF$ in this display,
\[
 D_iF=0\ (i\ne j),\qquad D_j^2F=0,\qquad D_jF\ne0.
\]
Lemma~\ref{E1:lem:indicator} therefore proves global regularity, not merely
regularity on the Gauss chart. The right degrees follow from the root
weights. These functions have fundamental left degree and are prime.
Each homogeneous kernel is one-dimensional: $r,t,s$ have respectively
two, four, and four candidate monomials, and the other six have one each.
Thus these polynomials span the lines also described by the
Albert-algebra and Clifford decompositions in the manuscript.

Put
\begin{align*}
 w&=p_1p_2(\xi_2\xi_3-\xi_1\xi_4),&
 \ell&=p_1p_3\xi_3,\\
 v&=p_1p_4(\xi_3\xi_8+\xi_4\xi_7),&
 q_5&=p_3\xi_5,\\
 z&=p_1p_3(\xi_1\xi_4\xi_7-\xi_2\xi_3\xi_7
                         -\xi_3^2\xi_6-\xi_3\xi_4\xi_5).
\end{align*}
The same differential equations prove regularity of these functions at
their respective left degrees
$\omega_1+\omega_2,\omega_1+\omega_3,\omega_1+\omega_4,
\omega_3,\omega_1+\omega_3$.
The seed and companions are
\[
\begin{gathered}
 (x_1,x_2,x_3,x_4)=(p_2,-\ell,-v,-w),\\
 (c_1,\ldots,c_8)=(p_1,p_4,q_1,-q_4,q_7,q_8,r,t),
 \qquad (x_1',x_2',x_3',x_4')=(z,h,p_3,q_5).
\end{gathered}
\]
Direct polynomial substitution gives
\begin{align*}
 x_1x_1'&=c_8x_2+c_5x_4,&
 x_2x_2'&=c_3x_1x_3+c_2c_5x_4,\\
 x_3x_3'&=c_6x_2+c_2c_4c_5,&
 x_4x_4'&=c_7x_2+c_1c_3c_8.
\end{align*}
Primality of $w,\ell,v$ follows also from their degrees, as in the
manuscript: a factorization would require a factor of left degree
$\omega_1$, with right degree $\varpi_2$ or $\varpi_4$, neither of which
can occur in their right degrees. Each companion-minus-parent degree
has a negative left Dynkin coordinate, so the companions are coprime
to their parents. Polynomial gcd tests give $\widehat E_i=E_i$.

\subsection{The rational inverse}
Write $E_3=c_6x_2+c_2c_4c_5$ and recover
\[
 p_1=c_1,\quad p_2=x_1,\quad p_4=c_2,\quad p_3=E_3/x_3.
\]
Then set
\begin{align*}
 \xi_1&=c_3/p_3,&\xi_4&=-c_4/p_1,&
 \xi_7&=c_5/p_3,&\xi_8&=c_6/p_4,\\
 \xi_3&=-x_2/(p_1p_3),&&
 \xi_2=\frac{-x_4/(p_1p_2)+\xi_1\xi_4}{\xi_3}.
\end{align*}
With $\Delta=\xi_2\xi_3-\xi_1\xi_4$, finish with
\[
 \xi_5=\frac{\xi_3c_7-\xi_1c_8}{p_2\Delta},
 \qquad
 \xi_6=\frac{\xi_2c_8-\xi_4c_7}{p_2\Delta}.
\]
Substitution in both directions gives the identity on the twelve
independent coordinates. The functions therefore
generate the full fraction field and are algebraically independent.
Termwise weight calculation gives the frozen sum
$(2\rho_{F_4};2\rho_{B_4})$. The exchange and weight lattices are exact,
and the polynomial Jacobian is
\[
 \det\frac{\partial(\mathbf x,\mathbf c)}
               {\partial(\mathbf p,\boldsymbol\xi)}
   =(p_1p_2p_3p_4)x_1x_2x_3x_4.
\]
Together with global regularity and the prime and coprimality assertions,
these prove the hypotheses of the manuscript's exceptional comparison
proposition. It follows that
\[\B(F_4,\Spin_9)=\operatorname{UB}(\Sigma_F)=\U_{\mathrm{LP}}(\Sigma_F),\]
with polynomial frozen coefficients.

\par\bigskip\Needspace{12\baselineskip}
\startcomponent{E2}{The B3 to G2 hypersurface}{The hypersurface $B_3\downarrow G_2$}
\section{The branching algebra and its seed}
Let $G=\Spin_7(\C)$ and $K=G_2(\C)$, with compatible positive
unipotent groups. Put $\B=\C[G]^{U_G^-\times U_K}$, with
$\B_{\lambda,\mu}\simeq\Hom_K(V_K(\mu),V_G(\lambda))$.
Number $B_3$ with its third root short, and let $\varpi_1$ be the
seven-dimensional $G_2$ weight. The restriction convention of
\cite[Section~4]{HP} is
\[
 \pi\omega_1=\pi\omega_3=\varpi_1,\qquad \pi\omega_2=\varpi_2
\]
The functions used below have degrees
\[
\begin{array}{c|ccccccc}
 &p_1&p_2&p_3&s&h&q&r\\\hline
\lambda&\omega_1&\omega_2&\omega_3&\omega_3&\omega_1+\omega_2&\omega_2&\omega_1+\omega_3\\
\mu&\varpi_1&\varpi_2&\varpi_1&0&\varpi_2&\varpi_1&\varpi_2.
\end{array}
\]
\begin{proposition}\label{E2:thm:main}
There are functions of these degrees for which
\begin{equation}\label{E2:eq:ring}
 \B=\C[p_1,p_2,p_3,s,h,q,r]/(qr-p_1p_2s-p_3h).
\end{equation}
The seed $\Sigma=(q;p_1,p_2,p_3,s,h)$ has exchange
$E_q=p_1p_2s+p_3h$, mutation $q'=r$, and
$\mathcal A_{\mathrm{LP}}(\Sigma)=\mathcal U_{\mathrm{LP}}(\Sigma)=\B$
with polynomial frozen coefficients.
\end{proposition}
In the order $(q,p_1,p_2,p_3,s,h)$ the exchange vector is
$B=(0,1,1,-1,1,-1)$. The seed is of cluster type $A_1$, and
$\wt(p_1p_2p_3sh)=(2,2,2;2,2)$.
Its seven degrees agree with those in \cite[Section~5]{PR}.

\section{Quotient coordinates and regularity}
For Chevalley generators of $B_3$, take
\[
 E_\alpha=E_1+E_3,\quad F_\alpha=F_1+F_3,\quad H_\alpha=H_1+H_3,
 \qquad (E_\beta,F_\beta,H_\beta)=(E_2,F_2,H_2).
\]
They satisfy the $G_2$ relations, with $\alpha$ short. For the mixed
Serre relation, $\ad E_1$ and $\ad E_3$ commute and their second
and third powers, respectively, vanish on $E_2$; hence
$(\ad(E_1+E_3))^4E_2=0$. The negative relation is identical.
This gives the standard embedding and the stated restriction of weights.

On $\C[a,b,c]$ consider
\begin{equation}\label{E2:eq:derivations}
 D_1=-\partial_a,\qquad D_2=a\partial_b,\qquad
 D_3=\partial_a+b\partial_c.
\end{equation}
These satisfy the $B_3$ Serre relations and decrease the grading
$\deg(a,b,c)=(1,2,3)$, so they integrate to a unipotent action.
The brackets
$[D_1,D_2]=-\partial_b$ and $[[D_1,D_2],D_3]=-\partial_c$,
together with $D_1$, give all translations. The action is transitive.
The generators of $U_K$ act by $b\partial_c$ and $a\partial_b$,
which fix the origin. Its stabilizer has dimension $9-3=6=\dim U_K$;
connectedness of unipotent subgroups gives
$U_G/U_K\simeq\mathbb A^3$.
The identification is $T_K$-equivariant for coordinate weights
$-\alpha,-(\alpha+\beta),-(2\alpha+\beta)$.
The action on functions is inverse-left pullback; its overall sign
does not affect the kernels in E1, Lemma~\ref{E1:lem:indicator}.

For the principal functions $p_i$, the Gauss chart is
$\B[(p_1p_2p_3)^{-1}]=\C[p_1^{\pm1},p_2^{\pm1},p_3^{\pm1},a,b,c]$.
Define
\begin{equation}\label{E2:eq:functions}
 q=p_2b,\qquad r=p_1p_3a,\qquad s=p_3c,\qquad h=p_1p_2(ab-c).
\end{equation}
The maximal derivative orders of $b,a,c,ab-c$ are, respectively,
$(0,1,0)$, $(1,0,1)$, $(0,0,1)$, $(1,1,0)$.
E1, Lemma~\ref{E1:lem:indicator}, proves global regularity.
The coordinate weights give the degrees above, and multiplication
gives $qr=p_1p_2s+p_3h$.

\section{Generation and the upper algebra}
Put $d=ab-c$. The polynomials
\begin{equation}\label{E2:eq:normal}
 a^i b^j c^k d^\ell,\qquad i,j,k,\ell\ge0,\quad\min(i,j)=0,
\end{equation}
form a basis of $\C[a,b,c]$, by the relation $ab=c+d$ with leading
monomial $ab$.
\begin{lemma}\label{E2:lem:orders}
The derivative orders of \eqref{E2:eq:normal} are
$(i+\ell,j+\ell,i+k)$. In a finite nonzero sum of distinct normal
monomials, each order is the maximum of the corresponding orders.
\end{lemma}
\begin{proof}
In coordinates $(a,b,c)$, the leading terms for $D_1$ and $D_2$ come
from $a^{i+\ell}b^{j+\ell}c^k$. Their exponents recover the indices,
since $\ell=\min(i+\ell,j+\ell)$. For $D_3$, use $(a,b,d)$,
where $D_3=\partial_a$. The leading terms are
$a^{i+k}b^{j+k}d^\ell$, and $k=\min(i+k,j+k)$.
At each maximal order the leading coefficients are therefore independent.
\end{proof}
\begin{proof}[Proof of Proposition~\ref{E2:thm:main}]
By E1, Lemma~\ref{E1:lem:indicator}, the coefficient space of
$\lambda$ has basis \eqref{E2:eq:normal} subject to
$i+\ell\le\lambda_1$, $j+\ell\le\lambda_2$, $i+k\le\lambda_3$.
Every homogenized element is
\[
 p^\lambda a^ib^jc^kd^\ell
 =p_1^{\lambda_1-i-\ell}p_2^{\lambda_2-j-\ell}
  p_3^{\lambda_3-i-k}r^iq^js^kh^\ell.
\]
All exponents are nonnegative, proving generation. The monomials
avoiding $qr$ map to distinct normal basis elements in each left
degree, so \eqref{E2:eq:ring} is the sole relation. The six initial
functions are algebraically independent. The exchange is primitive
and linear in $h$, hence irreducible, and mutation exchanges $q,r$.
For $S=\C[p_1,p_2,p_3,s,h]$,
\begin{equation}\label{E2:eq:intersection}
 S[q^{\pm1}]\cap S[(E_q/q)^{\pm1}]=S[q,E_q/q].
\end{equation}
Indeed, membership of $\sum_j a_jq^j$ in the second ring forces
$E_q^j\mid a_{-j}$ for every $j>0$. This proves the upper-algebra
equality without inverting any frozen coefficient \cite{LP}.
\end{proof}

\section{The valued presentation category}
Use the hereditary species
\[
 Q:\quad\xymatrix@C=2.7em{1&2\ar[l]&3\ar[l]_{(2,1)}},
 \qquad k=\Q,\quad L=\Q(\sqrt2),\quad(k_1,k_2,k_3)=(L,L,k),
\]
with bimodules ${}_LL_L$ and ${}_LL_k$. The valuation gives dimensions
over the source and target fields. Set
$\mathscr C=\Ch_2(\proj\mathbb S_Q)$, with commutative-square
morphisms and the degreewise split exact structure \cite{Fei}.
Write $X_{i,0}=(P_i\to0)$ and $X_{i,t}=f(\tau^{t-1}J_i)$ for
$1\le t\le3$, where $f$ denotes a minimal presentation. Together
with $\Id_i=(P_i\xrightarrow1P_i)$ these are all fifteen indecomposables.
Their projective terms are
\begin{equation}\label{E2:eq:presentations}
\begin{array}{c|ccc}
 &i=1&i=2&i=3\\\hline
t=0&P_1\to0&P_2\to0&P_3\to0\\
t=1&P_2\to P_3^{\oplus2}&P_1\oplus P_2\to P_3^{\oplus2}&P_2\to P_3\\
t=2&P_1\to P_2&P_1\to P_3^{\oplus2}&P_1\to P_3\\
t=3&0\to P_1&0\to P_2&0\to P_3.
\end{array}
\end{equation}
With $u=\sqrt2$, the three maps into $P_3^{\oplus2}$, in the displayed order,
are $(1,u)^{\mathsf t}$,
$\left(\begin{smallmatrix}0&1\\1&u\end{smallmatrix}\right)$,
and $(1,u)^{\mathsf t}$. The other nonzero maps are the path maps $1$.
The projective and injective rank matrices are
\[
 R=\begin{pmatrix}1&1&1\\0&1&1\\0&0&1\end{pmatrix},\qquad
 J=\begin{pmatrix}1&0&0\\1&1&0\\2&2&1\end{pmatrix}.
\]
For $D=\operatorname{diag}(2,2,1)$ and $E=R^{-\mathsf t}D$,
$\Phi=-E^{-1}E^{\mathsf t}$ recovers the nine positive-root ranks
and the translation indexing.

The solid arrows are
\[
\begin{gathered}
 X_{1,t}\to X_{2,t}\to X_{3,t}\quad(0\le t\le3),\\
 X_{2,t+1}\to X_{1,t},\quad X_{3,t+1}\to X_{2,t}\quad(0\le t<3),\\
 X_{1,3}\to\Id_1\to X_{1,2},\quad
 X_{1,2}\to\Id_2\to X_{1,1},\quad
 X_{3,1}\to\Id_3\to X_{3,0}.
\end{gathered}
\]
The irreducible spaces have $k$-dimension two, except for the two
arrows through $\Id_3$, which have dimension one. The endomorphism
fields are $L$ on orbits 1 and 2 and $k$ on orbit 3, with the same
rule for $\Id_i$. These data are obtained by solving the
commutative-square equations for the displayed maps and quotienting
radicals by compositions. The additional arrows
$X_{i,t}\dashrightarrow X_{i,t+1}$, $0\le t<3$, are translations,
not irreducible morphisms.

\section{Degree fibres and admissibility}
For $P_+(X)=\bigoplus_j P_j^{m_j^+(X)}$, set
\begin{equation}\label{E2:eq:degree}
 \bwt(X_{i,t})=\left(\sum_jm_j^+(X_{i,t})\omega_j;\pi\omega_i\right),
 \qquad \bwt(\Id_i)=(\omega_i;0),
\end{equation}
and extend by direct sums. On a union of full degree fibres, an object
is admissible when its complete solid-and-translation neighbourhood
belongs to this union, and its two products are homogeneous of the same degree,
distinct, relatively prime, and omit its own fibre variable. The
exponents at a neighbour $Y$ are
$\dim_k\Irr(X,Y)/\dim_k\End(Y)$ or
$\dim_k\Irr(Y,X)/\dim_k\End(Y)$, and one for translations.

Among the twelve dominant, nonzero-left-degree objects,
$X_{1,2},X_{2,2},X_{3,2},\Id_1$ fail neighbourhood closure.
Degree balance leaves $X_{1,1},X_{2,1},X_{3,1}$ among the rest.
The products at $X_{1,1}$ share $p_1$; those at $X_{2,1}$ share
$p_1$ and $q$. Thus only $\sigma=X_{3,1}=(P_2\to P_3)$ is admissible.
The smallest union of full fibres containing it and its neighbours is
\[
\begin{array}{c|l@{\qquad}c|l}
 q&X_{1,1},X_{3,1}&p_1&X_{1,0},X_{1,2},X_{3,2}\\
 p_2&X_{2,0}&p_3&X_{3,0}\\
 s&\Id_3&h&X_{2,1}.
\end{array}
\]
The incoming neighbours of $\sigma$ are $X_{2,1},X_{3,0}$ and its
outgoing neighbours are $X_{2,0},\Id_3,X_{3,2}$. At the two
orbit-2 neighbours the exponent is $2/2=1$. Hence
\[
 M_\sigma^{\rm in}=p_3h,\qquad M_\sigma^{\rm out}=p_1p_2s.
\]
The projected star gives the exchange and its signed exponent vector.
All its irreducible neighbours belong to the subcategory, so passing to the full
additive branching subcategory preserves the irreducibles at $\sigma$.
This is an interpretation of the initial seed, not of arbitrary LP mutation.

\par\bigskip\Needspace{12\baselineskip}
\startcomponent{E3}{The G2 to A2 hypersurface}{The hypersurface $G_2\downarrow A_2$}
\section{The branching algebra and its seed}
Let $G$ have type $G_2$, with $\alpha_1$ short, and let
$K=\SL_3$ be its long-root subgroup. Choose simple roots
$\beta_1=\alpha_2$, $\beta_2=3\alpha_1+\alpha_2$.
Their coroots are $H_2,H_1+H_2$, so
$\pi\omega_1=\varpi_2$ and $\pi\omega_2=\varpi_1+\varpi_2$.
The functions below have degrees
\[
\begin{array}{c|cccccc}
 &p_1&p_2&q&r&s&t\\\hline
\lambda&\omega_1&\omega_2&\omega_1&\omega_2&\omega_1&\omega_2\\
\mu&\varpi_2&\varpi_1+\varpi_2&\varpi_1&\varpi_2&0&\varpi_1.
\end{array}
\]
These agree with the six degrees in \cite[Section~4]{PR}, up to
interchanging the two $A_2$ fundamental weights.
\begin{proposition}\label{E3:thm:algebra}
There are functions of these degrees such that
\begin{equation}\label{E3:eq:ring}
 \B(G_2,\SL_3)=\C[p_1,p_2,q,r,s,t]/(qr-p_1t-p_2s).
\end{equation}
The seed $\Sigma=(q;p_1,p_2,s,t)$ has exchange $E_q=p_1t+p_2s$,
mutation $q'=r$, and
$\mathcal A_{\mathrm{LP}}(\Sigma)=\mathcal U_{\mathrm{LP}}(\Sigma)=\B(G_2,\SL_3)$,
with polynomial frozen coefficients.
\end{proposition}
In the order $(q,p_1,p_2,s,t)$, its exchange vector is
$B=(0,1,-1,-1,1)$, and $\wt(p_1p_2st)=(2,2;2,2)$.
This is again a cluster seed of type $A_1$.

\section{Quotient coordinates and regularity}
On $\C[a,b,c]$ put
\begin{equation}\label{E3:eq:operators}
 D_1=\partial_a+b\partial_c,\qquad D_2=a\partial_b.
\end{equation}
For $C=[D_1,D_2]=\partial_b-a\partial_c$ and
$Z=[D_1,C]=-2\partial_c$, one has
$[D_1,Z]=[D_2,C]=0$, and $Z$ is central. In particular, the
$G_2$ Serre relations hold. The derivations decrease the grading
$\deg(a,b,c)=(1,2,3)$ and integrate to a unipotent action.
The determinant of $D_1,C,Z$ is $-2$, so every orbit is open and
the action is transitive. The three positive long-root generators
act by $D_2,0,0$ and fix the origin. Since its stabilizer is connected
of dimension $6-3=3$, it is $U_K$. Thus
$U_G/U_K\simeq\mathbb A^3$, equivariantly for the coordinate weights
\[
 \wt(a)=(1,-1),\qquad\wt(b)=(-1,0),\qquad\wt(c)=(0,-1)
\]
in the $A_2$ fundamental-weight coordinates. As in E2, the action on
functions is inverse-left pullback.

The Gauss chart is
$\B[(p_1p_2)^{-1}]=\C[p_1^{\pm1},p_2^{\pm1},a,b,c]$.
Define
\begin{equation}\label{E3:eq:functions}
 q=p_1a,\qquad r=p_2b,\qquad s=p_1c,\qquad t=p_2(ab-c).
\end{equation}
The derivative orders of $a,b,c,ab-c$ are
$(1,0),(0,1),(1,0),(0,1)$, respectively. E1,
Lemma~\ref{E1:lem:indicator}, proves global regularity.
The coordinate weights give the stated degrees, and the relation is
$qr=p_1t+p_2s$.

\section{Generation and the upper algebra}
Use the normal basis $a^ib^jc^k(ab-c)^\ell$, $\min(i,j)=0$,
from E2, Section~3. For \eqref{E3:eq:operators} its derivative orders
are $(i+k,j+\ell)$. The noncancellation proof is the same:
in coordinates $(a,b,ab-c)$ the first operator is $\partial_a$,
with leading terms $a^{i+k}b^{j+k}(ab-c)^\ell$;
for the second operator the leading terms are
$a^{i+\ell}b^{j+\ell}c^k$. The minima of the first two exponents
recover $k$ and $\ell$, respectively.
\begin{proof}[Proof of Proposition~\ref{E3:thm:algebra}]
For $\lambda=A\omega_1+B\omega_2$, the highest-weight criterion
selects exactly the normal monomials with
$i+k\le A$ and $j+\ell\le B$. They homogenize as
\[
 p_1^Ap_2^B a^ib^jc^k(ab-c)^\ell
 =p_1^{A-i-k}p_2^{B-j-\ell}q^ir^js^kt^\ell.
\]
Thus the six functions generate. The normal monomials avoiding $qr$
remain independent in each left degree, so \eqref{E3:eq:ring} is the
sole relation. The five initial functions recover
\[
 a=\frac q{p_1},\qquad
 b=\frac{p_1t+p_2s}{p_2q},\qquad c=\frac s{p_1},
\]
and are algebraically independent. The exchange is primitive and
linear in $t$, hence irreducible. Apply the two-chart intersection
argument of E2, \eqref{E2:eq:intersection}, with
$S=\C[p_1,p_2,s,t]$. It gives the upper-algebra equality.
\end{proof}

\section{The valued presentation category}
Use the species
\[
 Q:\quad\xymatrix@C=3em{1\ar[r]^{(3,1)}&2},\qquad
 k_1=k=\Q,\quad k_2=L=\Q(u),\quad u^3=2,
\]
with bimodule ${}_LL_k$. The Cartan matrix is
$\left(\begin{smallmatrix}2&-3\\-1&2\end{smallmatrix}\right)$,
and the symmetrizer is $\operatorname{diag}(1,3)$ \cite{GLS}.

Set $\mathscr C=\Ch_2(\proj\mathbb S_Q)$, with the morphisms
and exact structure of E2. Its ten indecomposables are
$X_{i,0}=(P_i\to0)$, $X_{i,j}=f(\tau^{j-1}J_i)$, $1\le j\le3$,
and $\Id_i$, for $i=1,2$. The module presentations are
\begin{equation}\label{E3:eq:presentations}
\begin{array}{c|cc}
 &i=1&i=2\\\hline
 j=1&P_2\xrightarrow{1}P_1&P_2^{\oplus2}\xrightarrow{M}P_1^{\oplus3}\\
 j=2&P_2\xrightarrow{(1,u)^{\mathsf t}}P_1^{\oplus2}&
       P_2\xrightarrow{(1,u,u^2)^{\mathsf t}}P_1^{\oplus3}\\
 j=3&0\to P_1&0\to P_2
\end{array},\qquad
 M=\begin{pmatrix}-u&-u^2\\1&0\\0&1\end{pmatrix}.
\end{equation}
Their cokernel ranks over $(k,L)$ are
\[
\begin{array}{c|cc}
 &i=1&i=2\\\hline
 j=1&(1,0)&(3,1)\\j=2&(2,1)&(3,2)\\j=3&(1,1)&(0,1).
\end{array}
\]
These are $\Phi^{j-1}J_i$ for
$J=\left(\begin{smallmatrix}1&3\\0&1\end{smallmatrix}\right)$ and
$\Phi=\left(\begin{smallmatrix}2&-3\\1&-1\end{smallmatrix}\right)$;
hence they identify the translation indexing.
The solid arrows are
\[
\begin{gathered}
 X_{2,j}\to X_{1,j}\quad(0\le j\le3),\qquad
 X_{1,j+1}\to X_{2,j}\quad(0\le j<3),\\
 X_{1,1}\to\Id_1\to X_{1,0},\qquad
 X_{2,3}\to\Id_2\to X_{2,2}.
\end{gathered}
\]
Their dimensions are three except for the two arrows through $\Id_1$,
which have dimension one. The endomorphism fields are $k$ on orbit 1
and $L$ on orbit 2, with the same rule for $\Id_i$.
Taking commutative-square Hom spaces of \eqref{E3:eq:presentations}
and quotienting by compositions gives these eleven irreducible
bimodules. Add the six translations $X_{i,j}\dashrightarrow X_{i,j+1}$.

\section{Degree fibres and admissibility}
Set $e(X_{i,j})=e_i$, $e(\Id_i)=0$. For $M=\operatorname{coker}X$,
write $m_1=\dim_kM_1$, $m_2=\dim_LM_2$ and $\kappa=e_1+2e_2$.
Define the degree, additively on direct sums, by
\begin{equation}\label{E3:eq:degree}
 \bwt(X)=\bigl(e_1\omega_1+e_2\omega_2;
 (m_1-\kappa)\varpi_1+(\kappa-m_1+m_2)\varpi_2\bigr).
\end{equation}
This is a specified degree for this orientation, distinct from the
positive-projective degree in E2. Its dominant, nonzero-left-degree
objects are the five singleton fibres
\[
 (q,p_1,p_2,s,t)\longleftrightarrow
 (X_{1,2},X_{1,3},X_{2,2},X_{1,1},X_{2,1}).
\]
Indeed, $X_{2,3}$ has right degree $-2\varpi_1+3\varpi_2$,
$X_{1,0}$ and $X_{2,0}$ have right degrees
$-\varpi_1+\varpi_2$ and $-2\varpi_1+2\varpi_2$, and the
neutral presentations have degree zero. The negative presentation
$X_{1,3}=(0\to P_1)$ belongs to the dominant set because of its degree.

Use the admissibility rule of E2, Section~5.
The objects $X_{1,1},X_{1,3},X_{2,1},X_{2,2}$ have outside neighbours
$\Id_1,X_{2,3},X_{2,0},\Id_2$, respectively. Only
$\sigma=X_{1,2}=(P_2\to P_1^{\oplus2})$ has a closed neighbourhood.
Its incoming neighbours are $X_{2,2},X_{1,1}$ and its outgoing
neighbours are $X_{2,1},X_{1,3}$. At both cross-orbit neighbours the
exponent is $3/3=1$; the other arrows are translations. Thus
\[
 M_\sigma^{\rm in}=p_2s,\qquad M_\sigma^{\rm out}=p_1t.
\]
The products have equal degree, are coprime and distinct, and omit $q$.
This proves uniqueness and recovers both the exchange and $B$.
The full branching subcategory contains every irreducible neighbour
of $\sigma$, so these irreducibles are unchanged on restriction.

\par\bigskip\Needspace{12\baselineskip}

\clearpage
\startcomponent{A5}{The E6 categorical degree}{The $E_6\downarrow F_4$ categorical degree}
\section{The degree and its integral lattice}
Use the orientation $1\to3\leftarrow4\to5\leftarrow6$, $4\to2$.
For its projective Cartan matrix $C$, the columns are the dimension
vectors of the projectives:
\[
 C=\begin{pmatrix}
1&0&0&0&0&0\\0&1&0&1&0&0\\1&0&1&1&0&0\\
0&0&0&1&0&0\\0&0&0&1&1&1\\0&0&0&0&0&1
\end{pmatrix},\qquad
\Phi=-C^{\mathsf t}C^{-1},\quad\Psi=-C^{-1}C^{\mathsf t}.
\]
Put $M_{i,t}=\tau^{t-1}J_i$ and $X_{i,t}=f(M_{i,t})$, $1\le t\le6$;
thus $\underline{\dim}M_{i,t}=\Phi^{t-1}C^{\mathsf t}e_i$.
The other indecomposables are $X_{i,0}=(P_i\to0)$ and
$\Id_i=(P_i\xrightarrow1P_i)$. The positive projective term is the domain.
Set $e(X_{i,t})=e_i$, $e(\Id_i)=0$, and use the ordered probes
\[
\begin{array}{lll}
Z_1=M_{1,4},&Z_2=M_{1,5}\oplus M_{1,6},&Z_3=M_{2,2},\\
Z_4=M_{2,3},&Z_5=M_{3,4},&Z_6=M_{3,5},\\
Z_7=M_{4,4},&Z_8=M_{4,5},&Z_9=M_{5,1},\\
Z_{10}=M_{5,2},&Z_{11}=M_{6,3},&Z_{12}=M_{6,4}.
\end{array}
\]
For $h_j(X)=\dim\Hom(\operatorname{coker}X,Z_j)$, the degree is
\[
 \bwt_E(X)=\tfrac13(e(X),p^+(X),p^-(X),h_1(X),\ldots,h_{12}(X))N_E,
\]
where Appendix A of this component displays the complete integer matrix.
Its column order is $(\omega_1,\ldots,\omega_6;\varpi_1,\ldots,\varpi_4)$.
The matrix and probes are specified grading data; no canonical derivation
of them from the subgroup inclusion is asserted.

Put $\chi=(1,0,-1,0,1,-1)$. Direct multiplication gives
\[
 \chi\Psi-\chi=(-3,0,0,0,0,3),\qquad
 N_E\equiv
 (\chi,-\chi,\chi,0^{12})^{\mathsf t}
 (0,1,0,0,-1,0;0,1,1,1)\pmod3.
\]
Since $p^+(X_{i,t})-p^-(X_{i,t})=\Psi^t e_i$, these congruences
prove integrality of the degree; neutral presentations have $e=p^+-p^-=0$.
All features are additive under direct sums. The lattice generated by
the forty-eight feature vectors has rank thirty and index three.

\section{Check V22: Hom profiles and degree fibres}
The following table gives all nontrivial feature data. A digit string lists
successive coordinates; for example $010010=(0,1,0,0,1,0)$.
The vector $h$ has twelve coordinates in the displayed probe order.
For $X_{i,0}$ the data are $(e,p^+,p^-,h)=(e_i,e_i,0,0)$;
for $\Id_i$ they are $(0,e_i,e_i,0)$.
\begin{center}\small
\setlength{\tabcolsep}{6pt}
\begin{longtable}{@{}lrrr@{}}
\toprule
Presentation & $p^+$ & $p^-$ & $h_1\cdots h_{12}$\\\midrule\endhead
$X_{1,1}$ & $001000$ & $100000$ & $000000000000$ \\
$X_{1,2}$ & $010010$ & $000100$ & $000000000000$ \\
$X_{1,3}$ & $001010$ & $000101$ & $000000001000$ \\
$X_{1,4}$ & $011000$ & $100100$ & $101000001100$ \\
$X_{1,5}$ & $000010$ & $000100$ & $010110100110$ \\
$X_{1,6}$ & $000000$ & $000001$ & $011001111001$ \\
$X_{2,1}$ & $001010$ & $000100$ & $000000000000$ \\
$X_{2,2}$ & $011010$ & $100101$ & $001000001000$ \\
$X_{2,3}$ & $011010$ & $000200$ & $000100001100$ \\
$X_{2,4}$ & $001010$ & $100101$ & $001000101110$ \\
$X_{2,5}$ & $010000$ & $000100$ & $101110111101$ \\
$X_{2,6}$ & $000000$ & $010000$ & $010111210110$ \\
$X_{3,1}$ & $011010$ & $100100$ & $000000000000$ \\
$X_{3,2}$ & $011020$ & $000201$ & $000000001000$ \\
$X_{3,3}$ & $012010$ & $100201$ & $001000002100$ \\
$X_{3,4}$ & $011010$ & $100200$ & $101110101210$ \\
$X_{3,5}$ & $000010$ & $000101$ & $011111211111$ \\
$X_{3,6}$ & $000000$ & $000010$ & $111111221101$ \\
$X_{4,1}$ & $011010$ & $000100$ & $000000000000$ \\
$X_{4,2}$ & $012020$ & $100201$ & $000000001000$ \\
$X_{4,3}$ & $022020$ & $100301$ & $001000002100$ \\
$X_{4,4}$ & $012020$ & $100301$ & $001100102210$ \\
$X_{4,5}$ & $011010$ & $100201$ & $102110212211$ \\
$X_{4,6}$ & $000000$ & $000100$ & $111221321211$ \\
$X_{5,1}$ & $011010$ & $000101$ & $000000001000$ \\
$X_{5,2}$ & $012010$ & $100200$ & $000000001100$ \\
$X_{5,3}$ & $011020$ & $100201$ & $001000001110$ \\
$X_{5,4}$ & $011010$ & $000201$ & $001100102101$ \\
$X_{5,5}$ & $001000$ & $100100$ & $101110211210$ \\
$X_{5,6}$ & $000000$ & $001000$ & $111121220111$ \\
$X_{6,1}$ & $000010$ & $000001$ & $000000000000$ \\
$X_{6,2}$ & $011000$ & $000100$ & $000000001000$ \\
$X_{6,3}$ & $001010$ & $100100$ & $000000000110$ \\
$X_{6,4}$ & $010010$ & $000101$ & $001000001001$ \\
$X_{6,5}$ & $001000$ & $000100$ & $000100101100$ \\
$X_{6,6}$ & $000000$ & $100000$ & $101010110110$ \\
\bottomrule
\end{longtable}\end{center}
These profiles are obtained by solving, for every arrow $a:u\to v$,
$F_vM_a=(Z_j)_aF_u$. The indecomposables can be constructed over $\mathbb Q$
in the thirty-six dimension vectors above; their one-dimensional
endomorphism spaces certify that they are the required real-root modules.
Projective covers give $p^-$, and $p^+=p^- - C^{-1}\underline{\dim}M$.
An independent calculation from commutative squares gives
\[
 \dim\Hom(\operatorname{coker}X,M)
 =\dim\Hom_{\Ch_2}(X,f(M))
       -(p^-(X))^{\mathsf t}C\,p^+(f(M)).
\]
Both methods give the table. Substitution in the degree formula gives
the eighteen full fibres printed in the article. The excluded objects have
the following degrees; the three neutral objects on the last line have
degree zero, not negative degree.
\begin{center}\small\setlength{\tabcolsep}{5pt}
\begin{tabular}{@{}lll@{}}\toprule
Object & Source Dynkin coordinates & Target Dynkin coordinates\\\midrule
$X_{2,1}$ & $(0,1,3,-1,0,0)$ & $(-1,1,1,0)$ \\
$X_{2,2}$ & $(0,1,3,-1,1,0)$ & $(-2,1,2,0)$ \\
$X_{2,6}$ & $(0,0,-1,0,-1,0)$ & $(1,0,-2,2)$ \\
$X_{3,6}$ & $(0,0,-2,2,-2,0)$ & $(1,0,-1,1)$ \\
$X_{4,6}$ & $(0,0,-2,1,-1,0)$ & $(1,0,-2,2)$ \\
$\Id_4$ & $(0,1,2,0,-1,0)$ & $(-2,1,1,0)$ \\
$X_{5,5}$ & $(0,1,-1,1,-2,0)$ & $(0,0,-1,2)$ \\
$X_{5,6}$ & $(0,1,-2,1,-1,0)$ & $(1,0,-2,2)$ \\
$X_{6,5}$ & $(0,0,0,0,-1,0)$ & $(0,0,-1,2)$ \\
$X_{6,6}$ & $(0,1,-1,1,-1,0)$ & $(0,0,-1,2)$ \\
$\Id_2,\Id_3,\Id_5$ & $(0,0,0,0,0,0)$ & $(0,0,0,0)$\\
\bottomrule\end{tabular}\end{center}

All $2304$ commutative-square Hom spaces and their radical-square
quotients give the $77$ solid arrows in the article; add the $36$
translation arrows separately. Of the $35$ dominant objects with nonzero
source degree, $12$ have a neighbour outside this set. Of the remaining
$23$, $15$ have unequal incoming and outgoing degrees. The eight
admissible objects, in the article's mutable order, are
\[
 X_{1,2},\ X_{6,1},\ X_{1,1},\ X_{3,2},\ X_{1,5},\ X_{5,3},\ X_{3,1},\ X_{2,4}.
\]
Their products give the eight exchanges. The ten other fibres have total
degree $(2,2,2,2,2,2;2,2,2,2)$. Both grading maps have only unit
nonzero Smith factors. These are finite calculations from $C$, the probes,
and $N_E$; the fibre table and marks are not inputs to their reconstruction.

\clearpage
\section*{Appendix A. The integer matrix $N_E$}
The row groups follow $e_1,\ldots,e_6$, $p^+_1,\ldots,p^+_6$,
$p^-_1,\ldots,p^-_6$, and $h_1,\ldots,h_{12}$, in that order.
The denominator in the degree is $3$.
\begin{center}\small\setlength{\tabcolsep}{4pt}\renewcommand{\arraystretch}{1.1}
\begin{tabular}{@{}r|rrrrrr|rrrr@{}}
\toprule
& $\omega_1$&$\omega_2$&$\omega_3$&$\omega_4$&$\omega_5$&$\omega_6$&$\varpi_1$&$\varpi_2$&$\varpi_3$&$\varpi_4$\\\midrule
$e_{1}$ & 0 & 1 & -6 & 3 & -4 & 0 & 6 & 1 & -11 & 7 \\
$e_{2}$ & 0 & 3 & -6 & 3 & -6 & 0 & 9 & 3 & -15 & 6 \\
$e_{3}$ & 0 & 2 & -12 & 6 & -8 & 0 & 12 & 2 & -19 & 8 \\
$e_{4}$ & 0 & 3 & -18 & 9 & -12 & 0 & 18 & 6 & -33 & 12 \\
$e_{5}$ & 0 & 4 & -12 & 6 & -10 & 0 & 12 & 4 & -23 & 10 \\
$e_{6}$ & 0 & 2 & -6 & 3 & -5 & 0 & 6 & 2 & -13 & 8 \\
\midrule
$p^+_{1}$ & 3 & -1 & 6 & -3 & 4 & 0 & -6 & -1 & 11 & -4 \\
$p^+_{2}$ & 0 & 0 & 6 & -3 & 6 & 0 & -6 & -3 & 15 & -6 \\
$p^+_{3}$ & 0 & -2 & 15 & -6 & 8 & 0 & -12 & -2 & 22 & -8 \\
$p^+_{4}$ & 0 & -3 & 18 & -6 & 12 & 0 & -18 & -3 & 33 & -12 \\
$p^+_{5}$ & 0 & -4 & 12 & -6 & 13 & 0 & -12 & -4 & 26 & -10 \\
$p^+_{6}$ & 0 & -2 & 6 & -3 & 5 & 3 & -6 & -2 & 13 & -5 \\
\midrule
$p^-_{1}$ & 0 & 1 & -6 & 3 & -4 & 0 & 6 & 1 & -11 & 4 \\
$p^-_{2}$ & 0 & 0 & -6 & 3 & -6 & 0 & 6 & 3 & -15 & 6 \\
$p^-_{3}$ & 0 & 2 & -15 & 6 & -8 & 0 & 12 & 2 & -22 & 8 \\
$p^-_{4}$ & 0 & 6 & -12 & 6 & -15 & 0 & 12 & 6 & -30 & 12 \\
$p^-_{5}$ & 0 & 4 & -12 & 6 & -13 & 0 & 12 & 4 & -26 & 10 \\
$p^-_{6}$ & 0 & 2 & -6 & 3 & -5 & 0 & 6 & 2 & -13 & 5 \\
\midrule
$h_{1}$ & 0 & 3 & 0 & 3 & 0 & 0 & 0 & 3 & 0 & -3 \\
$h_{2}$ & 0 & 3 & 0 & 3 & 0 & 0 & 0 & 6 & 0 & -6 \\
$h_{3}$ & 0 & 0 & 9 & -3 & 0 & 0 & -9 & 3 & 6 & 0 \\
$h_{4}$ & 0 & -3 & 6 & 0 & 0 & 0 & -6 & 0 & 6 & 0 \\
$h_{5}$ & 0 & 0 & 0 & -6 & 6 & 0 & 0 & -3 & 0 & 3 \\
$h_{6}$ & 0 & -3 & 0 & -3 & 3 & 0 & 0 & -3 & 0 & 3 \\
$h_{7}$ & 0 & 3 & 3 & -3 & 0 & 0 & -3 & 0 & 3 & 0 \\
$h_{8}$ & 0 & -3 & 0 & 3 & -3 & 0 & -3 & -3 & 6 & 0 \\
$h_{9}$ & 0 & -3 & -3 & 0 & 6 & 0 & 0 & -3 & 6 & -3 \\
$h_{10}$ & 0 & -3 & -3 & 0 & 3 & 0 & 3 & -3 & 3 & -3 \\
$h_{11}$ & 0 & 0 & 0 & 3 & 0 & 0 & 0 & 0 & 3 & -3 \\
$h_{12}$ & 0 & 0 & 3 & 0 & 3 & 0 & 3 & 0 & 3 & -6 \\
\bottomrule
\end{tabular}\end{center}
The machine-readable data include this matrix, all forty-eight feature
and degree vectors, rational representatives of the thirty-six modules,
and the complete local-neighbourhood data. The exact checker uses only
these specified categorical inputs and rational linear algebra.

\clearpage
\startcomponent{S1}{The triality saturation certificate}{The triality saturation certificate}
The finite calculation in this component is the certificate for the
triality saturation proposition. It uses the geometric seed and the
branching counting theorem of the article; it is not part of the
separate proof of the triality algebra identification.

\section{The matrices and the weight fibres}
Write $m=(z_1,z_2,z_3,z_4;c_1,\ldots,c_6)$ in the geometric seed
order. The fourteen slacks $s=H_Dm$ are
\[
\begin{aligned}
 (s_1,\ldots,s_8)&=(c_1,z_3+c_1,c_2,z_2+c_2,c_3,z_1+c_3,c_4,c_5),\\
 (s_9,\ldots,s_{14})&=(c_6,z_4+c_6,z_2+c_6,z_2+z_4+c_6,\\
 &\hspace{2.6em}z_1+z_4+c_6,z_1+z_3+z_4+c_6).
\end{aligned}
\]
The geometric exchange and weight matrices are
\[
 B_D^{\rm geom}=\begin{pmatrix}
 0&1&1&-1&0&0&-1&0&0&0\\
 1&0&0&1&0&-1&0&0&0&-1\\
 -1&0&0&0&-1&1&1&1&0&0\\
 1&0&0&0&0&0&0&0&1&-1
 \end{pmatrix},
\]
\[
 (W_D^{\rm geom})^{\mathsf t}=\begin{pmatrix}
0&0&0&0&1&0&0&1&0&0\\
1&0&1&1&0&1&0&0&0&1\\
1&1&0&0&0&0&1&0&0&1\\
0&1&0&1&0&0&0&0&1&1\\
0&0&1&0&1&0&1&0&0&0\\
1&1&0&1&0&1&0&0&0&1
 \end{pmatrix}.
\]
Here $B_D^{\rm geom}W_D^{\rm geom}=0$, and the minors of $B_D^{\rm geom}$ in columns $(1,2,4,5)$
and of $W_D^{\rm geom}$ in rows $(1,2,3,4,5,6)$ have absolute value one.
Consequently $(W_D^{\rm geom})^{\mathsf t}\Z^{10}=\Z^6$ and
$\ker_\R (W_D^{\rm geom})^{\mathsf t}=\operatorname{im}_\R (B_D^{\rm geom})^{\mathsf t}$.
Thus, after choosing an integral $m_0$ of weight $\gamma$, the fibre is
identified with
\[
 P(m_0)=\{a\in\R^4:H_Dm_0+Aa\ge0\},\qquad A=H_D(B_D^{\rm geom})^{\mathsf t}.
\]
Explicitly,
\[
\setlength{\arraycolsep}{3.3pt}
 A^{\mathsf t}=\begin{pmatrix}
 0&1&0&1&-1&-1&0&0&0&-1&1&0&-1&0\\
 0&0&-1&-1&0&1&0&0&-1&0&-1&0&1&1\\
 -1&-1&1&1&1&0&1&0&0&0&0&0&-1&-1\\
 0&0&0&0&0&1&0&1&-1&-1&-1&-1&0&0
 \end{pmatrix}.
\]
It has rank four and $A^{\mathsf t}\zeta=0$ for
$\zeta=(1,3,1,2,1,3,2,1,1,1,1,1,1,1)^{\mathsf t}>0$.
Hence $Aa\ge0$ implies $a=0$, so every nonempty $P(m_0)$ is bounded.

\section{The active-constraint certificate}
\begin{proposition}\label{S1:prop:integral}
For every integral $m_0$, every vertex of $P(m_0)$ is integral.
The triality branching semigroup is therefore saturated.
\end{proposition}
\begin{proof}
At a vertex, choose four independent active rows indexed by $I$.
They give $A_Ia=-(H_D)_Im_0$. A unit determinant makes $a$ integral.
Among the $\binom{14}{4}=1001$ four-row choices, the absolute
determinants are zero in 489 cases, one in 496 cases, and two in the
sixteen cases in Table~\ref{S1:tab:bases}. There are no other values.
This is the finite determinant calculation used in the proof.

The five identities
\begin{equation}\label{S1:eq:slacks}
\begin{gathered}
 s_1+s_{14}=s_2+s_{13},\qquad s_3+s_{11}=s_4+s_9,\qquad
 s_3+s_{12}=s_4+s_{10},\\
 s_5+s_{13}=s_6+s_{10},\qquad s_9+s_{12}=s_{10}+s_{11}
\end{gathered}
\end{equation}
hold for every $m$. All slacks are nonnegative, so zero on either
side forces zero for both summands on the other side.
For each exceptional $I$, these implications force the active set
$J$ in the table, and $\det A_J=\pm1$. For example,
$I=\{1,3,6,10\}$ forces $s_5=s_{13}=0$ by the fourth identity;
$J=\{1,3,5,6\}$ then has determinant one. This proves integrality
in every case.

Every nonempty integral-weight fibre has a vertex, so it contains an
integral point. If the weight $N\gamma$ occurs, divide a corresponding
lattice point by $N$ to obtain a real point of the fibre over $\gamma$.
The integral vertex in that fibre and the branching counting theorem
show that $\gamma$ occurs. This proves saturation in $\Z^6$.
\end{proof}

\begin{table}[!htbp]
\centering
\setlength{\tabcolsep}{11pt}
\begin{tabular}{@{}ccc@{}}\toprule
Active set $I$&$\det A_I$&Forced unit-determinant set $J$\\\midrule
$1,3,6,10$&$-2$&$1,3,5,6$\\
$1,3,10,11$&$-2$&$1,3,4,9$\\
$1,6,10,14$&$2$&$1,2,6,10$\\
$1,10,11,14$&$2$&$1,2,9,10$\\
$2,4,8,9$&$-2$&$2,3,4,8$\\
$2,4,9,12$&$-2$&$2,3,4,9$\\
$2,8,9,13$&$2$&$1,2,8,9$\\
$2,9,12,13$&$2$&$1,2,9,10$\\
$3,6,7,10$&$2$&$3,5,6,7$\\
$3,7,10,11$&$-2$&$3,4,7,9$\\
$4,5,8,9$&$-2$&$3,4,5,8$\\
$4,5,9,12$&$-2$&$3,4,5,9$\\
$5,8,9,13$&$-2$&$5,6,8,9$\\
$5,9,12,13$&$-2$&$5,6,9,10$\\
$6,7,10,14$&$2$&$5,6,7,10$\\
$7,10,11,14$&$-2$&$7,9,10,11$\\\bottomrule
\end{tabular}
\caption{All nonunimodular four-row bases of $A$. Indices refer to the
printed order of the fourteen slacks.}
\label{S1:tab:bases}
\end{table}

\clearpage
\def\componentheading{References}
\phantomsection